\documentclass[12pt,a4paper]{book}
\usepackage{amsthm,amsfonts,amsmath,amscd,amssymb,latexsym}    
\usepackage[english,ngerman]{babel}
\usepackage{hyperref}
\usepackage{mathrsfs}
\usepackage{epsfig,graphicx,color}     
\usepackage{color}   
\usepackage{nomencl,makeidx}
\usepackage[all]{xy}  
\usepackage[titles]{tocloft}  

\usepackage{amsmath,amsthm,amsfonts,amsbsy,amsgen,amscd,amstext,amssymb}
\usepackage{dsfont}
\usepackage{latexsym}
\usepackage{bbm}
\usepackage{mathrsfs}
\usepackage{eufrak}
\usepackage{appendix}
\usepackage[all]{xy}
\usepackage{fancyhdr}
\usepackage{graphicx, psfrag, epsfig}
\usepackage{bbm}
\usepackage[left=1.45in,right=1.45in]{geometry}

\usepackage[titles]{tocloft}
\date{}

\numberwithin{equation}{chapter}

\newtheorem{PARA}{}[section] 
\newtheorem{theorem}[PARA]{Theorem} 
\newtheorem{corollary}[PARA]{Corollary} 
\newtheorem{claim}[PARA]{Claim}
\newtheorem{lemma}[PARA]{Lemma} 
\newtheorem{proposition}[PARA]{Proposition} 
\newtheorem{definition}[PARA]{Definition}    
\newtheorem{remark}[PARA]{Remark}

\newcommand{\sB}{\mathscr{B}}

\newcommand{\sE}{\mathscr{E}}    
    
\newcommand{\sG}{\mathscr{G}}

\newcommand{\sJ}{\mathscr{J}}    
    
\newcommand{\sL}{\mathscr{L}}    
\newcommand{\sM}{\mathscr{M}}

\newcommand{\sP}{\mathscr{P}}    
    
\newcommand{\sR}{\mathscr{R}}    
\newcommand{\sS}{\mathscr{S}}    
    
\newcommand{\sU}{\mathscr{U}}    
\newcommand{\sV}{\mathscr{V}}    
\newcommand{\sW}{\mathscr{W}}

\newcommand{\cA}{\mathcal{A}}    
    
\newcommand{\cC}{\mathcal{C}}    
    
\newcommand{\cF}{\mathcal{F}}

\newcommand{\cH}{\mathcal{H}}

\newcommand{\cJ}{\mathcal{J}}    
    
\newcommand{\cL}{\mathcal{L}}    
\newcommand{\cM}{\mathcal{M}}    
    
\newcommand{\cP}{\mathcal{P}}

\newcommand{\cS}{\mathcal{S}}    
    
\newcommand{\cU}{\mathcal{U}}    
    
\newcommand{\cW}{\mathcal{W}}

\newcommand{\C}{{\mathbb{C}}}    
\newcommand{\D}{{\mathbb{D}}}

\newcommand{\N}{{\mathbb{N}}}    
 
\newcommand{\R}{{\mathbb{R}}}

\newcommand{\Z}{{\mathbb{Z}}}    

\newcommand{\CF}{{\mathrm{CF}}}    
\newcommand{\HF}{{\mathrm{HF}}}

\newcommand{\Dom}{{\mathrm{Dom}}}
\newcommand{\tOmega}{{\widetilde{\Omega}}}

\newcommand{\tA}{{\widetilde A}}    
\newcommand{\tB}{{\widetilde B}}

\newcommand{\tD}{{\widetilde D}}

\newcommand{\tF}{{\widetilde F}}  
\newcommand{\tG}{{\widetilde G}}  
\newcommand{\tH}{{\widetilde H}}

\newcommand{\tJ}{{\widetilde J}} 
\newcommand{\tq}{\tilde{q}}
\newcommand{\tK}{{\widetilde K}}    
\newcommand{\tL}{{\widetilde L}}    
\newcommand{\tM}{{\widetilde M}}    
\newcommand{\tN}{{\widetilde N}}

\newcommand{\tT}{{\widetilde T}}    
\newcommand{\tU}{{\widetilde U}} 

\newcommand{\tV}{{\widetilde V}}

\newcommand{\tX}{{\widetilde X}}  
\newcommand{\tPsi}{{\widetilde \Psi}}

\newcommand{\tf}{{\widetilde f}}

\newcommand{\te}{{\widetilde e}}
\newcommand{\tnabla}{{\widetilde \nabla}}
    
\newcommand{\ts}{{\widetilde s}}    
\newcommand{\tp}{{\widetilde p}}    
\newcommand{\tu}{{\widetilde u}}    
\newcommand{\tv}{{\widetilde v}}    
\newcommand{\tx}{{\widetilde x}}

\newcommand{\tPhi}{{\widetilde\Phi}}    
    
\newcommand{\tpsi}{{\widetilde\psi}}

\newcommand{\txi}{{\widetilde\xi}}    
\newcommand{\teta}{{\widetilde\eta}}
\newcommand{\tom}{{\widetilde\om}}

\newcommand{\hx}{{\hat{x}}}
\newcommand{\hy}{{\hat{y}}}
\newcommand{\std}{{\text{std}}}
\renewcommand{\i}{{\mathbf{i}}}

\newcommand{\dbar}{{\bar\partial}}    
\newcommand{\ti}{\;\;\makebox[0pt]{$\top$}\makebox[0pt]{$\cap$}\;\;}
\newcommand{\one}    
{{{\mathchoice \mathrm{ 1\mskip-4mu l} \mathrm{ 1\mskip-4mu l}    
\mathrm{ 1\mskip-4.5mu l} \mathrm{ 1\mskip-5mu l}}}}    

\def\slashii#1{\setbox0=\hbox{$#1$}             
\dimen0=\wd0                                 
\setbox1=\hbox{\sl/} \dimen1=\wd1            
\ifdim\dimen0>\dimen1                        
\rlap{\hbox to \dimen0{\hfil\sl/\hfil}}   
#1                                        
\else                                        
\rlap{\hbox to \dimen1{\hfil$#1$\hfil}}   
\hbox{\sl/}                               
\fi}                                         %
\def\slashiii#1{\setbox0=\hbox{$#1$}#1\hskip-\wd0\hbox to\wd0{\hss\sl/\/\hss}}
\newcommand{\dvol}{\mathrm{ dvol}}

\newcommand{\loc}{{\mathrm{loc}}}             
             
\newcommand{\ev}{{\mathrm{ev}}}

\newcommand{\id}{\mathrm{ id}}      
\newcommand{\Id}{\mathrm{ Id}}    
\newcommand{\im}{\mathrm{ im }}       
\renewcommand{\Im}{\mathrm{ Im\,}}

\newcommand{\Hreg}{\cH_{\mathrm{reg}}}   
\newcommand{\Jreg}{\cJ_{\mathrm{reg}}}   
\newcommand{\HJreg}{\mathcal{HJ}_{\mathrm{reg}}}   
     
\newcommand{\Freg}{{\mathscr{F}_{\mathrm{reg}}}}

\newcommand{\Aut}{\mathrm{ Aut}}          
\newcommand{\Diff}{\mathrm{ Diff}}        
\newcommand{\End}{\mathrm{ End}}          
\newcommand{\eps}{{\varepsilon}}    
\newcommand{\om}{{\omega}}    
\newcommand{\Om}{{\Omega}}    
\newcommand{\Cinf}{C^{\infty}}

\newcommand{\inner}[2]{\bigl\langle #1, #2\bigr\rangle}

\def\NABLA#1{{\mathop{\nabla\kern-.5ex\lower1ex\hbox{$#1$}}}}    
\def\Nabla#1{\nabla\kern-.5ex{}_{#1}}    
\def\Tabla#1{\Tilde\nabla\kern-.5ex{}_{#1}}    
\def\abs#1{\mathopen|#1\mathclose|}    
\def\Abs#1{\left|#1\right|}    
\def\norm#1{\mathopen\|#1\mathclose\|}    

\renewcommand{\Tilde}{\widetilde}

\newcommand{\p}{{\partial}}

\makeatletter
\newcommand{\clap}[1]{\hbox to 0pt{\hss #1\hss}}%
\newcommand{\ligne}[1]
  { \hbox to \hsize{\vbox{\centering #1}} }
\newcommand{\haut}[3]{%
  \hbox to \hsize{%
    \rlap{\vtop{\raggedright #1}}%
    \hss
    \clap{\vtop{\centering #2}}%
    \hss
    \llap{\vtop{\raggedleft #3}}}}%
\newcommand{\bas}[3]{%
  \hbox to \hsize{%
    \rlap{\vbox{\raggedright #1}}%
    \hss
    \clap{\vbox{\centering #2}}%
    \hss
    \llap{\vbox{\raggedleft #3}}}}%

\renewcommand{\maketitle}[1]{%
  \thispagestyle{empty}\vbox to \vsize{%
 \begin{center}
\normalfont\Large DISS. ETH NO. 22118
  \end{center}   
    \vfill
    \ligne{\LARGE \@title}
    \vspace{5mm}
    \ligne{\Large \@author}
    \vspace{1cm}
    \vfill
    \vfill
    \bas{}{\@date}{}
    }%
  \cleardoublepage
  }
  \renewcommand{\date}[1]{\def\@date{#1}}
  \renewcommand{\author}[1]{\def\@author{#1}}
  \renewcommand{\title}[1]{\def\@title{#1}}

\makeatletter
  \title{\vspace{10pt}
         \vspace{10pt}
         \normalfont PhD Thesis \\
         \vspace{10pt}
         \vspace{10pt}
         \textbf{A Hardy Space Approach to Lagrangian Floer gluing}}
  \author{ \vspace{10pt}
           \vspace{10pt}
            \vspace{10pt}
          \vspace{10pt}
          \vspace{10pt}
          \vspace{10pt}
          \vspace{10pt}
          \vspace{10pt}
          \vspace{10pt}
          Tatjana \textsc{Sim\v{c}evi\'c}\\
          \vspace{10pt}
          \vspace{10pt}
          ETH Zurich \\
          \vspace{10pt}
          \vspace{10pt}
          \vspace{10pt}
          \vspace{10pt}
          \vspace{10pt}}
  \date{October 2014}

\makeatletter \newcommand \listoftodos{\section*{Todo list} \@starttoc{tdo}}
  \newcommand\l@todo[2]
    {\par\noindent \textit{#2}, \parbox{10cm}{#1}\par} \makeatother

\makenomenclature

\begin{document}

 \pagestyle{empty} 
  \maketitle


%
%

\chapter*{Acknowledgement}

First and foremost I would like to express my deep gratitude 
to my advisor, Professor Dietmar Salamon, for sharing his 
enthusiasm and insight into the subject, as well as for 
his guidance and help during the whole doctorate. 

\medskip 

I would like to thank the co-examiners Professor Paul 
Biran and Professor Urs Frauenfelder for their 
interest in my work. 
\medskip 

I thank to my colleagues from the ETH for the enjoyable 
working atmosphere and pleasant moments spent together. 
This work was partially supported 
by the Swiss National Science Foundation Grant 
200021-127136 and I would like to thank them for their financial 
support. 

\medskip 

Above all I would like to thank to my parents Velibor 
and Stojka, my brother Dalibor and my boyfriend Sa\v{s}a 
for their encouragement, support and understanding. 


\chapter*{Abstract}

\indent We develop a new approach to Lagrangian-Floer 
gluing. The construction of the gluing map is based on 
the intersection theory in some Hilbert manifold of paths $\sP $. 
We consider some moduli spaces of perturbed holomorphic 
curves whose domains are either strips or more general 
Riemann surfaces with strip-like ends. These moduli spaces 
can be injectively immersed, by taking the restriction to 
non-Lagrangian boundary, into the Hilbert manifold $\sP$. 

\indent Then we restrict our attention to some subsets 
$ \sM^{\infty} ( \sU ) $ and $ \sM^T(\sU) $ of the 
aforementioned moduli spaces of perturbed holomorphic strips. 
These moduli spaces consist of small energy curves whose non-Lagrangian 
boundary is conatined in the neighbourhood $\sU$
of a fixed Hamiltonian path $x$. Monotonicity results 
will gurantee that the elements of these moduli spaces are contained 
in some neighbourhood of the Hamiltonian path $x$. This will allow us 
to carry over the main analysis in suitable local coordinate charts.

\indent The moduli spaces $ \sM^T(\sU) $ and $ \sM^{\infty} (\sU) $ 
turn out to be embedded submanifolds of the Hilbert manifold of paths $\sP$. 
We prove that $ \sM^T(\sU) $ converges in the $C^1 $ topology 
toward  $ \sM^{\infty} (\sU) $. As an application of this convergence 
result we prove various gluing theorems. 
We explain the construction of Lagrangian-Floer homology and prove 
that the square of the boundary map is equal to zero. Here 
we restrict our discussion to the monotone case with minimal 
Maslov number at least three. We also prove that the homology 
is independent of the Hamiltonian and almost complex 
structure used in its definition. We include the exposition 
of the Lagrangian-Floer-Donaldson functor and Seidel homomorphism.

 \selectlanguage{english}

\newpage
\setcounter{page}{-1}
\pagenumbering{roman}
\tableofcontents
 
\thispagestyle{empty}
\pagestyle{fancy} 
\fancyhf{}
\pagenumbering{arabic}

\renewcommand{\chaptermark}[1]{\markboth{\chaptername \ \thechapter.\ #1}{}} 
\renewcommand{\sectionmark}[1]{\markright{#1}}
\fancypagestyle{plain} { 
  \fancyhead{} 
  \renewcommand{\headrulewidth}{0 pt} 
}

\fancyhead[RO]{\bfseries \small \rightmark\qquad\thepage} 
\fancyhead[LE]{\bfseries \small \thepage\qquad\leftmark} 

\makeatletter
\def\cleardoublepage{\clearpage\if@twoside \ifodd\c@page\else
    \hbox{}
    \thispagestyle{plain}
    \newpage
    \if@twocolumn\hbox{}\newpage\fi\fi\fi}
\makeatother \clearpage{\pagestyle{plain}\cleardoublepage}

\chapter{Introduction}

\indent Floer homology was developed in 1980's by A.Floer 
in the series of papers \cite{F1,F2,F3,F4}
and today it has various application in low dimensional 
topology and symplectic and contact topology. Three of its main 
technical ingredients are Gromov-compactness, transversality-Fredholm theory
and gluing. In this thesis we shall not discuss compactness and transversality, 
but we shall develop a new approach to gluing. 

\indent When we say gluing, we always think of it as the opposite of 
Gromov-compactness. Particularly, this means that two holomorphic curves (strips), 
that intersect in the appropriate sense can be ``glued'' together 
if they can be approximated by some nearby genuine holomorphic \
curve. Floer's idea was first to use cut-off functions and to construct 
a so called pregluing curve, which in general doesn't have to be 
holomorphic and then using the implicit function theorem to 
establish the existence of an actual nearby holomorphic curve. 
Our approach is based on the intersection theory in Hilbert manifolds. 
More precisely, we consider two different moduli spaces of pairs of 
holomorphic curves whose domains are either half-infinite strips or more general 
 truncated Riemann surfaces with strip-like ends. These Hilbert manifolds 
 are infinite dimensional and they can be injectively immersed 
 into some Hilbert manifold of paths $ \sP $. 
 The images of these moduli spaces represent Hardy submanifolds of $\sP$, 
 mentioned in the title. The existence of a nearby holomorphic curve is provided 
by the existence of a unique intersection point of these 
two submanifolds. This approach is discussed in more details below. 

\indent In chapter \ref{ch:hardy_floer} we explain how the Hardy 
space gluing theory developed in this thesis is used in the construction 
of Floer homology. In this chapter the exposition is restricted to monotone 
 Lagrangian submanifold with minimal Maslov number at least three.
 The rest of the thesis doesn't assume any restrictions on the 
 Lagrangian submanifolds and could potentially be used in more 
 general cases. In \cite{OH} besides the same assumption on 
 the minimal Maslov number was also assumed that the image $ \Im(\pi_1(L_i)) \subset 
 \pi_1(M) $ is a torsion subgroup, what allowed them to work with 
 $ \Z_2 $ coefficients. We will not assume this condition and we don't impose 
 any restrictions on the monotonicity factors, 
 which forces us to work with Novikov rings. For the reader who is not 
 familiar with Lagrangian Floer homology
 we shortly describe its construction and then we describe how 
 the technical part of our new approach 
 to gluing is applied in the relevant construction. 
 
 \indent The Floer chain complex 
 is a vector space over some Novikov ring $ \Lambda $
 generated by Hamiltonian paths. 
 The boundary operator $ \partial $ is defined by counting (mod $2$) the number of 
 elements of some zero-dimensional moduli space. This moduli space consists 
 of perturbed holomorphic strips that connect two Hamiltonian paths. 
 As the first application of our gluing approach we prove that 
 the square of the boundary map is equal zero. This allows us to define Lagrangian 
 Floer homology $ HF(L_0,L_1;H,J) $. 
 
 \indent The proof of the identity $ \partial^2 = 0 $ is based on the study 
 of the $2-$dimensional moduli space of index two perturbed holomorphic 
 strips that connect two Hamiltonian paths $ x$ and $z$. 
 This moduli space $ \cM^2(x,z;H,J) $ allows free $ \R $ action by translation. 
 The quotient space $ \widehat{\cM}^2 = \cM^2(x,Z;H,J)/ \R$ will in general not be 
 compact, but it can be compactified by adding the zero dimensional product 
 space consisting of broken trajectories. 
 The aim of the gluing theorem is to identify an end of $ \widehat{\cM}^2 $ 
 with a broken trajectory. More precisely, the purpose of the Floer gluing theorem 
 is to construct a diffeomorphism of an interval $ (T_0, +\infty) $ and 
 some open subset of $ \widehat{\cM}^2  $, such that compactification of the interval, 
 i.e. adding infinity, corresponds to adding the broken trajectory. 
 As a broken trajectory we mean a pair $ (u,v) $, 
 where $ u $ is a trajectory connecting $ x$ and $y$, whereas $v$ 
 connects $y$ and $z$. Gluing map assigns to each $ T \in ( T_0, \infty) $ 
 a holomorphic curve $ u_T \in \cM^2(x,z;H,J) $ such that after shifting 
 by $ T $ on the left $ u_T $ converges to $u$, and on the right to $v$. 
 Construction of the curve $u_T$ is the moment when our new approach 
 provides an alternative method. 
 
 \indent  The curve $ u_T $ is obtained as the isolated intersection point 
 of two Hilbert manifolds $ \cM^T (y, \cU) $ and $ \cM^-(x) \times \cM^+ (z) $. 
 The manifold $ \cM^T (y, \cU) $ consists of perturbed holomorphic curves
 with small energy whose domain is a strip $ [-T,T] \times [0,1] $ 
 and which are the boundary $ \{\pm T\} \times [0,1] $ contained in the neighborhood $ \cU $ 
 of a Hamiltonian path $y$, whereas the other two boundary components $ [-T,T] \times \{0,1\} $\
 lie on Lagrangian submanifolds. The other moduli space $\cM^-(x) \times \cM^+ (z) $
 consists of pairs of half infinite strips with fixed energy
 converging to $x$ and $z$ respectively. 
 We prove that these moduli spaces are indeed infinite dimensional 
 submanifolds of some Hilbert manifold of $W^{2,2} $ strips. 
 We also consider another Hilbert manifold $ \cM^{\infty}(y, \cU) $ 
 consisting of pairs of half infinite holomorphic curves which have sufficiently small
 energy and are at the boundary contained in $ \cU$. The bound on the energy 
 and the neighborhood of $y $ - $\cU $ are chosen in such a way that the 
 elements of $ \cM^{\infty}(y,\cU) $ as well as $ \cM^T(y, \cU) $ 
 are confined to a small neighborhood of $y$. 
 In chapter \ref{chp:mon} we prove some monotonicity results 
 which guarantee that the non-Lagrangian (free) boundary and the energy 
 of a perturbed holomorphic curve control its diameter.  
 Thus, these monotonicity results imply that
 the elements of these moduli spaces are contained in 
 a small neighborhood of a Hamiltonian path $y$ and therefore
 the main analysis can be carried out in Euclidean space 
 using appropriate coordinate charts.
 
 We also prove that the manifolds $  \cM^T(y, \cU) $
 and $ \cM^{\infty}(y,\cU) $ can be embedded by taking the restriction to the boundary 
 into some Hilbert manifold of paths $ \sP \times \sP $. Their images $ \cW^T $ and $ \cW^{\infty} $ 
 are the nonlinear Hardy spaces of the title, they are exactly those paths that can be extended 
 holomorphically to the corresponding strips. One of the most important result of the thesis is 
 that $ \cW^T $ converge to  $\cW^{\infty} $ in $C^1 $ topology. If we now 
 go back to the construction of the map $ u_T $ in the neighborhood of 
 the broken trajectory $ (u,v) $, we notice that a pair of paths $ (u(0),v(0))\in \sP\times \sP $
 is an isolated intersection point of $ \cW^{\infty} $ and $ \cM^-(x) \times \cM^+ (z) $. 
 As $ \cW^T $ converge towards $\cW^{\infty} $ this implies that for $ T $ sufficiently 
 large there will be a unique isolated intersection point of $ \cW^{T} $ and $ \cM^-(x) \times \cM^+ (z) $, 
 which we denote by $u_T$.
 
 \indent Another application of our gluing approach is 
 in the construction of the Floer connection homomorphism. Namely, for 
 two regular pairs $ (H^{\alpha}, J^{\alpha} ) $ 
 and $ (H^{\beta}, J^{\beta} ) $ we can consider homotopy connecting these 
 two pairs and the corresponding time dependent Floer equation. 
 The mod two count of solutions of such an equation defines homomorphism 
 of the corresponding chain complexes $ \Phi^{\beta \alpha}. $ We prove 
 that this homomorphism intertwines the Floer boundary 
 operator and hence defines a homomorphism of the corresponding 
 homology groups. The proof is again based on the study 
 of some $1-$dimensional moduli space and the identification 
 of its ends with the images of some gluing maps. 
 
 \indent The homomorphisms $ \Phi^{\beta \alpha} $ 
 are independent of the choice of the homotopy connecting the 
 regular pairs $ (H^{\alpha}, J^{\alpha} ) $ 
 and $ (H^{\beta}, J^{\beta} ) $. Thus two homomorphisms 
 $ \Phi^{\beta \alpha}_1 $ and $ \Phi^{\beta \alpha}_0$ 
 defined using two different homotopies are chain homotopy equivalent, 
 and hence induce the same map on the homology level. 
 
 \indent Finally we prove that two composable morphisms satisfy the composition 
 rule under the catenation of homotopies. More precisely three regular pairs, 
 $ (H^{\alpha}, J^{\alpha} ) $, $ (H^{\beta}, J^{\beta} ) $ and $ (H^{\gamma}, J^{\gamma} ) $
 and the corresponding homomorphisms $ \Phi^{\beta \alpha} $ and $ \Phi^{\gamma \beta } $ 
 satisfy the composition rule $  \Phi^{\gamma \beta } \circ \Phi^{\beta \alpha} =
 \Phi^{\gamma \alpha } $ and $ \Phi^{\alpha \alpha} = \Id $. Thus, this implies that 
 the Lagrangian Floer homology doesn't depend on the choice of the Hamiltonian and 
 the almost complex structure. This is where the presence of Novikov ring 
 makes the anlysis much more difficult and the proof requires 
 the construction of infinitely many morphisms $ \Phi_{\nu}^{ \gamma\alpha} $ 
 such that the composition  $  \Phi^{\gamma \beta } \circ \Phi^{\beta \alpha}$ 
 corresponds to the limit $ \lim\limits_{\nu\rightarrow \infty} \Phi^{\gamma \alpha}_{\nu} $
 and $\Phi^{\gamma \alpha } = \Phi^{\gamma \alpha }_0 $.
 Each two consecutive morphisms $ \Phi_{\nu}^{ \gamma \alpha}$ 
 are chain homotopic equivalent. 
 
 \indent The above properties of the homomorphisms $ \Phi^{\beta \alpha} $ 
 and the corresponding gluing theorems are done not just in the 
 case of strips, but also in the case of more general Riemann surfaces 
 with finitely many half-infinite strip-like ends. These results give rise 
 to a Lagrangian Floer-Donaldson functor from the category of 
 Lagrangian pairs to the category of vector spaces over the 
 Novikov ring $ \Lambda$ with $ \Z / 2\Z $ coefficients. 
 The objects are finite tuples of pairs of monotone Lagrangian submanifolds and the 
 morphisms, called \emph{string cobordisms}, are 2-manifolds with boundary and 
 finitely many strip-like ends and a monotone Lagrangian submanifold for each boundary 
 component. The Lagrangian Floer-Donaldson functor assigns to each 
 tuple of Lagrangian pairs the tensor product of its Floer homology groups 
 and to each string cobordism a morphism on Floer homology. 
 The last chapter of the thesis also includes the exposition of 
 the Lagrangian Seidel homomorophism. 
 
 \indent The thesis is organized as follows. In the second chapter 
 we prove the monotonicity lemma for holomorphic curves with Lagrangian 
 boundary conditions. As a corollary we prove some results which 
 guarantee that holomorphic strips with small energy and some 
 condition on the non-Lagrangian boundary are localized near 
 Lagrangian intersection point and hence can be studied in 
  suitable local coordinates. We also prove exponential decay 
 of finite energy holomorphic strips with Lagranian boundries 
 in the case of tame almost complex structure and clean intersection 
 of Lagrangian submanifolds. \\
 \indent In the third chapter we prove some linear elliptic estimates which will
 be crucial for the proofs of the main theorems in the fourth chapter. 
 We also prove elliptic regularity and surjectivity of some 
 specific linearized operator whose domain represent $W^{2,2} $ maps 
 on strips. Additional difficulties that don't occur in standard 
 Floer theory appear because of the presence of truncated surfaces, i.e. domains 
 with corners. The corresponding linearized operator will not be Fredholm 
 in this case, but it will still be surjective. 
 In the appendix of this chapter we discuss the abstract 
 interpolation theory, as it has various applications through 
 the whole thesis. \\
 \indent The fourth chapter contains the exposition of the
 Hilbert manifold setup necessary for the statement of the main theorems. 
 In the appendix of this chapter we discuss further 
 the interpolation Lions-Magenes \cite{LM} spaces which represent 
 the model space of the aforementioned Hilbert manifold of paths $ \sP $. 
 This chapter contains the statement and the proofs of 
 the main theorems. We prove that $ \cW^T $ converges 
 in $C^1 $ topology to $ \cW^{\infty} $ and we prove that 
 they are embedded submanifolds of the Hilbert path space $ \sP\times \sP$. 
 We also prove that $\sM^{\pm} (x) $, consisting of perturbed 
 holomorphic strips that converge to $x$, is a Hilbert submanifold 
 of some Hilbert manifold of strips $\sB$ and that it can be injectively 
 immersed into the path manifold $\sP$. \\ 
 \indent The last chapter is joint work with Prof. D. Salamon 
 and it contains various applications of the results from the 
 third chapter. We explain how the Floer gluing theorems can be reduced to 
 intersection theory in the path space $\sP $. We prove that the 
 square of the boundary map is equal zero and we prove the 
 aforementioned properties of the homomorphism $ \Phi^{\beta \alpha} $. 
 We include the exposition of the Floer-Donaldson category.

 \chapter{Monotonicity for holomorphic curves}\label{chp:mon}

\section{Main results}
It is well known that the monotonicity lemma holds 
for minimal surfaces $ u : \Omega \rightarrow M $.
\emph{Monotonicity} means that the area 
of the piece of the surface $u(\Omega)$, cut from $ u(\Omega) $
by a small ball of radius $r$ and centered on the surface,
is bounded bellow by $c \cdot r^2$. 
It is well known that this holds for holomorphic curves 
in the case that the boundary $ u(\partial \Omega ) $ isn't 
contained in the ball $ B_r(u(z))$. 
We shall prove that monotonicity holds also if 
 $$ 
 u( \partial \Omega ) \cap B_r( u(z_0))\neq \emptyset 
 $$
It suffices that the piece of a boundary within 
the ball lies on Lagrangian submanifolds which intersect 
cleanly. We illustrate this phenomenon in Figure \ref{fig_1}. \\

\begin{figure}[htp]
  \centering
 \scalebox{0.4}{\input{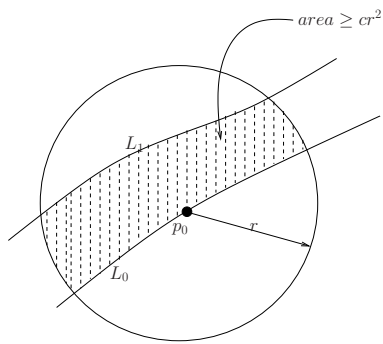}}
 \caption{Monotonicity for curves with Lagrangian boundary}\label{fig_1}
\end{figure}
Throughout this chapter we shall assume  the following:
\begin{itemize}
 \item[ (H)]  $ (M, \omega ) $ is a symplectic manifold 
without boundary and $ L_0 , L_1 \subset M $ are Lagrangian submanifolds 
without boundary that are closed subsets of $M$. The intersection 
$ \Lambda = L_0 \cap L_1 $ is compact and $N \subset M$ is a compact 
neighborhood of $ \Lambda $. 
\end{itemize}

The main theorem of this chapter is the following 
monotonicity theorem for holomorphic curves with 
Lagrangian boundary conditions. 

\begin{theorem}[Monotonicity lemma] \label{main_thm}
 Assume ~(H) and let $ J $ be an 
 $\omega -$tame almost complex structure. Suppose 
 that the intersection $ \Lambda = L_0 \cap L_1 $ is clean. 
 Equip $ M $ with the 
metric $ \langle \xi, \eta \rangle = \frac{1}{2} ( \omega (\xi, J\eta) + \omega (\eta,J \xi)). $ 
There exist positive constants $c_0$ and $r_0$ such that the following holds.
Let $ \Omega \subset \R \times [0,1] $ be a bounded open set, $ (s_0, t_0 ) \in \Omega $, 
$ 0 < r < r_0 $ . If a $ J-$holomorphic curve 
$ u : \Omega \rightarrow N $ which extends continuously 
to $ \overline{\Omega} $ satisfies the following
\begin{equation}\label{eq:lagbc}
\begin{split}
 u(s,0 ) \in L_0 , \text{ for all } s\in \R \text{ with } (s, 0 ) \in \Omega , \\
 u(s,1) \in L_1 ,   \text{ for all } s \in \R \text{ with } (s,1) \in \Omega
\end{split}
\end{equation}
 and 
 \begin{equation}\label{eq:mon_cond}
 u ( \overline{\Omega} \setminus \Omega ) \cap \overline{B_r(p_0) } = \emptyset 
 ,
 \end{equation}
 where $ p_0 = u(s_0, t_0) $. Then 
\begin{align}
  \displaystyle\int\limits_{u^{-1} ( B_r(p_0))} u^* \omega  \geq c_0 r^2. 
\end{align}
\end{theorem}
\begin{proof}
 See Section \ref{sec:proofs}. 
\end{proof}

The following standard monotonicity lemma 
(see for example \cite{Hu}) 
for holomorphic curves without 
boundary is a corollary of Theorem \ref{main_thm}. 

\begin{corollary}
 Let $ (M , \omega)$  be a symplectic manifold and $ N \subset M $ compact.
 Let $J$  be an $ \omega -$tame almost complex structure and
 equip $M $ with the metric 
 $ g = \frac{1}{2} ( \omega ( \cdot, J \cdot) + \omega ( J \cdot, \cdot) ) $. 
 Then there are constants $\epsilon_0, c_0 >0$ such that the following holds. 
 If $ \Omega \subset \C $ is an open and bounded set, $ (s_0, t_0) \in \Omega $, $ 0 < r < \epsilon_0 $ and
 \begin{itemize}
  \item $u: \Omega \rightarrow N $ is a smooth $J-$ holomorphic curve. 
  \item $ u : \overline{\Omega} \rightarrow N $ is continuous.
  \item $ u ( \overline{\Omega} \setminus \Omega)\cap \overline{B_r( u(s_0,t_0))} = \emptyset $ for $ r \leq \epsilon_0 $ 
 \end{itemize}
 Then 
 \begin{align}
  \int\limits_{u^{-1} ( B_r( u( s_0,t_0))} u^* \omega \geq c_0 r^2 
 \end{align}
\end{corollary}
\begin{proof}
After rescaling we may assume $\Omega \subset \R \times [0,1] $. Then 
 the boundary conditions are vacuous, so the assertion follows from Theorem 
 \ref{main_thm}. 
\end{proof}

\subsection*{Extension to the case of $t-$dependent
almost complex structures}
Let $ J_t $ be a smooth family of $ \omega -$tame 
almost complex structures and let $ I \subset \R $ be an interval. 
We consider $ J_t- $ holomorphic curves 
with Lagrangian boundary conditions, 
i.e. solutions $ u : I \times [0,1] \rightarrow N \subset  M $
of the following 
boundary value problem 
\begin{align}\label{eq:jhol}
 \partial_s u + J_t(u) \partial_t u = 0, \;\; u(s,i) \in L_i, i = 0,1.
\end{align}
Such holomorphic curves were studied by Floer \cite{F1, F2,F3}
and were used to define the Floer homology of Lagrangian intersection. 
The need of introducing $t-$dependent almost complex structure comes 
from tansversality issues. See for example  \cite{FHS}, 
where the construction leading to transversality involved such 
almost complex structures. 
Denote with $ E(u) $ the energy of such a holomorphic curve 
\begin{align}
 E(u) = \int\limits_{I \times [0,1] } u^* \omega = \int\limits_{I \times [0,1] } \abs{\partial_s u }_{J_t}^2 ds dt.
\end{align}

A set $ \Lambda_0 \subset L_0 \cap L_1 $ is called 
an {\bf isolated set} of intersections if there 
is an open neighborhood $ V \subset M $ of $ \Lambda_0 $ with 
compact closure such that $ \overline{V} \cap L_0 \cap L_1= \Lambda_0 $. 
In particular $ \Lambda_0 $ is compact. Any such $ V $ is called an 
{\bf isolating neighborhood } of $ \Lambda_0 $.  
Our second main result states that the energy of $ u $ 
and the non-Lagrangian boundary 
$\left. u\right|_{\partial I \times [0,1] } $
control the diameter of $ u$. 

\begin{theorem}\label{thm:mon1}
Assume ~(H). Suppose that $ \Lambda_0 \subset L_0 \cap L_1 $ 
is an isolated set of intersections. 
Let $ U $ be an open neighborhood of $\Lambda_0$  and 
let $ V $ be an isolated neighborhood of $ \Lambda_0 $ 
such that $ \overline{V} \subset U $. 
There exists $\hbar $ such that the following holds.
For any interval $ I=[a,b] \subset \R $  if a solution 
$ u : I \times [0,1] \rightarrow N \subset M $ of \eqref{eq:jhol}, satisfies  
$$
E(u) < \hbar , \;\;  and \; \; u|_{\partial I \times[0,1] } \in V ,
$$ then 
$$ 
u(s,t)\; \in U , \;\;\forall (s,t) \in I \times [0,1].
$$
\end{theorem}
\begin{proof}
 See Section \ref{sec:proofs}. 
\end{proof}

As a corollary of Theorem \ref{thm:mon1} we prove 
an analogous result for perturbed holomorphic strips. 
Let 
$$
[0,1]\times M\to\R:(t,p)\mapsto H(t,p)=H_t(p)
$$ 
be a time dependent Hamiltonian function. 
Denote by $\phi_t:\Om_t\to M$ the Hamiltonian
isotopy generated by $H$ via
$$
\p_t \phi_t(p) = X_{H_t} ( \phi_t(x)), \;\; \phi_0 = \Id
$$
Thus $\Om_t\subset M$ is the open set of 
all points $p_0\in M$ such that the solution 
of the initial value problem 
$\dot x(s)=X_{H_s}(x(s))$, $x(0)=p_0$,
exists on the interval $[0,t]$.

\begin{theorem}\label{thm:MON}
Assume~(H) and let $J=\{J_t\}_{0\le t\le1}\in\cJ(M,\om)$.
Suppose $\Lambda\subset L_0\cap\phi_1^{-1}(L_1)$ is an isolated 
set of intersections, let $U\subset M$ be an open 
neighborhood of $\Lambda$, and let $V\subset M$ be an 
isolating neighborhood of $\Lambda$ such that 
$\overline{V}\subset U$.  Fix a compact set $K\subset M$.
There is a constant $\hbar>0$ such that the following holds.
If $T>0$ and $u:[-T,T]\times[0,1]\to K$ satisfies
\begin{equation}\label{eq:floer1}
\p_su+J_t(u)\left(\p_tu-X_{H_t}(u)\right)=0,\qquad
u(s,0)\in L_0,\qquad u(s,1)\in L_1,
\end{equation}
and 
\begin{equation}\label{eq:EuV}
E_H(u) = \int_{-T}^T\int_0^1\abs{\p_su}^2\,dtds <\hbar,\qquad
u(\{\pm T\},t)\subset\phi_t(V)\;\forall t\in[0,1],
\end{equation}
then $u(s,t)\in\phi_t(U\cap\Om_t)$ for all $s\in[-T,T]$
and all $t\in [0,1]$. 
\end{theorem}

\begin{proof}
This theorem was proved in~\ref{thm:mon1} for $H=0$.
We reduce the general case to the case $H=0$ in two steps.

\medskip\noindent{\bf Step~1.}
{\it It suffices to assume that the Hamiltonian
$H:[0,1]\times M\to\R$ has compact support
(and hence $\Om_t=M$ for $0\le t\le 1$).}
\medskip\noindent

Shrinking $U$, if necessary, we may assume that
$\overline{U}$ is compact and
$$
\overline{V}\subset U\subset\overline{U}\subset \Om_1.
$$
Then the set
$$
\widehat{K} := [0,1]\times K
\cup\bigcup_{0\le t\le 1}\{t\}\times\phi_t(\overline{U})
$$
is a compact subset of $[0,1]\times M$.  Replace 
the Hamiltonian function $H$ 
by a function $\widehat{H}:[0,1]\times M\to\R$ with compact 
support that agrees with $H$ on a neighborhood 
of $\widehat{K}$. Denote by $\widehat{\phi}_t$ the Hamiltonian
isotopy generated by $\widehat{H}$.  
Then $\widehat{\phi}_t(U)=\phi_t(U\cap\Om_t)$
and $\widehat{\phi}_t(V)=\phi_t(V)$ for $0\le t\le 1$.
Hence a function $u:[-T,T]\times[0,1]\to K$ 
satisfies~\eqref{eq:floer1} and~\eqref{eq:EuV} 
if and only if it satisfies the same conditions
with $H$ and $\phi_t$ replaced by $\widehat{H}$ 
and $\widehat{\phi}_t$. This proves Step~1. 

\medskip\noindent{\bf Step~2.}
{\it Theorem~\ref{thm:MON} holds 
when $H$ has compact support.}

\medskip\noindent
 Theorem \ref{thm:mon1} implies the statement of 
 \ref{thm:MON} in the case $H=0$.
Under the assumption of Step~2 the assertion 
of Theorem~\ref{thm:MON} with $H\ne 0$ reduces 
to the assertion with $H=0$ by naturality. 
More precisely, define
$$
\tJ_t := \phi_t^*J_t,\qquad
\tL_0:=L_0,\qquad
\tL_1:=\phi_1^{-1}(L_1),\qquad
\tK := \bigcup_{t\in[0,1]}\phi_t^{-1}(K).
$$
Then $\tK$ is a compact subset of $M$
and $\tL_0,\tL_1\subset M$ satisfy~(H).
Hence the tuple 
\begin{equation}\label{eq:tilde}
(M,\om,\tL_0,\tL_1,\tH=0,\tJ,
\widetilde{\Lambda}=\Lambda,\tU=U,\tV=V,\tK)
\end{equation}
satisfies the hypotheses of Theorem~\ref{thm:MON}
with the Hamiltonian function equal to zero. 
Hence, by Theorem ~\ref{thm:mon1}, there exists a 
constant $\hbar>0$ such that the result holds for this tuple.

Now let $u:[-T,T]\times[0,1]\to M$ be a smooth function
that satisfies~\eqref{eq:floer1} and~\eqref{eq:EuV}
(with the above constant $\hbar$) and define 
$\tu:[-T,T]\times[0,1]\to M$ by
$$
\tu(s,t) := \phi_t^{-1}(u(s,t)).
$$
Then $\tu$ satisfies~\eqref{eq:floer1} and~\eqref{eq:EuV}
for the tuple~\eqref{eq:tilde}.  
Hence 
$$
\tu([-T,T]\times[0,1])\subset\tU=U
$$ 
and hence $u(s,t)\in\phi_t(U)$ for all $s$ and $t$.
This proves Theorem~\ref{thm:MON}.
\end{proof}

\begin{theorem}\label{thm:mon2}
Assume ~(H) and suppose that the intersection $ \Lambda = L_0 \cap L_1 $ 
is clean. 
Let $ g $ be some Riemannian metric on $ M $ 
and let $ d $ be the distance induced by $ g$. 
For every $ \epsilon> 0 $ there exists $ \hbar > 0 $ 
such that the following holds. 
For any interval $ I = [a,b]\subset \R $  if a 
solution $ u: I \times [0,1] \rightarrow N $ of 
\eqref{eq:jhol} satisfies the following:
\begin{itemize}
 \item There exist $ x, \; y \in \Lambda= L_0 \cap L_1 $ such that 
 $$ 
\sup\limits_t d( u( a, t ) , x ) < \frac{\epsilon}{12} \text { and }
 \sup\limits_t d(u(b, t ) , y ) < \frac{\epsilon}{12} .
$$
\item $ E(u) < \hbar, $
\end{itemize}
 then 
\begin{equation}
u(s,t ) \in B_{\epsilon } ( x ) \cap  B_{\epsilon } ( y )
, \; \; \text{ for all } (s,t) \in I \times [0,1].
\end{equation}
\end{theorem}
\begin{proof}
 See section \ref{sec:proofs}. 
\end{proof}

{\bf Outline of the chapter:}
The proofs of Theorems \ref{main_thm}, 
\ref{thm:mon1} and \ref{thm:mon2} are based on 
some properties of holomorphic maps such as 
isoperimetric inequality and exponential decay. 
In the next sections we discuss in more details 
these properties. In section \ref{sec:isoperimetric} 
we prove isoperimetric inequality in $ \R^{2n} $. 
Then we prove that the isoperimetric inequality holds also 
for short curves in a symplectic manifold
$M $ with Lagrangian boundary conditions.
We prove this using local Darboux charts that are adjusted to the 
clean intersection. Isoperimetric inequality is a crucial 
part of the proof of Theorem \ref{main_thm} as well as 
of the exponential decay. Theorems \ref{thm:mon1} and \ref{thm:mon2} 
are corollaries of the Theorem \ref{main_thm} and 
the exponential decay.

\section{The isoperimetric inequality}\label{sec:isoperimetric}

The isoperimetric inequality is an inequality involving the length 
of a closed curve and the area enclosed by this curves. 
It is often expressed by saying that among all curves of given length 
the circle encloses the greatest area. 
In other words, if $ \gamma \subset \R^2 $ is simple closed curve, $ A(\gamma) $
is the enclosed area and $ L(\gamma) $ is its length then 
\begin{align*}
 A(\gamma) \leq \frac{1}{4\pi} L(\gamma)^2 
\end{align*}
and the equality holds if and only if $ \gamma $ is a circle. 
The same holds in the case that $ \gamma $ is a curve 
with endpoints on some lines through the origin. 
Again, the maximal area of the curvilinear triangle
bounded by $ \gamma $ and the lines through the origin 
is in the case that $ \gamma $ is a piece of a circle. \\

We first prove the isoperimetric inequality 
in $ \R^{2n} $ for smooth curves with endpoints
on Lagrangian planes. 
Let $ \omega $ be the standard symplectic form in $ \R^{2n} $ 
and let $ \left |\cdot\right | $ denotes the standard Euclidean norm. We define 
the action, length and energy of a curve $ \gamma $ as follows. 

\begin{align}
 & A(\gamma) =\frac{1}{2} \int_0^1
\omega(\dot{\gamma}(t),\gamma(t)) dt \; 
,\;\;\notag \\ 
& L(\gamma) := \int_0^1 \abs{\dot{\gamma}(t)} dt, 
\; \;\notag \\
&E(\gamma) = \frac{1}{2} \int_0^1 \abs{\dot{\gamma}(t)}^2 dt. 
\end{align}
\index{ $A(\gamma)$}
\index{ $ E(\gamma)$} 
\index{$L(\gamma)$} 
\nomenclature{$A(\gamma)$}{Symplectic action of a curve $\gamma$ in the Euclidean space 
\nomrefeqpage}
\nomenclature{$E(\gamma)$}{Energy of a curve $\gamma$ in the Euclidean space 
\nomrefeqpage}
\nomenclature{$L(\gamma)$}{Length of a curve $\gamma$ in the Euclidean space 
\nomrefeqpage}
We prove the isoperimetric inequality for curves 
with Lagrangian boundary conditions using Fourier series as in \cite{MS}, 
where the analogous property was proved for closed curves. 

\begin{lemma}\label{lem:mon_lem1}
For all smooth curves $ \gamma : [0,1] \rightarrow \R^{2n} $ 
with $ \gamma(0)\in \R^n \times \{ 0\}= L_0 $ and 
$ \gamma(1) \in \R^d \times \{0\} \times \R^{n-d}= L_1 $ 
we have that 
\begin{equation} 
\abs{A(\gamma)} \leq \frac{1}{\pi} L(\gamma) ^2 .
\end{equation}
\end{lemma}
\begin{proof}
 Let 
$\gamma :[0,1]\rightarrow  \mathbb{C}^n,\;
\gamma=(z_1(t), z_2(t),...z_n(t))= ( x_1(t), \cdots, x_n(t), y_1(t), \cdots, y_n(t))\;,\;0\leq t \leq 1 $, where
$$
z_j(t)=x_j(t)+ i y_j(t)\;, \;\; 1\leq j\leq n .
$$ 
Because of the boundary conditions 
$ \gamma(0)\in \mathbb{R}^n \times {0} , \;\;
\gamma(1)\in \mathbb{R}^d \times{0}\oplus i({0}\times
\mathbb{R}^{n-d}),$ we have that for
$ z_{\nu}, \;\; 1\leq \nu \leq d $, 
$ z_{\nu} (0)\in \mathbb{R},
z_{\nu} (1)\in \mathbb{R} $. 
It follows that writing $ z_{\nu} $ using the Fourier series 
we have 
$$ 
z_{\nu}(t)= \sum\limits_{m=-\infty}^{+\infty} c_{\nu, m} e^{\pi
i mt} , \;\; 1\leq \nu \leq d,\;\; c_{\nu,m} \in \mathbb{R}.
$$

\noindent Let $ \gamma_1 $ be the first $d-$coordinates 
of the curve $ \gamma$ i.e. 
$\gamma_1(t)=(z_1(t), z_2(t),...z_d(t), 0,\cdots, 0) $. Then

\begin{equation}\label{gamma_1}
\gamma_1(t)= \sum\limits_{m=-\infty}^{+\infty} v_m e^{\pi
i mt} , \;\; 0\leq t \leq 1,\;\;v_m=(c_{1,m},c_{2,m},....c_{d,m}, 0, \cdots, 0). 
\end{equation}

\noindent 
Similarly for $d+1\leq \nu\leq n ,\;\;z_{\nu}(0)\in
\mathbb{R} ,z_{\nu}(1)\in i \mathbb{R} $ it follows that

$$ 
z_{\nu}(t)= \displaystyle \sum\limits_{m=-\infty}^{+\infty} c_{\nu, m}
e^{i(\frac{\pi}{2} +  m\pi)t} , \;\; d+ 1\leq \nu \leq n,\;\;
c_{\nu,m} \in \mathbb{R} 
$$
 Let $\gamma_2(t)=(0, \cdots, 0, z_{d+1}(t), z_{d+2}(t),...z_n(t)) $.Then
 
\begin{equation}\label{gamma_2}
\gamma_2(t)= \sum\limits_{m=-\infty}^{+\infty} v_m
e^{i(\frac{\pi}{2} +  m\pi)t} , \;\; 0\leq t \leq 1,\;\;
v_m=(0, \cdots, 0, c_{d+1,m},c_{d+2,m},....c_{n,m})
\end{equation}
First, notice that it is enough to prove that
 \begin{equation} \label{ch1.1_eqe}
 \abs{ A(\gamma)}  \leq \frac{2}{\pi} E(\gamma) 
 \end{equation}
Since, if $\gamma $ is naturally parametrized 
then 
$ E(\gamma) = \frac{1}{2} L(\gamma)^2 $. 
If $\gamma $ isn't parametrized by arc length, 
reparametrize it i.e. take a function 
$ \alpha :[0,1] \rightarrow [0,1] $ such that 
$ \abs{\frac{d}{dt} \gamma(\alpha(t)) }_e = L(\gamma ) $. 
Then it follows from \eqref{ch1.1_eqe} that 
$$ 
\abs{A(\gamma) } = \abs{A (\gamma \circ \alpha ) }
 \leq \frac{2}{\pi} E ( \gamma \circ \alpha) = 
\frac{1}{\pi} L(\gamma)^2 
$$
  It is enough to prove the inequality 
separately for $ \gamma_1 $ and $\gamma_2 $ as 
$\abs{ A ( (\gamma_1, \gamma_2 ))} \leq \abs{A(\gamma_1)} 
+ \abs{A (\gamma_2)} $
and 
$ E((\gamma_1, \gamma_2)) = E(\gamma_1) + E(\gamma_2) $.
Now the proof of Lemma \ref{lem:mon_lem1} follows directly from 
Lemma \ref{lem:mon_lem2}.
\end{proof}

\begin{lemma}\label{lem:mon_lem2}
 Let $\gamma_1$ and $\gamma_2$ be paths as in \eqref{gamma_1} 
 and \eqref{gamma_2}. Then
\begin{equation}
  \abs{\textsl{A}(\gamma_1)}\leq \frac{1}{\pi}E(\gamma_1) ,
\qquad
  \abs{\textsl{A}(\gamma_2)}\leq \frac{2}{\pi}E(\gamma_2).
\end{equation}
\end{lemma}
\begin{proof}
Let $\gamma_1(t)$ be as in \eqref{gamma_1} , and 
$$
\textsl{A}(\gamma_1)= -\frac{1}{2}\int_0^1 \omega(\gamma_1(t), \dot{\gamma_1}(t) ) dt
=  -\frac{1}{2} \int_0^1 \sum\limits_{m,n=-\infty}^{+\infty}
\omega(v_m e^{im\pi t}, in\pi v_n e^{in\pi t})dt.
$$ 
Let  
\begin{align*}
J_{mn}& = \int_0^1\omega(v_m e^{im \pi t},iv_n e^{in \pi t}) dt
= \frac{1}{\pi}\int_0^{\pi}\omega(v_m e^{im s},iv_n e^{ins})ds \\
 &= \frac{1}{\pi} \Big ( \underbrace{\int_0^{\pi} \cos(ms) \cos(ns) ds}_{A_{mn}} + 
 \underbrace{\int_0^{\pi} \sin(ms) \sin(ns)ds }_{B_{mn}} \Big ) \langle v_m, v_n \rangle.
\end{align*}
The last equality holds as $ \omega (v_m, v_n) = 0 $ , because $ v_m ,\; v_n \in \R^d $.
We compute the terms $ A_{mn}, \; B_{mn} $:
\[ A_{mn}=\begin{cases}
            \frac{\pi}{2}, &\text{m=n or m=-n}\\
            0, &\text{otherwise}
            \end{cases} \hspace{2cm}
B_{mn}=\begin{cases}
            \frac{\pi}{2}, &\text{m=n} \\
            \frac{-\pi}{2}, &\text{ m=-n}\\
            0, &\text{otherwise}.
            \end{cases} \]
Thus we get:
\[ J_{mn}=\begin{cases}
            |v_n|^2 , \;\; m=n\\
		0, \;\;otherwise.
            \end{cases} \]

It follows that
\begin{align*}
 \textsl{A}(\gamma_1)= -\frac{1}{2} \int_0^1 \sum\limits_{m,n=-\infty}^{+\infty}
\omega(v_m e^{im\pi t}, in\pi v_n e^{in\pi t})dt
= -\frac{\pi}{2} \sum\limits_{n=-\infty}^{+\infty} n
|v_n|^2.
\end{align*}
Similarly, we compute the energy $ E(\gamma_1) $:
\begin{align*}
 \textsl{E}(\gamma_1)&
 = \frac{1}{2}\int_0^1 \omega(\dot{\gamma_1}(t),i\dot{\gamma_1}(t) ) dt
 = \;\; \frac{1}{2} \int_0^1 \sum\limits_{m,n=-\infty}^{+\infty}
\omega(v_m m \pi i e^{im\pi t}, -\pi v_n n e^{in\pi t})dt \\
&= \frac{\pi^2}{2} \sum\limits_{m,n=-\infty}^{+\infty} mn \int_0^1
\omega(v_me^{im\pi t}, v_n i e^{in\pi t})dt\\
&=\frac{\pi^2}{2} \sum\limits_{n=-\infty}^{+\infty} n^2|v_n|^2
\end{align*}
 Now it is obvious that 
 $$
 |\textsl{A}(\gamma_1)|\leq
\frac{\pi}{2}\sum_{n}|n\|v_n|^2 
\leq \frac{1}{\pi} E(\gamma_1)
$$
Analogously we obtain for $ \gamma_2 $ as in \eqref{gamma_2}

\begin{align*}
 \abs{A(\gamma_2)}& =\frac{\pi}{2}
 \sum\limits_{n=-\infty}^{+\infty} \abs{n+ \frac{1}{2}}
 \abs{v_n}^2 \leq \frac{2}{\pi} \frac{\pi^2}{2} 
 \sum\limits_{n=-\infty}^{+\infty}
 \abs{n+\frac{1}{2}}^2\abs{v_n}^2 = \frac{2}{\pi}E(\gamma_2) 
\end{align*}

\end{proof}
\begin{remark}
 Notice that the isoperimetric inequality for curves with 
 end points in Lagrangian plane $ L $ 
 follows from Lemma \ref{lem:mon_lem2}. 
 Namely, we can take $ d=n $, i.e. $ L_0 = L_1 = L$, then
 it follows from lemma \ref{lem:mon_lem2}
 $$ 
A( \gamma ) \leq \frac{1}{\pi} E(\gamma) \leq \frac{1}{2\pi} L(\gamma)^2 .
$$
 In the case that $ \gamma $ is a closed curve we can take 
 a plane $ L_0 $ that divides $ \gamma $ into two pieces $ \gamma_i , \; i=0,1 $ 
 of equal length. For each curve $ \gamma_i $ the isoperimetric inequality holds 
 $$
 A(\gamma_i) \leq \frac{1}{2\pi} L(\gamma_i)^2 = \frac{1}{8\pi} L(\gamma)^2 .
$$
 Summing the previous inequalities for $ i=0,1 $ we obtain 
 
\begin{align}
  A(\gamma) \leq \frac{1}{4\pi} L(\gamma)^2. 
 \end{align}

\end{remark}

The next corollary easily follows from lemma \ref{lem:mon_lem1}. 

\begin{corollary}\label{cor:mon_cor1}
 Let  $Q\subset \R \times [0,1] $ be a submanifold with corners
 such that all the corners are contained in $ \R \times \{0,1\} $. 
 Then any smooth map 
 $$ 
 u: (Q, Q \cap (\R \times \{0\}), Q \cap (\R \times \{1\}) ) 
 \rightarrow (\R^{2n}, L_0, L_1 )= ( \R^{2n}, \R^n \times \{0\}, \R^d \times \{0\} \times \R^{n-d}), 
 $$
 satisfies the following: 
\begin{equation}
\int\limits_Q u^* \omega_{std} 
\leq \frac{1}{\pi}\cdot L( \left. 
u \right|_{u^{-1} (u( \partial Q ) \setminus (L_0 \cup L_1 ))} )^2
\end{equation}
where $ L$ denotes the Euclidean length.
\end{corollary}
\begin{proof} 
First notice that all boundary curves 
$ u( \partial Q ) $ are
of the following kind
\begin{itemize}
\item[1.] closed
\item[2.] Have both endpoints in $L_0$ or $ L_1$.
\item[3.] Have one boundary point in $L_0 $ and the other in $L_1$
\item[4.] Are contained in $ L_0 $ or $L_1$.
\end{itemize} 
Using Stoke's theorem we have:
\begin{align*}
\int\limits_Q u^* \omega_{std} & 
= \int\limits_Q u^* d\lambda
= \int\limits_{\partial Q} u^* \lambda 
= \sum\limits_{i} \int\limits_{\gamma_i} \lambda \\ 
&= \sum\limits_i \int_0^1 \lambda (\gamma_i(t)) (\dot{\gamma}_i(t) ) dt 
= \sum\limits_i \frac{1}{2} \int_0^1 \omega_{std} (\dot{\gamma}_i(t),\gamma_i(t))  dt.
\end{align*}
Notice that for $\gamma \subset L_i, \; i=0,1 $  the integral 
$ \int\limits_0^1 \omega_{std} (\dot{\gamma}(t), \gamma(t) ) = 0 $,
as $ \omega_{std}|_{L_i} \equiv 0, \; i=0,1 $. 
Thus,
$$
\int\limits_Q u^* \omega_{std} 
= \sum\limits_{i\in I} \frac{1}{2} \int\limits_0^1 \omega_{std}(\dot{\gamma}_i(t),\gamma_i(t) ) dt
= \sum\limits_{i\in I} A(\gamma_i) ,
$$ 
where $ I $ are the curves of type $ 1-3 $. 
In lemma \ref{lem:mon_lem1} we have proved that
$$ 
\abs{A(\gamma)} \leq \frac{1}{\pi} L(\gamma)^2 ,
$$ 
for any curve of the type 1-3. Thus we have : 
 $$ 
 \int\limits_Q u^*\omega_{std}
 \leq \frac{1}{\pi} \sum\limits_{i\in I} L(\gamma_i)^2 
 \leq \frac{1}{\pi} \Big ( \sum\limits_{i\in I} L(\gamma_i) \Big  )^2.
 $$
\end{proof}

\subsection{Isoperimetric inequality in symplectic manifolds}
The isoperimetric inequality holds not only for curves 
in Euclidean space, but also for short closed 
curves $ \gamma $ on symplectic manifolds or for 
short curves with Lagrangian boundary conditions. 
In both cases the 
area should be understood as a symplectic area.
In the case of a path with Lagrangian boundaries 
it should be understood as the area of a curvilinear triangle with one side 
$ \gamma $ and 
the other two sides on the Lagrangian submanifolds. 
There are many curvilienar triangles with one side $\gamma$ 
and the other two on Lagrangian submanifolds, but if $\gamma$ 
is sufficiently short all of them will have the same symplectic area. 
Throughout this section we shall assume the follwing:

\begin{itemize}
 \item[ ~(H1)] $L_0, L_1\subset M $ are Lagrangian 
 submanifolds of a symplectic manifold $(M, \omega) $ 
 and the intersection $ \Lambda = L_0 \cap L_1 $ is clean, i.e. 
 $$ T_p \Lambda = T_p L_0 \cap T_p L_1 , \;\; \forall p \in \Lambda $$
 We also assume that $ d = \text{dim}(\Lambda) $. 
\end{itemize}

Locally the clean intersection looks particularly nice. 
Namely for any point $ p \in \Lambda $ there exists a local symplectic chart
which maps $ L_0 $ into $ \R^n \times \{0\} $ and $ L_1 $ into 
$ \R^d \times \{0\} \times \R^{n-d} $. We first construct such 
Darboux charts adapted to clean intersections. 

\begin{lemma}[{\bf Darboux charts for clean intersection}]\label{lem:dar-clean}
Assume ~(H1). Then for any $ p\in L_0 \cap L_1=\Lambda $ there exists a neighborhood 
$ U(p) $ and a symplectomorphism 
$ \phi: U(p) \rightarrow  V(0) \subset \R^{2n} $ such that 

 \begin{equation}\label{eq_locsymp}
\begin{split}
  &\phi( U(p) \cap \Lambda ) \subset \R^d\times \{0\} , \\
 \;\;  &\phi ( U(p) \cap L_0 ) \subset \R^n \times \{0\} , \\
\;\; & \phi ( U(p)\cap L_1 ) \subset \R^d \times \{0\} \times \R^{n-d}  .
\end{split}
 \end{equation}
   
\end{lemma}
\begin{proof}
 Let $ p \in \Lambda $, from the Lagrangian neighborhood
 theorem we know that there exists a neighborhood of
 $ L_0, U(L_0) $ and a symplectomorphism 
$ \psi : U( L_0 ) \rightarrow  V( L_0 ) \subset T^*L_0 $ 
such that $ \psi (L_0 )= L_0 $. 
Choose local  symplectic coordinates in a neighborhood $ V(p) $ of
$ \psi(p)=p \in T^* L_0$ , i.e. a symplectomorphism 
$ \psi_1 : V(p) \rightarrow U(0) \subset \R^{2n} $,
such that in these coordinates $ \Lambda $, $ L_0 $ and
$ L_1 $ are given as follows 

 \[\begin{cases}
  \Lambda : \;\; ( x, 0, 0,\cdots, 0 ) , \;\;  x = ( x_1, x_2, \cdots, x_d) \\ 
    L_0 : \;\; (x, y, 0,\cdots,0 ), 
   \; \;  x = ( x_1, x_2, \cdots, x_d ) , 
    \; y= ( y_1, \cdots, y_{n-d} )\\
    L_1: \;\; (x, f(x, z), g(x,z), z), \;
  x= (x_1, \cdots, x_d), \; z= ( z_1, \cdots, z_{n-d} )
   \end{cases}
  \]
and 
$ f: \R^n \rightarrow \R^{n-d} , \;g: \R^{n} \rightarrow \R^d  $ 
are differentiable maps with $ f(x,0) = g(x,0) = 0 $. 
Let $ \Psi : \R^{2n} \rightarrow \R^{2n} $ be given by 
\begin{align}
 \Psi ( x, y, u, z) = ( x , y + f(x,z), u + g(x,z) , z) , 
\;\; x,u \in \R^d, \; y,z \in \R^{n-d}. 
\end{align}
Notice that $ \Psi (x, y, 0,0) = (x,y,0,0)\subset L_0 $
 and $ \Psi(x,0,0,z) = ( x, f(x,z), g(x,z) , z) \subset L_1 $,
 The mapping $ \Psi $ is a diffeomorphism locally 
 as its derivative is given by
\[
 d\Psi(x,y,u,z) = \left (\begin{matrix}
      1 \;& f_x\; & g_x\; & 0\; \\
      0 \;& 1 \;& 0\; & 0 \;\\
      0 \;& 0\; & 1\; & 0\; \\
      0 \; & f_z \;&g_z\; & 1 \;
  \end{matrix}\right )
\]
and the determinant of $ d \Psi  $ is $ 1 $. 
The mapping $ \Psi $ is a symplectomorphism if and 
only if $ f_z^T = f_z ,$
$ g_x^T= g_x $ and $ f_x + g_z^T = 0 $. 
From the fact that $ L_1 $ is Lagrangian
follows that $\Psi $ is also a symplectomorphism. 
Now the desired symplectomorphism 
$ \phi = \Psi^{-1} \circ\psi_1 \circ  
\psi : U(p) \rightarrow \R^ {2n} $ 
and satisfies the properties (\ref{eq_locsymp}). 
\end{proof}

The previous lemma follows also as a corollary of 
a more general result in the thesis of M.Pozniak (\cite{Po}). 
M.Pozniak proves the analog of Weinstein's neighborhood theorem for clean 
intersection. More precisely he proves that a neighborhood $ U(\Lambda ) $ is 
symplectomorphic to the neighborhood of $ \Lambda \subset T^* L_0 $ 
and such symplectomorphism maps $ L_0 $ into the zero section and $ L_1 $ into 
the conormal bundle of $ \Lambda $. Here we actually need a weaker result that 
concerns only the local representation of clean intersection. 
\begin{remark}
 Let $ \phi $ be a local Darboux chart constructed 
 in Lemma \ref{lem:dar-clean}, then the $ 1-$form 
 $\lambda = \phi^* \sum_{i=1}^n x_i dy_i $ satisfies 
 $ \omega = d \lambda $ and $ \lambda $ vanishes on $ TL_0 $ and 
 $ T L_1 $. 
\end{remark}

\begin{definition}[{\bf Darboux radius}]\label{def:darb_rad}
 Assume ~(H) and ~(H1) and fix a Riemannian metric $g $ on $M$. 
 A constant $ \delta > 0 $ is called a {\bf Darboux radius } 
 for $ ( M, L_0,L_1, \Lambda ) $ if for every $ p \in \Lambda $ 
 there exists a coordinate chart $ \phi : B_{\delta} (p) \rightarrow \R^{2n} $ 
 such that 
 \begin{itemize}
  \item $\phi^* \omega_{std} = \left. \omega\right|_{B_{\delta} (p)} $
  \item $\phi ( L_0 \cap B_{\delta} ( p)) = \R^n \times \{0\} \cap \phi ( B_{\delta} (p)) $
  \item  $\phi ( L_1\cap B_{\delta} ( p)) = \R^d \times \{0\}\times \R^{n-d} \cap \phi ( B_{\delta} (p)) $
  \item $ L_0 \cap B_{r} ( p) $ and $ L_1 \cap B_{r} (p) $ are connected for all $ 0< r \leq \delta $. 
 \end{itemize}
 Such a coordinate chart we call Darboux coordinate chart for 
 clean intersection. 
\end{definition}
\begin{remark}
 By Lemma \ref{lem:dar-clean} the set of Darboux radiuses is nonempty. 
 Moreover it follows from the definition if $ \delta  $ is a Darboux 
 radius so is every positive number $ r < \delta $. 
\end{remark}

\begin{lemma}\label{lem:mon_lem3}
Assume ~(H) and ~(H1).
 For every $ \delta > 0 $ there exists $ \rho > 0 $
such that the following holds.
Let $ \gamma: [0,1] \rightarrow N $ be a smooth curve 
with $ \gamma(i) \in L_i, \; i=0,1 $ . If $\ell( \gamma) < \rho $ then
$$ 
\gamma \subset B_{\delta} (x_0 ) 
$$ 
for some $x_0 \in \Lambda $. 
\end{lemma}

\begin{proof} 
Let $ \delta $ be an arbitrary positive number. 
Observe the following sets:
 \begin{align*}
 L_{0, \delta/2} := \{ x\in L_0 \cap N: d(x, \Lambda) \geq \delta/2 \}  \\
 L_{1, \delta/2} := \{ x\in L_1 \cap N: d(x, \Lambda)\geq \delta /2\}
 \end{align*}
Obviously 
$ L_{0, \delta/2} $ and $L_{1, \delta/2}$ are compact subsets
of $ L_0\cap N $ and $ L_1\cap N $. 
Let $ \rho >0 $ be a number that satisfies 
$$
d ( L_0, L_{1, \delta / 2} )\geq \rho  , \; \; d ( L_1, L_{0, \delta / 2} )\geq \rho .
$$
Obviously such a positive number exists 
and $ \rho < \delta / 2 $. Now if  $\ell(\gamma) < \rho $, 
it cannot happen that $ \gamma(0)\in L_{0, \delta/2} $ and 
$\gamma(1) \in L_{1, \delta /2} $ , as the distance between 
these sets is bigger than $ \rho $. 
Thus, for example $d(\gamma(0) , \Lambda ) < \delta/2 $, 
i.e there exists $ x_0 \in \Lambda $ such that 
$ d(x_0 , \gamma(0) ) \leq \delta/2 $. 
Then $ \gamma \subset B_{\delta} (x_0)  $.
\end{proof}

\begin{lemma}\label{mon:lem3a}
 Assume ~(H) and ~(H1). Let $ \delta > 0 $ be a Darboux radius 
 and let $ x_0 \in \Lambda $. Let $ \gamma : [0,1] \rightarrow B_{\delta} ( x_0 ) $ 
 be a smooth curve such that $ \gamma(i) \in L_i, \; i=0,1 $. 
 Denote 
 $$
 Q = \{ z\in \C : \abs{z} \leq 1, \; \text{ Re } ( z) > 0, \; \text{Im} ( z) > 0 \} , 
 $$
 Then there exist $ u : Q \rightarrow B_{\delta} ( x_0 ) $ that satisfies 
  $$ 
  u ( e^{\frac{\pi}{2} it} ) = \gamma(t),  \;\;  u( Q \cap \R) \subset L_0 , \;\;
  u ( Q \cap i\R ) \subset L_1  
  $$
\end{lemma}
\begin{proof}
 Let $ \alpha_i : [0,1] \rightarrow L_i \cap B_{\delta} (x_0)$ 
 be such that $ \alpha_i(0) = x_0 $  and $ \alpha_i(1) = \gamma(i) $. Such
 paths $ \alpha_i $ exist as $ \delta > 0 $ is a Darboux radius. 
 The loop that we obtain concatenating $ \alpha_i, i=0,1 $ and 
 $ \gamma $ is contractible as it is contained in a convex neighborhood 
 of a point. Thus there exists a desired map $ u$. 
\end{proof}

\begin{definition}\label{def:mon1}
 Assume ~(H) and ~(H1). Choose $ \delta > 0 $ so small that $ 3 \delta $ is still 
 a Darboux radius. Let $ \rho > 0 $ be the constant of Lemma 
 \ref{lem:mon_lem3} for such a $ \delta $. Let 
  $ \gamma : [0,1] \rightarrow M $
 be a smooth path satisfying 
 $$
 \gamma(i) \in L_i, \;\; i=0,1 , \ell(\gamma) < \rho, \;\; \gamma ( [0,1] ) \subset N 
 $$
Choose $ x_0 \in \Lambda $ such that $ \gamma( [0,1] ) \subset B_{\delta} ( x_0 ) $
(see Lemma \ref{lem:mon_lem3}) and let $ u : Q \rightarrow B_{\delta} (x_0) $ be 
as in Lemma \ref{mon:lem3a}. The 
 {\bf local symplectic action} of $ \gamma $ is a real number  
  \begin{align}
  a(\gamma):= \int_{Q} u^* \omega. 
 \end{align}
\end{definition} 
\nomenclature{$a(\gamma)$}{Symplectic action of acurve $\gamma$ on a symplectic manifold  
\nomrefeqpage}
\begin{claim}\label{cl:mon1}
 The local symplectic action is well defined, i.e. 
 it doesn't depend on the choice of $ x_0 $ and $ u $ 
 in the definition \ref{def:mon1}.
\end{claim}
\begin{proof}
 Let $ x_1 \in \Lambda $ be another point such that 
 $ \gamma( [0,1] ) \subset B_{\delta} ( x_1) $ and let 
 $u' : Q \rightarrow B_{\delta} ( x_1) $ satisfy 
 $ u' ( e^{i\pi/2 t } ) = \gamma(t) $. 
 Then $ d( x_0, x_1 ) < d ( x_0, \gamma(t)) + d ( \gamma(t), x_1) < 2 \delta $. 
 Thus $ x_0 $ and $ x_1 $ are contained in a single Darboux chart. 
 There exists a $ 1-$form $ \lambda $ on $ B_{2\delta} ( x_0) $ such 
 that $ d \lambda = \omega $. Such $ 1-$form $ \lambda $ 
 has additional property that it vanishes on Lagrangian submanifolds 
 $ L_0 $ and $ L_1 $. 
 Then 
 $$ 
 \int_{Q} u^* \omega= \int_{Q} u^* d\lambda
 = \int_{\partial Q } u^* \lambda = \int_0^1 \lambda ( \dot{\gamma}(t)) dt. 
 $$
 Similarly we have $ \int_{Q} (u')^* \omega= \int_0^1 \lambda ( \dot{\gamma}(t)) dt $. 
\end{proof}

\begin{remark}
 We see from the proof of claim \ref{cl:mon1} that it is possible to define 
 $$ a( \gamma) = \int_{\gamma} \lambda ,$$
 where $ \lambda \in \Omega^1 ( B_{\delta} ( x_0) ) $ is $ 1 -$form 
 which vanishes on Lagrangian submanifolds and $ d \lambda = \left.\omega \right|_{B_{\delta} ( x_0) } $. 
\end{remark}

Now we can prove the isoperimetric inequality 
for short curves with Lagrangian boundary conditions.

\begin{lemma}[{\bf Isoperimetric inequality}]\label{lem:mon_lem4}
Assume ~(H) and ~(H1). 
 There exist positive constants 
 $\rho_0 $ and $c$ such that for every 
 $ \gamma: [0,1] \rightarrow N , \; \gamma(i) \in L_i, i=0,1 $ the following holds:  \\
 If $ \ell( \gamma ) < \rho_0 $ then 
\begin{equation}
 a(\gamma) \leq c \ell(\gamma)^2. 
\end{equation}

\end{lemma}
\nomenclature{$\ell(\gamma)$}{The length of a curve $\gamma$ 
\nomrefpage}
\begin{proof} Let $ \delta $ be a Darboux radius. 
 We can cover $\Lambda $ with finitely many Darboux charts 
 as in Definition \ref{def:darb_rad}
 $ \phi_i : B_{\delta/2 } ( x_i ) \rightarrow \R^{2n} , 
\; x_i \in \Lambda$ such that  $ \|d\phi_i(x) \| \leq c' $ 
for some constant $c'$ and $ \forall x \in B_{\delta }(x_i) $ 
and for all $i $.
 Now take such 
  $ \rho_0 $ so that the claim of the lemma \ref{lem:mon_lem3}
is satisfied with the constant $ \delta_0 = \delta / 2$.
 It follows that if $ \ell( \gamma) < \rho_0 $ then
 $ \gamma \subset B_{\delta/2 } (x ) \subset B_{\delta} (x_i )$ 
for some $ i $.
 As $\gamma$ is contained in one Darboux chart 
we can use the fact that isoperimetric 
inequality holds in $ \R^{2n} $ and
 that the mapping $ \phi= \phi_i $ preserves 
the $ a(\gamma) $ and changes the length only up to some constant.
 More precisely we have
 
\begin{align*}
 a(\gamma) = \int_{Q} u_{\gamma} ^* \omega = \int_{Q} ( \phi \circ u_{\gamma} )^* w_{std} 
= A ( \phi \circ \gamma) \leq \frac{1}{\pi} L(\phi \circ \gamma ) ^2 
\end{align*}
Here the last inequality follows from Lemma \ref{lem:mon_lem1}. 
On the other hand, $ L(\phi \circ \gamma)  \leq c' \ell(\gamma) $,
thus we easily get
$$
 a(\gamma )\leq\frac{1}{\pi}(c')^2 \ell(\gamma)^2 = c \ell(\gamma)^2. 
$$
\end{proof}
 
The analog of the Corollary \ref{cor:mon_cor1}
holds for maps $ u : Q \rightarrow N $,
 provided that the image of $ u$ has small diameter. 

\begin{lemma}\label{lem:mon_lem5}
 Assume ~(H) and ~(H1). 
 Then there exist constants $ r_1 $ and $c$  
 such that the following holds. 
 Let $ Q \subset \R\times [0,1] $ be a compact 2-dimensional submanifold
 with corners such that the corners of $ Q $ are contained in $ \R \times \{0,1 \} $. 
 Denote $ \Gamma = \overline{ \partial Q \setminus \R \times \{ 0,1 \} } $. 
 Thus $ \Gamma $ is a compact 1-dimensional manifold with boundary and 
 $ \partial \Gamma \subset \R \times \{0, 1\} $. Let  
$$ u : (Q, Q \cap \R\times \{0\} , Q \cap \R\times \{1\}) \rightarrow (N, L_0, L_1 ) , 
$$
be a smooth map such that $ diam( u(Q) ) \leq r_1 $ then 
$$
 \int\limits_Q u^* \omega \leq c \ell^2 ( \left.u \right|_{\Gamma} ) ) .
$$
 \end{lemma}
\begin{proof}
We distinguish the following cases 
\begin{itemize} 
 \item [1)] Suppose first that $ Q \cap \R \times \{0\} \neq \emptyset $ 
and $ Q \cap \R \times \{ 1\} \neq \emptyset $. Let $ z_i \in Q \cap \R \times \{ i \} $. Let $ \delta > 0 $ be 
so small that $ 2 \delta $ is still a Darboux radius. Let $ \rho > 0 $ 
be the corresponding constant as in Lemma \ref{lem:mon_lem3} and 
take $ r_1 < \min \{ \delta, \rho\} $. 
As $ d ( u(z_0), u(z_1)) \leq r_1 < \rho $ it follows from 
Lemma \ref{lem:mon_lem3} that there exists $ x_0 \in \Lambda $ such that 
$ u(z_i) \in B_{\delta} ( x_0) $. Thus it follows that 
$ u ( Q) \subset B_{\delta + r_1} ( x_0) \subset B_{2\delta} ( x_0) $. 
We use the Darboux chart to map $ u( Q) $ into $ \R^{2n} $. 
\item[2)] Suppose now that $ Q \cap \R \times \{ 0 \} \neq \emptyset $ and 
$ Q \cap \R \times \{1\}= \emptyset $. The case $  Q \cap \R \times \{ 0 \} = \emptyset $ 
and $ Q \cap \R \times \{1\}\neq \emptyset $ is analog. In this case we can use the 
Lagrangian- Darboux local chart to map $ u( Q) $ into $ \R^{2n} $. 
\item[3)] If $ \R \times \{0,1\} \cap Q= \emptyset $ we can use the standard Darboux chart 
to map $ u( Q) $ into $ \R^{2n} $. 
\end{itemize}
In any of the three mentioned cases the value of $\int_Q u^* \omega $ 
won't change using the symplectomorphism and the length 
will change only up to a constant. 
Hence the assertion follows from the Corollary
\ref{cor:mon_cor1}.

\begin{figure}[htp]
\centering
 \scalebox{0.4}{\input{dom_mon.pstex_t}}
 \caption{Domain $ Q $.}
 \label{fig5}
\end{figure}

\end{proof}

\section{Exponential decay}

It is well known that holomorphic curves with bounded 
energy decay exponentially. Still, in the existing literature 
we weren't able to find the right reference for the case of 
holomorphic curves with boundary on Lagrangians that intersect cleanly
and in the case that $ J_t $ is only tame almost complex structure. 
For the case of transverse intersection of Lagrangian submanifolds
and compatible almost complex structure we refer 
to \cite{RS1}. For the sake of completeness we include the missing case 
in this section. 

The setup is the same as in the introduction and 
we shall prove that bounded energy
 solutions of \eqref{eq:jhol} decay exponentially. 
 In the case of infinite strip this means that
$$ 
p = \lim\limits_{s\rightarrow \infty} u(s,t) 
$$ 
exists and is an intersection point 
$ p \in \Lambda= L_0 \cap L_1 $. 
The convergence will be exponential with all the derivatives.
In the case of finite strips $ u : I \times [0,1] \rightarrow N $
we shall see that $ u(s,t) $ is close to some point 
in $ \Lambda $ 
for those $s $ far away from  $ \partial I $. 
\\

To state this more precisely, for any interval 
$I= [a,b)$  or $I= [a,b] $ denote with 
$ d ( s, \partial I ) $ the distance 
between $ s $ and the boundary of $ I $. 
Let 
\begin{equation}\label{eq:dr_I}
 D_r(I) := \Big \{ (s,t) \in I : 
\; d(s, \partial I ) \geq r , \;t \in [0,1] \Big \}. 
\end{equation}

Notice that in the case $ I= [a,b]$ we have that
\[D_r ( [a,b]) = \begin{cases} 
                  [a+r, b-r] \times [0,1] , \;\; r \leq (b-a) /2 \\
                  \emptyset , \;\; r > ( b-a) / 2
                \end{cases} 
\]
for $ I = [a, +\infty) $, we have that $D_r(I) = [a+r, +\infty ) \times [0,1] $. 
Fix a Riemannian metric $ g $ on $ M$ and denote by 
$ d $ be the distance induced by $ g $.  \\

\begin{proposition}\label{pr:exp-dec}
Assume ~(H) and ~(H1).
There exists $ \mu> 0 $ such that for all 
$ 0 < r_0  $  there exist $ \delta> 0 $ and 
$ c_i, \;i=0,1,2 $ with the following properties:
Assume that  $ u : I \times [0,1] \rightarrow N $ 
is a solution of \eqref{eq:jhol} with 
$$ 
E(u)< \delta ,
$$ 
then $ u $ satisfies the following:
\begin{align*}
   & i) \;\; E(u|_{D_r(I) } ) \leq c_0 e^{-2\mu r} E(u) ,  \\
  & ii) \;\; \abs{\partial_s u}_g \leq  c_1 e^{-\mu r } \sqrt{E(u)},
  \text { for all }\; (s,t) \in D_{r+ r_0} (I),   \\
   & iii) \;\;  \sup\limits_{(s_1, t_1), ( s_2, t_2 ) \in D_{r+ r_0}(I) } 
   d(u(s_1, t_1),u(s_2, t_2 )) < c_2  e^{-\mu r } \sqrt{E(u)}
\end{align*}
for all $ r\geq r_0 $.
\end{proposition}
\begin{proof}
 See subsection \ref{subsec:exp}. 
\end{proof}

Direct corollaries of the previous proposition are the following:
\begin{corollary}\label{ch1_cor0.6}
Assume ~(H) and ~(H1). 
 There exists a positive constant 
 $ \mu $ such that the following holds. 
 Assume that $ u : \R^{\pm} \times [0,1] \rightarrow N $ 
 is a solution of \eqref{eq:jhol} with finite energy, 
 $E(u) < +\infty$. Then
\begin{itemize}
 \item[i)] $u $ converges toward some point 
 $p\in \Lambda $, i.e. the limit
 $$ 
 \lim\limits_{s\rightarrow \infty } u(s,t) 
 = p\in \Lambda = \; L_0\cap L_1 .
 $$
\item[ii)] The convergence is exponential, 
i.e. there exist positive constants $ r_0, \; d_0 $ and 
$ d_1 $ that depend on $ E(u) $ such that
\begin{align}
\abs{ \partial_s u (s,t)} 
\leq d_0 e^{-\mu \abs{s} }  \notag \\
 d( u(s,t), p) \leq d_1 e^{-\mu \abs{s}}  
\end{align}
for all $ \abs{s} \geq r_0 $ and all $ t \in [0,1] $. 
\end{itemize}
\end{corollary}

\begin{corollary}\label{cor:exp_dec}
 For every $ \delta > 0 $ and for every $ r_1 > 0 $ 
there exist $ \hbar > 0 $ such 
that the following holds for any holomorphic curve
 $ u : I \times [0,1] \rightarrow N \subset M $. 
 If $ E(u) < \hbar $ then there exists
 $ x_0 \in \Lambda $ such that 
 \begin{align}
  d(u(s,t), x_0 ) < \delta, \;\; ( s,t) \in D_{r_1} ( I )
 \end{align}

\end{corollary}
\begin{proof}
 From proposition \ref{pr:exp-dec} we have that 
 $$ 
d  ( u( s_1, t_1 ), u(s_2,t_2)) < c e^{ - \mu r_1} \sqrt{ E(u)}
$$
 for all $ (s_i,t_i) \in D_{r_1} (I), \; i=0,1 $.
 Taking the energy sufficiently small 
 the right hand side of the previous inequality
 can be made arbitrary small.
 Suppose that $  c e^{ - \mu r_1} \sqrt{ E(u)} < \rho $, where 
 $ \rho > 0 $ is taken so small that the inequality
 $ d (u(s,0),u(s,1)) < \rho , \;\; (s,0) \in D_{r_1} (I)$ 
 implies that there exists $ x_0 \in \Lambda $ such that 
 $ d( x_0, u(s,0)) < \delta/2 $.
 Such $ \rho $ exists from lemma \ref{lem:mon_lem3}. 
 We may assume w.l.o.g. that $ \rho < \delta / 2 $. 
   Now we have 
 $$
 d( u(s_1,t_1), x_0 ) < 
d( u(s,0), x_0) + d ( u(s,0),u(s_1,t_1)) < \delta / 2 + \rho < \delta .
 $$
 and the previous inequality holds for all $ ( s_1,t_1) \in D_{r_1} (I) $. 
\end{proof}

 \begin{proposition}\label{prop:mon3}
 There exists a positive constant
 $ \mu $ such that the following holds. 
 Assume that
 $ u :I \times [0,1] \rightarrow M $ 
 is a solution of \eqref{eq:jhol}
 with finite energy $ E(u) < +\infty $.
 Then there exist constants $ c_k, \mu$ and $ r_k$, 
 which depend on $ E(u) $ such that
 \begin{equation}
 \|\partial_s u \|_{C^k ( D_r ( I))} \leq c_k e^{-\mu r } 
 \end{equation}
 for all $ r \geq r_k $.
 \end{proposition}
 The main ingredients of the proof of Proposition 
 \ref{pr:exp-dec} are the isoperimetric inequality and
 the mean value inequality. 
 We have proved the isoperimetric inequality in this setup 
 and it is left to prove the mean value inequality what we do in 
 the next subsection. 
 \subsection{The mean value inequality}
 The mean value inequality claims that a value of a certain function at a point 
 can be estimated from above by its mean 
 value. 
 $$ 
f(p) \leq \frac{c}{Vol(B_r(p))} \int\limits_{B_r(p)} f(x) dx .
$$
For example subharmonic functions satisfy the mean value inequality. 
We shall construct a function which is not subharmonic, but whose 
Laplacian can be estimated from below by a polynomial of degree 2. 
Such a function will satisfy a generalized mean value inequality. 
Before defining the desired function, we construct an appropriate family of metrics 
adapted to the Lagrangians. 
\begin{lemma}\label{mon_lem7}
There exists a smooth family of metrics 
$ g_t $ such that the following holds
\begin{itemize}
 \item [1)] $ L_i $ are totally geodesic with respect to $ g_i $ for $ i=0,1 $. 
 \item[2)] $g_t $ is compatible with $ J_t $ for all $ t \in [0,1] $.  
 \item[3)] $ J_i(p) T_pL_i $ is orthogonal complement of $ T_p L_i $ for 
 all $ p \in L_i $ and for $ i =0,1 $. 
\end{itemize}
\end{lemma}
\begin{proof}
  In \cite{Fr} is constructed a metric $ g_0 $ such 
that 
\begin{itemize}
 \item [i)] $ L_0 $ is totally geodesic with respect to $ g_0 $. 
 \item[ii)] $g_0 $ is compatible with $ J_0 $ 
 \item[iii)] $ J_0(p) T_pL_0 $ is orthogonal complement of $ T_p L_0 $. 
\end{itemize}
Analogously one can construct 
$ g_1 $ such that $ g_1 $ and $ L_1 $ satisfy the same 
conditions as $ g_0 $ and $ L_0 $. Then the metric $ \tilde{g}_t = ( 1-t) g_0 + t g_1 $ 
satisfies $1 ) $ and $ 3) $, but not $ 2 ) $. Finally taking  
$ g_t = \frac{1}{2} ( \tilde{g}_t  + J_t^T \tilde{g}_t J_t) $ 
we obtain a family of metrics $g_t $ that satisfy also $ 2) $. 
\end{proof}
Let $ g_t $ be a smooth family of metrics as in Lemma \ref{mon_lem7}
and let $ u : I \times [0,1] \to M $ be a $J_t $ holomorphic curve, 
i.e. a solution of \eqref{eq:jhol}. 
We define a smooth function $ e : I \times [0,1] \rightarrow \R $ as follows
\begin{align}\label{eq:func_e}
 e : I \times [0,1] \rightarrow \R, \;\; e(s,t) = g_t( \partial_s u(s,t), \partial_s u (s,t) ) = \abs{\partial_s u}^2_{t}.
\end{align}
Denote with 
$$ 
H_r(s,t) = B_r(s,t) \cap ( I \times [0,1] ), 
$$
where $ B_r(s,t) $ denotes a ball of radius $ r $ centered at the 
point $ (s,t)$. 
\begin{proposition}\label{pr:meanv}
 Let $ e $ be as in \eqref{eq:func_e}. 
 There exist positive constants $ \tilde{\mu} $ 
 and $ C $ such that for all $ r < \frac{1}{2} $ if 
 $$ 
 \int\limits_{H_r(s,t)} e(\xi, \tau ) d\xi d\tau < \tilde{\mu} ,
 $$
then  
\begin{equation}\label{eq:pr_meanv}
 e(s,t) = |\partial_s u|_t
 \leq C ( 1 + \frac{1}{r^2})\int\limits_{H_r(s,t)} e(\xi,\tau) d\xi d\tau,
\end{equation}
for all 
$ (s,t) \in D_r(I) 
= \{ (s,t) \in I \times [0,1] : \;  d(s, \partial I ) \geq r, \; t \in [0,1] \}$. 
\end{proposition}
\begin{proof}
In Lemma \ref{lem:mon_lem6} we prove that the function 
$e$ satisfies the following inequalities : 
$$\triangle e = \partial^2_s e + \partial^2_t e \geq -C_1( e + e^2), $$
and the normal derivative
$\displaystyle \left | \frac{\partial  e}{\partial t} \right |_{t=0,1} \leq C_2 \cdot e  ,$
for some positive constants $ C_1 $ and $ C_2 $. Thus the 
Claim follows from Theorem \ref{thm:kath_meanv} which was 
proved in \cite{We}. 
\end{proof}

Denote the intersection of an Euclidean ball 
with the half space by
$$ 
\mathbb{D}_r(x)= B_r(x) \cap \mathbb{H}^n \;\;,\; 
\mathbb{H}^n:= \{ (x_0,\bar{x}): x_0 \in [0,+\infty) ,
\; \bar{x}\in \mathbb{R}^{n-1}\}
$$
\begin{theorem}\label{thm:kath_meanv}\cite{We}
 For every $n\geq 2$ there exists a constant $ D $ and for all
$a,\;b$ there exist $\mu(a,b) >0 $ such that the following holds:
Consider ( partial) ball $\mathbb{D}_r(y)\subset \mathbb{H}^n$ for some
$r\geq 0$ and $y\in \mathbb{H}^n$. Suppose that $e\in C^2 ( \mathbb{D}_r(y),
[0,+\infty )) $ satisfies for some $A_0,\;A_1,\;B_0,\; B_1 \geq 0
$
\begin{equation}\label{eq:con_meanv}
\left\{
\begin{array}{ll}
\triangle e \geq -(A_0 +A_1 e + a e^{\frac{n+2}{n}})\;,\\
\frac{\partial}{\partial \nu}|_{\partial \mathbb{H}}\; e \leq B_0
 + B_1 e + b e^{\frac{n+1}{n}},
\end{array} \right.
\end{equation}
and 
$$ 
\displaystyle \int\limits_{\mathbb{D}_r(y)} e \leq \mu(a,b) 
$$
\noindent Then
$$ 
e(y)\leq D\Big ( A_0 + B_0 r + ( A_1^{\frac{n}{2}} + B_1^n + r^{-n})
\int_{\mathbb{D}_r(y)} e\Big ). 
$$
\end{theorem}
We prove that the function $ e $ as in \eqref{eq:func_e} 
satisfies the conditions \eqref{eq:con_meanv}. 
\begin{lemma}\label{lem:mon_lem6}
 Let $ e $ be as in \eqref{eq:func_e}. There exist positive 
 constants $ C_1 $ and $ C_2 $ such that 
\begin{itemize}
 \item[i)] $ \triangle e = \partial^2_s e + \partial^2_t e \geq -C_1( e + e^2).$
\item[ii)] $\displaystyle \left | \frac{\partial  e}{\partial t} \right |_{t=0,1} \leq C_2 \cdot e  .$
\end{itemize}

\end{lemma}
\begin{proof}
 Let $\xi(s,t)=\partial_s u(s,t) $ and $ \eta(s,t)=
\partial_t u(s,t) $. Let $ \nabla^t $ 
be Levi-Civita connection of the metric $ g_t $.
In order to make the notation less cumbersome, 
we shall write further $ \nabla $ instead of $ \nabla^t $.
Because of the compatibility we have $
\abs{\xi}_{t}= \abs{\eta}_{t} $ and as the connection is
Levi- Civita we have $ \nabla_t \xi = \nabla_s \eta $. 
Next, we have the following
\begin{align*}
\partial_s e(s,t) = 2g_{t} ( \nabla_s \xi, \xi), \;\; \;\;\;
\partial_t e(s,t) = 2g_{t} ( \nabla_t \xi, \xi) + (\partial_tg_{t}) ( \xi,\xi)
\end{align*}
\noindent and also
\begin{align*}
 \partial^2_s e(s,t)&=2g_{t} ( \nabla_s \xi, \nabla_s \xi ) + 
 2g_{t} ( \nabla_s \nabla_s \xi, \xi )= a_1(s,t) + a_2(s,t) \\
\partial^2_t e(s,t)&= 2g_{t} ( \nabla_t \xi, \nabla_t \xi ) 
+ 2g_{t} ( \nabla_t \nabla_t \xi, \xi ) + 
(\partial^2_t g_{t} ) (\xi, \xi) 
+2 \partial_t g_{t} ( \nabla_t \xi, \xi)\\
& = \sum\limits_{i=1}^4 b_i(s,t)
\end{align*}
\noindent The functions $ a_i(s,t) $ and $ b_i(s,t) $ correspond
to the summands on the left side of the equalities respectively. 
Notice that there exist a constant $ \bar{c}_1 $ such that 
\begin{align*}
\abs{b_3}   \leq \bar{c}_1|\xi|^2_{t} , \;\;
\abs{b_4}  \leq \bar{c}_1 ( \epsilon^2 \abs{\nabla_t \xi}^2_{t} +
\frac{1}{\epsilon^2} \abs{\xi}^2_{t} )
\end{align*}
\noindent In order to estimate $ a_2 $ and $ b_2 $ notice that:
 \begin{align*}
 \nabla_s\nabla_s \xi + \nabla_t\nabla_t\xi 
 & = \nabla_s(\nabla_s \xi +\nabla_t
\eta)+ \nabla_t\nabla_s \eta - \nabla_s \nabla_t\eta \\&
= \nabla_s (\nabla_s \xi +\nabla_t\eta) - R(\xi, \eta)\eta 
\end{align*} 
And we have
\begin{align*}
 \nabla_s \xi + \nabla_t \eta &= \nabla_s( -J_{t}(u)
\eta) + \nabla_t(J_{t}(u) \xi ) \\
&= -\nabla_s(J_{t}(u)) \eta-J_{t}(u) \nabla_s \eta
+\nabla_t(J_{t}(u) ) \xi + J_{t}(u) \nabla_t \xi \\
&= \nabla_t(J_{t}(u) ) \xi -\nabla_s(J_{t}(u)) \eta \\
&= (\partial_t J_{t} )  \xi + (\nabla_{\eta} J_{t} ) \xi -(\nabla_{\xi} J_{t}) \eta  
\end{align*}
It follows
\begin{align*}
\nabla_s (\nabla_s \xi + \nabla_t\eta ) &= \nabla_s \Big ( ( \partial_t J_{t} )  \xi + 
(\nabla_{\eta} J_{t} ) \xi - (\nabla_{\xi} J_{t}) \eta  \Big ) \\ 
&= (\nabla_s \partial_t J_{t} ( u ) ) \xi + \partial_t J_{t} ( u ) \nabla_s \xi + 
\nabla_s ( \nabla_{\eta} J_{t} ( u ) ) \xi + (\nabla_{\eta} J_{t } ) \nabla_s \xi \\
&  - \nabla_s ( \nabla_{\xi} J_{t} ) \eta  - ( \nabla_{\xi} J_{t} ) \nabla_s \eta  \\ &= \sum\limits_{i=1}^ 6 s_i (s,t ) 
\end{align*}
\noindent Where the terms $s_i$ correspond
to the summands in the preceding row respectively.
We write $ S_i = \langle s_i, \xi \rangle_{s,t} $. 
There exists a constant $\bar{c} $, which depends 
on $g_t$ the almost complex structure $ J_t $ and its derivatives, 
such that:
\begin{align*}
  & \abs{S_1}  \leq \bar{c}  \Big (  \abs{\xi}^2_{t} + \abs{\xi}^4_{t} \Big ), \;\;\; \abs{S_2} \leq \bar{c} \Big (\frac{1}{\epsilon^2} \abs{\xi}^2_{t}+ 
\epsilon^2 \abs{\nabla_s\xi}^2_{t}\Big ),  \\  &
 \abs{S_3} \leq \bar{c} \Big (  \abs{\xi}^2_{t} + \abs{\xi}^4_{t} \Big ) + \bar{c} \Big ( \epsilon^2 \abs{\nabla_t \xi } ^2 + 
\frac{1}{\epsilon^2} \abs{\xi}^4 \Big ),\; \;\;\;\abs{S_4} \leq \bar{c}\Big ( \frac{1}{\epsilon^2 }\abs{\xi}^4_{t}+ 
\epsilon^2 \abs{\nabla_t\xi}^2_{t}\Big ) \\  &  \abs{ S_5} \leq \bar{c} \Big (  \abs{\xi}^2_{t} + \abs{\xi}^4_{t} \Big ) + 
\bar{c} \Big ( \epsilon^2 \abs{\nabla_s \xi } ^2 + \frac{1}{\epsilon^2} \abs{\xi}^4 \Big ), \; \;\;\;\abs{S_6} \leq 
\bar{c} \Big (\frac{1}{\epsilon^2} \abs{\xi}^4_{t}+ \epsilon^2 \abs{\nabla_t\xi}^2_{t}\Big).
\end{align*}
 \noindent For $ \epsilon $ sufficiently small we get that 
 $$ 
 S= \abs{a_2} + \abs{b_2} + \abs{b_3} + \abs{b_4} 
 \leq 2 \abs{\nabla_s \xi }_{t} ^2 
 + 2\abs{\nabla_t \xi }_{t} ^2 + 
 C_1 ( \abs{\xi}_{t}^2 + \abs{\xi}_{t}^4 ) ,
 $$ 
 for some $ C_1 > 0 $. Thus
 \begin{align}
 \triangle e  &= 
 \sum\limits_{i=1}^2 (a_i + b_i ) + b_3 + b_4 \geq a_1 + b_1 -
 S \notag \\ 
 & \geq - C_1( \abs{\xi}_{t}^2 + \abs{\xi}_{t}^4 ) = -C_1 ( e + e^2 ).
 \end{align}
We will prove now the
second inequality in lemma \ref{lem:mon_lem6}. 
\begin{align}
  \frac{\partial e}{\partial \nu} = \frac{\partial e}{\partial
t}\Big |_{t=0,1}= \Big (2 \overbrace{ g_t( \nabla_t \xi,
 \xi )}^{A} + \overbrace{\partial_t g_t (\xi,  \xi)}^{B}\Big ) \Big |_{t=0,1} 
\end{align}
\begin{align*}
A&= g_t (\nabla_t\xi,\xi )= g_t( \nabla_s \eta, \xi ) =  g_t( \nabla_s(
J_{t} (u) \xi) , \xi )\\
& = g_t(J_t(u) \nabla_s \xi, \xi ) +
g_t((\nabla_{s} J_t(u))\xi,  \xi )
\end{align*}
Both terms in the last equality vanish at 
the time $ t=0,1 $. The first one vanishes as
$ L_i , \; i=0,1 $ are totally geodesic for $ g_i $ , 
thus $ \nabla_s \xi \in T_p L_i, \; p= u(s,i) $ 
and 
$ J_i T_p L_i $ 
is orthogonal complement of $ T_p L_i $. 
The second term vanishes as $ \nabla_s J_t $ is 
skew adjoint, what follows by differentiating
$ g_t(J_t(u) v, w) =  - g_t( v, J_t(u) w )$ in the direction of 
$\partial_su$. 
Thus, we have that
\begin{align*}
\displaystyle
\left | \frac{\partial e}{\partial \nu} \right | = 
\left | \left . \frac{\partial e}{\partial t }\right|_{t=0,1}\right |= 
\left |B\right|= \left| \partial_t g_t (\xi,  \xi)\right| \leq C_2 \abs{\xi}^2_t. 
\end{align*}
\end{proof}

\subsection{Exponential convergence}\label{subsec:exp}

In order to prove the mean value inequality
 we have used special metric,
 but the inequality \eqref{eq:pr_meanv} holds no 
matter which metric we use to define the norm of $ \partial_s u $. 
Until now we have mentioned three different metrics, 
one was just some Riemannian metric $ g $ on M, 
the other metric is given by pairing $ \omega $ and $ J_t $, i.e. 
$\langle \xi, \eta \rangle_{J_t}
= \frac{1}{2} \Big (\omega ( \xi, J_t ( \eta )) + 
\omega ( \eta, J_t \xi )\Big) $   
and the third one is the metric constructed in Lemma \ref{mon_lem7}. 
Denote with $\abs{\cdot}_g, \; \abs{\cdot}_{J_t}$  
and $ \abs{\cdot }_t $ the norms that correspond to these metrics. 
As $ N $ is compact all these metrics are equivalent, i.e. 
there exist positive constants $ K_1 $ and $ K_2$  such that 
\begin{align*}
 \frac{1}{K_1} \abs{\xi}_{J_t} 
 \leq \abs{\xi}_{t} \leq K_1 \abs{\xi}_{J_t} , \quad
  \frac{1}{K_2} \abs{\xi}_{g} 
  \leq \abs{\xi}_{t} \leq K_2 \abs{\xi}_{g}, 
\end{align*}
for all $ t \in [0,1] $. 
If the energy
$ E(u)  = 
\int\limits_{I \times [0,1] } \abs{\partial_su }_{J_t} ds dt$ 
is sufficiently small, more precisely if 
$ E(u) \leq \mu_1 = \frac{1}{K_1^2} \tilde{\mu} $, 
where $\tilde{\mu} $ is the constant from 
Proposition \ref{pr:meanv} then:

\begin{align}\label{eq:mon_pom1}
 \abs{\partial_s u(s,t)}^2_g &  
 \leq K_2^2 \abs{\partial_s u}^2_{t} 
 \leq  C \cdot K_2^2 \Big (1 + \frac{1}{r^2} \Big ) 
 \int\limits_{H_r ( s, t)} \abs{\partial_su }^2_t ds dt \notag \\ 
 &= \tilde{c} \Big (1 + \frac{1}{r^2} \Big ) 
 \int\limits_{H_r ( s, t)}\abs{\partial_su }^2_t ds dt
 \leq \tilde{c} K_1^2 \Big (1 + \frac{1}{r^2} \Big ) 
 \int\limits_{H_r ( s, t)} \abs{\partial_su }^2_{J_t} ds dt \notag \\ 
 & \leq c' \Big (1 + \frac{1}{r^2} \Big ) E(u|_{H_r(s,t)}), 
\end{align}
for $ (s,t) \in D_r ( I ) $. 
From here it follows that the 
length of the curve $ \gamma_s = u(s, \cdot) : [0,1]
 \rightarrow N , \; d(s, \partial I ) \geq r $ is small, 
 provided that the energy $ E(u) $ is small. 
\begin{proof}[ Proof of the proposition \ref{pr:exp-dec}]
Let $ I= [a,b] $ with $ a,b \in \R$ 
(the proof is the same for infinite interval).
Fix $ r_0> 0$, suppose w.l.g. that $ r_0 < 1/2$. 
Let $ \rho_0 $ be the constant from the Lemma \ref{lem:mon_lem4}.
Choose $ \delta $ so small that $ E(u) < \delta $ 
implies that the length $ \ell(\gamma_s) = \ell ( u(s, \cdot)) \leq \rho_0 $
for all $ s, \; d( s, \partial I )\geq r_0 $.
From the equation \eqref{eq:mon_pom1},
it follows that it is enough that
$ E(u) < \min \{ \frac{\rho_0^2}{c'( 1 + 1/r_0^2)},
\frac{\tilde{\mu}}{ K_1^2} \}= \delta $. 
Define
$$ 
e(r) := E(u|_{D_r(I)} ). 
$$
For $ r\geq r_0 $ we have:
\begin{align*}
 e(r) = \iint\limits_{D_r(I)} u^* \omega
 = \iint \limits_{D_r(I)} u^* d\lambda = a( u(b-r, \cdot) ) - a ( u( a+ r, \cdot )),
\end{align*}
this holds as all the curves $ \gamma_s= u(s, \cdot) $ 
are contained in the neighborhood 
where the symplectic form is exact 
$ \omega= d\lambda $ and $ \lambda $ 
vanishes on $ L_0 $ and $ L_1 $. Also,
\begin{align*}
 \dot{e}(r) = 
 -\int\limits_0^1 \abs{\partial_s u(b-r, t)}_{J_t}^2 dt 
 - \int\limits_0^1 \abs{\partial_s u(a+r, t)}_{J_t}^2 dt .
\end{align*}
Using the isoperimetric inequality, Lemma \ref{lem:mon_lem3}, 
and the previous two equalities we get 
\begin{equation}\label{eq_eder}
 \begin{split}
 e(r)& = a( u(b-r, \cdot) ) - a ( u( a+ r, \cdot ))  \\
	&\leq c ( \ell( \gamma_{a+r} )^2 + \ell(\gamma_{b-r})^2 ) \\
	&\leq c \Big (\int\limits_0^1 \abs{\partial_t u(a+r, t)}_{g}^2 dt 
	+ \int\limits_0^1 \abs{\partial_t u(b-r, t)}_{g}^2 dt \Big ) \\
	& \leq c' \Big ( \int\limits_0^1 \abs{\partial_s u(b-r, t)}_{J_t}^2 dt
	+ \int\limits_0^1 \abs{\partial_s u(-a+r, t)}_{J_t}^2 dt )  \\
	&= -c' \dot{e}(r),
\end{split}
\end{equation}
for $r \geq r_0 $. From the inequality \eqref{eq_eder}
follows part $ i) $ 
of the Proposition \ref{pr:exp-dec}
\begin{equation}
 e(r) \leq e(r_0) e^{-2\mu( r-r_0)} \leq c_0 e^{-2\mu r} E(u),
\end{equation}
where $ c_0 = e^{2 \mu r_0}$ and $ \mu = \frac{1}{2c'} $.
Take a point $ (s,t) \in D_{r+ r_0} (I) $ 
and a ball around that point,
$ B_{r_0} ( s,t)\subset D_r(I) $. From the inequality
(\ref{eq:mon_pom1}) we get: 
\begin{align}
 \abs{\partial_s u(s,t)}^2_g &  \leq c' ( 1 + \frac{1}{r_0^2} ) E(u|_{H_{r_0}(s,t)} ) \notag \\
			&\leq c' ( 1 + \frac{1}{r_0^2} ) E(u|_{D_r}) = \tilde{c}_1(r_0)  e(r) \notag\\
                        & \leq \tilde{c}_1(r_0) e^{-2\mu r} E(u) 
\end{align}
From here it follows $ ii)$ as we have 
 \begin{align*}
  \abs{\partial_s u(s,t)}_g \leq c_1 e^{-\mu r} \sqrt{E(u)}. 
 \end{align*}
Notice that the following inequality also holds 

\begin{align}\label{ch1_eq3.5}
  \abs{\partial_s u(s,t)}^2_g &  
  \leq c' ( 1 + \frac{1}{r_0^2} ) E(u|_{H_{r_0}(s,t)} ) 
  \leq \tilde{c}\cdot  e( d(s,\partial I ) - r_0) 				\notag \\
			& \leq c  e^{-2d(s,\partial I ) } E(u),
\end{align}
where the constant $ c$ depends on $ r_0$. 
Using the inequality \eqref{ch1_eq3.5} we can prove part $iii) $.
 Let $ (s_i,t_i)\in D_{r+r_0} (I) , \; i=1,2 $. 
 Suppose that $ d(s_1, \partial I ) = s_1-a $ and $ d(s_2, \partial I) =
 b-s_2 $. We have that 
\begin{align}
 d(u(s_1,t_1), u(s_2,t_2))\leq
 \overbrace{\int\limits_{s_1}^{s_2} \abs{\partial_s u(s, t_1)}_g ds}^{I_1} 
+ \overbrace{\int\limits_{t_1}^{t_2} \abs{\partial_t u(s_2, t}_g dt}^{I_2} \notag  .
\end{align}
and
\begin{align}\label{eq:i1}
 I_1& = \int\limits_{s_1}^{\frac{a+b}{2}} 
 \abs{\partial_s u(s, t_1)}_g ds + 
 \int\limits_{\frac{a+b}{2}}^{s_2} \abs{\partial_s u(s, t_1)}_g ds \notag \\
   & \leq c \sqrt{E(u)} \Big ( \int\limits_{s_1}^{\frac{a+b}{2}} e^{-\mu ( s-a)} ds 
   + \int\limits_{\frac{a+b}{2}}^{s_2} e^{-\mu(b-s)}ds \Big ) \notag \\
 & \leq c \sqrt{E(u)} \Big ( \frac{e^{-\mu (s_1-a)}}{\mu}
 - \frac{e^{-\mu (\frac{b-a}{2})}}{\mu} +
 \frac{e^{-\mu (b- s_2)}}{\mu} - 
 \frac{e^{-\mu (\frac{b-a}{2})}}{\mu} \Big )\notag \\
 & \leq c \sqrt{E(u)} \Big ( \frac{e^{-\mu (s_1-a)}}{\mu}
 + \frac{e^{-\mu (b- s_2)}}{\mu} \Big ) \leq \tilde{c} \sqrt{E(u)} e^{-\mu r}.
\end{align}
The first inequality in \eqref{eq:i1} follows from 
the inequality \eqref{ch1_eq3.5}. 
Similarly we have that $ I_2 $ satisfies the following. 
\begin{align}
 I_2 &= \tilde{c} \int\limits_{t_1}^{t_2} 
 \abs{\partial_s u(s_2, t)}_g dt 
 \leq c \sqrt{E(u)} \int\limits_{t_1}^{t_2} e^{-\mu(b-s_2)} dt \notag \\
 & \leq c  \sqrt{E(u)} e^{-\mu(b-s_2)}  \leq c  \sqrt{E(u)} e^{-\mu r} 
\end{align}
With the previous inequalities 
we have finished the proof of part $ iii) $ of
Proposition \ref{pr:exp-dec}. 

\end{proof}
\section{Proof of the main theorems}\label{sec:proofs}

In this section we prove the Theorems \ref{main_thm}, \ref{thm:mon1}
and \ref{thm:mon2}. The proof of the Theorem \ref{main_thm}
is based on the fact that the isoperimetric inequality holds also for curves 
with Lagrangian boundary conditions what we have proved 
in Lemma \ref{lem:mon_lem4} and \ref{lem:mon_lem5}. 

\begin{proof}[\bf Proof of the theorem \ref{main_thm}]
Let $ g $ be the metric on $M $ obtained by pairing $ \omega $ and  $J$ 
i.e. $g( \xi, \eta )= \frac{1}{2} (\omega ( \xi, J\eta ) + \omega ( \eta, J\xi )) $. Let 
 $ r_0 > 0 $ be such that 
 $$
 r_0 = \text{min} ( \inf\limits_{p\in N} inj ( M, p ), r_1 ) 
 ,$$
 where $ r_1 $ is the constant from Lemma \ref{lem:mon_lem5}. 
 
 There exists a constant $ c_1 $ such that 
 for all $ p \in M $  and $ v, \hat{v} \in T_p M  $ with $
 \| v \| < r_0 $ we have 

\begin{align}\label{eq:const_co}
  c_1\|d\exp_p ( v)\hat{v} \|\geq \|\hat{v} \|.
\end{align}

Let $ \Omega \subset S = \R \times [0,1] $ be a bounded open set
and let $u : \Omega \rightarrow M $ 
be a holomorphic curve that extends continuously to $ \overline{\Omega} $ 
and satisfies \eqref{eq:mon_cond} and \eqref{eq:lagbc}.

For $ (s_0, t_0)\in \Omega $ denote with $p_0 = u(s_0,t_0) $. 
We define the function $ f : \Omega \rightarrow \R $ as 
$$ 
f(s,t) = d( u(s,t),p_0 ).
$$
Let 
$$ \Omega_r = f^{-1} ( [0,r)) = \{ (s,t)\in \Omega | \: d(u(s,t),p_0) < r \} 
$$ 
and let 
$$ \Gamma_r = f^{-1} ( \{ r \} ) = \{ (s,t)\in \Omega| \; d(u(s,t),p_0) = r \} 
.$$
Note that $ \Gamma_0 = \{ (s,t) | \; u(s,t) = p_0\} $ is a finite set of points
(for the proof of this fact have a look at \cite{McDS2} for example).  
 Notice that 
$ f(s,t) = \rho(u(s,t)) $, where $\rho: M \rightarrow \R $ is given by 
$$
\rho(p) = d(p_0,p)= \|\exp_{p_0}^{-1}p \|.
$$
and the last equality holds for points $p$ that satisfy $ d( p, p_0 ) < r_0 $. 
Notice that the function $ f $ is smooth on the set $ \Omega_{r_0} \setminus \Gamma_0 $. 
Let $ v=\exp_{p_0}^{-1} p$, then we have 

\begin{align}\label{eq_cons_dr}
 \abs{d\rho(p) ( \hat{v} ) } &= 
\displaystyle\frac{\abs{\langle d (\exp_{p_0}^{-1})( p) ( \hat{v} ), 
\exp_{p_0}^{-1}p\rangle}}{\sqrt{\langle \exp_{p_0}^{-1} p, 
\exp_{p_0}^{-1} p \rangle }} \notag\\
 &\leq \frac{\| d \exp_{p_0} ^{-1} p ( \hat{v})\| \|v\| }{\|v\|}\notag\\
 &= \| (d\exp_{p_0}v )^{-1} ( \hat{v} )\| \notag \\
 & \leq c_1 \|\hat{v}\| 
\end{align}

where the last inequality follows from \eqref{eq:const_co}. 
From the inequality \eqref{eq_cons_dr} it follows that 

\begin{align}\label{eq:pom_mon1}
 \left |\frac{\partial{f}}{\partial s}\right| \leq c_1 \|\partial_s u \|, \qquad
 \left|\frac{\partial{f}}{\partial t} \right| \leq c_1 \| \partial_t u \|.
\end{align}

and hence 

\begin{align}\label{eq:pom_mon}
 \sqrt{\left |\frac{\partial{f}}{\partial s}\right|^2 +
\left |\frac{\partial{f}}{\partial t}\right |^2}
\leq \sqrt{2} c_1 \| \partial_s u \|
\end{align}

Finally define the function $a(r) $ as follows 
\begin{equation}\label{eq:sym_area}
 a(r) := \int_{\Omega_r} u^* \omega.
\end{equation}
{\bf Step 1:} Let $r_0 $ be as above
and let $r < r_0 $ be a regular value of the function $f$. 
Then the function 
$a$ is differentiable at the point $r$ and furthermore 
\begin{equation}\label{eq:der_syma}
 a'(r) \geq \frac{1}{\sqrt{2}c_1}\ell( \left. u\right|_{\Gamma_r} ) . 
\end{equation}
\begin{proof}
For $ \delta > 0 $ sufficiently small the interval $ [r,r+ \delta ] $ 
consists entirely of regular values of the function $ f $ and we have 
$$
a ( r+\delta) - a(r)= \int\limits_{f^{-1} ( [r,r+ \delta ] )} u^* \omega ,
$$
We reparametrize the set $ f^{-1} ( [r,r+\delta ] )$ 
 by using the gradient flow of 
$f$. Define the rescaled gradient flow $ \phi $  by 

\begin{align*}
 \phi: \Gamma_r \times (r-\delta ,r+ \delta ) \rightarrow f^{-1} ((r-\delta ,r+ \delta) ) , \\
 \phi( \cdot, r )= \mathrm{Id}_{\Gamma_r}, \;\; \partial_{\lambda} \phi( \cdot, \lambda ) 
 = \frac{\nabla f }{\|\nabla f\|^2} ( \phi( \cdot, \lambda ) ). 
\end{align*}
Then 
$f( \phi(\cdot, \lambda ) ) = \lambda $.
Here the gradient and the norm of $ \nabla f $ 
are understood with respect to the standard metric on $ \R^{2n} $. 

Suppose first that $ \Gamma_r $ doesn't intersect the 
boundary $ \partial S $. 
Parametrize $ \Gamma_r $ by a
smooth curve $ \gamma_r : [0,1] \rightarrow \Gamma_r $ and define 
$ \psi : [0,1] \times [r,r+ \delta] \rightarrow \Omega $ 
$$
\psi( \tau, \lambda ) := \phi( \gamma_r(\tau), \lambda ) . $$

Orient $ \Gamma_r $ such that the orientations 
of $ \R^2  $ and $ T_{(s,t)} \Gamma_r \oplus \R \nabla f ( s,t) $
agree. Assume without loss of generality that $ \gamma_r $ 
is an orientation preserving diffeomorphism. 
Then 

\begin{align}\label{eq:mon_1}
 a(r+ \delta ) - a(r) &= \iint_{[0,1] \times [r,r+ \delta]}\psi^* u^* \omega\notag  \\
                     &= \int_r^{r+\delta}\int_0^1 u^* \omega ( \partial_{\tau} \psi,
                       \partial_{\lambda}\psi ) d\tau d\lambda \notag \\
                     &= \int_r^{r+ \delta} \int_0^1 \omega ( du(\psi) ( \partial_{\tau} \psi ), 
		      du(\psi) \partial_{\lambda}\psi ) \notag \\
                     &= \int_r^{r+ \delta} \int_0^1\| du( \psi) \partial_{\tau} 
	      \psi\| \| du(\psi) \partial_{\lambda} \psi\| d\tau d \lambda.  
\end{align}

Here the last equality holds because $ \partial_{\tau} \psi $ and $ \partial_{\lambda} \psi $ 
form a positive basis and $ u $ is $ J -$holomorphic. 
On the other hand we have 
\begin{align}\label{eq:mon_2}
 \| du(\psi) \partial_{\lambda} \psi \| &= \frac{\| du( \psi ) \nabla f ( \psi) \|}{\|\nabla f ( \psi ) \|^2}\notag\\
 &= \frac{\|\partial_s u \partial_s f + \partial_t u \partial_t f\|}{\abs{\partial_s f}^2 + \abs{\partial_t f}^2}\notag\\
 & = \frac{\|\partial_s u \|}{\sqrt{\abs{\partial_s f}^2 + \abs{\partial_t f}^2 }}\notag \\
 & \geq \frac{1}{\sqrt{2} c_1 }. 
 \end{align}

The penultimate equality holds as $u$ is $ J-$holomorphic 
and hence $ \partial_t u= J \partial_s u $ 
and the last inequality holds from the inequality \eqref{eq:pom_mon}. 
Substituting \eqref{eq:mon_2} in \eqref{eq:mon_1} we obtain 
\begin{align}
 a(r+\delta) - a(r)&\geq  \frac{1}{\sqrt{2} c_1} \int_r^{r+\delta} \int_0^1 
 \| du( \psi) \partial_{\tau} \psi\|d\tau d\lambda \notag \\
 & \geq \frac{1}{\sqrt{2} c_1} \int_r^{r+\delta} \ell ( \left.u \right|_{\Gamma_{\lambda}}) d\lambda 
\end{align}
Similarly 
$$ a ( r ) - a ( r- \delta ) \geq \frac{1}{ \sqrt{2} c_1} \int_{r-\delta }^r \ell ( \left.u \right|_{\Gamma_{\lambda}}) d\lambda  
$$
Dividing  by $\delta > 0 $ and taking limit $ \delta \rightarrow 0 $ we obtain 
\begin{align}\label{ch2_eq0.14}
 a'(r) \geq \frac{1}{\sqrt{2} c_1} \ell ( \left. u\right|_{\Gamma_r} ). 
\end{align}

In the case  $\Gamma_r \cap \partial S \neq \emptyset$ we have that 
it either happens one of the situations in Figures \ref{fig3} or \ref{fig4} 
or it happens a mixture of these two cases. 
In the case as in Figure \ref{fig3} we cannot follow the flow for all 
time $ \tau \in [0, 1] $ but only for some time period 
$ I_{\delta}= [\epsilon_{\delta}, T_{\delta}] \subset [0,1] $, 
more precisely 
\begin{align*}
 I_{\delta} = \{ \tau \in [0,1] \Big|\; \lambda \mapsto \phi ( \gamma_r(\tau), \lambda ) 
							    \text{ exists on } [0,\delta] \} 
\end{align*}
The condition that the flow doesn't exist can be violated only 
for $ \tau $ close to $0 $ and $ 1 $. 

\begin{figure}[htp]
\centering
 \scalebox{0.7}{\input{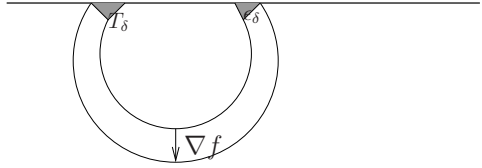}}
 \caption{It is not possible to follow the gradient flow for all times $t \in [0,1]$}
 \label{fig3}
\end{figure}
\begin{align}
 a(r+\delta) - a(r)\geq \frac{1}{\sqrt{2} c_1} 
 \int_r^{r+\delta} \int_{I_{\delta}} \| du( \psi) \partial_{\tau} \psi\|d\tau d\lambda
\end{align}
But still in the limit when $ \delta\rightarrow 0 $ we obtain that
$a'(r) \geq \frac{1}{\sqrt{2} c_1} \ell( \left.u \right|_{\Gamma_r} ) $
as $ I_{\delta} \rightarrow [0,1], \delta \rightarrow 0 $. 
In the second case as at the Figure \ref{fig4} we have that 
following the gradient flow we don't capture the whole area. 
In this case it still holds that 

$$
 a(r+\delta) - a(r)\geq \frac{1}{\sqrt{2} c_1} 
 \int_r^{r+\delta} \int_{0}^{1} \| du( \psi) \partial_{\tau} \psi\|d\tau d\lambda.
$$

Hence we obtain that $ a'(r) \geq \frac{1}{\sqrt{2} c_1} \ell( \left. u\right|_{\Gamma_r} ) $.  

\begin{figure}[htp]
 \centering
 \scalebox{0.7}{\input{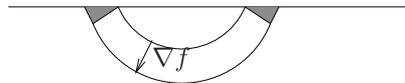}}
 \caption{Following the gradient flow one doesn't
  capture the whole area}
  \label{fig4}
\end{figure}

\end{proof}

{\bf Step 2 :} Let $ a(r) $ be defined as in \eqref{eq:sym_area}. 
There exists a constant $ c_0 > 0 $ such that 
  \begin{equation}
   a(r) = area ( u |_{\Omega_r}) \geq c_0 r^2 ,
  \end{equation}
for all $ r \leq r_0 $. \\
{\bf Proof:} In Lemma \ref{lem:mon_lem5}
we have proved that that the symplectic area is controlled 
by the length square:
 $$
 a( t ) \leq c \ell^2(u|_{\Gamma_t} ) .
$$
Thus we have that 
\begin{equation}\label{ch2_eq0.18}
 \ell(u|_{\Gamma_t } ) \geq \frac{1}{\sqrt{c}} \sqrt{a(t)}. 
\end{equation}
Substituting the inequality \eqref{ch2_eq0.18} in \eqref{eq:der_syma} we get
\begin{align}\label{eq:der_area}
 ( \sqrt{a(t)} ) ' \geq \frac{1}{2\sqrt{2}\sqrt{c} c_1}= \sqrt{c_0}, 
 \end{align}
for all regular values $ t \in [0, r_0 ]$.
From Sard's theorem we have that the set 
of singular values is compact and zero measure. 
The function $ \sqrt{a(t) } $ is monotone increasing, 
and hence differentiable almost everywhere. 
The set of regular values is open and full measure and 
hence can be written as a countable union of 
disjoint intervals $ I_n = ( \alpha_n, \beta_n) $. 
Integrating the inequality \eqref{eq:der_area} on
each interval $ I_n $ we get

\begin{align*}
 \sqrt{ a ( \beta_n)} - \sqrt{ a( \alpha_n)}  \leq \sqrt{c_0} ( \beta_n - \alpha_n ) = \sqrt{c_0}\mathcal{L} ( I_n ). 
\end{align*}

As 
\begin{align}
 \sqrt{a(r)} & \geq 
\sum\limits_{n=1}^{+\infty} ( \sqrt{a( \beta_n)} - 
\sqrt{ a ( \alpha_n)}) \geq\sqrt{ c_0} \sum\limits_n ( \beta_n - \alpha_n) \\ 
 & \geq\sqrt{c_0} \mathcal{L} ( \bigsqcup\limits_n I_n ) = \sqrt{c_0} r
\end{align}

Hence, we get
\begin{equation}
 a(r) \geq c_0 r^2
\end{equation}

for all $ r\leq r_0 $. 
\end{proof}

\medskip

\begin{proof}[{\bf  Proof of theorem \ref{thm:mon1}}]
Let $ W $ be an open neighborhood of 
$ (\Lambda \cap N)\setminus \Lambda_0 $.
such that $ \overline{W} \cap \overline{V} = \emptyset $. 
Let $ \tM = \R^2 \times M = \C \times M $.
We define compact sets $ \tA, \tB, \tN \subset \tM$ 
and closed $  \tL_0, \tL_1\subset  \tM$ as 
follows
\begin{equation}
 \begin{split}
 & \tA = [-1,1]\times [0,1] \times \overline{V}, \\
 &\tB = [-1,1]\times [0,1] \times ( ( N \setminus U ) \cup \overline{W} )\\
 &\tN = [-1,1]\times [0,1] \times N.\\
 &\tL_0 = \R \times \{0\} \times L_0, \qquad \tL_1 = \R \times \{1\} \times L_1
 \end{split}
\end{equation}
The tuple $ (\tM, \tN, \widetilde{L}_0, \widetilde{L}_1) $
 satisfies the assumptions of the Theorem \ref{main_thm}.
Let $ r_0 $ and $ c_0 $ be the corresponding constants, as in 
Theorem \ref{main_thm}. We choose positive constants $ \delta, \hbar $ 
and $ \epsilon $ such that the following holds
\begin{itemize}
\item[(i)]
$0<\rho<d(\tA,\tB)/2$ and $\rho<r_0$.
\item[(ii)] 
$0<\delta<1$ and $\delta < \frac{c_0\rho^2}{4}$.
\item[(iii)]
If $\gamma:[0,1]\to N$ is a smooth curve such that
$\gamma(0)\in L_0$ and $\gamma(1)\in L_1$ and 
$$
L(\gamma):=\int_0^1
\sqrt{\omega(\dot\gamma(t),J_t\dot\gamma(t))}\,dt<\epsilon
$$
then $\gamma([0,1])\subset V\cup W$. 
\item[(iv)]
If $u:[a,b]\times[0,1]\to M$ is a $ J_t $ 
holomorphic curve with energy $E(u)<\hbar$ 
and $\frac{b-a}{2}>\delta$ then 
the path $u_s:[0,1]\to M$ defined by 
$u_s(t):=u(s,t)$ has length $L(u_s)<\epsilon$
for $a+\delta\le s\le b-\delta$. 
\item[(v)]
$\hbar < \frac{c_0\rho^2}{2}$.
\end{itemize}
The existence of a constant $ \epsilon > 0 $ 
in $ \mathrm{(iii)} $ follows from Lemma \ref{lem:mon_lem3}. 
The existence of a constant $\hbar $ as in $\mathrm{(iv )} $ 
follows from  mean value inequality, Proposition \ref{pr:meanv}.

We claim that the assertion of Theorem~\ref{thm:mon1} 
holds with the above constant~$\hbar$. To see this, let 
$u:[a,b]\times[0,1]\to M$ be a $J_t $ holomorphic curve.
Assume first that $\delta < \frac{b-a}{2} $.
Then by (iv) we have $L(u(s, \cdot))<\eps$ for 
$s \in (a+\delta, b- \delta )$. 
Hence by (iii) we have
$
u([a+\delta,b-\delta]\times[0,1])\subset V\cup W.
$
We claim that 
\begin{equation}\label{eq:uV}
u([a+\delta,b-\delta]\times[0,1])\subset V.
\end{equation}
Suppose this is not the case.  Since $V\cap W=\emptyset$ it would then follow
that $u([a+\delta,b-\delta]\times[0,1])\subset W$.
Define $\tu:[b-\delta,b]\times[0,1]\to\tM$ by 
\begin{equation}\label{eq:tu}
\tu(s,t) := (-b+s+\im t,u(s,t)).
\end{equation}
Then $\tu$ takes values in $\tN$ and 
$\tu(b-\delta,t)\in\tB$ and $\tu(b,t)\in\tA$ 
for all $t\in[0,1]$. $\tu $ is also a $ \tJ $ holomorphic curve, 
where $ \tJ $ is given by 
\[
\tilde{J} =
\begin{pmatrix}
0& -1& 0\\
1 & 0& 0 \\
0 & 0& J_t
\end{pmatrix}
\]

Since $d(\tA,\tB)>2\rho$, by (i),
there is an element $(s_0,t_0)\in(b-\delta,b)\times[0,1]$
such that $\tp_0:=\tu(s_0,t_0)$ satisfies
$$
d(\tp_0,\tA)>\rho,\qquad d(\tp_0,\tB)>\rho.
$$
Since $\rho<r_0$, by~(i), it follows from 
Theorem~\ref{main_thm} that
\begin{equation}\label{eq:Ec0rho}
E(u;[b-\delta,b]\times[0,1]) +\delta
= E(\tu;[b-\delta,b]\times[0,1])
\ge c_0\rho^2.
\end{equation}
Hence
\begin{equation}\label{eq:hbar}
\hbar > E(u;[b-\delta,b]\times[0,1])
\ge c_0\rho^2-\delta > \frac{c_0\rho^2}{2}.
\end{equation}
This contradicts~(v).  Thus we have proved~\eqref{eq:uV}.

Next we claim that 
\begin{equation}\label{eq:uU}
u([a,b]\times[0,1])\subset U.
\end{equation}
If this does not hold we obtain a contradiction as above. 
Namely, there is a point $(s_0,t_0)\in[a,b]\times[0,1]$ 
such that 
$$
u(s_0,t_0)\notin U.  
$$
By~\eqref{eq:uV}, we must have
$s_0\in (a, a+ \delta) \cup ( b-\delta, b)$. 
 Suppose without loss of generality
that $s_0>b-\delta$ and define $\tu:[b-\delta,b]\times[0,1]\to\tN$ 
by~\eqref{eq:tu}. Then $\tu$ takes values in $\tN$ and
$$
\tp_0:=\tu(s_0,t_0)\in\tB,\qquad
\tu(b-\delta,t),\tu(b,t)\in\tA,
$$
for all $t\in[0,1]$. Since $d(\tA,\tB)>\rho$ and $\rho<r_0$ 
by~(i), it follows again from Theorem~\ref{main_thm} that
$u$ satisfies~\eqref{eq:Ec0rho} and~\eqref{eq:hbar},
in contradiction to~(v).  Thus we have proved~\eqref{eq:uU}
in the case $\delta< \frac{b-a}{2}$. 

If $\delta \geq \frac{b-a}{2}$ the map $\tu:[a,b]\times[0,1]\to\tM$ defined
by~\eqref{eq:tu} takes values in $\tN$, has energy 
$E(\tu)=b-a + E(u) \leq 2\delta+\hbar < c_0\rho^2$,
and maps the set $\{a,b\}\times[0,1]$ to $\tA$.
Hence it follows again from Theorem~\ref{main_thm}
that its image cannot intersect $\tB$ and so $u$ satisfies~\eqref{eq:uU}.
This proves Theorem~\ref{thm:mon1}.

\end{proof}

\begin{proof}[{\bf Proof of the theorem \ref{thm:mon2} }] 
Suppose that the claim isn't truth, i.e. suppose that 
there exists $ \epsilon > 0 $ and sequences 
$ \hbar_n \rightarrow 0 $ and $ I_n =[a_n, b_n ] $
and $J_t $ holomorphic curves $u_n:I_n \times [0,1] $  and points 
$x_n, y_n \in \Lambda $ such that 

$$ 
E(u_n ) < \hbar_n , \; \sup\limits_t d ( u_n ( a_n, t ), x_n ) < \epsilon/12 , 
\;\sup\limits_t d ( u_n ( b_n, t ), y_n ) < \epsilon/12
$$ 
but there exist $ (s_n, t_n ) \in I_n \times [0,1] $ 
such that $ p_n = u_n (s_n, t_n ) \notin B_{\epsilon} (x_n ) $. 
We make the same construction as in the proof of Theorem
\ref{thm:mon1}, i.e. we observe $ \tM = \C \times  M $, 
$ \tN = [-1,1]\times [0,1] \times N $ 
with Lagrangian submanifolds 
\begin{equation*}
\widetilde{L}_0 = \R\times \{0\} \times L_0, \qquad \widetilde{L}_1 = (\R \times \{1\} )\times L_1.
\end{equation*}
The tuple $ (\tM, \tN, \widetilde{L}_0, \widetilde{L}_1) $ 
satisfies the assumptions of the theorem \ref{main_thm}.
Let $ r_0 $ and $ c_0 $ be the constants as in Theorem \ref{main_thm}.
Let $ K_2 $ be positive constant such that 
\begin{align}
 \frac{1}{K_2} \abs{\xi}_g \leq \abs{\xi}_{J_t} \leq K_2 \abs{\xi}_g.
\end{align}
Take $r = \min \{ r_0 , \frac{\epsilon}{24K_2}\}$
and $ r_1= c_0 \frac{r^2}{2} $. 

We shall construct a new sequence $ q_n= u_n ( \ts_n, \tilde{t}_n)$ such that
$ B_{\frac{\epsilon}{12} } ( q_n ) $ contains no boundary points 
$ u_n|_{\partial I_n  \times [0, 1]} $ as well as the points 
$u_n(s,t) , \; (s,t ) \in D_{r_1}(I_n)$, where 
$ D_{r_1} ( I_n ) $ is defined as in \eqref{eq:dr_I}. 
From Corollary \ref{cor:exp_dec} follows that starting from some 
$n_0 $, for $ n \geq n_0 $ there exists 
$c_n\in \Lambda $ such that 
$$
\sup\limits_{s,t\in D_{r_1}( I_n) } d ( u_n(s,t), c_n) < \frac{\epsilon}{12}.
$$ 
One of the following cases occurs

\begin{itemize}
\item[1)] $ y_n, c_n \notin B_{\epsilon/ 6} (p_n) $. 
Then  $ q_n = p_n , \ts_n = s_n, \; \tilde{t}_n = t_n$ and  
$ B_{\epsilon/12} (q_n)$ 
doesn't contain points $u_n(s,t )$ for $ (s,t) \in D_{r_1}(I_n) $ and
$ u_n(s,t ), \; (s,t) \in \partial I_n \times [0,1] $

\item[2)] $y_n, c_n \in B_{\epsilon/ 6} (p_n) \cup B_{\epsilon/6} ( x_n ) $. 
Then $ B_{\epsilon/4} ( p_n ) \cup B_{\epsilon/4} ( x_n ) $ 
contains all points $ u_n(s,t) , \; (s,t) \in D_{r_1}(I_n) $ 
and all boundary points $ u|_{ \partial I_n \times [0,1] } $. 
In this case there exists a point $q_n = u_n ( \ts_n , \tilde{t}_n ) $ 
such that  $ d ( q_n, x_n )\geq \frac{\epsilon}{3} $  and 
$ d(q_n , p_n ) = \epsilon / 2 $ as at the Figure \ref{fig2a}.

\begin{figure}[htp]
\centering
 \scalebox{0.8}{\input{mon1.pstex_t}}
 \caption{}\label{fig2a}
\end{figure}

Such a point exists as the image of $u$ contains 
points from $ B_{\frac{\epsilon}{12}} (x_n ) $ 
and also $ p_n $, thus the ball centered at 
$ p_n $ with radius $ \frac{ \epsilon}{2}$ 
will intersect its image. 
Obviously  $ B_{\epsilon/12}(q_n)  $ contains no 
boundary points and no points $ u(s,t) , \; (s,t ) \in D_{r_1}(I_n) $.

\item[3)] If for example $ y_n \in B_{\epsilon/ 6} (p_n) $ and 
$ c_n \notin B_{\epsilon/ 6} (p_n) \cup B_{\epsilon/6} ( x_n ) $ 
( the reverse is equivalent), 
then choose again points $ (s'_n, t'_n ), \; q'_n =  u_n (s'_n, t'_n ) $
 such that $ d(q'_n, p_n ) = \frac{5 \epsilon}{12} $ and 
$  d(q'_n, x_n ) \geq \frac{\epsilon}{2} $. Then we have two cases \\ 
\indent a) If $c_n \notin B_{\epsilon/6} (q'_n )$ then $ q_n = q'_n $, 
$\ts_n = s_n' $ and $ \tilde{t}_n = t_n' $.  \\
\indent b) If $c_n \in B_{\epsilon/6} (q'_n )$, 
then choose points $(s''_n, t''_n ) , \; q''_n = u_n ( s''_n, t''_n )$
 such that $ d(q''_n, q'_n )=\frac{\epsilon}{3} $ and
 $ d(q''_n, x_n)\geq \epsilon / 6$. 
Then $ B_{\epsilon/ 12} ( q''_n )$ contains no 
boundary points as well as points $ u(s,t) , \; (s,t ) \in D_{r_1}(I_n) $ 
and $ q_n = q''_n $  and $ \ts_n = s_n'' , \; \tilde{t}_n = t_n''$.


\end{itemize}

In the case $ \ts_n \in ( a_n, a_n + r_1 ) $ we define 
$ \tu_n : [ a_n, a_n + r_1 ] \times [0,1] \rightarrow \tN $ 
by $ \tu_n (s,t) = ( - a_n + s + \im t , u_n (s,t)) $, otherwise 
if $ \ts_n \in ( b_n - r_1, b_n )  $ we define 
$ \tu_n : [ b_n - r_1, b_n] \times [0,1] \rightarrow \tN $ 
by $ \tu_n (s,t) = ( - b_n + s + \im t, u_n (s,t)) $. 
Suppose that $ \ts_n \in ( a_n , a_n + r_1 ) $, the other case is analog. 
We define $ \tq_n = ( - a_n + \ts_n + \im \tilde{t}_n, q_n )$.  Let $ s = a_n $ 
or $ s= a_n + r_1 $ 
then the distance 
$$
 d( \tilde{q}_n, \tu_n (s,t) ) 
= \inf \{ l(\gamma) : \; \gamma :[0,1] \rightarrow\tM ,
\; \gamma(0) =\tilde{q}_n,\; \gamma(1)= \tu_n (s,t) \} .
$$
Any such curve $\gamma$ can be written as
 $ \gamma(t)= (\gamma_1(t), \gamma_2(t))\in \C \times M.$ Then

\begin{align*}
\ell(\gamma) &\geq \frac{1}{2} \Big ( \int\limits_0^1 \abs{\dot{\gamma}_1 (t) }_e dt 
+ \int\limits_0^1 \abs{\dot{\gamma}_2 (t) }_{J_t} dt \Big ) \\ 
& \geq \frac{1}{2} \int\limits_0^1 \abs{\dot{\gamma}_2 (t) }_{J_t} dt 
  \geq \frac{1}{2K_2} \int\limits_0^1 \abs{\dot{\gamma}_2 (t) }_g dt \\
 & \geq \frac{1}{2K_2} d ( u_n(s,t) , q_n ) = \frac{\epsilon}{24K_2}\geq r  
\end{align*}
Thus, we see that 
$\tilde{u}_n^{-1} ( B_{r} ( \tilde{q_n} ) )  $ does not intersect 
the set $ \{ a_n \} \times [0,1] \cup \{ a_n + r_1 \} \times [0,1] $. 
From the monotonicity theorem for $\tilde{J}$ 
holomorphic curves, Theorem \ref{main_thm}, we have 
$$
A= Area ( \tilde{u}_n |_{\tilde{u}_n^{-1} ( B_{r} ( \tilde{q_n} ))} ) \geq c_0 r^2 .
$$ 
On the other hand 
\begin{align*}
A & = \int\limits_{\tilde{u}_n^{-1} ( B_{r} ( \tilde{q}_n ))} \tilde{u}_n^* \tilde{\omega} 
= \int\limits_{\tilde{u}_n^{-1} ( B_{r} ( \tilde{q_n} ))} \omega_{std}
 + \int\limits_{\tilde{u}_n^{-1} ( B_{r} ( \tilde{p_n} ))} u_n^*\omega \\ 
& \leq  \int_0 ^1 \int_{a_n}^{a_n + r_1} \omega_{std} + E(u_n)\\
&  \leq r_1 + E(u_n). 
\end{align*}
Therefore we have that 
$$
 c_0 r^2 \leq A \leq \frac{c_0 r^2}{2} + E(u_n) ,
$$
what is a contradiction as $ \lim\limits_{n\rightarrow + \infty } E(u_n) = 0 .$
\end{proof}

 \chapter{Fredholm theory on truncated surfaces}

This chapter is entirely devoted to the study of various properties
of some linear operators on truncated surfaces. These 
surfaces are essentially similar to half-infinite or finite strips.
We shall consider first the linearized operator $ D_A $ 
of the form 
$$ 
D_A \xi = \p_s \xi + A \xi, 
$$ 
where the operator $A$ is bijective and self-adjoint 
and $s$ independent. In order to establish
Fredholm properties on truncated surfaces we decided 
to work with Hilbert spaces unlike most authors do. 
Thus, the domain of the operator $D_A$ will be some $W^{2,2} $
space on strip with some boundary conditions. This approach 
has various advantages. One of the most important facts 
is that the trace spaces of such Hilbert spaces have particularly 
nice form and they can be described in terms of the 
domain of some power of the operator $A$. 
The operator $D_A$ will not be Fredholm if we 
take its domain to be the mentioned Hilbert space, 
as it will have infinite dimensional kernel, but reducing 
its domain by fixing some boundary conditions which can be expressed in terms of positive and 
negative eigenvector spaces of the operator $A$ we can prove that 
it is actually bijective. We allow later that the operator $A$ 
depends on time $s$, thus we consider the linearized operator 
$ D$ of the form 
$$ 
D\xi = \p_s \xi + A(s) \xi.
$$
The operator $A(s) $ is such that the difference $ D- D_A = K $ 
is a compact operator. Thus the operator $D$ as a compact perturbation of the 
operator $D_A$ inherits most of its properties. Our final goal 
will be to prove surjectivity of the operator $D$.

\section{Linear estimates in abstract setting}\label{SEC:lin_est}

\begin{PARA}[\bf{Hilbert triple and the operator in the time independent case}]\label{para:bij_lin}\rm
In this section we prove some linear elliptic estimates 
that will be crucial for the proof of Theorem \ref{thm:main_thm3}.
and \ref{thm:main_thm2}. We will not specify Hilbert spaces 
in which we work, as we shall allow later various applications 
of the results of this chapter. 
The approach is similar to  \cite{KM} and \cite{RS4}. 
\\
Consider the following three Hilbert spaces 
$ H^2 \subset H^1 \subset H^0 $ such that 
each inclusion is compact, dense and continuous. 
Throughout this section we shall assume the following 
hypothesis
\begin{description}
\item[(HA)] 
 Let $ A: H^1 \rightarrow H^{0}$ be a linear, bijective, 
 and self-adjoint operator with the domain of 
 $A^2$ equal to $H^2$, where 
 $$
  \Dom(A^2)=\{ \xi \in H^1| A(\xi) \in H^1 \}. 
  $$
\end{description}
If $ A$ satisfies ~(HA), then there exist positive constants $ C_j, \; j=1,2$ 
such that the following inequality holds for all $ \xi\in H^j $
 \begin{equation}\label{eq_eqnormA}
  \frac{1}{C_j}\| \xi\|_{H^j} \leq \| A\xi\|_{H^{j-1}} \leq C_j\| \xi\|_{H^{j}}, \;\;\; j=1,2. 
 \end{equation}
 The right side of the previous inequality follows 
 from Hellinger-Toplitz theorem, and the left side 
 from the open mapping theorem. 
\end{PARA}

\begin{PARA}[{\bf Intended Application}]\label{para:int_app}\rm
We shall apply the results from this section to the 
following three Hilbert spaces
\begin{equation}\label{eq:hilspaces}
\begin{split}
 & H^0 = L^2 ( [0,1] ) \\
 & H^1= H^1_{bc} ([0,1])= \left \{ \xi \in H^1 ([0,1], \R^{2n}) \;\bigg | 
                 \; \xi(i) \in \R^n \times \{0\}, \;\; i=0,1 \right \}  
     \end{split}
 \end{equation}
The Hilbert space $ H^1 $ will be actually the domain of some operator 
$A$ of the form $ A= J_0 \p_t + S(t) $. In that case 
the space $ H^2 $ will be just the domain of $A^2$. 
In the proof of Theorem \ref{thm:main_thm3}, the Hilbert space $ H^2 $ 
will be 
\begin{equation}\label{eq:h2}
 H^2 = H^2_{bc} ( [0,1] ) = \left \{ \xi \in H^2 ( [0,1], \R^{2n} ) \; \bigg |
  \begin{array}{l}
         \;\xi(i)  \in \R^n \times \{0\},\\ 
          \partial_t \xi(i) \in \{0\} \times \R^n, \;\; i=0,1
      \end{array}
      \right \}.
      \end{equation}
Another typical example of such triplet is also a Gelfand triple 
 $ V\subset H \subset V^* $. 
 \end{PARA}

\begin{PARA}[{\bf Intermediate Hilbert spaces and projections }]\label{rem:lines}\rm
 Assume ~(HA). Notice that each $ \xi \in H^0 $ 
 can be uniquely written in the form $ \xi = \sum\limits_{i} x_{i} e_{\lambda_i} $,
 where $ e_{\lambda_i} $ are the eigenvectors of the operator $A$ 
 forming an orthonormal basis of $ H^0$. 
 Let 
 $$ 
 E^{1/2}:= [ H^1, H^0] _{1/2}, \quad
 E^{3/2}:= [H^2,H^1]_{1/2} 
 $$
 (for more details on interpolation spaces 
 we refer to Appendix \ref{SEC:app}).
 As $ H^2 = \text{Dom}(A^2) $ it follows that 
\begin{equation}\label{eq:h3/2}
 E^{3/2} = \text{Dom} ( \abs{A}^{3/2}) =
 \Big \{ \xi = \sum\limits_{i} x_{i} e_{\lambda_i} \Big 
| \; \sum\limits_{i}\abs{\lambda_i}^3 \abs{x_{i}}^2 < + \infty \Big \}.
\end{equation}
and analogously can be described $ E^{1/2} $
\begin{equation}\label{eq:h1/2}
 E^{1/2} = \text{Dom} ( \abs{A}^{1/2}) =
 \Big \{ \xi = \sum\limits_{i} x_{i} e_{\lambda_i} \Big 
| \; \sum\limits_{i}\abs{\lambda_i} \abs{x_{i}}^2 < + \infty \Big \}.
\end{equation}
For $ \xi \in E^{3/2} $ we denote by $ \|\xi\|_{3/2} $ 
the norm of $ \xi= \sum\limits_{i} x_{i} e_{\lambda_i} $, 
we have 
$$
\|\xi\|_{3/2}^2 := \sum_{i}  \abs{\lambda_i}^3 \abs{x_{i}}^2.
$$
Let $ E^{\pm}\subset H^0 $ be positive and negative eigenvector spaces
 $$ 
E^{\pm}:=\bigg \{ \xi =
 \sum\limits_{\pm\lambda_i> 0} x_{i} e_{\lambda_i}\in H^0 \bigg\}. 
$$
 Then the Hilbert space $ E^{3/2}$ decomposes as
 \begin{equation}\label{eq:dec}
    E^{3/2} = (E^+\cap E^{3/2}) \oplus (E^-\cap E^{3/2} ) 
 \end{equation}
 Each $ \xi \in E^{3/2} $ can be written uniquely in the form 
$ \xi = \xi^+ + \xi^- $, where $ \xi^{\pm} \in E^{\pm}\cap E^{3/2} $.
 Denote with $ \pi^{\pm} $ the projections 
\begin{equation}\label{eq:proj}
  \pi^{\pm} : E^{3/2}\rightarrow E ^{\pm}\cap E^{3/2}.
\end{equation}
It is crucial that $ \text{Dom}(A^2) = H^2 $, otherwise 
we would not have the decomposition of $E^{3/2} $ 
into its positive and negative subspaces and the projection would not 
be well defined. 
We have an analogous projection 
\begin{equation}
 \pi^{\pm}: E^{1/2} \rightarrow E^{\pm} \cap E^{1/2}.
\end{equation}

\end{PARA}

\begin{PARA}[{\bf Hilbert spaces with strip-like domains}]\label{para:hilsp}\rm
Let $I $ be an interval in $ \overline{\R} $
and let $ H^2\subset H^1 \subset H^0 $ be as in \ref{para:bij_lin}.
Observe the following three Hilbert spaces
\begin{align*}
\cW^0(I) &=  L^2( I, H^0) \\
 \cW^1(I) &= L^2 (I, H^1) \cap W^{1,2} ( I, H^0) \\
  \cW^2(I)&= L^2 ( I, H^2) \cap W^{1,2} ( I, H^1) \cap W^{2,2} ( I, H^0) 
\end{align*}
Equip $\cW^0(I), \cW^1(I) $  and $ \cW^2(I) $ with the following norms
\begin{align*}
 \|\xi\|_{\cW^0(I)}^2 &:= \int_I \|\xi\|_{H^0}^2 ds \\
 \|\xi\|_{\cW^1(I)}^2 &:= \int_I \|\xi\|_{H^1}^2 ds + \int_I \| \p_s \xi\|_{H^0}^2 ds \\
 \|\xi\|_{\cW^2(I)}^2 &:= \int_I \|\xi\|_{H^2}^2 ds + \int_I \| \p_s \xi\|_{H^1}^2 ds + \int_I \| \p^2_s \xi \|_{H^0} ds.  
\end{align*}
As the inclusion $ H^i \hookrightarrow H^{i-1}, \;\; i=1,2$ 
is continuous it follows that there exist constants $c', c'' $ such that 
$$
\| \xi \|_{\cW^0(I)} \leq c' \| \xi \|_{\cW^1 (I)} \leq c'' \| \xi \|_{\cW^2 (I)} 
$$
holds for all $ I \subset \overline{\R} $. 
\end{PARA}
\begin{PARA}\label{para:hil_striplike}
Notice that in the case $ H^i, \; i=0,1 $ 
are as in \eqref{eq:hilspaces}, then the Hilbert space $ \cW^1 (I) $ 
is isometric to the following space
\begin{equation}\label{eq:H1bc}
H^1_{bc} (I \times [0,1] ) :=  
\left\{ \xi \in W^{1,2} (I \times  [0,1], \R^{2n}) \; \bigg |
\xi(s,i) \in \R^n \times \{0\} , i=0,1 \right \}.
\end{equation}
In the case that $ H^2 = \Dom(A)^2 $ the space 
$ \cW^2(I) $ is isometric to the following Hilbert space 
\begin{equation}\label{eq:h2bc}
H^2_{bc} ( I \times [0,1] ) :=
\left\{\xi \in H^2 ( I \times [0,1], \R^{2n} ) \,\bigg|\,
\begin{array}{ll}
\xi(s,i) \in \R^n \times \{0\} , i=0,1 \\
A \xi(s,i) \in \R^n \times \{0\} , i=0,1
\end{array}\right\}.
\end{equation}
In particular, if the space $ H^2 $ is given by \eqref{eq:h2} then 
the space $\cW^2(I) $ is 
\begin{equation}\label{eq:H2bc}
W^{2,2}_{bc} ( I \times [0,1] ) :=
\left\{\xi \in H^2 ( I \times [0,1], \R^{2n} ) \,\bigg|\,
\begin{array}{ll}
\xi(s,i) \in \R^n \times \{0\} , i=0,1 \\
\p_t \xi(s,i) \in \{0\} \times \R^n , i=0,1
\end{array}\right\}.
\end{equation}
The Hilbert space $ \sW^0(I) $ is just the standard 
$ L^2( I \times [0,1],\R^{2n} ) $. 
\end{PARA}
Observe the following linear operator
\begin{equation}\label{eq:DA}
\begin{split}
 & D_A: \cW^i(I) \rightarrow \cW^{i-1} ( I ), \;\;\; i=1,2\\
 & D_A(\xi) = \partial_{s}\xi(s, t)  + A  \xi(s, t),
   \end{split}
\end{equation}
where $A$ satisfies the assumption ~(HA).

\begin{theorem}\label{thm:main_inq}
Let $i=1 $ or $i=2 $ and let $ D_A $ be defined as in \eqref{eq:DA}.
Let  $E^{i-1/2} , E^{\pm} $ and $ \pi^{\pm} $ 
be as  in \ref{rem:lines} and let $ \cW^i(I)$ 
be as in \ref{para:hilsp}. 
\begin{itemize}
 \item[i)] 
 There exists a constant $c_i > 0 $ 
 such that for any interval $ I= [a,b] $ 
 and for all $ \xi \in \cW^i ([a,b] ) $ 
 the following inequality holds 
  \begin{equation}\label{eq:inq_inj}
   \| \xi \|_{\cW^i([a,b])} \leq c_i \left ( \| D_A \xi \|_{\cW^{i-1}([a,b])} + 
     \|\pi^+( \xi(a)) \|_{i-\frac{1}{2}} + \| \pi^- ( \xi(b))\|_{i-\frac{1}{2}} \right )
  \end{equation}
  Furthermore the mapping 
 \begin{equation}
  \begin{split}
   &F: \cW^i([a,b]) \to \cW^{i-1}([a,b]) \times (E^+\cap E^{i-\frac{1}{2}}) \times (E^-\cap E^{i-\frac{1}{2}}) \\
   &F( \xi ) = \left ( D_A \xi, \pi^+ ( \xi(a)), \pi^- ( \xi(b)) \right )
  \end{split}
  \end{equation}
  is bijective. 
 \item[ii)] The maps
 \begin{equation}
  \begin{split}
 & F^{\pm} : \cW^{i}( \R^{\pm} ) \rightarrow \cW^{i-1}(\R^{\pm} ) \times E^{\pm}\cap E^{i-\frac{1}{2}}\\
  & F^{\pm} ( \xi ) = ( D_A \xi , \pi^{\pm} ( \xi (0, \cdot ))
  \end{split}
\end{equation}
are bijective. 
There exists a constant $ c_i>0 $ 
such that for all $ \xi \in \cW^i( \R^{\pm})$ 
the following inequality holds 
\begin{equation}\label{eq_inj}
\|\xi\|_{\cW^i(\R^{\pm} )} \leq c_i\bigg( \|D_A \xi \|_{ \cW^{i-1}(\R^{\pm} )} + 
\| \pi_{\pm} (\xi  (0, \cdot))\|_{i-\frac{1}{2}} \bigg).
\end{equation}

\end{itemize}
\end{theorem}

\begin{proof} 
We prove the theorem in the next four steps. \\
\medskip
\noindent {\bf Step 1. Proof of the inequality \eqref{eq:inq_inj} in the case $i=1 $.}\\
\medskip 
\begin{equation}\label{eq:help1}
\begin{split}
 \int_a^{b} \| D_A \xi \|^2_{H^0}
  &=  \int_a^{b} \langle \partial_s \xi + A \xi, \partial_s \xi + A \xi \rangle_{H^0} ds \\
  & = \int_a^{b} \Big (\| \partial_s \xi \|^2_{H^0} + \| A \xi\|^2_{H^0}\Big ) ds + 
  \int_a^ {b} \partial_s \langle \xi, A \xi \rangle_{H^0} ds.
  \end{split}
 \end{equation} 
 Thus $ L=\int_a^{b} ( \| \partial_s \xi\|^2_{H^0} + \| A \xi\|^2_{H^0} ) ds $ satisfies
 \begin{align*}
  L = \int_a^{b} \| D_A \xi \|^2_{H^0} ds  - \langle \xi(b) , A \xi(b)\rangle_{H^0} +  \langle \xi(a), A \xi(a)\rangle_{H^0}
\end{align*}
From the inequality \eqref{eq_eqnormA} and 
the previous equality we obtain
\begin{align}\label{eq:H1est}
 \| \xi\|_{\cW^1([a,b])}^2 
 &\leq C_1^2 L = C_1^2 \Big ( \| D_A\xi\|^2_{\cW^0([a,b])} -  \langle \xi(b), A\xi(b)\rangle_{H^0}  +  \langle \xi(a), A \xi(a)\rangle_{H^0}  \Big )\notag \\ 
 & \leq    c_1 \Big (  \| D_A\xi\|^2_{\cW^0([a,b])} + \| \pi^+ ( \xi(a,\cdot))\|_{1/2}^2  + \| \pi^- ( \xi(b,\cdot))\|_{1/2}^2\Big )  .
 \end{align}
In the previous inequality $ \| \cdot\|_{1/2} $ norm of a  
$ \xi (a, \cdot) = \sum_{\lambda} x_{\lambda} e_{\lambda} $ 
 is given by 
 $$
  \| \xi(a, \cdot) \|_{1/2}^2 = \sum_{\lambda} \abs{\lambda} \abs{x_{\lambda}}^2 
  $$
  thus  $ \|\pi^+(\xi(a, \cdot))\|^2_{1/2}= \sum_{\lambda>0} \lambda \abs{x_{\lambda}}^2 $ and analogously is given 
  the norm of $ \|\pi^- (\xi(a, \cdot))\|_{1/2} $. \\
  \medskip 
  
{\bf Step 2. Proof of the inequality \eqref{eq:inq_inj} in the case $i=2 .$ }\\

\medskip 
 To shorten the notation we shall write 
 $ \| \cdot\|_{\cW^i }, \; i=0,1,2 $ for $ \|\cdot\|_{\cW^i([a,b] ) } $. 
 Substituting $ A\xi $ in the inequality \eqref{eq:H1est} we obtain  
 
\begin{align}\label{eq:aw1}
 \| A \xi\|_{\cW^1}^2 &  \leq c_1 \Big ( \| D_A ( A (\xi)) \|^2_{\cW^0} + \| \pi^+ ( A ( \xi(a))\|_{1/2}  + \| \pi^- ( A \xi(b))\|_{1/2}^2\Big ) \notag \\ 
    		    &\leq c_1 \Big ( \| A(D_A ( \xi))\|^2_{\cW^0} +  \| \pi^+ ( A ( \xi(a))\|_{1/2}  + \| \pi^- ( A \xi(b))\|_{1/2}^2\Big ) \notag \\	
                              &\leq c_1' \Big ( \| D_A\xi\|^2_{\cW^1} + \| \pi^+ ( \xi(a,\cdot))\|^2_{3/2} + \| \pi^- ( \xi(b, \cdot))\|_{3/2}^2  \Big ).
\end{align}
The last inequality follows from \eqref{eq_eqnormA} and the following 
observation 
$$ 
\| A \eta \|_{\cW^0}^2 = \int_a^b  \| A \eta\|_{H^0}^2 ds \leq C_1^2 \int_a^b \| \eta \|^2_{H^1} ds \leq C_1^2 \| \eta \|^2_{\cW^1} . 
$$
As the embedding $ H^1\hookrightarrow H^0 $ is continuous 
and  \eqref{eq_eqnormA}  holds we have 
$$
\| \xi \|_{H^0} \leq c' \| \xi\|_{H^1} \leq c' C_1 \| A \xi \|_{H^0} \leq c'' \| A \xi \|_{H^1}
$$
Integrating the previous inequality on interval $ [a,b] $
and using \eqref{eq:aw1} we obtain 
\begin{equation}\label{eq:help0}
\| \xi \|_{\cW^0}^2 \leq c_2' \Big ( \| D_A\xi\|^2_{\cW^1} + \| \pi^+ ( \xi(a,\cdot))\|^2_{3/2} +
\| \pi^- ( \xi(b))\|_{3/2}^2  \Big ).
\end{equation}
We have that 
$ \| \partial_s \xi \|^2_{H^1} \leq  2 (\| D_A \xi\|^2_{H^1} + \| A \xi\|^2_{H^1} )$. 
Integrating this inequality we obtain that 
\begin{equation}\label{eq:help2}
\begin{split}
\int_a^b \|\p_s \xi \|^2_{H^1}  ds  &\leq  2 \Big ( \int_a^b \| D_A \xi \|^2_{H^1} ds + \int_a^b \| A \xi \|^2_{H^1} ds\Big ) \\
 						       &\leq c_3 \Big ( \|D_A \xi \|^2_{\cW^1} + \| A\xi \|^2_{\cW^1 } 	\Big )\\
						       & \leq c_4 \Big ( \| D_A \xi \|^2_{\cW^1} + \| \pi^+ ( \xi(a,\cdot))\|^2_{3/2} + \| \pi^- ( \xi(b, \cdot))\|_{3/2}^2  \Big ).
\end{split}
\end{equation}
The last inequality of  \eqref{eq:help2} follows from \eqref{eq:aw1}. 
Finally, the following inequality also holds
$$ 
\| \p^2_s\xi\|^2_{H^0} \leq 2 \Big ( \| \p_s ( D_A \xi ) \|^2_{H^0} + \| \p_s ( A\xi)\|^2_{H^0} \Big ).
$$
Integrating this inequality on the interval $ [a,b] $ we obtain 
\begin{align}\label{eq:help3}
\int_a^b \| \p^2_s \xi \|_{H^0} ds  & \leq 2 \Big ( \| D_A \xi \|^2_{\cW^1} + \| A \xi \|^2_{\cW^1} \Big ) \notag \\
                                                  & \leq c_5  \Big ( \| D_A \xi \|^2_{\cW^1}  + \| \pi^+ ( \xi(a,\cdot))\|^2_{3/2} + \| \pi^- ( \xi(b, \cdot))\|_{3/2}^2  \Big ). 
\end{align}
The inequality \eqref{eq:help3} follows from \eqref{eq:aw1}. 
Summing the inequalities \eqref{eq:help0}, \eqref{eq:help2} and 
\eqref{eq:help3}, we obtain
\begin{align}
 \|\xi\|_{\cW^2([a,b])} \leq c_2 \Big ( \|D_A \xi \|_{\cW^1([a,b]) } 
 +\|\pi^+ ( \xi(a))\|_{3/2} + \| \pi^- ( \xi(b))\|_{3/2} \Big ).  
\end{align}
Thus we have proved the inequality \eqref{eq:inq_inj}.
This inequality implies that the mapping $ F $ 
is injective and has closed range. 
We still have to prove that it is surjective. \\
{\bf Step 3. Surjectivity of the operator $F$.} \\
\medskip
We prove surjectivity in the case $i= 2$. The proof of surjectivity 
in the case $i=1 $ is analogous. 

Let $ \eta \in \cW^1([a,b]) $ and $ \zeta^{\pm} \in E^{\pm}\cap E^{3/2} $.
We prove the existence of $ \xi \in \cW^2([a,b] )$ which satisfies:
\begin{equation}\label{eq_da}
 D_A \xi = \partial_{s} \xi(s, \cdot)  + A \xi(s, \cdot)  = \eta, 
 \;\; \pi^+ ( \xi (a, t)) = \zeta^+, \; \pi^- ( \xi (b,t)) = \zeta^-
\end{equation}

Let $ \xi = \sum\limits_{\lambda} \xi_{\lambda} ( s, \cdot) $ and 
$ \eta= \sum\limits_{\lambda} \eta_{\lambda}(s, t) $ and 
$ \zeta^{\pm} = \sum\limits_{\pm \lambda> 0} \zeta_{\lambda} $,
where $ \xi_{\lambda } $, $\eta_{\lambda} $ and $ \zeta_{\lambda} $ 
are the eigenvectors of the operator $A$. 
In order to find the solution of \eqref{eq_da} we need
to solve the following equations
\begin{equation}\label{eq:ksi}
\begin{split}
 &\partial_{s} \xi_{\lambda} ( s, \cdot) + \lambda \xi_{\lambda} ( s, \cdot) =
 \eta_{\lambda} ( s, \cdot ) , \;\; \text{for all} \;\; \lambda \\
 &\xi_{\lambda} (a, \cdot) = \zeta_{\lambda} ( \cdot), \;\; \lambda > 0\\
 & \xi_{\lambda} (b, \cdot)= \zeta_{\lambda} ( \cdot) , \;\; \lambda < 0
 \end{split}
 \end{equation}
The solutions $ \xi_{\lambda} $ of the previous equation is given by
\begin{equation}\label{eq:xilamb}
 \begin{split}
 \xi_{\lambda} (s, t) = \int\limits_{a}^{s} \eta_{\lambda} (y, t) e^{-\lambda ( s - y)} dy + 
 e^{-\lambda (s -a)} \zeta_{\lambda} (t) , \;\;\lambda > 0\\
\xi_{\lambda} (s, t) = e^{\lambda (b -s)} \zeta_{\lambda} (t)  
- \int\limits_{s}^{b} \eta_{\lambda} (y, t) e^{\lambda (  y-s)} dy   , \;\;\lambda < 0.
\end{split}
 \end{equation}
It is left to prove that $ \xi \in \cW^2( [a,b]) $ what is, because of \eqref{eq_eqnormA}, 
equivalent to the following 
\begin{align*}
 \sum\limits_{\lambda} \lambda^4 \int\limits_a^b \|\xi_{\lambda} \|_{H^0}^2 ds < + \infty\\
 \sum\limits_{\lambda} \lambda^2 \int\limits_a^b \|\p_s\xi_{\lambda} \|_{H^0}^2 ds < + \infty\\
 \sum\limits_{\lambda}  \int\limits_a^b \|\p_s^2\xi_{\lambda} \|_{H^0}^2 ds < + \infty
\end{align*}
Remember that $ \eta \in \cW^1( [a,b] ) $ and $ \zeta \in E $ 
what is equivalent to the following
\begin{equation}\label{eq:pomoc2}
 \begin{split} 
 & \sum\limits_{\lambda} \lambda^2 \int_a^b \|\eta_{\lambda}\|^2_{H^0} ds < + \infty \text{ and } 
  \sum\limits_{\lambda}  \int_a^b \|\partial_s \eta_{\lambda}\|^2_{H^0} ds < + \infty \\
&   \sum\limits_{\lambda} \abs{\lambda}^3 \|\zeta_{\lambda} \|_{H^0}^2 < + \infty
 \end{split}
 \end{equation}
 As $ \xi_{\lambda} $ satisfies the equation \eqref{eq:ksi} 
 and $ \eta_{\lambda} $ satisfy \eqref{eq:pomoc2} it is enough to 
 prove 
 \begin{equation}\label{eq:ksi1}
  \sum\limits_{\lambda} \lambda^4 \|\xi_{\lambda} \|^2_{\cW^0([a,b])} = \sum\limits_{\lambda} \lambda^4 \int_a^b \|\xi_{\lambda}\|^2_{H^0} ds < + \infty.
 \end{equation}
 Write $ \xi_{\lambda} = v_{\lambda} + w_{\lambda} $, where 
 \begin{align*}
  v_{\lambda}(s,t) = \int_a^s \eta_{\lambda} (y,t) e^{-\lambda ( s-y) } dy, \; \lambda> 0 \\
  w_{\lambda}(s,t) = e^{-\lambda ( s- a)} \zeta_{\lambda}, \; \lambda > 0
 \end{align*}
and analogously for $ \lambda < 0 $. Notice that $ v_{\lambda} = K_{\lambda} * \chi_{[a,s)} \cdot \eta_{\lambda} $, 
where 
\[ K_{\lambda} (s,t) = \begin{cases}
                        e^{-\lambda s} , \;\; s\geq 0 \\
                        0, s < 0 
                       \end{cases}
\]
for $ \lambda > 0 $. Obviously $ \| K_{\lambda} \|_{L^1} = \frac{1}{\lambda} $. 
Denote by $ d\mu_s(\lambda) = e^{-\lambda(s-y)} \chi_{[a,s)} (y) dy $. 
Then 
$$
\int_a^b d \mu_s(y) \leq \frac{1}{\lambda}, \;\;\forall s\in (a,b). 
$$
We can apply Jensen's inequality 
$$
\|v_{\lambda}(s,\cdot)\|_{H^0} \leq \int_a^b e^{-\lambda(s-y)} \chi_{[a,s)} \| \eta_{\lambda}\|_{H^0} dy = K_{\lambda} * f_{\lambda},
$$
where $ f_{\lambda} = \chi_{[a,s)}\|\eta_{\lambda}\|_{H^0} $. 
From Young's inequality we obtain 
$$
\| v_{\lambda}\|_{\cW^0([a,b])} \leq \| K_{\lambda}\|_{L^1} \|\eta_{\lambda} \|_{\cW^0([a,b])}
\leq \frac{1}{\lambda} \|\eta_{\lambda} \|_{\cW^0([a,b])}.
$$
Thus it follows from \eqref{eq:pomoc2} that 
\begin{equation}\label{eq:pomoc3}
\sum\limits_{\lambda} \lambda^4 \int_a^b\|v_{\lambda}\|^2_{H^0} ds 
					\leq \sum\limits_{\lambda}\lambda^2 \int_a^b \|\eta_{\lambda}\|^2_{H^0} ds< +\infty.
\end{equation}
On the other hand $ \|w_{\lambda}\|^2_{\cW^0([a,b])} \leq \frac{1}{2\lambda} \|\zeta_{\lambda}\|^2_{\cW^0([a,b])} $, 
thus it follows from \eqref{eq:pomoc2} that 
\begin{equation}\label{eq:pomoc4}
\sum\limits_{\lambda} \lambda^4 \|w_{\lambda}\|^2_{\cW^0([a,b])} < +\infty. 
\end{equation}
From \eqref{eq:pomoc3} and \eqref{eq:pomoc4} follows that $\xi $ satisfies 
\eqref{eq:ksi1}. 
This proves $ i)$. \\
\medskip 
{\bf Step 4. Proof of $ii)$.} 
The proof of part $ii) $ is analogous to the proof 
of part $i) $. 
We just explain the differences. Let for example $ a=0 $ 
and $ b=+\infty $. In this case in order to prove 
the inequality \eqref{eq_inj} we repeat the same 
procedure as in the proof of the inequality \eqref{eq:inq_inj}
just with $ b= + \infty $. Notice that in this case 
$$ \int _0^{+\infty} \p_s \langle \xi, A \xi \rangle_{H^0} = 
- \inner{\xi(0)}{ A \xi(0)} = 
- \| \pi^+ ( \xi(0))\|_{1/2} + \|\pi^- ( \xi(0)) \|_{1/2}.$$
The proof of the surjectivity of the maps 
$ F^{\pm} $ is also analog to the proof of surjectivity 
of the mapping $ F$, thus we use again eigenspace decomposition 
of the function $ \xi, \zeta $ and $ \eta $. The solution $ \xi $ 
of the boundary value problem 
$$ 
\p_s \xi + A \xi = \eta , \pi^+ ( \xi(0))= \zeta^+ 
$$
is given analogously to the equation \eqref{eq:xilamb}. 
The eigenvectors $ \xi_{\lambda} , \lambda > 0 $ are given 
as in \eqref{eq:xilamb} in the case $ a = 0 $, whereas $ \xi_{\lambda}, \lambda < 0 $
are given by 
$$ 
\xi_{\lambda} (s, t) = - \int\limits_{s}^{+\infty} \eta_{\lambda} (y, t) e^{\lambda (  y-s)} dy.
$$ 
The rest of the proof is word by word the same. 
\end{proof}

\section{Elliptic regularity}\label{sec:cor_setup}
In this section we shall prove some corollaries of the Theorem 
\ref{thm:main_inq} for the specific choice of the linear 
operator $A$ and its domain. 
 \begin{PARA}[{\bf The time independent case}] \label{PARA:Cor_time_independent}\rm
 Let $H^1= H^1_{bc}([0,1] ) $  and $ H^0 = L^2([0,1], \R^{2n} )$ 
 be as in \eqref{eq:hilspaces} and suppose that 
 the operator $A:H^1 \to H^0 $ has the following form
 \begin{equation}\label{eq:opA}
 A= J_0\p_t + S(t) :H^1_{bc}([0,1])\to L^2([0,1]) 
 \end{equation}
 where $J_0 $ is the standard complex structure. 
 We assume that the operator $A$ is bijective and self-adjoint. 
 Let $ E^{\pm}\subset L^2([0,1]) $ be generated by 
 positive and negative eigenvectors as in \ref{rem:lines} 
 corresponding to the above operator $A$. 
 Let $I=[a,b] $ or $ I = \R^{\pm} $ and let 
 \begin{align}\label{eq:h1_bc}
 H^1_{bc} ( I \times [0,1] )  &:= \{\xi \in W^{1,2} ( I \times [0,1] ) \Big| \xi(s,i) \in \R^n \times \{0\} \}\notag \\
 &= W^{1,2} ( I, H^0 ) \cap L^2( I, H^1)
 \end{align}
 Let $ \cW^1_{\mp} (I) $ be its 
 subspace defined as follows 
 \begin{equation}\label{eq:w1mp}
 \begin{split}
  \cW^1_{\mp} ( I) &:= \{ \xi \in H^1_{bc} ( I \times [0,1] ) \Big | \xi(a) \in E^-, \; \xi(b) \in E^+ \}\\
		    &=\{ \xi \in W^{1,2} ( I, H^0) \cap L^2 ( I, H^1) \Big | \xi(a) \in E^-, \; \xi(b) \in E^+ \}
  \end{split}
 \end{equation}
 Analogously in the case $ I = \R^{\pm} $ we can define $\cW^1_{\mp} ( \R^{\pm} ) $ 
 as 
 $$
 \cW^1_{\mp} ( \R^{\pm}) := \{ \xi \in H^1_{bc} ( \R^{\pm} \times [0,1] ) \bigg | \xi(0) \in E^{\mp} \}. 
 $$
 Let 
\begin{equation}\label{eq:h2_bc}
\begin{split}
  H^2_{bc} ( I \times [0,1])  &= \left \{\xi \in W^{2,2} (I\times [0,1], \R^{2n})\bigg|
  \begin{array}{l}
   \xi(s,i) \in \R^n \times \{0\},\; i=0,1 \\
   A\xi (s,i)\in \R^n \times \{0\} , \; i=0,1 
  \end{array}
\right\}\\
&= W^{2,2}( I, H^0) \cap W^{1,2} ( I, H^1) \cap  L^2 ( I, \Dom(A^2)) 
\end{split}
 \end{equation}
 Analogously we define a Hilbert subspace $ \cW^2_{\mp} (I) $ as 
 follows 
 \begin{equation}\label{eq:w2pm}
 \cW^2_{\mp}(I) := \{ \xi \in  H^2_{bc} ( I \times [0,1] ) \Big | \xi(a) \in E^-, \;\; \xi(b) \in E^+ \} ,
 \end{equation}
 and in a similar way we define in the case $ I=\R^{\pm} $ the space $ \cW^2_{\mp}(I) $. 
 From Theorem \ref{thm:main_inq} we derive some useful corollaries. 
 The next corollary follows directly from the mentioned theorem. 
 \end{PARA}
 \begin{corollary}[{\bf Bijective linearized operator}] \label{cor:bij_lin_op}
  Let $I=[a,b] $ or $ I = \R^{\pm} $ and let $A$ 
  be as in \eqref{eq:opA}. Let $ \cW^{i}_{\mp} (I) $,$ i=1,2 $
  be defined as above. Denote with $D_A $ 
  the following linear operator
  \begin{equation*}
  \begin{split}
   & D_A :\cW^1_{\mp} (I)\to L^2( I\times [0,1] )\\
   & D_A \xi = \p_s \xi + A \xi.
   \end{split}
  \end{equation*}
  The operator $D_A $ is bijective and similarly 
  $$ D_A : \cW^2_{\mp} (I) \to H^1_{bc} ( I \times [0,1] ) $$ 
  is bijective. 
 \end{corollary}
 \begin{theorem} \label{thm:reg} 
 Let $ I=[a,b] $. Then the following hold
 \begin{itemize}
  \item[i)]{\bf($ W^{1,2} $ regularity)} Let $ \xi, \eta \in L^2(I \times [0,1] ) $ satisfy the following equality 
 \begin{equation}
  \int_0^1 \inner{\xi,}{D_A^*\zeta}_{H^0}ds = \int_0^1 \inner{\eta}{\zeta}_{H^0} ds, \;\; \forall \zeta\in \cW'(I),  
 \end{equation}
 where 
 $$ \cW'(I) = \{ \xi\in L^2( I, H^1) \cap W^{1,2} ( I, H^0) \Big | \zeta(a) \in E^+ , \;\; \zeta(b) \in E^- \} 
 $$
 and $ D_A^*\zeta = -\p_s \zeta  + A \zeta .$
 Then $ \xi \in \cW^1_{\mp}(I) $ and $ D_A\xi = \eta $. 
 \item[ii)] {\bf ($W^{2,2} $ regularity)}  Let $ \xi \in \cW^1_{\mp} (I) $ satisfy $ D_A \xi = \eta $ 
 for some $ \eta \in H^1_{bc} ( I \times [0,1] ) $. Then $ \xi \in \cW^2_{\mp} ( I) $. 
 \item[iii)]{\bf  ($ W^{1,p} $ regularity, $p> 2 $)} Let  $ \xi \in \cW^1_{\mp} (I) $ be such that $ D_A \xi = \eta \in L^p (I \times [0,1] ) $. 
Then $ \xi \in W^{1,p}_{bc} ( I \times [0,1] ) := \{ \xi \in W^{1,p} ( I \times [0,1] ) \Big | \xi(s,i) \in \R^n \times \{0\} , \; i=0,1 \} $.
 \end{itemize}

 \end{theorem}
\begin{proof}$ i) $
Let $ \txi\in \cW^1_{\mp}(I) $ be a unique solution of 
 the equation $ D_A \txi = \eta $. From Corollary \ref{cor:bij_lin_op} 
 we have that such $ \txi $ exists and is unique. Notice that 
 for all $ \zeta \in \cW'(I) $ we have 
 $$
 \int_0^1 \inner{\txi}{D_A^* \zeta}_{H^0}ds = \int_0^1 \inner{D_A\txi}{\zeta}_{H^0}ds = \int_0^1 \inner{\eta}{\zeta}_{H^0}.
 $$
 Thus we have that 
 $$
 \int_0^1 \inner{\txi - \xi}{D_A^* \zeta}_{H^0}=0 \; \forall \zeta \in \cW'(I).
 $$
 One can prove analogously as in Corollary \ref{cor:bij_lin_op} that 
 the operator $ D_A^*: \cW'(I) \to L^2 ( I \times [0,1] ) $ is bijective, 
 thus we have that $ \xi = \txi$.\\
 \medskip 
 $ii) $ Let $ \txi \in \cW^2_{\mp} (I) $ be a unique solution 
 of the equation $ D_A \xi = \eta $. From Corollary \ref{cor:bij_lin_op} 
 it follows that such $ \txi $ exists and is unique. Notice that the difference 
 $ \xi' = \txi - \xi $ satisfies the equation $ D_A \xi' = 0 $ and $ \xi' \in \cW^1_{\mp} (I)$. 
 Thus, it follows from Corollary \ref{cor:bij_lin_op} that $ \txi = \xi $. \\
 \medskip 
$iii)$ Let $ \teta $ be the extension of $ \eta $ on the whole of $ \R \times [0,1] $. 
 \[\teta = \begin{cases}
            \eta (s,t), \;\; s\in I \\
            0, \;\;  s\notin I
           \end{cases}
 \]
 It follows from Lemma \ref{lem:lp_est} that the linear operator 
 $$ 
 D_A : W^{1,p}_{bc} ( \R \times [0,1] ) \to L^p ( \R \times [0,1] ) 
 $$
 is bijective. 
 Thus there exists a unique $ \txi \in W^{1,p}_{bc} ( \R \times [0,1] ) $ 
 such that $ D_A \txi = \teta $. Now in the case that $ I = [a,b] $ for example 
 it follows using eigenvector decomposition that $ \txi(a) \in E^- $, as $ D_A\txi=0 $ 
 on the interval $ (-\infty, a] \times [0,1] $ and analogously 
 $ \txi(b) \in E^+ $. As  this implies that $ \txi \in \cW^1_{\mp} (I) $
 and $ D_A ( \txi - \xi ) = 0 $, we have from Corollary \ref{cor:bij_lin_op} 
 that $ \txi= \xi \in W^{1,p}_{bc} ( I \times [0,1] ) $. 
 
 \end{proof}

\begin{PARA}[{\bf The time dependent case}]\label{para:bij_lim}\rm
In the previous two sections we have examined linearized operators of 
the form $ D_{A} = \p_s + A  $, where the operator $ A $ was bijective, 
self-adjoint and time independent. 
Now we allow the operator $A$ to depend on time- $s$ as well,  
but for simplicity we assume that the operators $A(s)$ 
have the following form
$$
A(s)= J_0 \p_t + S(s,t),
$$
where $ J_0 $ is the standard complex structure and 
$S\in W^{1,2} ( I \times [0,1], \R^{2n\times 2n} ) $. 
Let $ H^0 = L^2([0,1])$ and suppose that $ H^1:=\Dom(A(s)) $ and $H^2:= \Dom(A(s)^2) $ 
are $s$ independent. In other words the operators $ A(s) $ can 
be written as the sum of some time independent operator $A$ which satisfies 
~(HA) and some matrix valued function $R$ of $W^{1,2} $ class. 
Thus, 
$$ D \xi =  \p_s \xi + A(s) \xi = \p_s \xi + A \xi + R(s,t) \xi ,$$
where $ R \in W^{1,2} ( I \times [0,1], \R^{2n\times 2n} )$. 
Let $I $ be either an interval of the form $ I= [a,b] $ 
or $ I= \R^{\pm} $ and let $ H^i_{bc} ( I \times [0,1] ), \; i=1,2 $ be 
as in \eqref{eq:h1_bc} and  \eqref{eq:h2_bc} respectively and let $ H^0_{bc} ( I \times [0,1] )$ 
be just the standard $ L^2(I\times [0,1] ) $. Observe the 
linear operator 
\begin{equation}\label{eq:opD}
\begin{split}
  &D : H^i_{bc} ( I \times [0,1] )\to H^{i-1}_{bc} ( I \times [0,1] ) ,\;\; i=1,2\\
  &D \xi= \p_s \xi + J_0 \p_t\xi + S(s,t)\xi  = \p_s\xi  + A(s)\xi = \p_s \xi + A \xi + R(s,\cdot)\xi 
\end{split}
\end{equation}
\begin{description}
 \item[(H1)] In the case $ I = \R^{\pm} $ we additionally assume that 
 the 
 $$
 \lim\limits_{s\to \pm \infty} \|S(s,t) -S^{\pm} (t)\|_{C^1([s, \infty) \times [0,1] ) }=0,
 $$
 where $ S^{\pm}:[0,1] \to \R^{2n\times 2n}$ 
 are smooth functions and that the limit operators 
 $$ 
 A^{\pm} = J_0 \p_t + S^{\pm} (t):H^1\to H^0 
 $$
 satisfy ~(HA). 
\end{description}

\end{PARA}
The analogous statement as in Theorem \ref{thm:reg} 
holds in the case of time-dependent operator $A$. We formulate 
and prove the analogous theorem. 
\begin{theorem}[{\bf Regularity }]\label{thm:reg1}
 Let $ I= [a,b] $ and let the linear operator 
 $D$ be as in \eqref{eq:opD}. 
 Then the following statements hold 
 \begin{itemize}
  \item[i)] Let $ \xi, \eta \in L^2 ( I \times [0,1] ) $ and suppose that 
  the following equality 
  $$\int_I \inner{\xi}{D^*\zeta}_{H^0} ds = \int_I\inner{\eta}{\zeta}_{H^0} ds, $$
  holds for all $ \xi \in \cW'(I) = \{ \xi\in L^2( I, H^1) \cap W^{1,2} ( I, H^0) \Big | \zeta(a) \in E^+ , \;\; \zeta(b) \in E^- \} $, 
  where $ D^* = D_A^* + R^T = -\p_s + A + R^T $. 
  Then $ \xi $ is a strong solution of the equation $ D \xi = \eta $ and $ \xi \in \cW^1_{\mp} (I) $. 
  \item[ii)] Let $ \xi \in \cW^1_{\mp} (I) $, $ \eta \in H^1_{bc} ( I \times [0,1] ) $ satisfy 
  $ D \xi = \eta $. Then $ \xi \in \cW^2_{\mp}(I) $. 
  \item[iii)]  Let $ \xi \in \cW^1_{\mp} (I) $ and $ \eta \in L^p( I \times [0,1) , \; p> 2$
  satisfy the equation $ D \xi = \eta $. Then $ \xi \in W^{1,p}_{bc} ( I \times [0,1] ) $. 
 \end{itemize}
\end{theorem}
\begin{proof} 
 $i) $ Notice that 
 \begin{align*}
  &\int_I \inner{\xi}{D^* \zeta}_{H^0} ds = \int_I \inner{\xi}{D_A^* \zeta + R^T \zeta}_{H^0} ds \\
  &= \int_{I} \inner{\xi}{D_A^* \zeta}_{H^0} ds + \int_I \inner{R\xi}{\zeta}_{H^0} ds. 
 \end{align*}
 Thus we have that 
 \begin{align*}
  \int_I\inner{\xi}{D_A^*\zeta}_{H^0} ds = \int_I\inner{\eta- R \xi}{\zeta}_{H^0}ds = \int_I \inner{\eta'}{\zeta}_{H^0}ds,
 \end{align*}
 and $ \eta' = \eta- R \xi\in L^2(I \times [0,1] ) $. Thus it 
 follows from Theorem \ref{thm:reg}, that $ \xi \in \cW^1_{\mp} (I) $ 
 and it satisfies the equality $ D_A \xi = \eta' $. Thus $ D(\xi) = D_A \xi + R \xi = \eta $.\\
 \medskip 
 $ ii) $ As $ D \xi = D_A \xi + R \xi = \eta $, and both $ \xi $ and $ R $ are 
 $ W^{1,2} $ functions their product will be an $ L^p $ function for any $p > 1 $. 
 Thus, $ D_A \xi = \eta - R \xi \in L^p ( I \times [0,1] ) $, $p > 2 $. 
 From Theorem \ref{thm:reg} it follows that $ \xi \in W^{1,p}_{bc} ( I\times [0,1] ) $. 
 This implies that the product $ R \cdot \xi \in W^{1,2}_{bc} ( I \times [0,1] )$ 
 and $ D_A \xi = \eta' \in W^{1,2}_{bc} ( I \times [0,1] ) $. Again, from 
 Theorem \ref{thm:reg} we have that $ \xi \in \cW^2_{\mp} (I) $. 
 The proof of part $iii) $ is analogous to the proof of $ ii) $ 
 and we shall not repeat it. 
\end{proof}

\begin{theorem}[{\bf Estimates}]\label{thm:main_inq1}
 Let $H^i, \; i=0,1,2 $ be as in \ref{para:bij_lim}
 and let $D$ be as in \eqref{eq:opD}. 
 Let $ A = J_0 \p_t + S_1 (t) : H^1 \to H^0 $ 
 and $ B =  J_0 \p_t + S_2 (t) : H^1 \to H^0 $ satisfy ~(HA). 
 Let $ E^{\pm}_A, E^{\pm}_B$ be Hilbert spaces generated by positive 
 and negative eigenvectors of the operators $ A$ and $ B $ 
 respectively as in \ref{rem:lines}. 
 Denote with $\pi^{\pm}_A $ and $  \pi^{\pm}_B $  
 the corresponding projections, as in \ref{rem:lines}. 
 \begin{itemize}
  \item [i)] Let $ I = [a,b] $. Then there exist a constant $ c > 0 $ 
  and a compact operator 
  $$
  K: H^2_{bc} ( I \times [0,1] )\to H^1_{bc} ( I \times [0,1] ) 
  $$
  such that the following inequality holds for all 
  $ \xi \in H^2_{bc} ( I \times [0,1] ) $. 
  \begin{equation}\label{eq:main_ingD}
   \| \xi \|_{2,2} \leq c \Big ( \| D \xi \|_{1,2} + \| K \xi \|_{1,2} + \| \pi_A^+ ( \xi(a)) \|_{3/2} + \| \pi^-_B (\xi(b))\|_{3/2} \Big ).
  \end{equation}
  \item[ii)] Let $ I = \R^{\pm} $ and assume ~(H1). 
  Then there exist a constant $ c > 0 $ and a compact 
  operator $ K: H^2_{bc} ( \R^{\pm} \times [0,1] )\to H^1_{bc} ( \R^{\pm} \times [0,1] ) $ 
  such that the following inequality holds 
  for all $ \xi \in H^2_{bc} ( \R^{\pm} \times [0,1] ) $. 
  \begin{equation}\label{eq:pp0}
   \| \xi \|_{2,2} \leq c \Big ( \|  D \xi \|_{1,2} + \|K \xi \|_{1,2} + \| \pi^{\pm}_A ( \xi (0)) \|_{3/2} \Big ). 
  \end{equation}
 \end{itemize}
\end{theorem}

\begin{proof}
We prove this theorem in the next four steps.\\
\medskip 

\noindent {\bf Step 1.} {\bf Proof of the inequality \eqref{eq:main_ingD} in the case that $ A= B $.} \\
\medskip 
 Denote with $D_{A}$ linear operator 
 $ D_{A} = \p_s + A= J_0 \p_t + S_1(t)$. 
 From Theorem \ref{thm:main_inq} we have the following: 
 \begin{align}\label{eq:lin_fin}
   \| \xi \|_{2,2} & \leq c \Big ( \|D_{A} \xi \|_{1,2} + \| \pi^+_A ( \xi(a))\|_{3/2} + \|\pi^-_A (\xi(b))\|_{3/2} \Big ) \notag \\
		  & \leq c \Big ( \| D\xi \|_{1,2} + \| ( D - D_A ) \xi \|_{1,2} + \| \pi^+_A ( \xi(a))\|_{3/2} + \|\pi^-_A (\xi(b))\|_{3/2} \Big ) \notag  \\
		  & \leq c \Big ( \| D\xi \|_{1,2} + \| ( S - S_1)\xi \|_{1,2} + \| \pi^+_A ( \xi(a))\|_{3/2} + \|\pi^-_A (\xi(b))\|_{3/2} \Big )\notag \\
		  & \leq c \Big (\| D\xi \|_{1,2} + \|R \xi \|_{1,2} + \| \pi^+_A ( \xi(a))\|_{3/2} + \|\pi^-_A (\xi(b))\|_{3/2} \Big ). 
 \end{align}
 Notice that the difference $ S(s,t) - S_1(t) =R(s,t)$ and the matrix valued function 
 $R \in W^{1,2} ( I \times [0,1] ) $. 
 The operator $K(\xi ) := R\xi $ is a compact operator. 
 This follows by the following observation 
\begin{align*}
\| K(\xi )\|_{1,2} = \| R\xi \|_{1,2} &\leq \| dR \xi \|_{L^2} + \|R d \xi \|_{L^2} + \| R \xi \|_{L^2} \\
& \leq \| dR \|_{L^2} \| \xi \|_{L^{\infty}}  + \|R\|_{L^4} \|d\xi \|_{L^4} + \|R\|_{L^4} \| \xi \|_{L^4} \\
&\leq \|R\|_{1,2} \| \xi \|_{L^{\infty}} + c\|R\|_{1,2} \| \xi \|_{1,4} + c\|R\|_{1,2} \| \xi\|_{L^4} 
\end{align*}
As the embedding $ W^{2,2}( I \times [0,1] ) \hookrightarrow W^{1,4}( I \times [0,1] ) $ 
is compact as well as the embedding $W^{2,2}( I \times [0,1] ) \hookrightarrow L^{\infty} $, 
we have that the operator $K $ is compact. \\
\medskip

\noindent {\bf Step 2. Proof of the inequality \eqref{eq:pp0} in the case that $ A^{\pm} = A $. }\\
\medskip
 We do the proof in the case of positive half-infinite strips. 
 The case of negative strips is analogous. 
 Consider the following linear maps 
\begin{align*}
  F , F_{A}:
  H^2_{bc} ( \R^{+} \times [0,1] , \R^{2n} ) 
  \rightarrow H^1_{bc} ( \R^{+} \times [0,1] ) \times E^{+}_A , \\
  F(\xi)= ( D \xi, \pi^+ (\xi(0))), \;\; F_{A} ( \xi) = ( D_{A} \xi,\pi^+(\xi(0))) . 
\end{align*}
From Theorem \ref{thm:main_inq}
it follows that the map $ F_{A} $ is bijective and it satisfies 
the estimate:
\begin{align}\label{eq:linz_inf}
 \| \xi\|_{2,2} \leq c \big ( \| D_{A} ( \xi ) \|_{1,2} + \|\pi^{+}_A( \xi(0) )\|_{3/2} \Big ). 
\end{align}
 The operator $ F $ is just a compact perturbation of 
the operator $ F_{A} $ what can be proved as follows.
 Convergence $ S(s,t)\stackrel{C^1}{\longrightarrow} S_1(t) $
implies that for $ s_0 $ sufficiently large we have
$$ 
\|S(s,t) - S_1(t)\|_{C^1 ( [s_0, + \infty) \times [0,1] )} \leq \frac{1}{4c} , 
$$
where $ c $ is the constant of the inequality \eqref{eq:linz_inf}.
Let $ \beta : \R \rightarrow [0,1] $ be a smooth cut-off function 
with
\[\beta (s)=\begin{cases}
            1, s\leq s_0 , \\
            0, s\geq s_1 >> s_0
           \end{cases}
\]
and $ \| \beta \|_{C^1} \leq 2 $. 
From \eqref{eq:linz_inf} we obtain
\begin{align}\label{eq:lin_D}
\| \xi\|_{2,2} &\leq c \big ( \| D \xi \|_{1,2} + \| ( D - D_{A} ) \xi \|_{1,2} + \|\pi^{+}_A( \xi(0)) \|_{3/2} \Big )\notag \\
            & \leq c \Big ( \| D \xi \|_{1,2} + \| (S(s,t) - S_1(t)) \xi\|_{1,2} + \| \pi^{+}_A(\xi(0)) \|_{3/2} \Big ) \notag \\
            & \leq c \Big (  \| D \xi \|_{1,2} + \|( S(s,t)- S_1(t) ) \beta \xi \|_{1,2}  \notag \\
            & \qquad + \|( S(s,t) - S_1(t) )( 1-\beta) \xi\|_{1,2} +  \|\pi^{+}( \xi(0))\|_{3/2} \Big ) \notag \\
            & \leq 2c \Big ( \| D \xi \|_{1,2} 
               + \|K \xi\|_{W^{1,2} ([0, s_1] \times [0,1] ) } 
              + \|\pi^{+}_A ( \xi(0))\|_{3/2}   \Big ). 
\end{align}
Here the operator $K $ is given as multiplication 
by $ (S- S_1) \beta $ which has compact support and is $ W^{1,2} $ 
function. \\
\medskip
\noindent {\bf Step 3. Proof of \eqref{eq:main_ingD} in general case.} \\
In Step $ 1)$ we have proved the following inequality 
\begin{equation}\label{eq:pm1}
 \| \xi \|_{2,2} \leq c \Big ( \| D \xi \|_{2,2}  + \|K \xi \|_{1,2} + \| \pi^+_A ( \xi(a))\|_{3/2} + \|\pi^-_A ( \xi(b))\|_{3/2} \Big )
\end{equation}
and the same inequality follows when the operator $A$ is 
substituted with the operator $B $ on the right side 
of the inequality. Let $ \beta : [a,b] \to [0,1] $ be a smooth cut-off function 
such that 
\[\beta (s)=\begin{cases}
            1,\;\; s\leq  a + \frac{ b-a}{4} , \\
            0, \;\;s\geq  b - \frac{ b-a}{4}
           \end{cases}
 \]
 Apply the inequality \eqref{eq:pm1} to $ \beta \xi $, and the 
 same type of the inequality just with $A$ substituted with $B$ to $ (1-\beta) \xi $. 
 Thus the following two inequalities hold
 \begin{equation}\label{eq:pp1}
  \| \beta \xi \|_{2,2} \leq c \Big ( \| D ( \beta \xi ) \|_{1,2} + \| \beta K \xi \|_{1,2} + \|\pi^+_A ( \xi(a))\|_{3/2} \Big )
 \end{equation}
 and analogously we have 
 \begin{equation}\label{eq:pp2}
  \| ( 1- \beta ) \xi \|_{2,2} \leq c \Big ( \| D ( ( 1- \beta ) \xi ) \|_{1,2} + \| ( 1- \beta )K \xi \|_{1,2} + \| \pi^-_B( \xi (b) ) \|_{3/2} \Big ). 
 \end{equation}
Summing the inequalities \eqref{eq:pp1} and \eqref{eq:pp2} we obtain 
\begin{align}
 \| \xi \| \leq c \Big ( & \| \beta D \xi \|_{1,2} + \| ( 1- \beta ) D \xi \|_{1,2} + \| \dot{\beta} \xi \|_{1,2} + \|K \xi \|_{1,2}  \notag \\ & +  \|\pi^+_A ( \xi(a))\|_{3/2}  + \| \pi^-_B ( \xi (b) ) \|_{3/2} \Big )\notag \\
   \leq c \Big (  & \| D \xi \|_{1,2} + \| \xi \|_{1,2} + \| K\xi \|_{1,2} +
   \|\pi^+_A ( \xi(a))\|_{3/2}  + \| \pi^-_B ( \xi (b) ) \|_{3/2} \Big )
\end{align}
As, also the embedding $ H^2_{bc} ( I \times [0,1] ) \to H^1_{bc} ( I \times [0,1] ) $ 
is compact, the claim follows. \\
{\bf Step 4: Proof of \eqref{eq:pp0} in general.} \\
We prove the inequality in the case of positive half strips. 
The proof in the case of negative strips is analogous. Let $ \beta : \R^+ \to [0,1] $ 
be a smooth cut-off function with the properties: 
\[\beta (s) = \begin{cases}
               1, s \geq s_1 \\
               0, s\leq s_0
              \end{cases}
\]
Applying the results of Step $2)$ to $ \beta \xi $ and the limit operator $ A^+ $ we obtain 
\begin{equation}\label{eq:pp3}
 \begin{split}
  \| \beta \xi \|_{2,2} &\leq c \Big ( \| D ( \beta \xi )) \|_{1,2} + \| \beta K \xi \|_{1,2} + \| \pi^+_{A^+} ( \beta \xi (0)) \|_{3/2} \Big ) \\
  & \leq  c \Big ( \| \beta D \xi \|_{1,2} + \| \dot{\beta} \xi \|_{1,2} + \|K \xi \|_{1,2} \Big ) \\
  & \leq c \Big ( \| D \xi \|_{1,2} + \|K' \xi \|_{1,2} \Big )
 \end{split}
\end{equation}
where $ K' $ is compact operator. 
Apply next the inequality \eqref{eq:main_ingD} to $ (1-\beta ) \xi $ 
and the operator $A$ on compact interval $ [ 0, s_1] \times [0,1] $. 
We have 
\begin{equation}\label{eq:pp4}
 \begin{split}
  \| ( 1-\beta ) \xi \|_{2,2} &\leq c \Big (  \| D ( ( 1-\beta )\xi \|_{1,2} + 
  \| ( 1- \beta )K \xi \|_{1,2} + \| \pi^+_A ( \xi (0)) \|_{3/2}  \Big ) \\
  &\leq c \Big ( \| ( 1- \beta ) D \xi \|_{1,2} + \| K\xi\|_{1,2} + \| (1-\beta) \xi\|_{1,2} \| \pi^+_A ( \xi (0)) \|_{3/2} \Big ) \\ 
  & \leq c \Big ( \| D \xi \|_{1,2} +  \|K' \xi\|_{1,2} +\| \pi^+_A ( \xi (0)) \|_{3/2} \Big )
 \end{split}
\end{equation}
Summing the inequalities \eqref{eq:pp3} and \eqref{eq:pp4} we obtain the inequality 
\eqref{eq:pp0}.
\end{proof}

\begin{PARA}[{\bf Closed Image}]\label{para:cl_im}\rm 
Let $D$ be the operator as in \eqref{eq:opD}. 
We prove that $D$ has closed image. 

\begin{lemma}\label{lem:closed_img}
Let $i=1$ or $i=2 $ and let $ I = [a,b] $ or $ I = \R^{\pm}$. 
Let $D$ be as in \eqref{eq:opD}. In the case of infinite strips 
$ I = \R^{\pm} $ we assume ~(H1). Then the image of the 
operator $D$ is closed. 
\end{lemma}
\begin{proof}
 We prove Lemma \ref{lem:closed_img} in the case $ I =[a,b] $
 and $i=2 $. The proof in the case $ I = \R^{\pm} $ and 
 in the case $i=1 $ is analogous. 
 Let $A = J_0 \p_t + S(t) : H^1 \to H^0 $
and $ E^{\pm} $ be as in \ref{PARA:Cor_time_independent}
and let $ \sV = \cW^2_{\pm} (I)$ be defined as 
in \eqref{eq:w2pm}. \\
 \medskip
 {\bf Step 1.} The restriction of the operator $D$ to $\sV $ 
 is a Fredholm operator. \\
 Let $ D_A = \p_s + A $. Then it follows from 
the corollary \ref{cor:bij_lin_op} that 
the  restriction of operator $ D_A $ to $\sV $ is a
bijective operator. Thus particularly the restriction of 
the operator $D_A $ to $ \sV $ is a Fredholm operator of index $0$. 
On the other hand the operator $ D = D_A + ( S(s,t) - S_1(t)) = D_A + K $,
where $ K : H^2_{bc} ( I \times [0,1] ) \to H^1_{bc} ( I \times [0,1] ) $ 
is a compact operator. 
Thus the operator $D$ is a compact 
perturbation of a Fredholm operator, hence it is also a compact operator
of the same index. 

 {\bf Step 2.} The operator $ D : H^2_{bc} ( I\times [0,1] ) \to H^1_{bc} ( I \times [0,1] ) $
 has closed image. \\
 \medskip
 Let $ X = H^2_{bc} ( I \times [0,1] ) $, $ Y = H^1_{bc} ( I \times [0,1] ) $ 
 and let $ \sV \subset X $ be as in Step 1. Denote with $ Y_0 $ the image 
 of $ \sV $ via $D$, i.e. $ Y_0 = D(\sV) $. Then it follows 
 from Step 1 that $ Y_0 \subset Y $ is closed and finite codimension subspace. 
 We need to prove that $Y_1= D( X) $ is also closed. 
 Notice that $ Y_0 \subset Y_1 \subset Y$. Observe natural projection $ \text{pr}: Y_1 \to Y/ Y_0 $. 
 As $ Y / Y_0 $ is finite dimensional space and $ \text{p} ( Y_1) $ is a vector subspace
 it follows that $ \text{pr}(Y_1) $ is finite dimensional and hence also closed in $ Y / Y_0 $. 
 Thus $ Y_1 $ =  $ \text{pr}^{-1} ( \text{pr}( Y_1)) $ is closed in $Y$. 
 \end{proof}
We prove in Section \ref{SEC:ELreg} that $ D$ is actually surjective. 

\end{PARA}

\section{Unique continuation and surjectivity}\label{SEC:ELreg}

\begin{PARA}[{\bf Elliptic regularity at the corner}] \label{para:dence_img}\rm

Here we shall prove elliptic regularity at the corner 
which is reduced using reflection argument to the elliptic 
regularity at the boundary. As a corollary we prove that 
the operator $ D $ as in \eqref{eq:opD} has dense image and as 
a corollary we prove that it is also surjective. 

Let $ \epsilon \in \overline{R}$ be positive and define  
\begin{equation}\label{eq:om}
 \Omega= [0,\epsilon) \times [0,1], \quad \tOmega= ( - \epsilon, \epsilon) \times [0,1] 
\end{equation}
Denote by $ C^{\infty}_{c, bc} ( \tOmega ) $ the following set 
 \begin{equation}\label{eq:cinf_bc}
  C^{\infty}_{c, bc }(\tOmega) =  
 \left \{ \phi \in C^{\infty}_c (\tOmega, \R^{2n}) |
 \begin{array}{l}
  \phi(s,i)\in \R^n \times \{0\}, i=0,1
 \end{array}
\right \}
 \end{equation}
 Similarly we define the set $ C^{\infty}_{c,bc} ( \Omega) $ by
 \begin{equation*}
  C^{\infty}_{c, bc }(\Omega) =  
 \left \{ \phi \in C^{\infty}_c (\Omega, \R^{2n})  |
 \begin{array}{l}
  \phi(s,i)\in \R^n \times \{0\}, i=0,1
 \end{array}
\right \}
 \end{equation*}
Notice that a function $ \phi \in C^{\infty}_{c, bc }(\Omega) $ 
doesn't vanish on $ \{ 0\} \times [0,1] $. 

A direct corollary of Lemma B.4.9
 in \cite{MS} is the following: 
\begin{claim}\label{ch4_cor2.2}
If a function $ u \in L^2_{loc} (\tOmega, \R^{2n} ) $ 
satisfies the following equality
\begin{equation}\label{ch4_eq2.3}
\int\limits_{\widetilde{\Omega}} \langle \partial_s \phi + J_0\partial_t \phi, u \rangle = \int\limits_{\widetilde{\Omega}} \langle \phi , v\rangle ,
\;\; \forall \phi \in C^{\infty}_{c, bc} (\widetilde{\Omega}, \R^{2n} ) 
\end{equation}
where $v \in L^2_{loc} ( \widetilde{\Omega})$. 
Then the following holds
\begin{itemize}
\item[1)] $ u \in H^{1}_{loc} (\widetilde{\Omega}, \R^{2n}) $ and $ u(s,i) \in \R^n \times \{0\} $  for $ i=0,1 $. 
\item[2)] $ -\partial_s u + J_0 \partial_t u = v $
\end{itemize} 
\end{claim}

\begin{lemma}\label{lem:vanish_bdy}
Let $ \Omega $ be as in \eqref{eq:om}.
Suppose that $\eta \in L^2_{loc}( \Omega) $ satisfies 
\begin{equation}\label{ch4_eq2.2}
\int\limits_{\Omega} \langle \partial_s \phi + J_0 \partial_t \phi , \eta \rangle = 
 \int\limits_{\Omega} - \langle \zeta, \phi \rangle
\end{equation}
for all $\phi\in  C^{\infty}_{c, bc }(\Omega)$ and some  $\zeta \in L^2_{loc}(\Omega) $. 
Then the following holds
\begin{itemize}
 \item [1)] $ \eta \in H^1_{loc}(\Omega),\;\; \eta(s,0),\eta(s,1)\in\R^{n}\times 0 $  
 and $\eta(0,t)=0,$ for a.e.  $t\in [0,1]$.
\item[2)] $ -\partial_s \eta + J_0 \partial_t \eta = -\zeta $. 
\end{itemize}
 
\end{lemma}

\begin{proof}[Proof of Lemma \ref{lem:vanish_bdy}]
 The proof follows from the Claim \ref{ch4_cor2.2}
 and a reflection argument. Let $ \eta $ and $\zeta$ satisfy the 
 equation \eqref{ch4_eq2.2}. We shall extend both $ \xi $ 
 and $ \eta $ to $\widetilde{\Omega} $ in two different ways. 

{\bf I)} Let $\widetilde{\eta}, \tilde{\zeta} $ be odd and even
 extensions of $\eta$ and $\zeta$.  
\[  
\widetilde{\eta}(s,t)=\begin{cases}
	           \eta (s,t),  \;\;s \geq 0  \\ \\
		    -\overline{\eta(-s,t)} \;\; s < 0 
		\hspace{2cm}
\end{cases}
 \tilde{\zeta}(s,t)=\begin{cases} 
		\zeta(s,t) , \;\;s\geq 0\\ \\
		\overline{\zeta(-s,t)},  \;\;s< 0 
\end{cases}
\]
										
Here $\overline{\zeta}$ represents the image of $\zeta$ 
made by symmetry with respect to the plane $x_1=x_2=\cdots =x_n=0$,
or equivalently
 
 \[
\mathbf{\overline{\zeta}} =
\begin{pmatrix}
1 & 0\\
0 & -1
\end{pmatrix} \cdot \zeta
\] 

We shall prove that the extended functions
$\widetilde{\eta}$ and $\tilde{\zeta}$ satisfy the equation 
~\eqref{ch4_eq2.3}, i.e. for all 
$ \phi \in C^{\infty}_{c,bc}(\widetilde{\Omega}) $ 
the following holds:

\begin{equation}\label{eq2}
\int\limits_{\widetilde{\Omega}} \langle \partial_s \phi + J_0\partial_t \phi,\widetilde{\eta} \rangle = - \int\limits_{\widetilde{\Omega}} \langle \phi , \tilde{\zeta}\rangle 
\end{equation}

Let $ \phi\in C^{\infty}_{c, bc} (\widetilde{\Omega}) $ 
be an arbitrary function.
Define functions $ \phi_0$ and $\phi_1$ in the following way
\begin{align*}
\phi_0(s,t) =\frac{1}{2}\Big ( \phi(s,t)+ \overline{\phi(-s,t)} \Big), \;\;\;\;
\phi_1(s,t) =\frac{1}{2}\Big ( \phi(s,t)- \overline{\phi(-s,t)} \Big)
\end{align*}
Obviously $ \phi_0(-s,t)= \overline{\phi_0(s,t)} $, $ \phi_1(-s,t)= 
-\overline{\phi_1(s,t)} $ and $ \phi= \phi_0 + \phi_1 $. 
It also holds
\begin{align*}
\partial_s \phi_0(-s,t) =-\overline{\partial_s \phi_0(s,t)},\;\;
 \partial_t \phi_0(-s,t) = \overline{\partial_t \phi_0(s,t)}.
\end{align*}

Hence, we have that $\partial_s \phi_0(-s,t)+ J_0\partial_t\phi_0(-s,t)
= -(\overline{\partial_s \phi_0(s,t)+ J_0\partial_t\phi_0(s,t)})$. 
For $\phi_1$ the reverse holds i.e.
$$ 
\partial_s \phi_1(-s,t)+ J_0\partial_t\phi_1(-s,t)
= \overline{\partial_s \phi_1(s,t)+ J_0\partial_t\phi_1(s,t)}. 
$$ 
Now it is easy to see that 
 $$
 \int\limits_{\widetilde{\Omega}} 
 \langle \partial_s \phi_1 + J_0\partial_t \phi_1,\widetilde{\eta} \rangle
 =0=\int\limits_{\widetilde{\Omega}} \langle \phi_1,\tilde{\zeta} \rangle $$
 and also
 
 \begin{align}\label{eq3}
 \int\limits_{\widetilde{\Omega}} 
 \langle \partial_s \phi_0 + J_0\partial_t \phi_0,\widetilde{\eta} \rangle 
 &= 2\int\limits_{\Omega} \langle \partial_s \phi_0 +
 J_0\partial_t \phi_0,\eta \rangle\notag\\
 -\int\limits_{\widetilde{\Omega}} \langle  \phi_0 ,\tilde{\zeta} \rangle
 &= -2 \int\limits_{\Omega} \langle  \phi_0 ,\zeta \rangle
 \end{align}
 
 By assumption \eqref{ch4_eq2.2} the integrals 
 on the right side of \eqref{eq3} are equal, it follows
 
 \begin{equation}
 \int\limits_{\widetilde{\Omega}} \langle \partial_s \phi + 
 J_0\partial_t \phi,\widetilde{\eta} \rangle =
 - \int\limits_{\widetilde{\Omega}} \langle \phi , \tilde{\zeta}\rangle .
\end{equation}
Thus, we have proved the equation \eqref{eq2}. 
From corollary \ref{ch4_cor2.2} we have that 
$\widetilde{\eta}\in H^1_{bc}(\widetilde{\Omega},\R^{2n} )$
as $ \eta = \widetilde{\eta}|_{\Omega} $ we have that $\eta $ also satisfies

$$ 
\eta(s,i)\in \R^n \times \{0\} ,  \; \eta \in H^1_{loc} ( \Omega ) .
$$

We prove that $ \widetilde{\eta}(0,t)=\eta(0,t) \in 0\times \R^{n} $ 
for almost every $t.$ Choose a sequence of smooth functions
$\eta_i(s,t) $ that converges to $\widetilde{\eta} $ 
on every compact subset of $\widetilde{\Omega} $ in $ W^{1,2} $ norm. 
Notice that $\widetilde{\eta}=
\frac{1}{2}\Big ( \widetilde{\eta}(s,t) -\overline{\widetilde{\eta}(-s,t)} \Big ) $.
Then the sequence
$ h_i(s,t)= \frac{1}{2} \Big( \eta_i(s,t) -\overline{ \eta_i(-s,t) } \Big) $ 
also converges to $\widetilde{\eta} $ in $ W^{1,2} $ norm.
We also have that 
 $ h_i(0,t)= \frac{1}{2} \Big ( \eta_i(0,t) -\overline{\eta_i(0,t)} \Big )
 \in \{0\} \times \R^{n}.$
 As $ h_i(0,t) \stackrel{L^2}{\longrightarrow} \eta(0,t) = \widetilde{\eta}(0,t) $
 it follows that $ \eta(0,t) \in \{0\} \times \R^{n} $ for almost every $t$. 
\\
 {\bf II)} We will extend now $ \eta $ and $\zeta$ 
 reverse than in the case {\bf I)}, i.e. $\eta $ even
 and $\zeta$ odd. 
 Define $\widetilde{\eta} $ and $\tilde{\zeta}$ as
\[  
\widetilde{\eta}(s,t)=\begin{cases}
	           \eta (s,t), \;\; if \;\;s \geq 0  \\ \\
		    \overline{\eta(-s,t)}\;\; if\;\; s < 0 
		\hspace{2cm}
\end{cases}
 \tilde{\zeta}(s,t)=\begin{cases} 
		\zeta(s,t) , \;\; if \;\;s\geq 0\\ \\
		-\overline{\zeta(-s,t)}, \; \; if \;\;s< 0.
\end{cases}
\]
	
 Extended functions $\widetilde{\eta}$ and $\tilde{\zeta} $ 
 satisfy \eqref{eq2}. In order to prove that 
 one can use the same decomposition of function $\phi$. 
 This time the integral with $\phi_1$ will be doubled 
 and the integral with $\phi_0$ will vanish. 
 Therefore we conclude from the Claim \ref{ch4_cor2.2} that 
 $\widetilde{\eta}\in W^{1,2}_{loc} (\widetilde{\Omega}) $. 
 In the same way as in {\bf (I)} we prove that 
 $$ \widetilde{\eta}(0,t)=\eta(0,t) \in \R^n \times\{0\} \; \text{for a.e. } t. $$  
 
 Thus from {\bf I)} and from what we have just proved 
 $$ \eta(0,t)\in (\R^{n}\times \{0\} ) \cap (\{ 0\} \times \R^n )= 0 $$ 
 for almost every $t\in [0,1].$
 \end{proof}
 
\medskip 
\end{PARA}

\begin{PARA}[{\bf Unique continuation}]\label{para:unq_con}\rm
Let $ I = [0, \epsilon ) $, where $ \epsilon $ is 
possibly infinite and let $ D $ be an operator as in \eqref{eq:opD}.
In this paragraph we observe the mapping 
$$
\xi \mapsto ( D\xi, \xi(0))= ( \p_s \xi + J_0 \p_t \xi + S \xi, \xi (0)) 
$$
In the case $ \eps= +\infty $ we suppose $~(H1) $. 
This mapping is injective, 
i.e. any $ \xi $ with $ \xi(0, \cdot) = 0 $ and $ D\xi = 0 $
has to vanish everywhere $ \xi\equiv 0 $. 
This follows from Agmon and Nirenberg trick 
and we shall not discuss the details of the proof here. 
It is proved by careful study of the function $ \ln(\|\xi\|) $ 
and can be seen in  \cite{RS3}. 

\begin{lemma}\label{lem:unq_con}
Let $i=1 $ or $i= 2$ and let $ D$ be an operator 
as in \eqref{eq:opD}. Then the mapping 
 \begin{equation}
  \begin{split}
   & H^i_{bc} ( [0, \epsilon) \times [0,1] ) \to H^{i-1}_{bc} ( [0, \epsilon) \times [0,1] )\times H^{i-1/2}_{bc}\\
   & \xi \mapsto ( D\xi, \xi(0))
  \end{split}
 \end{equation}
is injective. The analog holds for the operator 
$D^*= - \p_s + J_0 \p_t + S(s,t)^T $. 
Namely, the mapping 
 \begin{equation}
  \begin{split}
   & H^{i}_{bc} ( [0, \epsilon) \times [0,1] ) \to H^{i-1}_{bc} ( [0, \epsilon) \times [0,1] )\times H^{i-1/2}_{bc}\\
   & \xi \mapsto ( D^*\xi, \xi(0))
  \end{split}
 \end{equation}
is injective.
\end{lemma}
\begin{proof}
The proof is verbatim the same as the proof of  Lemma 3.3 
in \cite{RS3} and we shall not repeat it here. 
\end{proof}
 
\end{PARA}

\begin{corollary}\label{lem:inq_D}
Let $ D $ be an operator of the form \eqref{eq:opD}. 
\begin{itemize} 
 \item[i)] Let $ I = [a,b] $. There exists a constant $c > 0 $ such 
 that the following inequality holds for all $ \xi \in H^2_{bc} ( I \times [0,1]) $. 
 \begin{equation}
 \| \xi \|_{2,2} \leq c \left ( \| D \xi \|_{1,2} + \| \xi(a,\cdot)\|_{3/2} + \| \xi (b, \cdot) \|_{3/2} \right )
\end{equation}
\item[ii)] Suppose that $ I = \R^{\pm} $ and assume ~(H1). Then there 
exist positive constant $c $ such that the following inequality holds
for all $ \xi \in H^2_{bc} ( \R^{\pm} \times [0,1] ) $. 
\begin{equation}\label{eq:main_lin}
 \|\xi\|_{2,2} \leq  c \Big ( \|D \xi \|_{1,2} + \| \xi(0, \cdot) \|_{3/2} \Big )
\end{equation}
\end{itemize}
\end{corollary}
\begin{proof}
We shall prove part $ii) $, the proof of $ i) $ is analogous. 
It follows directly from the inequality \eqref{eq:pp0} 
in Theorem \ref{thm:main_inq1} that
\begin{equation}\label{eq:hlp1}
 \| \xi \|_{2,2} \leq c \Big (  \| D \xi \|_{1,2} + \| K\xi \|_{1,2} + \| \xi (0)\|_{3/2}\Big ),
\end{equation}
as $ \| \pi^+ (\xi(0))\|_{3/2} \leq \| \xi(0)\|_{3/2} $. 
Remember also that the operator 
$  H^2_{bc} ( \R^+ \times [0,1] ) \ni \xi\mapsto 
 K \xi \in H^1_{bc} ( \R^+\times [0,1] ) $ is compact.
From the inequality \eqref{eq:hlp1} it follows that the 
operator $ \xi \mapsto ( D \xi, \xi(0)) $ has closed image and finite 
dimensional kernel. By Lemma \ref{lem:unq_con} it follows that 
it is bijective onto its image. 
From the open mapping theorem it follows 
that its inverse is bounded and we can omit the middle term, i.e. $ \| K \xi \|_{1,2} $
of the inequality 
\eqref{eq:hlp1}. Thus we have proved the required inequality. 
\end{proof}

\begin{corollary}\label{cor:surj}
Let $ I= [a,b] $ or $ I = \R^{\pm} $. 
Suppose that the operator 
$ D:H^i_{bc}( I \times [0,1] ) \to H^{i-1}_{bc} ( I \times [0,1] ), \; i=1,2 $ 
 has the form \eqref{eq:opD}, in the case $ I = \R^{\pm} $ 
 we assume ~(H1). Then the operator $D$ is surjective. 
  \end{corollary}
  \begin{proof}
  
  {\bf Step 1.} Surjectivity of the operator 
  $$
  D :H^1_{bc} ( I \times [0,1] ) \to L^2 (I \times [0,1] ). 
  $$
  \medskip 
  Let $ \Omega = I \times [0,1]$ and let $\eta \in L^2 ( \Omega ) $ 
  be orthogonal to the image of $D$. Then we have that 
  $$
  \int_{\Omega} \inner{\p_s \xi + J_0 \p_t\xi + S\xi}{\eta}= 0
  $$
   holds for all $ \xi \in H^1_{bc} (\Omega) $. 
  Particularly this implies that for all 
  $\phi\in C^{\infty}_{c, bc} (\Omega) $ ( defined in \eqref{eq:cinf_bc})
  the following equality holds 
  $$ 
  \int_I \int_0^1 \inner {\p_s \phi + J_0 \p_t \phi} {\eta} ds dt 
  = -\int_I \int_0^1 \inner{S^T \eta}{\phi} 
  =-\int_I\int_0^1 \inner{\zeta}{\phi} ds dt. 
  $$
  It follows from Lemma \ref{lem:vanish_bdy} that 
  $ \eta \in H^1_{bc} ( I \times [0,1] ) $ 
  and that $ \left. \eta\right|_{\p I} = 0 $ and $ \eta $ 
  is a strong solution of the equation 
  $$
  -\p_s \eta + J_0 \p_t \eta = -\zeta = - S^T \eta.
  $$
  From Lemma \ref{lem:unq_con} it follows that $ \eta = 0 $. 
  Thus the image of the operator 
  $$
  D : H^1_{bc} ( \Omega ) \to L^2 ( \Omega ), \;\; 
  D\xi = \p_s \xi + J_0 \p_t \xi + S(s,t)\xi
  $$
  is dense and as its image is also closed (see Lemma \ref{lem:closed_img}) 
  we have that $D$ is surjective. \\
  {\bf Step 2.} Let $ \cW^1_{\mp} ( I) $ be defined as in 
  \eqref{eq:w1mp}. Then there exist smooth functions 
  $ \xi_i \in H^1_{bc} (I \times [0,1] ) , \; i=1, \cdots m$
  such that 
  $$ 
  D : \cW^1_{\mp} ( I ) \cup \text{Span}\{ \xi_1, \cdots, \xi_m\} \to L^2( I \times [0,1] ) 
  $$
  is surjective. \\
  \medskip
  Notice that the operator $ D  $ can be written in the form 
  $ D= D_A + R $, where $ R ( \xi ) $ is given as a multiplication 
  by some $ W^{1,2} $ matrix valued function. 
  As the operator $ D_A : \cW^1_{\mp}(I) \to L^2 $ 
  is bijective (Corollary \ref{cor:bij_lin_op}), we have that 
  the operator $ D: \cW^1_{\mp}(I) \to L^2  $ is Fredholm of index $0$, as a
  compact perturbation of the operator $D_A$. From Step 1 it 
  follows that there exist $ \xi_i, \;\; i=1, \cdots, m $ such that 
  the restriction of the operator $D$ to the $ \cW^1_{\mp} (I) \cup \text{Span}\{ \xi_1, \cdots, \xi_m\} $ 
  is surjective. Notice that each $ \xi_i $ can be approximated by smooth elements 
  $ \xi_i^k, \; k=1, \cdots , \infty $ which also satisfy the condition 
  $D_A \xi^k_i \in H^1_{bc} ( I \times [0,1] ) $.  Thus we have that 
  $\xi^k_i \in H^2_{bc} ( I \times [0,1] ) $. Thus for sufficiently 
  large $k$ we have that the restriction of the operator $D$ to 
  $\cW^1_{\mp} (I) \cup \text{Span} ( \xi^k_1, \cdots, \xi^k_m) $ 
  is surjective. Thus, we can assume w.l.o.g. that $ \xi_i \in H^2_{bc} ( I \times [0,1])$ are smooth. 
 \\
 {\bf Step 3.} Let $ \cW^2_{\mp} (I) $ be defined as in \eqref{eq:w2pm} and let 
 $ \xi_i, i=1, \cdots, m $ be as in Step 2. Then 
 $$
  D: \cW^2_{\mp} (I) \cup \text{Span} \{ \xi_1, \cdots, \xi_m\} \to H^1_{bc} ( I \times [0,1] ) 
 $$
 is surjective. \\
 Let $ \eta \in H^1_{bc} ( I \times [0,1] ) $. From Step 2 it follows that 
 there exist $ \xi \in \cW^1_{\mp} (I) $ and $ \alpha_i \in \R$ such that 
 $ D ( \xi + \alpha_i \xi_i) = \eta $. We prove that $ \xi \in \cW^2_{\mp} (I) $ 
 actually. First notice that 
 $$ D\xi = \eta - \sum\limits_i\alpha_i D \xi_i 
 = \eta - \sum\limits_i \eta_i  = \eta' \in H^1_{bc} ( I \times [0,1] ) 
 $$
 thus we have that 
 $$ 
 D \xi = D_A \xi + R \xi = \eta' \;\; \;\; \Rightarrow \;\;\;  D_A \xi = \eta' - R \xi = \teta
 $$
 As $ R $ is a $ W^{1,2} $ function and $ \xi $ as well, we have that their product 
 is an $ L^p $ function for any $ p < \infty $. Thus the function $ \teta \in L^p( I \times [0,1] ) $
 and it follows from Theorem \ref{thm:reg1} that $ \xi \in W^{1,p}_{bc} ( I \times [0,1] ) $
 for some $ p > 2 $. This implies that the product $ R \xi $ is actually a $ W^{1,2} $ function 
 and it also satisfies the right boundary condition and hence $ \teta \in H^1_{bc} ( I \times [0,1] ) $. 
 From Corollary \ref{cor:bij_lin_op} we have that $ \xi \in \cW^2_{\mp} (I) $. \\
 Steps 1-3 prove that the operator $D$ is surjective.  
  \end{proof}

  \section{Appendix}
 \subsection{Abstract interpolation theory}
Let $ H $ and $ W $ be Hilbert spaces and let the 
inclusion $W\hookrightarrow H $ be continuous and dense. 
We define the space $ \cW = \cW(0, + \infty ) $ as follows
\begin{align*}
 \cW= \cW(0, + \infty) &= \Big \{ x : x \in L^2((0, + \infty), W ) ),
 \; \frac{\partial u}{\partial s} \in L^2 ( (0 , +\infty), H ) \big \}  
\\      &= \Big \{  x \in L^2 ( (0, +\infty), W ) \cap W^{1,2} ( (0, +\infty), H) \Big \}
\end{align*}
with the norm 
\begin{align*}
 \|x\|_{\cW}^2 & = \|x\|^2_{L^2( (0, + \infty), W)} +
              \|\dot{x}\|^2_{L^2((0, +\infty), H)} 
= \int_0^{+ \infty} \Big ( \|x(s)\|^2_{W} + \|\dot{x}(s)\|^2_H\Big ) ds 
 \end{align*}

\begin{remark}
The space $ \cW$ is a Hilbert space and the space 
$C^{\infty}_c ( [0, +\infty), W ) $ is dense in $\cW $.
\end{remark}
\begin{definition}
The trace space $V$ of the space $ \cW $ is given by 
\begin{align*}
 V= Tr(\mathcal{W}) &:= \Big \{ \xi \in H :\;\; \exists x \in \mathcal{W}, \; x(0) = \xi  \Big \} \\
\|\xi\|_{1/2}= \|\xi\|^2_{V} &:= \inf\limits_{x\in \mathcal{W}, \; x(0) = \xi }
\int_0^{+\infty}  ( \|\dot{x} (s)\|^2_H + \|x(s)\|_W^2 ) ds 
\end{align*}
\end{definition}
\begin{remark}\rm
 One could also use finite interval $I= (0,1) $ or $I= \R $ 
 instead of the interval $(0, + \infty ) $ to define the space $ \cW (I ) $,
 but the trace space $V=Tr ( \cW) $ will always be the same.   
\end{remark}
It is easy to see that the norm $ \|\cdot\|_{1/2}$ is really a norm
and that with respect to this norm the space $V$ is a Banach space. 

\begin{definition}\label{ap_def1.4}
 Suppose that $H, W $ are Hilbert spaces with dense 
 and continuous inclusion $W\hookrightarrow H $.  
 Let $ A: W \rightarrow H $, be a linear operator 
that satisfies the following

\begin{itemize}
 \item[i)] A is self-adjoint with the domain $ D(A) = W $, i.e.
 \begin{align*}
 \langle Ax, y\rangle_H = \langle x, Ay \rangle_H ,\;
\end{align*}
for all $ x, y\in W = D(A) $. It follows from Hellinger-Toplitz theorem 
that $A$ is also continuous. 
\item[ii)] A is positive, $ \langle Ax, x \rangle \geq 0 $ for all $ x \in W $.
\item[iii)] Suppose also that $ A $ is bijective.  
Hence, there exists a positive constant $ c_0 $ such that 
 \begin{equation}\label{ap:eq_1}
  \frac{1}{c_0}  \|x\|_W  \leq \| Ax \|_H \leq c_0 \|x\|_W 
 \end{equation}
 The right inequality follows from continuity of $A$, and left from 
 open mapping theorem. 
\end{itemize}
 For $ \theta\in [0,1] $ we define the
 {\bf intermediate (interpolation) space }
 $$ 
 [W, H]_{1-\theta, A}= \text{Dom}( A^{\theta}). 
 $$
\end{definition}

 \begin{definition}
 We say that a self-adjoint operator $ A : W \rightarrow H $
 is a {\bf purely point operator} ( or {\bf has a purely point spectrum} ) 
 if the following holds: \\
 There exists an $ H-$orthogonal decomposition
 $ H = \bigoplus H_i $ , where each $ H_i = \langle e_i \rangle $,
 and $ e_i $ is an eigenvector of the operator $A$,
 i.e. $ A (e_i)= \lambda_i e_i $. \\
  A sufficient condition that a symmetric operator 
  has eigenvectors which form a Hilbert space basis 
  is that it has a compact inverse.
  \end{definition}
  
  \begin{remark}\label{ap_rem1.1}\rm
  If $ A $ is a purely point operator 
  which satisfies the requirements of the Definition 
  \ref{ap_def1.4}, then for $ H \ni \xi= \sum\limits_i a_i e_i $ we have
  $$[W,H]_{1-\theta, A} =  \text{Dom}(A^{\theta} ) =
  \Big \{ \xi \in H , \; \xi = \sum\limits_i a_i e_i : 
  \; \sum\limits_i \lambda_i^{2\theta} \abs{a_i}^2 < + \infty \Big \}  .$$
   One could do the same for an operator $ A $ 
   which is not necessarily positive (Thus the condition $ii) $
   in definition \ref{ap_def1.4} is superfluous. 
   Namely, the operator 
   $ \abs{A} $ is positive and self-adjoint
   and we can define 
    $$
    [W,H]_{1-\theta, A} =  \text{Dom}(\abs{A}^{\theta} ) = 
    \Big \{ \xi \in H , \; \xi = \sum\limits_i a_i e_i : 
    \; \sum\limits_i \abs{\lambda_i}^{2\theta} \abs{a_i}^2 < + \infty \Big \}  .
    $$
   Notice that the space $[W,H]_{1-\theta, A}$ is a Hilbert space with 
   the scalar product
   $$ 
   \xi= \sum\limits_ia_i e_i, \;\; \eta=\sum\limits_i b_i e_i, \;\;
   \langle \xi, \eta\rangle_{1-\theta, A} =\sum\limits_i a_i b_i \abs{\lambda_i}^{2\theta} 
   $$ 
  
 \end{remark}
 
 \begin{remark}\rm
In our intended applications of this theory the operator $ A $
will be the square root of the Laplacian,
$ A = \sqrt{-\triangle} = i \partial_t $ 
or some compact perturbation of $\sqrt{-\triangle} $. 
The spaces $ W $ and $H $ will be some 
$ H^k([0,1] )= W^{k,2}([0,1] ) $ spaces with certain boundary conditions. 
\end{remark}

In the next theorem we prove that the space $ [W,H]_{1/2, A} $
doesn't depend on the operator $ A $ and that 
it is the same as the trace space $ V= Tr(\cW) $. 

\begin{theorem}\label{apthm1}
Let $ A $ be an operator as in Remark \ref{ap_rem1.1}
and  let $\xi\in H$. Then $ \xi \in [W, H]_{1/2, A}= D( \sqrt{A} )$ 
if and only if $\xi \in V=Tr(\cW)  $ and 
there exists a constant  $c>0 $ such that 
for all $\xi \in D( \sqrt{A} ) $ 
$$\frac{1}{c} \|\xi\|_{1/2} \leq \|\xi\|_{1/2,A} \leq c \|\xi\|_{1/2} .$$
\end{theorem}
{\bf Proof:} We shall devide the proof of 
this theorem into three steps \\
{\bf Step 1:} For all $ x \in C^{\infty}_c ( [0, + \infty), W ) $ 
we have that 
\begin{equation}\label{apeq1}
 \|x(0)\|^2_{1/2,A} = \|\xi\|_{1/2,A} ^2
 \leq  c \int_0^{+\infty}  \Big (\|\dot{x}(s)\|^2_H + \|x(s)\|_W^2 \Big ) ds 
\end{equation}
\begin{proof}
 As $ W \subset [W,H]_{1/2, A} $ we have that
 $ x(s) \in [W,H]_{1/2, A} $, for all $s \in [0, + \infty) $.
 Let $ f= \dot{x}(s) + A(x(s)) $, then 
  \begin{align}\label{eq:ap_pom1}
  \frac{d}{ds} \frac{1}{2} \|x\|_{1/2,A}^2 &= \langle x, \dot{x}\rangle_{1/2,A} =\notag \\
   & = \langle x, f-A x\rangle_{1/2,A} = \langle x, f\rangle_{1/2,A} - \|Ax\|^2_H 
 \end{align}
 The last equality holds as
 $ \langle Ax, x \rangle_{1/2,A} =\sum_{i=1}^{+\infty} \lambda_i^2 \langle x, e_i \rangle^2 =\|Ax\|^2_H $. 
 Integrating the inequality \eqref{eq:ap_pom1} we obtain 
 \begin{align*}
 \|\xi\|^2_{1/2,A}= \|x(0)\|^2_{1/2,A } & = \int_0^{+\infty} \|Ax\|_H^2 -\int_0^{+\infty} \langle x, f\rangle_A ds \\
              &\leq \int_0^{+\infty}\|A(x(s))\|^2_H dt + \frac{1}{2} \int_0^{+\infty}\| A(x(s))\|^2_H ds + \frac{1}{2}\int_0^{+\infty} \|f\|^2_H ds \\
              &\leq \int_0^{+\infty} \Big ( \frac{1}{2} \|\dot{x}(s) + A(x(s))\|_H^2 + \frac{3}{2}\|A(x(s))\|^2_H \Big ) ds \\
              & \leq c_1^2 \int_0^1 \Big ( \|\dot{x}(s)\|^2_H + \|x(s)\|^2_W \Big )ds
 \end{align*}
 and the constant $ c_1= \max \left \{ 1, \sqrt{\frac{5}{2}}c_0 \right \}$, 
 where $ c_0 $ is the constant from the inequality \eqref{ap:eq_1}.
\end{proof}

 {\bf Step 2:} If $ x \in L^2 ([0,+ \infty), W) \cap W^{1,2} ([0,+\infty) , H)=\cW $, 
 then $\forall \tau \in [0,+ \infty ),  \; x(\tau) \in D( \sqrt{A} ) $. 
 \\ 
 \begin{proof}
For functions $x \in C^{\infty}_c([0,+\infty), W ) $ this obviously holds, 
as $ W\subset D( \sqrt{A} ) $.  As the set $ C^{\infty}_c ([0,+\infty), W)$ 
is dense in $ \cW $ choose sequence $x_k \longrightarrow x $.
The sequence $x_k $ is Cauchy w.r.t. the norm
$ \int_0^{+\infty} \Big ( \|\dot{x}(s)\|^2_H + \|x(s)\|^2_W \Big )dt$. 
Therefore this sequence is also Cauchy w.r.t. the norm $\|\cdot\|_{1/2,A} $ 
in $D( \sqrt{A} )$ (Remark that the inequality proved 
in the first step holds for every $\tau \in [0,+\infty) $, not just $\tau=0 $ ).
Therefore 
 $$ x(\tau) = \lim\limits_{k\rightarrow \infty} x_k (\tau) \in D( \sqrt{A} ). $$  
 \end{proof}
 {\bf Step 3:} There exists $\delta > 0 $ such that for all $\xi \in D( \sqrt{A} )$
 there exists $x\in \mathcal{V} $ with $x(0)=\xi $ and such that 
 \begin{equation}\label{apeq2}
 \|\xi\|^2_{1/2,A} \geq \delta \int_0^{+\infty} \Big ( \|\dot{x}(s)\|^2_H + \|x(s)\|^2_W \Big ) ds 
 \end{equation}
\begin{proof}
 Let $ \xi = \sum\limits_{i=1}^{+\infty} \xi_i ,\;x(s) = \sum\limits_{i=1}^{+\infty} x_i (s) $.
 As $\xi \in D ( \sqrt{A} )$ we have $ \sum\limits_i \lambda_i \xi_i^2 < +\infty $. 
 The solution of the following problem 
\begin{align}\label{ap:eq:int1}
 \dot{x}(s) + A(x(s))= 0, \;\; x(0) = \xi
 \end{align}
 is given by $ x= \sum\limits_i x_i(s) = \sum\limits_i \xi_i e^{-\lambda_i s } $. 
 We shall prove that $ x\in \mathcal{W} $ and that
  \begin{align}
  \|\xi\|^2_{1/2, A} \geq c \int_0^{+\infty} \Big ( \|\dot{x}(s)\|^2_H + \|A(x(s))\|^2_H \Big ) ds,
 \end{align}
for some positive constant $ c $. As $ x $ satisfies 
the equality (\ref{ap:eq:int1}) it is 
enough to prove that $\int\limits_0^{+\infty} \|A(x(s))\|_H ds < +\infty $. 
Notice that 
\begin{align}
 \int_0^{+\infty} \|A(x_i)\|_H^2 &= \int_0^{+\infty} \lambda_i^2 \xi_i^2 e^{-2\lambda_i s} ds  \\
                                  &= \lambda_i \xi_i^2 \int_0^{+\infty} \lambda_i e^{-2\lambda_i s} ds 
                                  = \frac{1}{2} \lambda_i \xi_i^2  
\end{align}
hence, 
$$ 
\int_0^{+\infty} \|A(x)\|_H^2 ds = 
\int_0^{+\infty} \sum\limits_i \|A(x_i)\|^2_H = \frac{1}{2} \|\xi\|_{1/2, A}^2 . 
$$
As $ \|A(x)\|_H \geq \frac{1}{c_0} \|x\|_W $ we have 
\begin{equation}
\|\xi\|^2_{1/2, A} = 
2 \int_0^{+\infty} \|A(x)\|_H^2 ds \geq
\delta \int_0^{+\infty} \Big ( \|\dot{x}(s)\|^2_H + \|x(s)\|^2_W \Big ) ds,
\end{equation}
where $ \delta$ behaves as $ \frac{1}{c_0^2} .$ 
 
\end{proof}

\subsection{ $L^p$ estimates}\label{sec:lp}

In section \ref{SEC:lin_est} we have considered 
the linearized operator $D_A $ given by 
\begin{equation*}
\begin{split}
D_A &: W^{2,2}_{bc} ( I \times [0,1] ) \to W^{1,2}_{bc} ( I \times [0,1] )\\
D_A &\xi = \p_s \xi + A \xi = \p_s \xi + J_0 \p_t \xi + S(t) \xi,
\end{split}
\end{equation*}
where the operator $ A $ satisfies ~ (HA). 
Now we consider the same operator, but as the domain we 
take $ W^{1,p}_{bc} ( \R\times [0,1] ) $ defined by 
\begin{equation}
 W^{1,p}_{bc} ( \R\times [0,1] ) := \left \{ \xi \in W^{1,p} ( \R \times [0,1] ) \Big|\xi(s,i) \in \R^n \times \{0\} \right \}. 
\end{equation}
The next lemma is an analogous of the Lemma 2.4 in 
\cite{S1} and the proof is almost verbatim taken 
from there. 
\begin{lemma}\label{lem:lp_est}
 Suppose that $ S: [0,1] \to R^{2n \times 2n} $ is a smooth function 
 and suppose that the operator $ A= J_0 \p_t + S(t) $ satisfies ~(HA). 
 Let $ 1<p< \infty $ and let $D_A$ be given by 
 \begin{equation*}
\begin{split}
D_A &: W^{1,p}_{bc} ( \R \times [0,1] ) \to L^p ( \R \times [0,1] )\\
D_A &\xi = \p_s \xi + A \xi = \p_s \xi + J_0 \p_t \xi + S(t) \xi,
\end{split}
\end{equation*}
is bijective and there exists a constant $c > 0 $ such that 
the following inequality 
\begin{equation}\label{eq:lp_est}
 \| \xi \|_{W^{1,p}} \leq c \| D_A \xi \|_{L^p},
\end{equation}
holds for all $ \xi \in W^{1,p}_{bc}(\R \times [0,1] ) $.
\begin{proof}
{\bf Step 1. } The claim holds in the case $ p= 2 $. \\
Notice that 
\begin{align*}
 \| D_A \xi \|_{L^2}^2 &= \int\limits_{-\infty}^{+\infty} \langle \p_s \xi + A \xi, \p_s \xi + A \xi \rangle_{L^2} ds \\
		      & = \int\limits_{-\infty}^{+\infty} \Big ( \| \p_s \xi \|^2_{L^2} + \| A \xi \|^2_{L^2} \Big ) + 
		      \int\limits_{-\infty}^{+\infty} \p_s \inner{\xi}{A\xi}ds \\
		      &= \int\limits_{-\infty}^{+\infty} \Big ( \| \p_s \xi \|^2_{L^2} + \| A \xi \|^2_{L^2} \Big )\\
		      & \geq c' \| \xi \|_{W^{1,2}( \R \times [0,1] )}. 
\end{align*}
Thus, we have proved that the operator $ D_A $ is injective 
and has a closed image. To prove that it is also 
surjective we can use eigenvector decomposition.  Let $ \eta \in L^2 ( \R \times [0,1] ) $. 
Write $ \eta(s,t) = \sum_{\lambda} \eta_{\lambda}(s,t) $, where $ \eta_{\lambda} $ 
are the eigenvectors of the operator $A$. Observe the equation 
$ D_A \xi = \p_s \xi + A \xi = \eta $. Then the solution $ \xi = \sum_{\lambda} \xi_{\lambda} $ 
of this equation is given by 
\begin{equation*}
 \begin{split}
  \xi_{\lambda}(s,t) = \int_{-\infty}^s e^{-\lambda ( s- \tau)} \eta_{\lambda} ( \tau, t) d \tau ,\;\; \lambda > 0\\
  \xi_{\lambda} (s,t) = -\int_s^{\infty} e^{-\lambda ( s- \tau)} \eta_{\lambda} ( \tau, t) d \tau, \;\; \lambda< 0 
 \end{split}
\end{equation*}
Notice that $ \xi = \int\limits_{-\infty}^{+\infty} K(s-\tau) \eta(\tau,t)d\tau $, 
where $ K$ decays exponentially. \\
{\bf Step 2.} Let $ p \geq 2 $. There exists a constant $ c_1 > 0 $ such that
\begin{equation}
 \| \xi \|_{W^{1,p} ( [0,1] \times [0,1] )} \leq c_1 \left ( \| D_A \xi \|_{L^p([-1,2]\times [0,1] )} + \| \xi\|_{L^2([-1,2] \times [0,1] )} \right) 
\end{equation}
holds for all $ \xi \in W^{1,p} ( [-1,2] \times [0,1] ) $. Moreover, if $ \xi \in W^{1,2} $ and 
$ D \xi \in L^p_{\loc} $ then $ \xi \in W^{1,p}_{\loc} $.\\
From Calderon-Zygmung inequality we have that 
\begin{align*}
 \| \nabla\xi \|_{L^p( [0,1] \times [0,1] ) }\leq c \| \dbar \xi \|_{L^p ( [- 1/2, 3/2] \times [0,1] )}. 
\end{align*}
This implies the above estimate with $L^2 $ norm on the right side replaced 
by $ L^p $ norm. Let $ \Om' = [- 1/2, 3/2] \times [0,1]  $ and $ \Om=[-1,2]\times[0,1] $ we have 
\begin{align*}
 \| \xi \|_{W^{1,p} ( [0,1] \times [0,1] )} &  \leq c \Big ( \| D_A \xi \|_{L^p ( \Om')} + \| \xi \|_{L^p ( \Om')} \Big ) \\
  & \leq  c \Big ( \| D_A \xi \|_{L^p ( \Om')} + \| \xi \|_{W^{1,2} ( \Om')} \Big ) \\
  &\leq c \Big (\| D_A \xi \|_{L^p ( \Om')} + \| D_A \xi \|_{L^2( \Om) }  + \| \xi\|_{L^2(\Om)} \Big )\\
  &\leq c \Big ( \| D_A \xi \|_{L^p ( \Om)} +  \| \xi\|_{L^2(\Om)} \Big )
\end{align*}
The second inequality follows from Sobolev embedding.\\
{\bf Step 3.} Consider the norm 
$$
\| \xi \|_{2,p} := \Big ( \int\limits_{-\infty}^{\infty} \| \xi(s,\cdot)\|^p_{L^2([0,1])}\Big )^{1/p}.
$$
There exists constants $ c_2, c_3 > o $ such that, if $ \xi \in W^{1,2}_{bc} ( \R\times[0,1] ) $ 
and $ D \xi \in L^p( \R \times [0,1] ) $, then $ \xi \in W^{1,p}_{bc} ( \R \times [0,1] ) $ and 
\begin{align*}
 \| \xi \|_{2,p} \leq c_2 \| D \xi \|_{L^p} , \;\; \| \xi \|_{W^{1,p}} \leq  c_3 ( \| D \xi \|_{L^p} + \| \xi \|_{2,p} ). 
\end{align*}
From Step 1 it follows that $ \xi\in W^{1,p}_{\loc} $, to prove that 
$ \xi \in W^{1,p} $ it is enough to prove the two estimates. The 
first inequality follows from Young's convolution inequality. 
From Step 1 we have that $ \xi = K * \eta $, where $ K(s) $ decays 
exponentially 
$$
\| \xi \|_{2,p} = \| K * \eta \|_{2,p} \leq \| K \|_{L^1 ( \R, L^2 ( 0,1))} \| \eta \|_{L^p ( \R, L^2(0,1) ) }\leq c_2 \| \eta\|_{L^p(\R \times [0,1] )}
$$
and the last inequality follows as $ \| \eta \|_{L^2( 0,1)} \leq  \| \eta \|_{L^p (0,1)} $
for $ p \geq 2 $. To prove the second inequality, we shall use the result from 
Step 1 and the inequality $ (a+b)^p \leq 2^p ( a^p + b^p) $. 
\begin{align*}
 \| \xi \|^p_{W^{1,p} ( [k,k+1]\times[0,1] )} &\leq 2^p c_1^p \Bigg ( \int\limits_{k-1}^{k+2} \| D \xi \|^p_{L^p([0,1])}ds + 
 \Big ( \int\limits_{k-1}^{k+2} \| \xi \|^2_{L^2([0,1])} ds \Big )^{p/2} \Bigg )\\
 &\leq 2^p c_1^p \Bigg (  \int\limits_{k-1}^{k+2} \| D \xi \|^p_{L^p([0,1])}ds + 3^{p/2-1} \int\limits_{k-1}^{k+2} \| \xi\|^p_{L^2([0,1])} \Bigg )\\
  &\leq 3^{p/2} 2^p c_1^p \int\limits_{k-1}^{k+2} \Big ( \| D\xi \|_{L^p}^p + \| \xi\|_{L^2}^p \Big )ds
\end{align*}
Taking the sum over all $ k \in \Z $ we obtain the desired inequality. \\
{\bf Step 4.} Proof of Lemma for $ p > 2 $.\\
From Step 3, we have that 
$$
\| \xi \|_{1,p} \leq c \| D\xi \|_{L^p}, 
$$
holds for all $ \xi \in C^{\infty}_{bc} ( \R \times [0,1] ) $ with compact 
support. Thus by density of such functions the above inequality holds 
for all $ \xi \in W^{1,p}_{bc} ( \R \times [0,1] ) $. Thus the operator 
$ D_A : W^{1,p}_{bc} \to L^p $ is injective and has closed image. Let 
$ \eta \in L^p( \R \times [0,1] ) \cap L^2 (  \R \times [0,1] ) $. 
By Step 1 there exists $ \xi \in W^{1,2}_{bc}(  \R \times [0,1] ) $
such that $ D \xi = \eta $. By Step 3 we have that $ \xi \in W^{1,p}_{bc} ( \R \times [0,1] ) $
and thus the image is dense and hence is onto. \\
{\bf Step 5.} Case $ 1<q<2 $. \\
By duality we have 
\begin{align*}
  \| D_A \xi \|_{L^q} &= \sup\limits_{\|\eta\|_{L^p} \leq 1} \left |  \int_{\Omega} \inner{D_A\xi}{\eta} ds dt \right| \\
 		    & \geq \sup\limits_{\|\eta\|_{W^{1,p}}\leq 1 } \left |  \int_{\Omega} \inner{D_A\xi}{\eta} ds dt \right|\\
 		    &= \sup\limits_{\|\eta\|_{W^{1,p}}\leq 1 }\left |  \int_{\Omega} \inner{\xi}{D^*_A\eta} ds dt \right|\\	    &= c \| \xi\|_{L^q} 
\end{align*}
The last equality follows from the bijectivity 
of the operator $ D_A* :W^{1,p}_{bc} \to L^p $, where $ D_A^* = - \p_s + J_0 \p_t + S(t) $. 
Particularly there exists $ \eta $ such that $ \| \eta \|_{W^{1,p}}\leq 1 $ 
and $ D_A^* \eta = \frac{c\abs{\xi}^{q-2} \overline{\xi}}{\|\xi\|^{q-1}_{L^q}} $. 
From the above inequality it follows that $D_A $ is injective and 
analogously as in Step 4 we have that its image is dense. 
\end{proof}

\end{lemma}

\chapter{Hardy space approach to gluing}\label{chp:hardy}
 
Symplectic Floer homology was introduced by Floer in  
 \cite{F1,F2,F3,F4}.
 Floer Gluing theorem is one of the main  technical ingredients
 in the construction of the Floer homology. 
 Together with  the compactness and linear 
 elliptic Fredholm theory is used to define Floer homology and prove its various 
 properties.  
 
 In this chapter we introduce a new approach to gluing of perturbed 
 holomorphic curves with Lagrangian boundary conditions. The motivation 
 comes on one hand from the work of Salamon, Robbin and Ruan in 
 \cite{RS4, RRS}, where similar 
 techniques were used in integrable case, on the other hand form
 the work of Kronheimer and Mrowka \cite{KM} in the Seiberg-Witten setting. \\
 
 Perturbed holomorphic curves with Lagrangian boundary conditions 
 are solutions $ u : I \times [0,1] \to M $ of the Floer equation 
 \begin{equation}\label{eq:per_hol1}
  \p_s u + J_t(u) ( \p_t u - X_{H_t} ( u)) = 0, \qquad u(s,0) \in L_0, \;\; u(s,1) \in L_1,
 \end{equation}
 where $ L_i, \; i=0,1 $ are Lagrangian submanifolds of a symplectic manifold $M$, 
 $J_t $ is a smooth family of almost complex structures 
 and $ I \subset \R $ is an interval. 
 The perturbation here comes 
 in the form of a Hamiltonian vector field $ X_{H_t} $.
 We shall be interested in the cases when the interval $ I $ 
 is either a half infinite interval or $ I=[-T,T] $. We introduce the 
 moduli space $ \sM^{\infty} $ consisting of the pairs of half-infinite 
 perturbed holomorphic curves, 
 i.e. solutions of \eqref{eq:per_hol1}, and analogously we 
 introduce the moduli space $ \sM^T$ of finite strips $ u : [-T, T] \times [0,1] \to M $
 satisfying \eqref{eq:per_hol1}. The moduli space
 $ \sM^T $ is a Hilbert submanifolds of certain Hilbert manifold of strips
 $ \sB^T$ consisting of $ W^{2,2} $ 
maps $ u : I \times [0,1] \to M $ satisfying 
 the standard Lagrangian boundary condition 
 $$ 
 u(s,i) \in L_i , \;\; i=0,1 
 $$
 and the condition on the first derivative 
 $$ 
 J_i(u(s,i)) (\partial_t u (s,i) -X_{H_t} (u(s,i)) \in T_{u(s,i)} L_i, \; i=0,1 
 $$
 and analogous result holds for $ \sM^{\infty}\subset \sB^{\infty} $. 
 This setup is needed in order to establish the necessary estimates for 
 the nonlinear Hardy spaces. 
 
 In section \ref{SEC:hilmf}, we introduce the Hilbert manifolds of strips
 and the Hilbert manifold $ \sP^{3/2} $ consisting of 
 $ W^{3/2, 2} $ paths that satisfy Lagrangian boundary conditions 
 and a certain condition on the first derivative. 
 The path space $ \sP^{3/2} $ is actually the trace manifold 
 of the Hilbert manifold of strips $ \sB^T$, namely restricting 
 an element $ u \in \sB^T $ to the free boundary $ u(\pm T, \cdot) $ we obtain 
 an element of $ \sP^{3/2} $. 
 For a Hamiltonian path $ x \in \sP^{3/2} $, we denote by $ \sU $ 
 some small neighborhood of $ x$ within $ \sP^{3/2} $. 
 We focus our attention to some subsets of the moduli 
 spaces $ \sM^{\infty} $ and $ \sM^T$, which we shall denote 
 by $ \sM^{\infty}(\sU) $ and $ \sM^T(\sU) $. These subsets consist only 
 of those elements of the aforementioned moduli spaces which have sufficiently 
 small energy and are on the free boundary ( for example $ \{ \pm T \} \times [0,1] $)
 close to the given Hamiltonian path $x$. Thus we assume that $ u(\pm T, \cdot) \in \sU $, 
 where $ \sU $ is the aforementioned neighborhood of a Hamiltonian path.
 
 In Theorem \ref{thm:MON} we have proved monotonicity results for the 
 solutions of \eqref{eq:per_hol1}. These results guarantee that 
 small energy solutions of  \eqref{eq:per_hol1}, that at the ends 
 are close to the Hamiltonian path $ x$ are confined to a small 
 neighborhood of $ x$. This will imply that the elements of 
 $ \sM^{\infty}(\sU) $ and $ \sM^T(\sU) $ are contained in a 
 small neighborhood of $x$ and will allow us to work in suitable 
 coordinate charts, thus the main analysis can be done in the standard 
 model in Euclidean space using suitable coordinate charts.
 Consider the restriction maps  
 \begin{align*}
  i^T : \sM^T \to \sP^{3/2} \times \sP^{3/2}, \;\; i^T( u) = ( u(-T,\cdot), u(T, \cdot)) \\
  i^{\infty} : \sM^{\infty} \to \sP^{3/2} \times \sP^{3/2}, \;\; i^{\infty} (u^-,u^+) = ( u^- (0, \cdot), u^+(0, \cdot)). 
 \end{align*}
 We prove that these maps are injective immersions
 and the restrictions of $ i^{\infty} $ and $ i^T$ to 
 $ \sM^{\infty}(\sU) $ and $ \sM^T( \sU) $ are embeddings. 
 The images $ \sW^{\infty} = i^{\infty} ( \sM^{\infty}(\sU) ) $
 and $ \sW^T= i^T(\sM^T( \sU)) $ are the {\bf nonlinear
 Hardy spaces} of the title. In section \ref{SEC:conv}   we prove that 
 $ \sW^T $ converge to $ \sW^{\infty} $ in the $ C^1 $ topology. 
 This is the main result of this chapter. 
 
 This chapter is organized as follows: 
 \begin{itemize}
  \item [i)] In section \ref{SEC:Mainres} we explain the setup and state the main theorems. 
  \item[ii)] In section \ref{SEC:hilmf} we prove that the sets of 
  paths and strips, $ \sP^{3/2}, \sB^T $ and $ \sB^{\infty} $ 
  are indeed Hilbert manifolds and we explicitly 
  construct coordinate charts on these manifolds. 
  \item[iii)] Sections \ref{SEC:prof_main_thm3}, \ref{SEC:embedding} 
  and \ref{SEC:conv} contain the proofs of the main theorems. 
  They also rely on linear elliptic estimates from the previous chapter. 
  The hard of the proof is the convergence theorem, proved in section
  \ref{SEC:conv}. In the appendix \ref{SEC:app} we recall some properties 
  of Lions- Magenes interpolation which we use in various places within 
  the thesis. 
   \end{itemize}
\section{Main results}\label{SEC:Mainres}

In this section we explain the setup and state the main theorems. 
Let $ ( M, \om ) $ be a symplectic manifold without boundary 
and let 
$$
 \cL = \cL (M, \om ) 
$$
 be the set of all compact Lagrangian 
submanifolds $ L\subset M $ without boundary.
Throughout we abbreviate
$$
\R^+ := [0,\infty),\qquad\R^-:=(-\infty,0].
$$

\begin{PARA}[{\bf Hamiltonian Paths}]\label{para:Ham_path}\rm
Let $ L_0, L_1 \in \cL(M, \om) $. 
Denote by $\cH(M)$ the space of smooth functions 
$H:[0,1]\times M\to\R$
and by $\cJ(M,\om)$ the space of smooth families 
of $\om$-compatible almost complex structures
$J=\{J_t\}_{0\le t\le 1}$ on $M$. For $H\in\cH(M)$ 
denote $H_t:=H(t,\cdot)$ for $0\le t\le 1$ and let
$
[0,1]\to\Diff(M,\om):t\mapsto\phi_t
$ 
be the Hamiltonian isotopy generated by $H$ via 
\begin{equation}\label{eq:phiH}
\p_t\phi_t=X_{H_t}\circ\phi_t,\qquad
\phi_0=\id,\qquad \iota(X_{H_t})\om=dH_t.
\end{equation}
A Hamiltonian function $ H $ is called  {\bf regular} 
for $ (L_0,L_1 ) $ if the Lagrangian submanifolds $ L_0 $ and 
$ \phi_1^{-1} ( L_1) $ intersect transversally. 
The set of regular Hamiltonian functions will be denoted 
by $\Hreg(M,L_0,L_1)$.  Intersection points of $ L_0 $ and 
$ \phi_1^{-1} (L_1) $ correspond to solutions 
$x:[0,1]\to M$ 
of Hamilton's equation
\begin{equation}\label{eq:CRIT}
\dot x(t)=X_{H_t}(x(t)),\qquad
x(0)\in L_0,\qquad x(1)\in L_1.
\end{equation}
Denote the set of solutions of~\eqref{eq:CRIT} by 
\begin{equation}\label{eq:Ham_path}
\cC(L_0,L_1;H) := \bigl\{x:[0,1]\to M\,\big|\,
x\mbox{ satisfies }\eqref{eq:CRIT}\bigr\}.
\end{equation}
The set $\cC(L_0,L_1;H)$ can be also seen as the 
set of critical points of the perturbed symplectic action functional 
on the space $\sP$ of paths in $M$ connecting $L_0$ to $L_1$. 
\end{PARA}

\begin{PARA}[{\bf Floer Equation}]\label{para:floer}\rm
Fix a regular Hamiltonian function $ H= \{H_t\}_{0\le t\le 1} $ and 
a smooth family of almost complex structures
$J=\{J_t\}_{0\le t\le 1} \in \cJ(M, \om) $
and $ L_0, L_1 \in \cL( M, \om ) $. 
For a smooth map $ u : \R\times [0,1] \to M $ 
the {\bf Floer equation} has the form 
\begin{equation}\label{eq_FLOER}
\p_su+J_t(u)(\p_tu-X_{H_t}(u))=0,\qquad
u(s,0)\in L_0,\qquad u(s,1)\in L_1,
\end{equation}
The {\bf energy} of a solution $u$ of~\eqref{eq_FLOER} 
is defined by 
$$
E_H(u) 
:=
\frac12\int_{-\infty}^\infty\int_0^1
\Bigl(\abs{\p_su}_t^2+\abs{\p_tu-X_{H_t}(u)}_t^2
\Bigr)\,dtds.
$$
Here $\inner{\xi}{\eta}_t:=\om(\xi,J_t\eta)$
denotes the Riemannian metric determined by $\om$ and $J_t$.
If the energy is finite then the limits 
\begin{equation}\label{eq:limit}
x^\pm(t) := \lim_{s\to\pm\infty}u(s,t)
\end{equation}
exist and belong to $\cC(L_0,L_1;H)$ (see for example~\cite{RS3} and Proposition \ref{pr:exp-dec}).
The convergence is with all derivatives, uniform in $t$,
and exponential.
\end{PARA}

\begin{PARA}[{\bf Hilbert Manifold of Paths}]\label{para:hil_path}\rm
Let $ L_0, L_1 \in \cL(M, \om) $ and let
$ \sP^1_L $ denote the Hilbert manifold of paths with Lagrangian 
boundary conditions. More precisely 
\begin{equation}\label{eq:hilpath}
 \sP^1_L := \left \{ \gamma \in W^{1,2} ([0,1], M ) | \gamma(i) \in L_i, \; i=0,1 \right \}.
\end{equation}
We prove in section \ref{SEC:hilmf} that this set is a Hilbert manifold. 
Consider the following two Hilbert space bundles over $ \sP^1_L $. 
 \begin{equation}\label{eq:bundles}
\xymatrix    
@C=15pt    
@R=20pt    
{    
 \sE^1\ar[dr] \ar[rr] &   & \ar[dl]\sE^0 & \\
&\sP^1_L &
}
\end{equation}
with fibers
\begin{align*}
& \sE^0_{\gamma} = L^2 ( [0,1], \gamma^* TM ) \\
 & \sE^1_{\gamma} = \{ \xi \in W^{1,2}( [0,1], \gamma^* TM ) | \xi (i) \in T_{\gamma(i)} L_i, \; i=0,1 \}.
\end{align*}
Let $ \sE^{1/2}_{\gamma} $ be the following interpolation 
space 
$$
\sE^{1/2}_{\gamma} =  [\sE^1_{\gamma}, \sE^0_{\gamma}]_{1/2} 
$$
In the appendix \ref{SEC:app} we explain more about the interpolation 
theory relevant for our setting. 
The almost complex structure $ J \in \cJ(M, \om) $ and 
Hamiltonian $ H \in \cH(M) $ determine a section 
$ \sS : \sP^1_{L} \to \sE^0 $ via 
$$
\sS(\gamma)(t) = J_t(\gamma) ( \dot{\gamma}(t) - X_{H_t}(\gamma(t)))
$$
Denote by $ \sP^{3/2}(H,J) $ the following set 
\begin{equation}\label{eq:hil_path}
 \sP^{3/2}(H,J) = \left \{ \gamma\in \sP^1_L | \sS(\gamma) \in \sE^{1/2}_{\gamma}\right \}.
\end{equation}
We prove in \ref{para:hil_P3/2} that the set $ \sP^{3/2}(H,J) $ is 
a Hilbert manifold. We also give some more details about the 
interpolation space on which $ \sP^{3/2}(H,J) $ is modelled.

\end{PARA}

 \begin{PARA}[{\bf Hilbert Manifolds of Strips}]\label{para:hil_strips}\rm 
Fix two Lagrangian submanifolds $ L_0,L_1 \in \cL(M, \om ) $,
 fix regular Hamiltonian function $ H \in \cH_{\text{reg}} (M, L_0,L_1)$ 
(see \ref{para:Ham_path})
 and almost complex structure $ J \in \cJ(M, \om ) $. 
 Let $ x \in \cC(L_0,L_1;H) $ be a Hamiltonian path. 
 Observe the moduli space of infinite holomorphic strips, 
 i.e. stable and unstable manifold defined as follows
 \begin{equation}\label{eq:stab_uns}
\sM^{\pm}(x;H,J) = 
\left \{ u \in W^{2,2}_{loc} ( \R^{\pm} \times [0,1] , M )\bigg |
\begin{array}{lll} 
  u \text{ satisfies } \eqref{eq_FLOER},\\
  E_H ( u ) < +\infty\\
 \lim\limits_{s\rightarrow \pm \infty} u ( s,t) = x ( t)
\end{array}
 \right \},
\end{equation}
and the moduli space of finite strips 
\begin{align}\label{eq:mod_finite}
 \sM^T(H,J)= \{ u \in W^{2,2} ( [-T,T] \times [0,1] , M) \Big | \; u \text{ satisfies } \eqref{eq_FLOER} \}.
\end{align}
We shall prove that the moduli spaces 
$ \sM^{\pm} (x; H,J) $ and $ \sM^T(H,J ) $ 
are Hilbert manifolds. 
Their ambient manifolds are Hilbert 
manifolds of strips $ \sB^{\pm}(x) $ and $ \sB^T $ with the 
following boundary conditions 
\begin{equation}\label{eq:bound_data}
\begin{split}
  & u(s,i) \in L_i, \;\; i=0,1 \\
   & J_i(u(s,i)) (\partial_t u (s,i) -X_{H_t} (u(s,i)) \in T_{u(s,i)} L_i, \; i=0,1
  \end{split}
\end{equation}
Thus 
 \begin{equation*}
 \sB^{\pm}(x) = \left \{  u \in W^{2,2}_{loc} ( \R^{\pm} \times [0,1], M )\bigg|
 \begin{array}{l}
  u \text{ satisfies } \eqref{eq:bound_data}\\
  \lim\limits_{s\rightarrow \pm\infty} u(s,t) = x(t)
 \end{array} \right \} 
\end{equation*}
and analogously 
\begin{equation*}
 \sB^T = \left \{
 u \in W^{2,2} ( [-T,T]\times [0,1], M)|  u \text{ satisfies } \eqref{eq:bound_data}
 \right \}.
\end{equation*} 
 We prove in \ref{para:hilbert_strips} that $ \sB^T $ and $ \sB^{\pm} $ 
 are Hilbert manifolds and in  Section \ref{SEC:prof_main_thm3} we prove 
 the following theorem. 
\begin{theorem}\label{cor:main_thm3}

Let $ \sM^{\pm} ( H, J) $ and $ \sM^T (H, J ) $ be defined by \eqref{eq:stab_uns} 
and \eqref{eq:mod_finite} and let $ \sB^{\pm} $ and $ \sB^T $ be as above. 
\begin{itemize}
\item[a)] The sets $ \sM^{\pm} ( x;H,J) $ and $ \sM^T ( H,J ) $ are Hilbert 
 submanifolds of the Hilbert manifolds $ \sB^{\pm}(x) $ and $ \sB^T $ respectively. 
 \item[b)] The maps $ i^{\pm} $ and $ i^T $ defined by  
 \begin{equation}\label{eq:map_i}
 \begin{split}
 &i^{\pm} : \sM^{\pm}( x; H, J) \to \sP^{3/2}(H,J)\\
  & i^{\pm}(u) =  u(0,\cdot) \\
 & i^T : \sM^T( H, J)\rightarrow \sP^{3/2}(H,J) \times \sP^{3/2}(H,J)\\
  &  u \mapsto ( u( -T, \cdot), u(T, \cdot))
 \end{split}
\end{equation}
are injective immersions. 
\end{itemize}
\end{theorem}
\begin{proof}
See section \ref{SEC:prof_main_thm3}.
\end{proof}
Denote by $ \sM^{\infty} ( x; H,J ) $ 
the product of stable and unstable manifolds 
 \begin{equation}\label{eq:prod_mfld}
   \sM^{\infty} ( x; H,J )= \sM^+(x;H,J) \times \sM^- (x;H,J). 
 \end{equation}
and denote by $ i^{\infty} $ the product of the maps $ i^{\pm} $
\begin{equation}\label{eq:iinf}
\begin{split}
 i^{\infty} : \sM^{\infty} ( x; H,J )\to \sP^{3/2}(H,J) \times \sP^{3/2}(H,J)\\
 i^{\infty} ( u^+,u^-) = ( i^+ (u^+), i^-(u^-)) = ( u^+(0, \cdot), u^- (0, \cdot)).
 \end{split} 
\end{equation}
It follows from Theorem \ref{cor:main_thm3} that the mapping $ i^{\infty} $ 
is an injective immersion, as a product of injective immersions. 

 \end{PARA}
 \begin{PARA}[{\bf Convergence and embedding theorem }] \label{para:conv_thm}\rm
 Let the moduli spaces $ \sM^T (H,J) $ and $ \sM^{\infty} ( x; H,J ) $ be 
 defined by \eqref{eq:mod_finite} and \eqref{eq:prod_mfld}.  
We consider only those elements of these moduli spaces 
which have sufficiently small energy and which are sufficiently 
close on the boundary to a Hamiltonian path $x\in \cC( L_0,L_1;H) $. 
 To explain this more precise we fix some small 
 neighborhood $ U $ of a point 
 $ p = x(0)\in L_0 \cap\phi_1 ^ {-1} ( L_1)$
 that doesn't contain any other intersection points of
 $ L_0 \cap \phi_1 ^ {-1} ( L_1)$.
 We have assumed that the Hamiltonian 
 $ H \in \cH_{\text{reg}}( L_0,L_1;M) $
 thus the intersection $ L_0 \cap \phi_1 ^ {-1} ( L_1)$ is transverse
 and compact. Let $ U_t = \phi_t ( U ) $, where 
 $ \phi_t $ is Hamiltonian isotopy \eqref{eq:phiH}.
 Let $ V \subset \overline{V} \subset U $ be a neighborhood of $ x(0) $ 
 and let $ V_t = \phi_t ( V ) $. Suppose that on $ U_t $ we have
 constructed coordinate charts $ f_t : U_t \rightarrow \R^{2n} $. 
 We construct such family of coordinate charts in section \ref{SEC:hilmf}. 

 By monotonicity result, Theorem \ref{thm:MON}, 
 there exists a constant $ \hbar > 0 $ such that every solution 
 $ u: [-T, T]\times [0,1] \to M $ of \eqref{eq_FLOER} 
 with $ E(u) < \hbar $ and $ u ( \pm T, t) \in V_t $ 
 satisfies $ u ( s,t) \in U_t $ for every $ s,t $. 
 The analogous result holds for half-infinite holomorphic strips
 $ u \in \sM^{\pm} ( x;H,J ) $ with $ u(0,t)\in V_t $.
 
 Let $ \sU $ be a neighborhood of a Hamiltonian path 
 $ x \in \sP^{3/2}(H,J) $ in the Hilbert manifold 
 $ \sP^{3/2}(H,J) $ defined in \eqref{eq:hil_path}.
 Shrinking $ \sU $ if necessary we may assume that every 
 $ \gamma\in \sU $ satisfies $ \gamma(t) \in V_t $, 
 where $ V_t $ is as above. For $\hbar $ and $ \sU $ 
 as above we define the following sets 
 \begin{equation}\label{eq:mod_emb}
\begin{split}
  &\sM^{\infty}( x, \sU) = \{ ( u^+,u^-) \in \sM^{\infty}( x; H,J) | 
  \; u^{\pm} (0, \cdot)\in \sU, \; E( u^{\pm} ) < \hbar \},\\
  &\sM^T(\sU) = \{ u\in \sM^T( H,J) | \; u (\pm T, \cdot) \in \sU , \;\; E(u) < \hbar \}.
\end{split}
 \end{equation}
Here $ \sM^{\infty} ( x; H,J ) $ is defined by \eqref{eq:prod_mfld}
and $\sM^T(H,J) $ by \eqref{eq:mod_finite}. 
 All holomorphic curves $ u \in \sM^T(\sU) $
 are contained in coordinate charts $ U_t= \phi_t(U) $ 
 and the same holds for  $ ( u^+,u^-) \in \sM^{\infty}(x, \sU)$ too. 
 Thus, instead of working with holomorphic curves in $M $ we can work in local 
 coordinates in $ \R^{2n} $, which is much simpler for the analysis. 
 The main theorems are the following:

\begin{theorem}\label{thm:main_thm2}
 Let $ \sM^{\infty}(x, \sU)$ and $\sM^T ( \sU ) $ be defined by 
\eqref{eq:mod_emb}.
Let $ i^{\infty} $ and $ i^{T} $ be defined by \eqref{eq:map_i}
and \eqref{eq:iinf}. 
   There exists an open neighborhood $\sU $ of a Hamiltonian 
   path $ x $ such that the restrictions of the maps $ i^{\infty} $ 
   and $ i^T $ to $ \sM^{\infty}(x, \sU ) $ and $ \sM^T( \sU)$ are embeddings
   for all $ T \geq 1 $. 
\end{theorem}
\begin{proof}
See section \ref{SEC:embedding}.
\end{proof}

 \begin{theorem}\label{thm:main_thm3}
   Assume the notation as in Theorem \ref{thm:main_thm2}.
   Let 
   $$ 
   \sW^{\infty}(x, \sU) := i^{\infty} ( \sM^{\infty}(x, \sU) )
   \text{ and } \sW^T ( \sU) := i^T ( \sM^T( \sU) ). 
   $$
   Then after possibly shrinking the neighborhood $\sU $ 
   the manifolds $ \sW^{T} ( \sU)$ converge to $ \sW^{\infty} ( x,\sU) $ in the
   $ C^1- $ topology.
 \end{theorem}
 \begin{proof}
 See section \ref{SEC:conv}.
 \end{proof}
 
 \begin{remark}\label{rem:ham0}\rm
 It is enough to prove Theorems \ref{thm:main_thm2} and
\ref{thm:main_thm3} in the case that the Hamiltonian 
 is identically equal zero. 
 Let $ \phi_t $ be the Hamiltonian isotopy as in \eqref{eq:phiH}. 
 By naturality, we can consider the tuple 
 \begin{equation}
  \tu(s,t) = \phi_t^{-1} ( u(s,t)) ,\; \tL_0 = L_0, \;\tL_1 = \phi_1^{-1} ( L_1 ) , \;\; \tJ_t = \phi_t^* J_t.
  \end{equation}
If $ u $ is the solution of the equation \eqref{eq_FLOER}, then 
$\tu: \R^{\pm} \times [0,1]\to M $ satisfies 
unperturbed {\bf Floer} equation, i.e. $ \tu $ is  a $ \tJ $ 
holomorphic curve, 
\begin{equation}\label{eq:Floer1}
 \partial_s \tu + \tJ_t( \tu ) \partial_t \tu = 0, \;\; \tu(s,0) \in \tL_0, \; \tu(s,1) \in \tL_1.
\end{equation}
By naturality, to Hamiltonian paths $ x \in \cC( L_0, L_1; H ) $ correspond 
intersection points $ \tx = x(0) \in \tL_0\cap \tL_1 $.  
By assumption the Hamiltonian $H$ is regular, thus
the intersection $ \tL_0 \cap \tL_1 $ is transverse. 
Hence, every solution $ \tu: \R^{\pm} \times [0,1] \to M $ 
of the equation \eqref{eq_Floer1} with finite energy 
converges exponentially to an intersection point $ \tx \in \tL_0 \cap \tL_1 $
(see Proposition \ref{pr:exp-dec} for the case of tame almost complex structure and 
the clean intersection of Lagrangian submanifolds).

 \end{remark}

 \end{PARA}

\section{Hilbert manifold setup}\label{SEC:hilmf}

\begin{PARA}[{\bf Hilbert manifold of paths with Lagrangian boundary conditions}]\label{para:hil_paths}\rm
In this section we introduce a collection of Hilbert manifolds of paths with 
Lagrangian boundary conditions. We prove that 
they are infinite dimensional Hilbert manifolds and 
we explicitly construct coordinate charts on them.  
Fix $ L_0, L_1 \in \cL(M) $ and a regular Hamiltonian 
$H \in \cH_{\text{reg}} ( M, L_0, L_1 ) $. 
\begin{proposition}\label{prop:PLk}
{\bf (i)}
The set
\begin{equation}\label{eq:PL1}
\sP^1_L := \left\{\gamma\in W^{1,2} ( [0,1], M)\,|\,
\gamma(0)\in L_0,\,\gamma(1)\in L_1\right\}.
\end{equation}
is a smooth Hilbert submanifold of $W^{1,2} ( [0,1], M)$.

\smallskip\noindent{\bf (ii)}
The set
\begin{equation}\label{eq:PL2}
\sP^2_{L} := \left\{\gamma\in W^{2,2} ( [0,1], M)\,\Bigg|\,
\begin{array}{l}
\gamma(0)\in L_0,\gamma(1)\in L_1, \\
J_0(\gamma(0))(\dot\gamma(0)- X_{H_0}(\gamma(0)))\in T_{\gamma(0)}L_0 \\
J_1(\gamma(1))(\dot\gamma(1)- X_{H_1}(\gamma(1)))\in T_{\gamma(1)}L_1
\end{array}\right\}.
\end{equation}
is a smooth Hilbert submanifold of $W^{2,2} ( [0,1], M)$.
\end{proposition}
\begin{proof}
Denote by $ \sP^{k} = W^{k,2} ( [0,1] , M )$. The set $ \sP^k $ is a smooth 
Hilbert manifold modeled on $ W^{k,2} ( \gamma^* TM) $ and local charts are given via 
exponential map. 
 Denote by 
$$
\ev_M:\sP^1\to M\times M
$$
the evaluation map defined by 
$$
\ev_M(\gamma):=(\gamma(0),\gamma(1)).
$$
This map is smooth
and its first derivative at 
$\gamma$ is the linear map
$$
d\ev_M(\gamma):W^{1,2}([0,1],\gamma^*TM)
\to T_{\gamma(0)}M\times T_{\gamma(1)}M,
$$
given by 
$$
d\ev_M(\gamma)\xi = (\xi(0),\xi(1)).
$$
This map is clearly surjective for every 
$\gamma\in\sP^1$.  Hence $\ev_M$ is a submersion
and hence the preimage 
$$
\sP^1_L=\ev_M^{-1}(L_0\times L_1)
$$
is a smooth submanifold of $\sP^1$ whose 
tangent space at $\gamma$ is the preimage of 
$T_{\gamma(0)}L_0\times T_{\gamma(1)}L_1$
under $d\ev_M(\gamma)$.  This proves~(i).

To prove~(ii) we define the evaluation map
$$
\ev_{TM}:\sP^2\to TM\times TM
$$
by 
$$
\ev_{TM}(\gamma)= \Big (\gamma(0),v_0,
,\gamma(1), v_1 \Big)
$$
where $ v_i =J_i(\gamma(i))(\dot\gamma(i)- X_{H_i}(\gamma(i))) , \; i=0,1. $
Then 
$$
\sP^k_{L} = \ev_{TM}^{-1} ( TL_0 \times TL_1 ). 
$$
We prove that the mapping $\ev_{TM} $ is also a submersion
and hence $ \sP^2_{L} $ is a Hilbert submanifold of $ \sP^2 $. 
Choose a Riemannian metric on $M$. 
Fix a point $p\in M$ and a tangent vector $v\in T_pM$.
The Riemannian metric determines an isomorphism
$$
T_{(p,v)}TM\cong T_pM\oplus T_pM
$$
as follows.  Think of a tangent vector of $TM$ at $(p,v)$
as an equivalence class of curves $\R\to TM:t\mapsto(\gamma(t),X(t))$
with $\gamma(0)=p$ and $X(0)=v$, where two curves are equivalent iff
they have the same derivative at $t=0$ in some (and hence every) 
coordinate chart on $TM$ containing the point $(p,v)$. 
The isomorphism 
$
T_{(p,v)}TM\to T_pM\oplus T_pM
$
is then given by
$$
[\gamma,X]\mapsto (\dot\gamma(0),\nabla X(0))
$$
where $\nabla X$ denotes the covariant 
derivative of a vector field along $\gamma$. 
With this understood the derivative of the 
map $\ev_{TM}$ is given by
$$
d\ev_{TM}(\gamma)\xi
= \left((\xi(0),\eta(0)),(\xi(1),\eta(1))\right),
$$
Denote for $ i=0,1 $ 
\begin{align*}
 \zeta_i & = \Big (\nabla_{\xi(i)}J_i(\gamma(i)\Big)\Big (\dot\gamma(i) - X_{H_i} ( \gamma(i))\Big) + 
 J_i(\gamma(i))\nabla\xi(i)\\
 \mu_i  & = J_i(\gamma(i))\Big (\nabla_{\xi(i)}X_{H_i}(\gamma(i))\Big )
\end{align*}
Then we have 
\begin{align*}
\eta(0) =  \zeta_0 - \mu_0, \qquad 
\eta(1) =  \zeta_1 - \mu_1 
\end{align*}
This map is surjective because $J_0(\gamma(0))$ and $J_1(\gamma(1))$
are isomorphisms of the respective tangent space of $M$.
This proves the proposition.
\end{proof}
Another way to prove that $ \sP^1_{L} $ and $ \sP^2_{L} $ 
are Hilbert manifolds is to construct explicitly coordinate charts on them. 
In the case of $ \sP^1_L $ one can use only the exponential map 
of an appropriately chosen metric $g_t $
which has the property that $ L_i $ are totally geodesic 
with respect to $ g_i $ for $ i=0,1 $ 
(see Lemma \ref{mon_lem7} for a construction of such metrics).
In this case we obtain coordinate charts which map $ \sP^1_L $ into 
its model space, i.e. 
\begin{align}\label{eq:w12bc}
 W^1_{bc} ( [0,1], \R^{2n} ) = \left \{ \xi \in W^{1,2} ( [0,1], \R^{2n} ) | \xi (i) \in \R^n \times \{0\} , \; i =0,1 \right\}. 
\end{align} 
In the case of $ \sP^2_{L} $ one cannot hope to find a metric 
such that the local charts are given only by its exponential map.
This is true because the definition of $ \sP^2_{L} $ involves 
two boundary conditions, one of these conditions includes the 
derivative of a curve $ \gamma $ and we don't have enough control 
of the derivative of the exponential map. 
Still, we can explicitly construct local charts that map $ \sP^2_{L} $ 
into its model space, i.e. into 
\begin{equation}\label{eq:w22bc}
W^{2,2}_{bc}([0,1]) = \left \{ \xi \in W^{k,2} ( [0,1] , \R^{2n} ) \bigg | 
\begin{array}{ll}
 \xi(i) \in \R^n \times \{0\},\; i=0,1 \\
 \partial_t \xi(i) \in \{0\} \times \R^n, \; i=0,1 
\end{array}
 \right \}.
 \end{equation}
We shall first construct local coordinate charts in the case 
that Hamiltonian $ H=0 $ and then we shall reduce the case $H\neq 0 $
to the case $ H=0 $ by naturality. 
\end{PARA}

\begin{lemma}\label{le:phi}
Let $ L_0, L_1 \in \cL(M) $, let $J\in \cJ(M, \om)$ and let
$\alpha :[0,1]\to M$ be a smooth path with Lagrangian
boundary conditions, $ \alpha(i) \in L_i , \; i=0,1 $.
There is an open set
$
\tU\subset [0,1]\times M
$
and a smooth map $ f:\tU\to\R^{2n}$ such that the following holds:
\begin{itemize}
 \item[i)] 
 $$
\alpha(t) \in \tU_t:=\left\{p\in M\,|\,(t,p)\in \tU\right\}.
$$
\item[ii)] The map $ f_t( \cdot):=f(t,\cdot):\tU_t\to f_t( \tU_t) = W_t $ 
is a diffeomorphism for every $t $ 
and it satisfies:
\begin{equation}\label{eq:loc_chart}
\begin{split}
f_{i}( L_i \cap \tU_i) &=  (\R^n \times \{0\}) \cap W_i,  \; i=0,1  \\
((f_{i}) _* J_i)(x,0) &= J_{std} , \;\; ( x,0) \in W_i , \;\; i=0,1 . 
\end{split}
\end{equation}
\item[iii)] $ \partial_t f_t(p)= 0 $  for $ t=0,1 $ for all $p \in \tU_t $. 
\item[iv)] 
 If $ \alpha (t) $ is a constant path $\alpha(t) = p $
 then $ \tU$ can be chosen such that $ \tU_t $ is independent of $t$. 

\end{itemize} 

\end{lemma}

\begin{proof} [Proof of Lemma \ref{le:phi}]
We devide the proof into four steps.\\

\medskip
{\bf Step 1. Construction of the trivialization of  $\alpha^* TM $:}\\ 
There exists a smooth map 
$$
e_{\alpha} : [0,1] \times \R^{2n} \to \alpha^* TM  
$$
such that 
\begin{itemize}
\item[i)] $ e_{\alpha} (t) : \R^{2n} \to T_{\alpha(t)} M $ is a vector space 
isomorphism for all $t$. 
\item[ii)] $ e_{\alpha} (t) J_{\text{std}} = J_t(\alpha(t))e_{\alpha(t)} $ for all $ t $, where
$ J_{\text{std}} $ denotes standard complex structure in $ \R^{2n} $. 
\item[iii)] $ \om ( e_{\alpha} (t) \cdot, e_{\alpha}(t) \cdot ) = \om_{\text{std}} (\cdot,\cdot), $
where $ \om_{\text{std}} $ denotes the standard symplectic form in $ \R^{2n} $. 
\item[iv)] $e_{\alpha}(i) : \R^n \times \{0\} \to T_{\alpha(i)} L_i, \;\; i=0,1 $

\end{itemize}

Let $ \{ e_i\}_{1\leq i \leq n }$ be the standard basis of $ \R^n \times \{0\} $,
then $ \{ e_i, J_{\std} e_i \}_{1\leq i \leq  n} $ is the standard basis 
of $ \R^{2n} $. We first construct a trivialization $ \te_{\alpha} $ of $\alpha^* TM $
which satisfies the conditions $ i) - iii) $. 
 Let $g_t$ be the metric obtained by pairing $ \om $ and $J_t $, 
 $ g_t (\cdot, \cdot)= \om ( \cdot, J_t (\cdot)) $. 
 
 Let $\{v_i\}_{1\leq i \leq n} $ be an orthonormal basis of $ T_{\alpha(0)}L_0 $
 with respect to $ g_0 $. We define 
\begin{align*}
\te_{\alpha}(0) e_i = v_i, \quad    \te_{\alpha}(0)(J_{\std} e_i) =  J_0( \alpha(0)) v_i, \;\;\text{for } i=1,\cdots n 
\end{align*}  
and extend $ \te_{\alpha}(0) $ linearly. 
We define $ A_t (p) : T_p M \to T_p M$ as follows 
\begin{align}\label{eq:At} 
 g_t(p)( A_t(p) \cdot, \cdot) = \frac{1}{2} (\partial_t g_t)(p) ( \cdot, \cdot ). 
\end{align}
We define 
 $$
 \te_{\alpha}(t) e_i:= v_i(t) , \quad i=1, \cdots , n
 $$
where $ v_i(t) $ is a solution of the following equation
\begin{align}\label{eq:un_fr}
  &\widetilde{\nabla}_t v_i(t) =
 -\frac{1}{2} A_t(\alpha) v_i(t) +\frac{1}{2} J_t (\alpha) A_t(\alpha) J_t(\alpha)v_i(t), \\
 & v_i(0) = v_i , \notag 
\end{align}
where $ \tnabla_t $ is given by 
$$ 
\widetilde{\nabla}_t v := 
\nabla_t v - \frac{1}{2} J_t 
\left ( (\nabla \dot{\alpha} J_t) v +  (\partial_t J_t) v \right ),
$$
$\nabla= \nabla^t$ is Levi-Civita connection of the metric $g_t$ 
and $A_t$ is defined by~\eqref{eq:At}. 
Notice that both left and right side of the equation (\ref{eq:un_fr}) 
commute with $J_t $. Thus if $ v(t) $ is a solution of (\ref{eq:un_fr}), 
then also $ J_t(\alpha(t)) v(t) $ is a solution of the same equation. 
Also if $ v(t) $ and $ w(t) $ are the solutions of the equation 
\eqref{eq:un_fr} then  
$$ 
0= \frac{d}{dt} \Big (g_t( v(t), w(t))+ g_t ( J_t( \alpha) v(t), J_t (\alpha) w(t)) \Big )
= 2 \frac{d}{dt} g_t ( v(t), w(t)).
$$
Thus the vectors $ \{v_i(t), J_t(\alpha(t)) v_i(t)\}_{i=1,n} $ form an
orthonormal (with respect to $g_t$ ) basis for all $ t\in [0,1] $. 
Thus we can define 
$$
\te_{\alpha}(t)  ( J_{\std} (e_i) ) := J_t(\alpha(t) v_i(t), \quad  i=1, \cdots n 
$$
and extend it by linearity. 
Notice that the trivialization $\te_{\alpha} $ 
satisfies the conditions $ i) - iii) $, whereas 
the condition $iv) $ doesn't have to be necessarily 
satisfied for $t=1 $. There exists some Lagrangian 
subspace $ V\subset \R^{2n} $ such that 
$$
 \te_{\alpha}(1) : V \to T_{\alpha(1)} L_1. 
$$
There exists a smooth path $ U(t) $ of unitary matrices 
with the  property $ U(0) = \Id $ and $ U(1) : \R^n \times \{0\} \to V $ 
(as $U(n) $ is connected ). Thus the trivialization 
$$
e_{\alpha} (t) := \te_{\alpha}(t) U(t) , 
$$
satisfies all the required properties. 
\medskip

{\bf Step 2. Construction of local coordinate charts}\\

Let $ h_t $ be a smooth family of metrics as in Lemma \ref{mon_lem7}. 
Let $ r_{\alpha(t),t} $ be the injectivity radius of the metric $ h_t $
at the point $ \alpha(t) $. Let 
\begin{equation}\label{eq:tneigh}
 W'_t =B_{r_{\alpha(t),t }}(0) \subset \R^{2n} 
\text{ and } U'_t= B_{r_{\alpha(t),t}} ( \alpha(t)) \subset M .  
\end{equation}
Let $ \psi'_t : W'_t \to U'_t $ be a smooth family of diffeomorphisms 
given by 
$$ 
\psi'_t ( \xi) = \exp_{\alpha(t),t} ( e_{\alpha}(t)\xi ) ,
$$
where $ e_{\alpha}(t) $ is the trivialization constructed in Step 1. 

The map 
\begin{equation}\label{eq:tphi_t}
 \phi'_t= (\psi'_t)^{-1}: U'_t \rightarrow W'_t 
\end{equation}
has the following properties:
\begin{align*}
\phi'_0( L_0 \cap U'_0)  = W'_0\cap ( \R^n \times \{0\})
 ,\;\phi'_1( L_1 \cap U'_1) = W'_1 \cap (\R^n \times \{0 \}).
\end{align*}
Denote with $ J_t' $ the push forward of
$ J_t $ via $ \phi'_t $, $ J'_t := (\phi'_t)_* J_t $.
Notice that as $ d\phi'_t ( \alpha(t))= e_{\alpha}(t)^{-1},$
the almost complex structure
$$
J'_t(0)  = d\phi'_t(\alpha(t)) J_t (\alpha(t))d\phi'_t(\alpha(t))^{-1}
$$
 satisfies 
$$ 
J'_t (0) = J_{std}, \;\; \text{ for all } t \in [0,1]  
$$
In the next step we make $ J'_t $ 
standard on the whole $ \R^n \times \{0\} $.

{\bf Step 3. Adapting the charts to the almost complex structure}\\

\medskip

There exist open sets $ V, \tV \subset [0,1] \times \R^{2n} $  
such that
\begin{align*}
 V_t= \{ p | (t,p) \in V \} , \;\; \tV_t = \{ p | ( t,p) \in \tV \}
\end{align*}
are open neighborhoods of $ 0 \in \R^{2n} $. 
There exist a smooth map 
$ \tPhi : \tV \to V $ such that 
$$ 
\tPhi_t = \tPhi( t, \cdot ) : \tV_t \rightarrow V_t ,
$$ 
is a diffeomorphism for every $ t $ and 
$ \widetilde{J}_t:= (\tPhi_t \circ \phi'_t)_* J_t $ satisfies
\begin{align*}
\widetilde{J}_t( x, 0) = J_{std}, \; \;\; \text{ for all } (x,0) \in V_t \cap (\R^n \times \{0\})
\end{align*}
and $ \phi'_t $ is a diffeomorphism from \eqref{eq:tphi_t}.

  Let $ \tPsi_t : \R^{2n} \rightarrow \R^{2n} $ be given by
  \[ \tPsi_t( x,y) = \left ( \begin{matrix}
                              x\\
                              0
                             \end{matrix}\right )
               + J'_t( x,0) \left ( \begin{matrix}
                                      y \\
                                      0
                                     \end{matrix}
                                        \right ),
                          \]
 where $ J'_t = (\phi'_t)_* J_t $ and $ \phi'_t $ is given by \eqref{eq:tphi_t}.
 Then 
  \[  d\tPsi_t(x,0) ( \hat x, \hat y)  = \left ( \begin{matrix}
                              \hat x\\
                              0
                             \end{matrix}\right )
               + J'_t( x,0) \left ( \begin{matrix}
                                     \hat  y \\
                                      0
                                     \end{matrix}
                                        \right )
                          \]
  and obviously $ d\tPsi_t( 0,0) = \one $ for all $t$.
  Hence there exist a smooth map $ t \mapsto r_t > 0 $ such that 
   $ \tPsi_t : V_t = B_{r_t} (0) \rightarrow \tV_t = \tPsi_t( V_t) $ 
  is a diffeomorphism  for every $t$. 
  If necessary shrink $ \tV_t $ so that $ \tV_t \subset W'_t $, 
  where $ W'_t $ are as in \eqref{eq:tneigh}.
  Notice that 
  \begin{align*}
   d\tPsi_t(x,0) \circ J_{std} = J'_t(x,0) d\tPsi_t (x,0)
  \end{align*}
  for all $ t \in [0,1] $. 
  The desired map $ \tPhi_t $ is the inverse of $\tPsi_t$. 
  The open set $ \tV= \{ ( t, x) \; | \; x \in \tV_t \} $ and 
  $ V $ is given analogously. 
 This proves the second step.\\
 \medskip 
 
{\bf Step 4:} {\bf Construction of the map $ f_t $} \\
 Let $ \tU_t = (\phi'_t)^{-1} ( \tV_t) $, where $ \tV_t $ are the 
 open sets as in Step 3 and $ \phi'_t $ the map \eqref{eq:tphi_t}. 
 The map 
 $ f_t: \tU_t \rightarrow V_t $  is given by 
 $$ 
   f_t = \tPhi_t\circ\phi'_t .
$$
Obviously, $f _t $ satisfies~\eqref{eq:loc_chart}.
 If $ \partial_t f_t \neq 0 $ for $ t=0,1 $ 
instead of $ f_t $ take $ f_{\beta(t)} $, 
where $ \beta(t) $ is a smooth cut-off function satisfying
$ \dot{\beta}(0) = \dot{\beta}(1) = 0 $. 

Finally  in the case $ \alpha(t ) = p $ if it is necessary shrink $ V_t $ such that 
$ f_t^{-1}( V_t) = U_p $, where $ U_p $ is some fixed neighborhood of the point $p$.  
\end{proof}

\begin{PARA}[{\bf Construction of Coordinate Charts on $ \sP^i_L$ }]\label{para:coordcharts}\rm
Now we are able using the map $ f $ constructed in Lemma \ref{le:phi} 
to construct coordinate charts on $ \sP^i_L, \; i=1,2 $. We shall
first construct these charts in the case that the Hamiltonian 
function $ H$ is identically equal zero, 
and then in the general case, when $H$ is an arbitrary function. 

\begin{corollary}[{\bf The case $ H=0$}]
 Assume that the Hamiltonian function $H= 0$ and let $ \sP^i_L, \; i=1,2 $ 
 be as in \eqref{eq:PL1} and \eqref{eq:PL2}. Let $ f_t : \tU_t \to \R^{2n} $ 
 and $ \alpha $ be as in Lemma \ref{le:phi}. 
 Define $\sU_{\alpha}^i \subset \sP^i_L, \; i=1,2 $ 
 and $ \Phi_{\alpha} : \sU_{\alpha}^i \to W^{i,2}_{bc} ( [0,1], \R^{2n} ) $
 by 
 $$ \sU_{\alpha}^i = \left \{ \gamma \in \sP^i_{L} | \gamma(t) \in \tU_t \right \} $$
 and 
 $$ 
 \Phi_{\alpha} ( \gamma) (t) = f_t (\gamma(t)) 
 $$
 Then $ \Phi_{\alpha} $ is a coordinate chart on $ \sP^{i}_L , \; i=1,2 $. 
\end{corollary} 

\begin{corollary}[{\bf General Case}]\label{pr:chart} 
Let $ \sP^i_L, i=1,2 $ be as in \eqref{eq:PL1} and \eqref{eq:PL2}. 
Let $ f_t : \tU_t \to \R^{2n} $ and $ \alpha $ be as in 
Lemma \ref{le:phi}. Let $ \phi_t $ be the Hamiltonian isotopy 
defined by \eqref{eq:phiH} and denote $ U_t = \phi_t (\tU_t ) $. Define 
$\sU_{\alpha}^i \subset \sP^i_L$ and $ \Phi_{\alpha} : \sU^i_{\alpha} \to W^{i,2}_{bc} ( [0,1], \R^{2n}) $
by 
$$
\sU^i_{\alpha} = \{ \gamma \in \sP^{i}_L | \gamma (t) \in U_t \} 
$$
and 
$$ 
\Phi_{\alpha} (\gamma) (t) = f_t ( \phi_t^{-1} ( \gamma(t))) = F_t ( \gamma(t)). 
$$
Then $ \Phi_{\alpha} :\sU^i_{\alpha} \to W^{i,2}_{bc} ( [0,1], \R^{2n})$
is a coordinate chart on $ \sP^i_L , \; i=0,1$. 
\end{corollary}
\end{PARA}

\begin{PARA}[{\bf Hilbert space bundles $ \sE^0 $ and $ \sE^1 $ }] \label{para:hil_bundl}\rm
In the Proposition \ref{prop:PLk} we have introduced 
manifolds of paths $ \sP^1_{L} $ and $ \sP^2_{L}$,
but our main interest will be another Hilbert manifold, 
denoted by $ \sP^{3/2}=\sP^{3/2}( H,J) $,
which is in some sense an intermediate manifold 
between these two manifolds 
$$
\sP^2_{L} \subset \sP^{3/2} \subset \sP^1_L 
$$
Consider the following two Hilbert space bundles over 
$ \sP^1_L $ 
\begin{equation}\label{eq:bundles1}
\xymatrix    
@C=15pt    
@R=20pt    
{    
 \sE^1\ar[dr] \ar[rr] &   & \ar[dl]\sE^0 & \\
&\sP^1_L &
}
\end{equation}
with fibers 
\begin{equation}\label{eq:fib}
\begin{split}
 &\sE^0_{\gamma} = L^2 ([0,1], \gamma^*TM)\\
 &\sE^1_{ \gamma}= \left\{ \xi \in W^{1,2} ([0,1], \gamma^*TM) \;| 
\; \xi(i) \in T_{\gamma(i)}L_i, \; i=0,1 \right \}.
\end{split}
\end{equation}
Note that $ \sE^1_{\gamma} $ is a dense subset of $ \sE^0_{\gamma} $ 
and the inclusion of $ \sE^1_{\gamma} $  into  $ \sE^0_{\gamma} $
is a compact operator. The tangent bundle of $ \sP^1_L $ is $ \sE^1 $. 
The almost complex structures $ J\in \cJ(M, \om)$ and Hamiltonian 
$ H\in \cH_{\text{reg}}( M, L_0, L_1) $ determine a section 
$ \sS : \sP^1_{L} \rightarrow \sE^0 $ via 

\begin{equation}\label{eq:sec_s}
\sS(\gamma) (t) = J_t(\gamma) \Big (\dot{\gamma} (t)- X_{H_t} ( \gamma(t))\Big ) , \;\; 0 \leq t \leq 1 . 
\end{equation}
Notice that 
$$ 
\sP^2_{L}=\{ \gamma \in \sP^1_{L} \; | 
\sS(\gamma) \in \sE^1_{\Lambda} \} .
$$
 \end{PARA}
 
 \begin{PARA}[{\bf The interpolation subbundle $ \sE^{1/2}$}] \label{para:int_subb}\rm
 Let $ \sP^i_L , \;\; i=1,2$ be as in \eqref{eq:PL1} and \eqref{eq:PL2}
 and let $ \sE^i_{\gamma}, i=0,1 $ be as in \eqref{eq:fib}. 
Let 
$$ 
H :=  L^2 ( [0,1], \R^{2n}) \text{ and } W :=  W^{1,2}_{bc} ( [0,1], \R^{2n}) 
$$
be as in \eqref{eq:w12bc}.
Remember that the Hilbert manifold $ \sP^1_L $ is modelled 
on the Hilbert space $W$, whereas the Hilbert manifold $ \sP^2_L $
is modelled on the Hilbert space $ W^{2,2}_{bc} ( [0,1] ) $ defined 
by \eqref{eq:w22bc}. 
There exist an isomorphism 
\begin{align*}
 e_{\gamma} :  H\rightarrow \sE^0_{\gamma} \qquad \text{and} \qquad
e_{\gamma} :  W\rightarrow \sE^1_{\gamma}. 
\end{align*}
There are many different ways to construct such isomorphism. 
Any local chart on $ \sP^1_L $ gives us a trivialization of its 
tangent bundle $ \sE^1 $, as well as of $ \sE^0 $. 
Thus we can use charts constructed in \ref{para:coordcharts} or 
we can also use the trivialization from the Step 1 of Lemma \ref{le:phi}. 
Let  $ \sE^{1/2}_{\gamma} $ be the following interpolation 
space
\begin{equation}\label{eq:int_fib}
\sE^{1/2}_{\gamma}  = [ \sE^1_{\gamma} , \sE^0_{\gamma}]_{1/2}.
\end{equation}
It is characterized as the domain of the square root  $ A^{1/2} $
of any self-adjoint positive definite operator on $ \sE^0_{\gamma} $ 
with domain $ \sE^1_{\gamma}$. Alternatively, it can be defined as
the set of initial conditions $ \xi(0)\in\sE^0_{\gamma}  $ of $ L^2 $ 
functions $ \xi : [0,1] \rightarrow \sE^1_{\gamma}$ whose composition 
with inclusion $ \sE^1_{\gamma}\hookrightarrow \sE^0_{\gamma} $ is 
of class $ W^{1,2} $. 
The Hilbert space  $ \sE^{1/2}_{\gamma} $ is 
isomorphic via  $e_{\gamma} $  to the interpolation space
$$ 
V:= [W,H]_{1/2} . 
$$
More precisely, let $ \xi= (\xi^1, \xi^2 ) \in V $, 
where $ \xi^1 $ denotes the first $n $ coordinates 
and $ \xi^2 $ are the last $n$ coordinates. Then evidently
\begin{align}\label{eq:int_space}
  \xi^1 \in [H^1, L^2]_{1/2} = H^{1/2}([0,1], \R^{n}), \;\;
\xi^2 \in [H^1_0, L^2]_{1/2}= H^{1/2}_{00}([0,1], \R^n ),
\end{align}
where $ H^{1/2}_{00} $ is Lions-Magenes' space \cite{LM}. 
We explain more in the Appendix \ref{SEC:app} about the interpolation 
theory relevant for this setting.
Thus, the Hilbert interpolation space 
$V= [ W, H]_{1/2} $ is just the product
$ V= H^{1/2}([0,1], \R^{n})\times H^{1/2}_{00}([0,1], \R^n )) ,$
and 
$$ 
e_{\gamma} : V  \stackrel{\cong}{\longrightarrow}  \sE^{1/2}_{\gamma} .
$$
Let $ \sE^{1/2} $ be the bundle over $ \sP^1_L $
with the fiber $ \sE^{1/2}_{\gamma} $  over $ \gamma$.

\begin{equation}\label{eq:bundles2}
\xymatrix    
@C=15pt    
@R=20pt    
{    
 \sE^1\ar[dr] \ar[r] &  \sE^{1/2}\ar[d] \ar[r] & \ar[dl]\sE^0 & \\
&\sP^1_L &
}
\end{equation}
\end{PARA}

\begin{PARA}[{\bf The Hilbert manifold $ \sP^{3/2}$ }]\label{para:hil_P3/2}\rm
Let $ \sE^{1/2}_{\gamma } $ be as in \eqref{eq:int_fib}
and let $ \sS $ be defined \eqref{eq:sec_s}. 
Denote by $\sP^{3/2}= \sP^{3/2} ( H,J) $ 
the following set
\begin{equation}\label{eq:sp3/2}
 \sP^{3/2}=\sP^{3/2}(H,J) := \left \{ \gamma \in \sP^1_L \; | 
 \sS(\gamma) \in \sE^{1/2}_{\gamma} \right \}.
\end{equation}
We shall prove that the set $ \sP^{3/2} $ is a Hilbert manifold modelled 
on the following Hilbert space 
\begin{equation}\label{eq:h3.2bc} 
H^{3/2}_{bc} := [ W^{2,2}_{bc} ([0,1] ), W^{1,2}_{bc} ([0,1] )]_{1/2},
\end{equation}
where $ W^{i,2}_{bc} ( [0,1] )= W^{i,2}_{bc} ( [0,1], \R^{2n} ), \; i=1,2 $ 
are as in \eqref{eq:w12bc} and \eqref{eq:w22bc}. 
Notice that $ H^{3/2}_{bc}  $ can be written as the following space 
\begin{equation}\label{eq:h32bc}
  H^{3/2}_{bc} =
\left \{ ( \xi_1, \xi_2 ) \in H^1( [0,1], \R^n) \times H^1_0( [0,1], \R^n)\bigg |
\begin{array}{ll} 
\partial_t \xi_1 \in H^{1/2}_{00}, \\
\partial_t\xi_2 \in H^{1/2}
\end{array}
\right  \} .
\end{equation}
\end{PARA}

\begin{lemma}\label{lem:loc_path}
Let $ \sP^{3/2} $ be defined by \eqref{eq:sp3/2} 
and let $H^{3/2}_{bc} $ be as in \eqref{eq:h32bc}. 
The set $ \sP^{3/2} $ is a Hilbert manifold
modelled on the interpolation space $H^{3/2}_{bc} $. 
The Hilbert manifold structure is defined by the following construction. 
 Let $ \alpha \in \sP^2_L \subset \sP^{3/2} $ be a smooth path and let 
 $\{ U_t\}_{0\leq t \leq 1} $ be a smooth family of open sets 
as in Corollary \ref{pr:chart}.
We define $ \sU_{\alpha} $ and $ \Phi_{\alpha} : \sU_{\alpha} \to H^{3/2}_{bc} $ as follows 
\begin{align*} 
\sU_{\alpha}= \{ \gamma \in \sP^{3/2} | \gamma(t) \in U_t , \forall t \in[0,1] \}  
\end{align*} 
and 
\begin{equation}\label{eq:locch}
\Phi_{\alpha} ( \gamma)(t) = f_t(\phi_t^{-1} (\gamma(t))) = F_t ( \gamma(t)),
\end{equation}
where $ f_t $ is a family of diffeomorphisms constructed in Lemma \ref{le:phi}
and $ \phi_t $ is Hamiltonian isotopy \eqref{eq:phiH}.
Then $ \Phi_{\alpha} $ is a local chart on $ \sU_{\alpha} $. 
 \end{lemma}

\begin{proof}
Let $ F_t = f_t \circ \phi_t^{-1} $ be as in equation \eqref{eq:locch}
and let $ \gamma \in \sU_{\alpha}$. 
The mapping $F_t^*= dF_t(\gamma(t))^{-1} $ induces the following isomorphisms
 \begin{align*}
  F_t^*=  H=L^2 ( [0,1], \R^{2n} )
  \rightarrow \sE^0_{\gamma} ,\;\; F_t^*: 
  W= W^{1,2}_{bc} ( [0,1], \R^{2n} ) \rightarrow \sE^1_{\gamma} 
 \end{align*}
Thus $ F_t^* $ induces the isomorphism 
\begin{align*}
 (F_t)^* : V=[ W,H]_{1/2} \rightarrow \sE^{1/2}_{\gamma}.
\end{align*}
Let $\hat{\gamma}(t)= F_t( \gamma(t)) = ( \xi_1(t), \xi_2 (t))  $, 
where $ \xi_1 $ denotes the first $n-$coordinates and $ \xi_2 $ the 
last $n-$coordinates. Then $ \hat{\gamma} $ satisfies 
$ \hat{\gamma}(i) \in \R^n \times \{0\}, \; i=0,1 $. 
Thus we have 
\begin{equation}\label{eq:xi12}
\xi_1 \in H^1([0,1], \R^n ), \;\; \xi_2 \in H^1_0( [0,1], \R^n ) 
\end{equation}
Denote by $ \hat{J}_t = (F_t)_* J_t $.
As $ \sS(\gamma) \in \sE^{1/2}_{\gamma} $ we have 
$ (F_t)_* ( J_t( \gamma) (\partial_t \gamma- X_{H_t} (\gamma)) \in V $. 
The following equalities hold
\begin{align}\label{eq:FT}
 &(F_t)_* \Big( J_t( \gamma) (\partial_t \gamma- X_{H_t} (\gamma)\Big) 
 = dF_t ( \gamma(t)) J_t ( \gamma(t)) \Big (\partial_t \gamma(t)- X_{H_t} ( \gamma)\Big) = \notag \\
 & (( F_t)_* J_t) ( \hat{\gamma}) \Big (dF_t ( \gamma) \partial_t \gamma- dF_t (\gamma) X_{H_t}\Big) = \notag \\
  &\hat{J}_t (\hat{\gamma}) \Big ( \partial_t \hat{\gamma} 
 - ( \partial_t F_t) ( \gamma) - df_t \cdot d\phi_t^{-1} \cdot X_{H_t} ( \gamma ) \Big)= \notag \\
 & \hat{J}_t ( \hat{\gamma}) 
 \Big ( \partial_t \hat{\gamma}  - ( \partial_t f_t )  \phi_t^{-1} (\gamma) -
 df_t ( \phi_t^{-1}(\gamma)) \partial_t \phi_t^{-1} ( \gamma) - df_t(\phi_t^{-1}(\gamma)) d\phi_t ^{-1}(\gamma) X_{H_t} (\gamma ) \Big )= \notag \\
 & \hat{J}_t ( \hat{\gamma}) \Big ( \partial_t \hat{\gamma} -( \partial_t f_t ) ( \phi_t^{-1} (\gamma) ) \Big ) \in V 
 \end{align}
 From Lemma \ref{le:phi} we have that $ \partial_t f_t = 0 $ for $ t=0,1 $, 
 this implies that 
 $$ 
 \hat{J}_t ( \hat{\gamma})( \partial_t f_t) ( \phi_t^{-1} (\gamma))\in H^1_0 ( [0,1],\R^{2n} ) \subset V 
 $$
 and from \eqref{eq:FT} it follows that 
 \begin{equation}\label{eq:loc}
  \hat{J}_t( \hat{\gamma}) \partial_t \hat{\gamma} \in [ W,H]_{1/2} = V= H^{1/2}( [0,1], \R^n ) \times H^{1/2}_{00} ( [0,1], \R^n )
 \end{equation}
  From Lemma \ref{le:phi} we have that 
 $ \hat{J}_t(x,0) = J_{\text{std}}, \;\; t=0,1 $ for $ x \in \R^n , \; \abs{x} < r $,
 together with \eqref{eq:loc} and \eqref{eq:int_space}
 this implies that 
$$
\partial_t \xi_1 \in [H^1_0, L^2]_{1/2}= H^{1/2}_{00}, \;\;\;
\partial_t \xi_2 \in [H^1, L^2]_{1/2}= H^{1/2} .
$$

Thus  we have $\hat{\gamma}\in  H^{3/2}_{bc}$. 
Notice also that if we consider smooth paths $ \alpha $ and $\beta$ 
and local charts $ \Phi_{\alpha} $ and $ \Phi_{\beta} $ given 
by \eqref{eq:locch}, then the transition map $\Phi_{\beta \alpha} = \Phi_{\beta} \circ \Phi_{\alpha}^{-1} $ 
is given by 
$$ 
\Phi_{\beta \alpha}(\xi)(t) = f_t' \circ f_t^{-1}(\xi(t)) 
$$
and $ f_t $ and $ f_t' $ are as in Lemma \ref{le:phi}. 
As these maps are diffeomorphisms it follows that the 
transition maps are also diffeomorphisms. 

\end{proof}
\medskip

\begin{PARA}[{\bf The Hilbert manifold of strips}]\label{para:hilbert_strips}\rm 
Let $ \sB^{\pm} (x) $ and $ \sB^T $ be defined as in 
 \ref{para:hil_strips}. We prove that they are infinite dimensional Hilbert 
 manifolds modelled on the following Hilbert spaces 
 \begin{equation}
  W^{2,2}_{bc} ( I \times [0,1]) = \left \{\xi \in W^{2,2} (I\times [0,1], \R^{2n})\bigg|
  \begin{array}{l}
   \xi(s,i) \in \R^n \times \{0\},\; i=0,1 \\
   \p_t \xi(s,i) \in \{0\} \times \R^n , \; i=0,1 
  \end{array}
\right\}
 \end{equation}
 where $ I= \R^{\pm} $ in the case of infinite strips
 and $ I= [-T,T] $ in the case of finite strips. 
 We also prove that restricting an element $ u \in \sB^{\pm} (x) $ to 
 the free (non Lagrangian) boundary we obtain an element of the Hilbert 
 manifold of paths $ \sP^{3/2} $. The next result extends Lemma 
 \ref{le:phi} to global strips.
 \end{PARA}
 \begin{lemma}\label{lem:chart_strips}
 Let $ u \in \sB^{+}(x) $ be a smooth strip
 such that $ u(s,t) = x(t) $ for $ s\geq s_0 $, for some large $s_0$. 
 There exits an open set $ U \subset \R^+ \times [0,1] \times M $ and 
 a smooth map $ f : U \to \R^{2n} $ such that the following holds: 
 \begin{itemize}
  \item [i)] $$ u(s,t) \in U_{s,t} := \{ p \in M | (s,t,p) \in U \} $$.
  \item[ii)] The mapping $ f_{s,t} = f(s,t,\cdot): U_{s,t} \to \R^{2n} $ is a diffeomorphism
  onto its image and it satisfies 
  \begin{equation}\label{eq:prop}
  \begin{split}
  & f_{s,i} ( U_{s,i} \cap L_i) \subset \R^n \times \{0\} , \;\; i =0,1 \\
   &df_{s,i} (q) J_i(q) = J_{std}\; df_{s,i}(q), \;\; q\in L_i\cap U_{s,i},\; i=0,1\\
   &\p_t f_{s,t} (p) + df_{s,t} (p) X_{H_t} (p) = 0, \;\; t\sim 0 \text{ and } t \sim 1 , \;\; p\in U_{s,t}
   \end{split}
  \end{equation}
   \item[iii)] Besides, the neighborhoods $ U_{s,t} $ and the mapping $ f_{s,t} $ 
 can be chosen to be $ s $ independent for large $s$. 
 \end{itemize}
 
\end{lemma}
\begin{proof} 
The boundary conditions \eqref{eq:bound_data} include 
Hamiltonian vector field. We shall first use naturality to reduce this equation 
to an equation without the Hamiltonian term. Let $ \phi_t $ be the Hamiltonian isotopy 
\eqref{eq:phiH}. Let
\begin{align*}
 & \tu(s,t) = \phi_t^{-1} ( u(s,t) ),  \; \tJ_t = (\phi_t)^* J_t ,\\
 & \tL_i = \phi_i^{-1} ( L_i ),\; i=0,1 
\end{align*}
Notice that $ \tu(s,t) $ satisfies the following boundary conditions
\begin{align*}
 &\tu(s,i) \in \tL_i, \;\; i=0,1 \\
 & \tJ_i( \tu(s,i)) \p_t \tu(s,i) \in T_{\tu(s,i)} \tL_i , \; i=0,1.
\end{align*}
We construct an open set $ \tU \subset \R^+ \times [0,1] \times M $ 
such that 
$$
\tu(s,t) \in \tU_{s,t} := \{ p \in M | ( s,t,p) \in  \tU \} 
$$
and a smooth map $ \tf: \tU \to \R^{2n} $ such that 
$$
\tf_{s,t} (\cdot):= \tf(s,t, \cdot): \tU_{s,t} \to \R^{2n}
$$
is a diffeomorphism onto its image and it satisfies the 
following properties
\begin{equation}
 \begin{split}
 & \tf_{s,i} ( \tU_{s,i} \cap \tL_i ) \subset \R^n \times \{0\} , \\
 & (\tf_{s,i})_* \tJ_{s,i} ( x,0) = J_{\text{std}}, \;\; (x,0) \in \R^n \times \{0\} \cap \tf_{s,0} ( \tU_{s,0}) \\
 & \p_t \tf_{s,t} (p) = 0, \;\; \text{ for } \; t \sim 0 \text{ and } t \sim 1. 
 \end{split}
\end{equation}
The open neighborhoods  $ \tU_{s,t} $ can be chosen to be 
$ s$ independent for large $s$ as well as the mapping $ \tf_{s,t} $. 
The construction of the maps $ \tf_{s,t} $ satisfying the above properties 
is analogous to the construction of the map $ f_t $ in Lemma \ref{le:phi}. 
For this reason we shall only sketch the construction of the maps $ \tf_{s,t} $. 
We first construct a trivialization $ e_{\tu} $ of the bundle $ \tu^* TM $ such that 
$ e_{\tu} (s,t) : \R^{2n} \to T_{\tu(s,t)} M $ satisfies
\begin{equation}\label{eq:utriv}
 \begin{split}
  &e_{\tu} (s,i) ( \R^n \times \{0\} ) = T_{\tu(s,0)} \tL_i, \;\; i=0,1 \\
  &e_{\tu}(s,t) \circ J_{\std} = \tJ_{t} ( \tu(s,t)) \circ e_{\tu} ( s,t). 
 \end{split}
 \end{equation}
 Notice that the smooth curve $ \tu(s,t) = x(0), \; s\geq s_0 $.
 Construct as in Step 1 of Lemma \ref{le:phi}, 
 for a reference curve $ \alpha(t) = x(0) $, a trivialization 
 $$ 
 e_{\tu}(s,t) : \R^{2n} \to T_{x(0)} M = T_{\tu(s,t)} M \text{ for }  s\geq s_0 . 
 $$
 This trivialization satisfies properties \eqref{eq:utriv} for $ s\geq s_0 $
 and is $ s $ independent. 
 Next extend the trivialization by parallel transport along $ \tu(s,t) $. 
 More precisely we define 
 $$
 e_{\tu} (s,t) v := P_s( \tu(s,t), x(0)) e_{\tu} ( T_0,t) v ,
 $$
 where $ P_s( \tu(s,t), x(0)) $ denotes parallel transport in $ s $ direction 
 along $ \tu $ from the point $ x(0) = \tu(T_0,t) $ to the point $ \tu(s,t)$. 
 This parallel transport should be taken with respect to the connection 
 $ \widetilde{\nabla}: = \nabla - \frac{1}{2} \tJ_t \nabla \tJ_t $
 and $ \nabla=\nabla^t $ is a Levi-Civita connection of the metric $ h_t $
 as in Lemma \ref{mon_lem7}. Such parallel transport has the property
 $$
 P_s( \tu(s,t), x(0)) J_t(x(0))v = J_t( \tu(s,t)) P_s(\tu(s,t), x(0)) v, \;\;\; \forall v. 
 $$
 The trivialization $ e_{\tu} (s,t) $ obtained in this way satisfies 
 the requirements \eqref{eq:utriv}. 
 
Let $ r_{s,t} $ be the injectivity radius of the metric $ h_t $ at the 
point $ \tu(s,t) $. We define a mapping 
$ \psi_{s,t}' : B_{r_{s,t}}(0) \to B_{r_{s,t}} (\tu(s,t))= U_{s,t}'$
by 
$$ \psi_{s,t}' ( \xi ) = \text{exp}_{\tu(s,t)} ( e_{\tu}(s,t) \xi ) .$$
This mapping is obviously a diffeomorphism onto its image. Let $ \phi_{s,t}' $ 
be its inverse. Let $ J_{s,t}' = ( \phi_{s,t}')_* \tJ_t $. 
It follows from the properties \eqref{eq:utriv} of 
the trivialization and metric $ g_t $ that 
\begin{align*}
 \phi_{s,i}' ( \tL_i \cap U_{s,t}') \subset \R^n \times \{0\} , \;\; i=0,1 \\
  J_{s,t}'(0) = J_{std}.
\end{align*}
Analogously as in Lemma \ref{le:phi} in Step 3, we can 
make $ J_{s,t}' $ standard on $ \R^n \times \{0\} $
by composing with an appropriate isomorphism $ \tPhi_{s,t} $.
Thus, the rest of the construction is completely analogous to 
the construction of the maps $ f_t $ in Lemma \ref{le:phi}
and the mapping $ \tf_{s,t} $ is given as a composition 
of $ \phi_{s,t}' $ and $ \tPhi_{s,t} $. 
Finally the mapping $ f_{s,t} : U_{s,t} = \phi_t ( \tU_{s,t} )  \to \R^{2n} $
is given by 
$$ f_{s,t} ( p) = \tf_{s,t} ( \phi_t^{-1} (p)) $$
and it satisfies all the properties \eqref{eq:prop}. 
\end{proof}
\begin{definition}\label{cor:chart_strips}
Let $ u \in \sB^+(x) $ and $ f_{s,t}:U_{s,t} \to \R^{2n}$ 
be as in Lemma \ref{lem:chart_strips}.
Define $ \sU_u \subset \sB^+ (x) $ 
and $ \Phi_u : \sU_u \to W^{2,2}_{bc} ( \R^+ \times [0,1] ) $ by 
\begin{align*}
 &\sU_u = \{ v\in \sB^+ (x) |  v(s,t) \in U_{s,t} \} \\
 &\Phi_u (v) (s,t) = f_{s,t} (v(s,t))
\end{align*}
Then $ \Phi_u :\sU_u \to W^{2,2}_{bc} ( \R^+ \times [0,1] ) $ 
is a coordinate chart on $ \sB^+(x) $. This defines a Hilbert manifold structure 
on $ \sB^+(x) $. 
\end{definition}

\begin{proposition}\label{pr:rest_bdy}
Let $ \sB^+ (x) $ be defined as in \ref{para:hil_strips}
and let $ \sP^{3/2}(H,J) $ be as in \ref{para:hil_P3/2}. 
Observe the restriction to the non-Lagrangian boundary
$$ 
\sR: \sB^+(x) \to \sP^{3/2} (H,J), \qquad v \mapsto v(0,\cdot) .
$$
The mapping $ \sR $ is a smooth surjective submersion. 
\end{proposition}
\begin{proof}
 Let $ u \in \sB^+(x), \; \Phi_u $ and $ \sU_u $ be as in Definition \ref{cor:chart_strips}. 
 Denote with $ \alpha $ the smooth path $ \alpha(t) = u(0,t) $. 
 In a neighborhood $ \sU_u $ the map $ \sR $ is given by 
 $$ 
 \sR = \Psi_{\alpha} \circ r \circ \Phi_u,
 $$
 where $ \Psi_{\alpha} $ is the inverse of the local chart 
 $ \Phi_{\alpha} $ constructed in Lemma \ref{lem:loc_path} 
 and 
 $$ 
 r : W^{2,2}_{bc} ( \R^+\times [0,1],\R^{2n} ) \to H^{3/2}_{bc} 
 $$
 is a restriction map $r(\xi) = \xi(0,\cdot) $. 
 As $ d\sR =  d\Psi_{\alpha} \circ r \circ d\Phi_u $
 and both $ d\Psi_{\alpha} $ and $d\Phi_u $ are 
 bijective it follows from Proposition \ref{pr:trace} that $\sR $ is submersion. 
 Surjectivity of the map $ \sR $ follows again from Proposition \ref{pr:trace}.
 As the mapping $ r$ is surjective and $ \Phi_u $ and $ \Psi_{\alpha} $ 
 are diffeomorphisms, it follows that the mapping $ \sR $ maps 
 a neighborhood of a smooth map $ u \in \sB^+(x) $ 
 onto a neighborhood of $ \alpha = u(0, \cdot) $. As the set of smooth 
 $ \alpha = u(0,t) $ is dense in $ \sP^{3/2} $ we have that the mapping $ \sR $
 is also surjective. 
\end{proof}

\section{Proof of the Theorem \ref{cor:main_thm3}}\label{SEC:prof_main_thm3}
  
In this section we prove Theorem \ref{cor:main_thm3}.
The proof is based on the study of the vertical differential 
and main ingredients of the proof are already contained in 
the previous chapter. 

\begin{PARA}[{\bf Vertical differential}]\label{para:ver_diff} \rm 
Let $ \sB= \sB^{\pm} (x) $ or $ \sB = \sB^T $, where 
$ \sB^{\pm} (x)$ and $ \sB^T $ are defined in \ref{para:hil_strips}. 
Let $ \sE $ be a Hilbert space bundle over $ \sB$ with 
the fiber over each $ v \in \sB $, $ \sE_v = W^{1,2}_{bc} ( v^* TM) $, where
$$
W^{1,2}_{bc} ( v^* TM)= \{ \xi \in W^{1,2}(v^* TM) |\xi(s,i) \in T_{v(s,i)} L_i, \;\; i=0,1 \}. 
$$
Observe a section $ \sS : \sB \to \sE $ 
of this bundle given by 
$$
\sS(v) = \dbar_{J_t,X_{H_t}} v = 
\p_s v + J_t(v) ( \p_t v - X_{H_t} (v)) .
$$
Given $ u \in \sS^{-1} (0) $ denote by $ D_u $ 
the {\bf vertical differential} 
 $$
 D_u = \pi_u \circ d\sS(u): T_u \sB \to \sE_u 
 $$
 where $\pi_u : T_{(u,0)} \sE=T_u \sB \oplus \sE_u \rightarrow \sE_u  $
 is the projection to the fiber. The vertical differential 
 is given by
 \begin{equation}
  D_u ( \hat{u}) = \nabla_s \hat{u} + J_t(u) ( \nabla_t \hat{u} - \nabla_{\hat{u}} X_{H_t} (u) )
  + (\nabla_{\hat{u}} J_t(u)) ( \p_t u - X_{H_t} (u))
 \end{equation}
 and it is independent of the choice of connection. 
 Notice that if we take $ \nabla= \nabla^t $ to be 
 the Levi-Civita connection of the metric $h_t $ 
 which has the property that $L_i, i=0,1 $ are totally 
 geodesic with respect to $h_i$, then the tangent space 
 $T_u \sB $ can be described as follows
 \begin{equation}\label{eq:tan_sp}
  T_u \sB := \left\{ \hat{u} \in W^{2,2} ( u^* TM) \Big|
  \begin{array}{l}
  \hat{u}(s,i) \in T_{u(s,i)} L_i, \;\; i=0,1 \\
   D_u ( \hat{u}(s,i)) \in T_{u(s,i)} L_i, \;\; i=0,1
  \end{array}\right\}.
 \end{equation}

\end{PARA}

\begin{PARA}[{\bf Linearized operator at infinity}]\label{linop_inf}\rm 
Let $ x $ be a solution of \eqref{eq:CRIT}, we observe 
the vertical differential at $x$ and we prove that it is symmetric. 
One can analogously as in Lemma 2.3 in \cite{RS3} 
prove that it is bijective. 
\begin{theorem}\label{thm:lin_infty}
 Denote with $ g_t $ a metric that we obtain by 
pairing $J_t $ and $ \omega $. 
$$ 
g_t(p)(v,w)= \om_p(v, J_t(p) w )= \langle v, w \rangle_t 
$$ 
Let $ x :[0,1] \to M $ be a solution of \eqref{eq:CRIT}. 
The operator 
$$ 
A : W^{1,2}_{bc} ( x^* TM) \to L^2 (x^* TM) 
$$ 
given by 
\begin{align*}
 A(\hx) = J_t(x) ( \nabla_t \hx - \nabla_{\hx} X_{H_t}) 
\end{align*}
is symmetric with respect to the following scalar product
$$
 \langle \xi, \eta \rangle_{L^2} = \int_0^1 g_t( \xi(t), \eta(t)) dt 
 =  \int_0^1 \langle \xi(t), \eta(t)\rangle_t dt. 
$$
\end{theorem}
\begin{proof}
It is enough to prove that 
$$ 
\int_0^1 \langle \hy,J_t(x) ( \nabla_t \hx - \nabla_{\hx} X_{H_t})\rangle_t=
\int_0^1 \langle J_t(x) (\nabla_t \hy - \nabla_{\hy} X_{H_t} ) , \hx\rangle_t  
$$
for all vector fields $ \hx,\hy\in W^{1,2}_{bc} ( x^* TM)  $ . 
Denote with $ \dot{g}_t (p)(v,w) = \omega(p)(v, \dot{J}_t(p) w ) $, where 
$\dot{J}_t(p)=\frac{d}{dt}J_t(p)$. 
\begin{align*}
 I&= \int_0^1 \langle \hy, J_t(x) ( \nabla_t \hx - \nabla_{\hx} X_{H_t}) \rangle_t dt \\
  &=  \int_0^1 \langle \hy, J_t(x) \nabla_t \hx \rangle_t dt -
	\int_0^1 \langle \hy,J_t(x) \nabla_{\hx}X_{H_t}\rangle_t dt \\
  &= 	-\int_0^1\langle J_t(x) \hy, \nabla_t \hx \rangle_t  - 
      \int_0^1 \langle \hy,J_t(x) \nabla_{\hx}X_{H_t}\rangle_t dt \\
  &=
      \int_0^1\langle \nabla_t ( J_t(x) \hy), \hx \rangle_t -  
      \overbrace{\langle J_1(\hy(1)), \hx(1)\rangle_1}^{bc=0}\\
      & + \overbrace{\langle J_0(\hy(0)), \hx(0)\rangle_0}^{bc=0} + \int_0^1 \dot{g}_t ( J_t(x) \hy, \hx)
      -\int_0^1 \langle \hy,  J_t(x)\nabla_{\hx}X_{H_t}\rangle_t dt\\
  &= \int_0^1\langle \dot{J}_t(x) \hy, \hx \rangle_t dt +
      \int_0^1 \langle (\nabla_{\dot{x}} J_t)\hy, \hx\rangle_t dt  +
      \int_0^1 \langle J_t(x) \nabla_t \hy, \hx\rangle_t\\
      & + \int_0^1 \dot{g}_t ( J_t(x) \hy, \hx)
      -\int_0^1 \langle \hy, J_t(x) \nabla_{\hx}X_{H_t}\rangle_t dt
\end{align*}
 Notice that the first and the fourth term of 
 the previous equality cancel out. This follows by 
 differentiating the sum $ 0= \frac{d}{dt}(\om(w, J_t(p) v ) + \om ( v, J_t(p) w )) $
Thus we have 
\begin{align*}
 I&= \int_0^1 \langle (\nabla_{\dot{x}} J_t)\hy, \hx\rangle_t dt  +
 \overbrace{\int_0^1 \langle J_t(x) (\nabla_t \hy-\nabla_{\hy} X_{H_t}), \hx \rangle_t dt}^J
 \\&+ \int_0^1\langle J_t(x) \nabla_{\hy}X_{H_t}, \hx\rangle_t dt -
 \int_0^1 \inner{\hy}{J_t\nabla_{\hx} X_{H_t}}_t
\end{align*}
Write
$$
 \int_0^1 \langle J_t(x)\nabla_{\hy}X_{H_t}, \hx\rangle_t dt =
  \int_0^1 \langle \nabla_{\hy} ( J_t X_H) , \hx\rangle_t -
  \int_0^1 \langle (\nabla_{\hy} J_t) X_{H_t} , \hx\rangle,
  $$
and similarly 
$$\int_0^1 \langle \hy, J_t \nabla_{\hx} X_{H_t} \rangle_t dt = 
\int_0^1 \langle \hy, \nabla_{\hx} (J_t X_H) \rangle_t dt- \int_0^1 \langle \hy, (\nabla_{\hx} J_t) X_{H_t} \rangle_t .
$$
As $ J_t X_{H_t} = \nabla H $, we have that 
$$ 
\langle \nabla_{\hy} \nabla H, \hx \rangle_t= \langle \hy, \nabla_{\hx} \nabla H \rangle_t.
$$
Thus it follows that 
\begin{align*}
I= J + \int_0^1 \langle  (\nabla_{\dot{x}} J_t) \hy, \hx \rangle_t dt -
\int_0^1 \langle ( \nabla_{\hy} J_t ) X_{H_t}, \hx \rangle_t + \int_0^1 \langle \hy, (\nabla_{\hx} J_t) X_{H_t} \rangle_t dt 
\end{align*}
As $ \nabla_{\hy} J_t $ is skew symmetric it follows that
$$ 
-\int_0^1 \inner {(\nabla_{\hy} J_t)\dot{x} }{\hx}_t = \int_0^1 \inner {(\nabla_{\hy} J_t) \hx}{ \dot{x}}_t.
$$
Now the sum 
\begin{equation}
\int_0^1 \langle  (\nabla_{\dot{x}} J_t) \hy, \hx \rangle_t dt + \int_0^1 \inner {(\nabla_{\hy} J_t) \hx}{ \dot{x}}_t + \int_0^1 \langle (\nabla_{\hx} J_t) \dot{x}, \hy \rangle_t dt 
\end{equation}
is equal to zero and this follows from 
the compatibility of $\omega $ and $ J_t$.
 For more details we refer to the Appendix  in 
 \cite{MS}.
\end{proof}
\end{PARA}

\begin{lemma}[{\bf Unitary trivialization}]\label{lem:triv1}
Let $u\in\sB^+(x) $ be smooth such that $ u(s,t) = x(t), s\geq s_0 $. 
There exists an open set $ U \subset \R^+ \times [0,1] \times M $ 
such that 
$$
u(s,t) \in U_{s,t} = \{ p \in M | (s,t,p)\in V \} 
$$
and a smooth map 
$ \Phi: U \times \R^{2n} \to TM $ 
such that $ \Phi_{s,t} = \Phi( s,t,\cdot) $ 
has the following properties:
\begin{itemize}
 \item[i)] $\Phi_{s,t}(p) : \R^{2n} \to T_p M, \;\; p \in U_{s,t} $ 
	    is a vector space isomorphism. 
 \item[ii)] $\Phi_{s,t} (p) $ is complex, i.e. 
	  $$ 
	    \Phi_{s,t}(p) J_0 = J_t(p) \Phi_{s,t}(p) ,
	    \;\; p \in V_{s,t}.
	    $$
  \item[iii)] $ \Phi_{s,i}(q) : \R^n \times \{0\} \to T_q L_i, \;\; q \in L_i \cap U_{s,i}, \; i=0,1$
  \item[iv)] The mapping $ \Phi_{s,t} $ is $ s$ independent for 
     $s$ sufficiently large, thus  $\Phi_{s,t} = \Phi_t$ for $ s $ sufficiently large. 
   \item[v)] $\Phi_t(x(t)) $ is symplectic 
   $$
   \omega(\Phi_t(x(t))\cdot, \Phi_t(x(t)) \cdot)= \omega_0 ( \cdot, \cdot).  
   $$
\end{itemize}
\end{lemma}
\begin{proof}
\medskip
In steps 1-4 we assume that the Hamiltonian term vanishes, thus 
$ u(s,t) = p_0 $ for $ s \geq s_0 $. In the last Step we reduce the 
case $H\neq 0 $ to the case $H=0 $ by naturality. \\
 {\bf Step 1.} {\bf Construction of the trivialization $ \Phi_{\infty}(t) $ of $T_{p_0} M  $ 
 satisfying the properties $i)-v)$.} \\
 \medskip
 This trivialization can be constructed analogously 
 as the trivialization $ e_{\alpha} $ in the 
  Step 1 of the proof of Lemma \ref{le:phi}, or in the following way. 
  There exists a smooth family of Lagrangian planes $T_{p_0} L_t $ connecting 
  $ T_{p_0} L_0 $ and $ T_{p_0} L_1 $. Let $ e_i(t) \in T_{p_0} L_t , \; i=1, \cdots, n$
  be orthonormal frame with respect the metric $  \om( \cdot, J_t(p_0) \cdot ) $. 
  Observe the unitary frame $ \{ e_i(t) J_t(p_0) e_i(t)\}_{i=1,\cdots,n} $. 
  Let $ \{ v_i, J_0 v_i\} _{i=1,n} $ be the standard basis of $ \R^{2n} $
  ( $ v_i $ is the orthonormal basis of $\R^n \times \{0\} $). 
  We define the trivialization 
  $$ \Phi_{\infty} :[0,1] \times \R^{2n} \to T_{p_0} M $$ 
  by 
  $$ \Phi_{\infty} (t) v_i = e_i(t), \;\; \Phi_{\infty} (t)J_0 v_i= J_t(p_0) e_i(t) .$$
  Notice that the the constructed trivialization maps $ \R^n \times \{0\} $ to $ T_{p_0} L_i, \; i =0,1 $
  and that is unitary. 
 \medskip
 
 {\bf Step 2.}{\bf Extension of the trivialization 
$ \Phi_{\infty}(t) $ to the neighborhood of $p_0$.}\\
\medskip
There exists an open set $ V \subset I \times M $ such that 
$$
p_0 \in V_t = \{ p \in M | (t, p) \in V \},
$$
and a smooth map $ \tPhi: V \times \R^{2n} \to TM$ such that 
$ \tPhi_t := \tPhi(t, \cdot) $ satisfies the following 
\begin{itemize}
 \item [1)] $\tPhi_t(p) : \R^{2n} \to T_p M $, $ p \in V_t $ is a vector space 
	    isomorphism for all $ p \in V_t $. 
 \item[2)] $ \tPhi_t(p) J_0 = J_t(p) \tPhi_t(p) $ 
	      for all $ p \in V_t $. 
 \item[3)] $\tPhi_i(p):\R^n \times \{0\} \to T_p L_i $ for all 
             $p \in L_i\cap V_i, \; i=0,1 $.
 \item[4)] $ \tPhi_t( x(t))= \Phi_{\infty}(t)$. 
\end{itemize}
Let $ g_t $ be a family of metrics as in Lemma \ref{mon_lem7}
Let $ \nabla^t $ be a Levi-Civita connection of the metric $g_t$ 
and let $ \tnabla^t $ be a complex linear connection 
associated to $ \nabla^t $ 
 $$
 \tnabla^t_{\lambda}v = \nabla^t_{\lambda} v - \frac{1}{2} J_t(\nabla_{\lambda}^t J_t) v. 
 $$
 For a point $ p $ in a geodesic neighborhood of $ p_0 $ we define 
the trivialization $ \tPhi_t(p) $ as follows 
$$
\tPhi_t(p) v := P_{\gamma}(p_0, p)\Phi_{\infty} (t) v, 
$$
where $ P_{\gamma} $ denotes parallel transport (with respect to $\tnabla^t$) 
along geodesic  $ \gamma(\lambda) $ connecting $ p_0$ and $p$. 
As $P_{\gamma} $ commutes with $ J_t $ it follows that 
$\tPhi_t $ satisfies $2) $. As $ L_i, \; i=0,1 $ are totally geodesic 
with respect to $ \tnabla^i $, we have that $3) $ is also satisfied. 
We can assume that $U_{s,t}:=V_t , s\geq s_0.$\\
 
\medskip
{\bf Step 3. Extension of the trivialization $\Phi_{\infty}(t) $ to the trivialization 
of $ u^*TM $.}\\
\medskip
Use parallel transport with respect to $ \tnabla^t $, defined as in Step 2, 
along $u(s,t) $ in the direction of $s$. In this way we construct a 
smooth mapping 
$$ 
\Phi: \R^+ \times [0,1] \times \R^{2n} \to u^* TM ,
$$
or equivalently we construct the trivialization 
$$
\Phi_{s,t}(u(s,t))=\Phi(s,t,u(s,t)) : \R^{2n} \to T_{u(s,t)}M.  
$$
defined by 
$$
 \Phi_{s,t}(u(s,t)) := P_{u}( u(s_0,\cdot), u(s,\cdot)) \Phi_{\infty}(t)v ,
 $$
 where $ P_{u} ( u(s_0,\cdot), u(s,\cdot)) $ denotes parallel transport 
 along $u$. \\
 \medskip
 
 {\bf Step 4. Extension of the trivialization $ \Phi_{s,t}(u(s,t)) $ to some neighborhoods
 $U_{s,t} $ of $u(s,t)$.}\\
 This can be done analogously as in Step 2, using 
 parallel transport along geodesics. Thus the neighborhoods 
 $U_{s,t} $ are just geodesic neighborhoods of $u(s,t) $.\\
 {\bf Step 5.} {\bf Reducing the general case to the case $H=0 $. }\\
 Let $ \tu(s,t)= \phi_t^{-1} (u(s,t)) $ and $ \tJ_t = \phi_t^* J_t $, 
 where $ \Phi_t $ is Hamiltonian isotopy \eqref{eq:phiH}. 
 Then $ \tu(s,t) = p_0 , s \geq s_0 $ and we can construct as in Steps 1-4 the trivialization 
 $ \tPhi $ satisfying the properties $ i) - v) $ applied to the curve $ \tu $ and the 
 almost complex structure $ \tJ_t $. Then the mapping 
 $ \Phi $ defined as follows 
 $$ \Phi_{s,t} (p) := d\phi_t ( \phi_t^{-1} (p)) \tPhi_{s,t} ( \phi_t^{-1} (p)) $$
 satisfies the properties $i) -v) $. 
 
\end{proof}

\begin{PARA}[{\bf Conjugate operator}]\label{rem:triv1}\rm 
Let $u$ be a solution of the equation \eqref{eq_FLOER}
and suppose that $u$ converges exponentially 
toward $x\in \cC(L_0,L_1;H)$. Notice that $u$ is just of $ W^{2,2} $ class, 
though it is smooth away from the non-Lagrangian boundary. 
Let $ D_u $ be the vertical differential as in \ref{para:ver_diff}
and let $ \Phi_{s,t} $ be as in Lemma \ref{lem:triv1}.
Without loss of generality we can assume that $ u(s,t) \in U_{s,t} $, 
where $ U_{s,t} $ are as in Lemma \ref{lem:triv1}.
We abbreviate $ \Phi= \Phi_{s,t} (u(s,t)) $. 
It follows from the properties of the trivialization $ \Phi $ 
that the conjugate operator
$ D = \Phi^{-1} \circ D_u\circ \Phi $ has the following form 
\begin{equation}\label{eq:conj_opD}
D\xi = \Phi^{-1}  D_u\circ (\Phi\xi ) = \p_s \xi + J_0 \p_t \xi + S(s,t) \xi  = \p_s \xi + A(s) \xi 
\end{equation}
The matrix valued function 
$ S \in W^{1,2}(\R^+\times [0,1], \R^{2n\times 2n} )$ 
is given by
$$
 S(v) =\overbrace{\Phi^{-1} \nabla_s (\Phi v)}^{C(v)} + 
J_0 \overbrace{\Phi^{-1} ( \nabla_t (\Phi v) - \nabla_{\Phi(v)} X_{H_t}  )}^{B(v)} +
 \overbrace{\Phi^{-1}(\nabla_{\Phi(v)}J)( \p_t u - X_{H_t})}^{E(v)}
$$
Define a smooth matrix valued function $ B_{\infty} : [0,1] \to \R^{2n\times 2n} $
by 
$$
  B_{\infty} (t) v = 
 \Phi_t(x(t))^{-1} \Big ( \nabla_t ( \Phi_t( x(t)) v ) - \nabla_{\Phi_t(x(t)) v } X_{H_t} ( x(t)) \Big ), 
$$
where $ \Phi_t $ is given as in part $ iv) $ of Lemma \ref{lem:triv1}. 
From exponential decay of $u$ it follows that 
the $ C,E $ converge exponentially to zero and the function $ B $ 
converges toward $  B_{\infty} $. One can also see that $ J_0 B_{\infty} (t) $ 
is symmetric for all $t$. Besides the matrix valued functions 
$C(s,i),E(s,i) $ and $ J_0 ( B(s,i)- B_{\infty}(i) ) $ map $ \R^n \times \{0\} $
into itself for $ i=0,1 $. These boundary properties will 
imply that if we observe the operators $ A(s) = J_0 \p_t + S(s,t) $ 
and we fix $ H^1= \Dom(A(s)) $, as in \eqref{eq:hilspaces}, 
then also $ H^2= \Dom(A(s)^2) $ will be independent of $s$. 
\end{PARA}

\begin{corollary}\label{cor:conv}
 Let $ H^i, \; i=0,1 $ be as in \eqref{eq:hilspaces} and let 
 $ C,B, E $ and $ B_{\infty} $ be as in \ref{rem:triv1}. 
 Then for  any $ k \in \N $ we have 
 \begin{align*}
 & \lim\limits_{s\rightarrow \infty} \| C \|_{C^k ( [ s, +\infty)\times [0,1]  )} = 0 
 , \lim\limits_{s\rightarrow \infty} \|E \|_{C^k ( [ s, +\infty)\times [0,1] )} = 0 \\
 & \lim\limits_{s\rightarrow \infty} \| B - B_{\infty} \|_{C^k ( [ s, +\infty)\times [0,1] )} = 0.  
 \end{align*}
Moreover 
\begin{itemize}
\item[i)] The operator $ J_0\p_t + J_0 B_{\infty}  : H^1 \to  H^0 $ is bijective and self-adjoint. 
\item[ii)] The functions $ C(s,i), E(s,i) $ and $ J_0 ( B(s,i) - B_{\infty}(i) ) $ map 
$ \R^n \times{0} $ to itself for $ i=0,1 $. 
\end{itemize}
\end{corollary}
\begin{proof}
 The fact that $ C, E $ and the difference $ B - B_{\infty} $ 
 converge to zero, follows from the fact that $u$ converges 
 exponentially to $x(t)$. Remember that $\nabla $ 
 in the definition of the matrix valued functions $B, C $ and $E$ 
 is the Levi-Civita connection of the metric $g_t $ such that 
 $L_i $ are totally geodesic with respect to $g_i,\; i=0,1 $.
 This implies that $ C $ and $E$ satisfy  $ii) $. 
 To prove that $ J_0 ( B(s,i) - B_{\infty}(i) ): \R^n \times \{0\} \to \R^n \times \{0\} $
 it is enough to prove that $ \p_s B(s,i) : \R^n \times \{0\}\to \{0\} \times \R^n $. 
 We prove this fact in the case that Hamiltonian term vanishes. Notice 
 that the trivialization was constructed such that the general case $ H \neq 0 $ 
 can be reduced to the case $H=0 $. The following equalities hold 
 \begin{equation}
  \begin{split}
   &\nabla_s ( \Phi_{s,t}(u) B(s,t) v ) = \nabla_s \nabla_t ( \Phi_{s,t}(u) v ) \\
   & \nabla_s ( \Phi_{s,t}(u)) B(s,t)v + \Phi_{s,t} (u) \p_s B(s,t) v = \nabla_t \nabla_s ( \Phi_{s,t} (u) v) + R ( \p_su, \p_t u) \Phi_{s,t}(u) v 
  \end{split}
 \end{equation}
 Where $ R $ denotes the Riemann curvature tensor. 
 Notice that the first term on both left and right side 
 of the upper equality vanishes, as $ \Phi_{s,t}(u) = P_s  \circ \Phi_t(x(t) $, 
 where $ P_s $ denotes parallel transport along $u$ in the direction of $s$. 
 Thus we have 
 \begin{equation*}
  \Phi_{s,t} (u) \p_s B(s,t) v= R ( \p_su, \p_t u) \Phi_{s,t}(u) v.
 \end{equation*}
 It is left to prove that $ \left. R(\p_s u, \p_t u ) \Phi_{s,t}(u) v\right|_{t=0} \perp T_{u(s,0)}L_0 $
 and analogously for $t=1 $. Thus it is enough to prove that 
 $$
 \langle\left. R(\p_s u, \p_t u ) \Phi_{s,t}(u)v\right|_{t=0},  w\rangle_0 = 0, \;\;\;\;\forall w \in T_{u(s,0)} L_0 
 $$
From the properties of the curvature $R$ we have 
\begin{equation}
 \inner{ R(\p_s u, \p_t u ) \Phi_{s,t}(u)v}{w}_0=\inner{R(\Phi_{s,t}(u)v,w)\p_su}{\p_t u}.
\end{equation}
As $ R(X,Y) Z = \nabla_x \nabla_y Z - \nabla_Y \nabla_X Z $ and all three 
vector fields $ \left. \Phi_{s,t}(u)v\right|_{t=0} $, $w $ and $ \p_s u $ are tangent to $L_0 $
and as $L_i $ are totally geodesic for $g_i $ we have that $ R(\Phi_{s,t}(u)v,w)\p_su $ 
is also tangent to $L_0 $. As $ \p_t u $ is orthogonal to $ T_{u(s,0)} L_0 $ we have 
that $  \inner{ R(\p_s u, \p_t u ) \Phi_{s,t}(u)v}{w}_0=0 $ for all $ w \in T_{u(s,0)} L_0 $.

\end{proof}

\begin{theorem}\label{thm:ver_diff}
Let $ \sB $ and $ u\in \sB $ be as in \ref{para:ver_diff} 
and let $ D_u : T_u \sB \to \sE_u $ be vertical differential 
as in \ref{para:ver_diff}. Then the following holds
\begin{itemize}
\item[a)] The operator $D_u $ is surjective. 
\item[b)] Suppose that $ \sB = \sB^+ $. Then there exists a constant 
$c> 0 $ such that the following inequality holds for all $ \hat{u} \in T_u \sB $

\begin{equation}\label{eq:inq_D}
\| \hat{u}\|_{2,2} \leq c \Big ( \| D_u \hat{u}\|_{1,2} + \| \hat{u}(0, \cdot)\|_{3/2} )
\end{equation}
Analogous results holds in the case $ \sB = \sB^T $. 
For every $ T > 0 $ there exists a constant $c$ such that 
 the following inequality holds for all $ \hat{u} \in T_u \sB^T $. 
\begin{equation}\label{eq:inq_D1} 
\| \hat{u}\|_{2,2} \leq c \Big ( \| D_u (\hat{u}) \|_{1,2} + \| \hat{u}(-T, \cdot)\|_{3/2} + \| \hat{u}(T, \cdot)\|_{3/2} \Big )
\end{equation}
\end{itemize}
\end{theorem}
\begin{proof}
$a)$ Let $ \Phi $ be the trivialization of $ u^*TM $ as in Lemma 
\ref{lem:triv1}. Let $ D$ be the conjugate operator 
as in \eqref{eq:conj_opD}, $D = \Phi^{-1} \circ D_u \circ \Phi  $.  
Conjugation by $\Phi $ identifies the tangent space $ T_u \sB $ with 
the Hilbert space $ H^2_{bc} ( I \times [0,1] ) $  given by 
\begin{equation*}
\begin{split}
 H^2_{bc} ( I \times [0,1] ) & = \left \{\xi\in W^{2,2} ( I \times [0,1] ) \Big | 
 \begin{array}{l}
 \xi(s,i) \in \R^n \times \{0\} , \;\; i=0,1 \\
   D\xi(s,i) \in \R^n \times \{0\}, \;\; i=0,1 
 \end{array}
 \right\}\\
 &= W^{2,2} ( I, L^2([0,1] )) \cap L^2 ( I, \Dom(A(s))^2 ) 
 \end{split}
\end{equation*}
notice that from \ref{rem:triv1} and \ref{cor:conv} we have that
$ \Dom(A(s)^2) $ is $s$ independent. 
It follows from Corollary  \ref{cor:surj} that $D_u $ is 
surjective. \\
\medskip
$b)$ Let $ \hat{u}(s,t) = \Phi_{s,t}(u(s,t))\xi(s,t) $, 
where $ \Phi$ is the trivialization constructed in Lemma \ref{lem:triv1}
and let $ D = \Phi^{-1} \circ D_u \circ \Phi  $ be as in \eqref{eq:conj_opD}. 
As the operator $D$ has the form \eqref{eq:opD}, 
the conclusion of Corollary \ref{lem:inq_D}
holds. Thus the inequality \eqref{eq:inq_D} follows from Corollary 
\ref{lem:inq_D} part $b)$, whereas the inequality \eqref{eq:inq_D1} 
follows from  the same Corollary part $a)$. 

\end{proof}

\begin{proof}[{\bf Proof of Theorem \ref{cor:main_thm3}}]
\bigskip
This is just an easy corollary of the previous Theorem. 
\noindent{\bf a)} We prove that $ \sM^+ \subset \sB^+ $ is a smooth Hilbert 
submanifold. The proof that $ \sM^T \subset \sB^T $ is a smooth submanifold is analogous. 
Let $ u \in \sM^+ $ and let $ D_u $ be the vertical differential as in \ref{para:ver_diff}. 
It is enough to prove that $ D_u : T_u \sB \to \sE_u $ is surjective, as this implies that 
the section $ \sS = \dbar_{J_t, X_t} $ is transverse to the $ 0-$section and hence 
the set $ \sM^+ = \sS^{-1} ( 0 ) $ is a smooth submanifold. This follows from Theorem 
\ref{thm:ver_diff} part $a)$. \\
\medskip
\noindent{\bf b)} To prove that the maps $ i^{\pm}:\sM^{\pm} \to \sP^{3/2} $ are immersions
notice that  $ T_u \sM^{\pm}\cong \text{Ker} (D_u)  $, thus it follows 
directly from the inequality \eqref{eq:inq_D}  that the  maps $ di^{\pm} (u) $
are injective.   
In the case of finite strips it follows from the inequality \eqref{eq:inq_D1} 
that the mapping $i^T $ is an immersion. 
That the maps $ i^{\pm} $ and $ i^T $ are also injective 
follows from the unique continuation of holomorphic curves. 
\end{proof}

\section{Embedding into the path space}\label{SEC:embedding}

In this section we prove Theorem \ref{thm:main_thm2}. 

\begin{PARA}[{\bf Zero Hamiltonian}]\label{para:zero_ham}\rm
 Let $ x \in \cC ( L_0, L_1,H) $ and let  $ \sM^{\pm} ( x; H,J ) $ and $\cM^T(H,J) $ 
be defined as in \ref{para:hil_strips}. We have proved in Theorem 
\ref{cor:main_thm3} that these moduli spaces are Hilbert manifolds and that 
they can be immersed into the Hilbert manifold of paths $ \sP^{3/2} (H,J) $.
In Remark \ref{rem:ham0} we have explained that it is enough 
to study these manifolds in the case that the Hamiltonian function $ H= 0 $. 
Let $ \tJ$, $ \tx$  and $ \tL_i, \; i=0,1 $ be as in Remark \ref{rem:ham0}. 
Notice that it follows from \ref{linop_inf} that the 
linearized operator 
$$ 
A : W^{1,2}_{bc} (\tx^* TM) \to L^2 (\tx^*TM) , \;\; A(\hat{x} ) := \tJ_t(\tx) \p_t \hat{x} 
$$
is bijective and self-adjoint.  This will be crucial for all the proofs. 
We abbreviate 
$ \sM^{\infty} = \sM^{\infty} ( \tilde{x}; 0, \tJ) $,
$ \sM^T = \sM^T ( 0, \tJ ) $ and $ \sP^{3/2} = \sP^{3/2} ( 0, \tJ ) $. 
Here we shall consider only those curves that 
are close enough to the constant curve 
$$
p:= \tx \in\tL_0 \cap \tL_1 
$$
 and have sufficiently small energy.
 We first explain when a path $ \alpha\in \sP^{3/2} $
 is in an $ \epsilon $ neighborhood of a constant path $p$. 
 Then we introduce subsets
$ \sM^{\infty}_{\epsilon} $ and $ \sM^T_{\epsilon} $
of the moduli space of holomorphic curves $ \sM^{\infty} $ and $ \sM^T $  
and we prove that they can be embedded, 
by taking the restriction to the boundary,
into the Hilbert manifold of paths $ \sP^{3/2} $.
To simplify the notation we omit $ \sim$. 
\end{PARA}

\begin{PARA}[{\bf $\eps-$neighborhood of  $p\in L_0 \cap L_1 $}]\label{ch5_rem1}\rm
 We assume that the Hamiltonian is $0$.
 Local chart in the neighborhood of $p$ within 
the Hilbert manifold $ \sP^{3/2} $ 
is given as in Lemma \ref{lem:loc_path}.
Remember that the local chart in the neighborhood 
of a constant path $ p \in \sP^{3/2} $, 
 $ \Phi_p: \sU_p \rightarrow \sW_p $, 
 is given by $ \Phi_p( \gamma ) = f_t(\gamma) $, where
 $f_t: U_p \rightarrow \R^{2n} $ is a smooth 
 family of maps, $ U_p\subset M $ is an open 
 neighborhood of $p$ that doesn't contain other intersection 
 points of $ L_0 \cap L_1 $ and $f_t $ has the following
 additional properties:

\begin{itemize}
 \item [1)] $ \;\; f_t : U_p \rightarrow f_t ( U_p)\subset \R^{2n} $ is
 a diffeomorphism for all $t$ and $f_t( p) =0$ for all $t$.

\item[2)] $ f_i( L_i\cap U_p) = (\R^n \times \{0\}) \cap f_i ( U_p), \;\;\; i=0,1 $, 
 and $\left. (\partial_t f_t)\right|_{t=0,1}= 0 $.     

\item[3)] If $ \tilde{J}_t = (f_t)_* J_t $, then $ \tilde{J}_t (x,0) = J_{std} $ 
 for all $ t \in [0,1] $ and for all $ (x,0) \in (\R^n \times \{0\}) \cap f_t( U_p)$.
\end{itemize}
We say that a curve $ \gamma \in \sP^{3/2}$ is in the $\epsilon $ neighborhood 
of $p\in \sP^{3/2} $ and we write $ \gamma \in \cU_{\eps}(p) $ 
iff  $\gamma \subset U_p $ and 
 $ \xi(t) := f_t(\gamma(t)) = \Phi_p(\gamma)(t)\subset H^{3/2} $ 
satisfies 
$$ 
 \| \xi \|_{3/2} < \epsilon.
$$
Notice that it makes sense to define $ \epsilon-$ neighborhood
of a constant path $p$ only in the case that $ \epsilon>0 $ 
is sufficiently small. \\
\medskip
We define analogously an $ \eps -$ neighborhood of 
a constant strip $ p \in \sB^{\pm} $ (or $ p\in \sB^T$). 
A strip $ u \in \sB^{\pm} $ is in an $ \eps -$neighborhood of $p$ if 
$ \text{Im}(u) \subset U_p $ and 
$ \xi(s,t) =f_t(u(s,t)) $ satisfies
$$ 
\| \xi \|_{2,2} < \epsilon, 
$$
where $f_t $ is a local chart as above. 
\end{PARA}

\begin{PARA}[{\bf Monotonicity}]\label{para:mon}\rm
 Let $ U_p $ be the neighborhood of a point $p$ 
as in Remark \ref{ch5_rem1}. 
We shall be interested only in those holomorphic curves 
that are contained in this neighborhood. 
In Theorems \ref{thm:mon1} and \ref{thm:mon2}, we prove that the 
energy of a holomorphic curve
$ u: I \times [0,1] \rightarrow N $ and 
$ \sup\limits_{t}\sup\limits_{s\in \partial I} d( u(s, t), p) $
control the distance 
$ \sup\limits_{s,t} d( u(s,t), p) $.
Let $ \hbar $  and $ \epsilon _0 > 0 $ be such that 
each holomorphic curve $ u $ which satisfies 
$$
E(u) < \hbar, \;\;\; \left. u \right|_{\partial I \times [0,1] }
\in \cU_{\epsilon_0} ( p ) 
$$
is contained in  $U_p $, i.e. 
$$
u ( s,t ) \in U_p, \;\; \forall ( s,t) \in I \times [0,1].
$$
Here $ \cU_{\epsilon_0}(p) $ denotes the $ \epsilon_0 $ 
neighborhood of a constant path $ p $ in the Hilbert 
manifold $ \sP^{3/2} $ as in definition \ref{ch5_rem1}. 
\end{PARA}
\begin{definition}\label{def:emb_mfld}
 Let $ \epsilon_0 $ and $ \hbar $ be as in \ref{para:mon}
 and let $ \sM^T $ and $ \sM^{\infty} $ be as in \ref{para:zero_ham} 
 For $ \epsilon \leq \epsilon_0 $ we
 define the following subsets of the moduli 
spaces of holomorphic strips:

\begin{align*}
 &\sM^{\infty}_{\epsilon} = \bigg \{ ( u^+, u^-) \in \sM^{\infty}
 \bigg | \;\; u^{\pm} (0, \cdot) \in \cU_{\epsilon} (p)
 , \;\;\; E(u^{\pm}) < \hbar \bigg \} \\ 
 &\sM^T_{\epsilon} = \bigg \{ u\in \sM^T \;\bigg | \;\; 
 u(\pm T, \cdot )\in \cU_{\epsilon} (p) ,
 \; E(u) < \hbar  \bigg \}. 
\end{align*}
Here $ \cU_{\epsilon} $ denotes $ \epsilon-$neighborhood of a constant path 
$p$ in the Hilbert manifold of paths $ \sP^{3/2} $ as in \ref{ch5_rem1}.
\end{definition}

\begin{PARA}[{\bf Local setup}]\label{para:loc_setup}\rm 
 Holomorphic curves $ u \in \sM^T_{\epsilon}
\;( ( u^+, u^- ) \in \sM^{\infty}_{\epsilon}) $ are
contained in the small neighborhood $U_p$ of the point $p\in M$, 
as explained in \ref{para:mon}. 
Thus instead of observing holomorphic curves on $M $ we can 
work in $\R^{2n}$.  
We reformulate the setup in local coordinates.\\
\medskip
Let $ v = \Phi_p(u) = f_t (u) $ be the image of a holomorphic curve 
$u$ contained in $U_p$. Here $ f_t $ is as in \ref{ch5_rem1}. 
As $ u $ is a $ J_t $ holomorphic curve it follows that 
$v = f_t ( u) $ satisfies the following equation 
 \begin{equation}\label{eq:per_hol}
 \overline{\partial}_{\tJ,\tX}v
 = \partial_s v + \tJ_t(v) ( \partial_t v- \tX_t(v) )=0,
 \end{equation}
 where $ \tX_t $ and $ \tJ_t $ have the following properties.
 \begin{itemize}
  \item[1)] $ \tX_t $ is a smooth vector field given by 
  $$ 
  \tX_t (x) = \partial_t f_t ( f_t^{-1} (x)) .
  $$
  \item[2)] $ \tX_t(x) = 0 $ for $ t=0,1 $ and for all $ x \in \R^{2n} $. 
  \item[3)] $ \tX_t(0) =0 $ for all $ t \in [0,1] $. 
   \item[4)] $ \tJ_t= (f_t)_* J_t $ is a smooth family 
   of almost complex structures with the properties
\begin{equation}\label{eq:prop_J}
 \tJ_t (x,0) = J_{\std}, \;t=0,1, \;\;  x \in \R^n. 
\end{equation}
  \end{itemize}
New boundary data are given by 
 \begin{align}
  v(s,0) \in \tL_0 = f_0 ( L_0 )\subset\R^n \times \{0\}, 
  \qquad v(s,1) \in \tL_1 = f_1 ( L_1) \subset \R^n \times \{0\}
 \end{align}
 Hence it follows that $ v \in H^2_{bc}( I \times [0,1], \R^{2n} )$,
 where $ H^2_{bc} ( I \times [0,1] ) $ is given by \eqref{eq:H2bc}. 
The condition that the intersection of 
Lagrangian submanifolds $ L_0\cap L_1 $ is transverse 
translates into the following condition in $ \R^{2n} $.
Let $ \widetilde{\psi}_t $ be the 
flow of the vector field $ \tX_t $
$$
\partial_t \tpsi_t( x ) = \tX_t ( \tpsi_t(x)), \;\; \tpsi_0 = \one.
$$
Notice that $ \tpsi_t = f_t \circ f_0^{-1} $. 
Then 
$$
L_0 \pitchfork L_1  \qquad \text{ if and only if } \qquad \tpsi_1 ( \tL_0) \pitchfork \tL_1 .
$$
 \end{PARA}
 \begin{PARA}[{\bf Linearized operator}]\label{para:lin_op}\rm
 Let $ \dbar_{\tJ,\tX} $ be as in \eqref{eq:per_hol} and let $ D_0 $ 
 be its linearization at the origin. 
 \begin{align*}
 D_0 \xi = \partial_s \xi + \tJ_t(0) ( \partial_t \xi - d\tX_t(0)\xi)= \partial_s \xi +  A\xi 
\end{align*}
Then the linear operator 
\begin{align*}
  &A: H^1_{bc}([0,1], \R^{2n})\rightarrow L^2 ( [0,1], \R^{2n} ), \;\;  \\
 & A = \tJ_t(0) ( \partial_t  - d\tX_t(0) ) 
\end{align*}
is bijective and self-adjoint, where 
\begin{align*}
 H^1_{bc} ( [0,1] ) = \left \{ \xi \in H^1 ([0,1], \R^{2n} ) \bigg |
 \; \xi(i) \in \R^n \times \{0\}, \; i=0,1  \right \}. 
 \end{align*}
The operator $ A$ 
is conjugate via $ df_t(p) $ to the operator  $ B= J_t(p) \partial_t $, 
$$
A \xi= df_t(p)J_t(p)\partial_t (df_t(p)^{-1}\xi ). 
$$
As the operator $ B= J_t(p) \partial_t$ is bijective and self-adjoint 
it follows that the operator $ A $ is also bijective and self-adjoint with
respect to the appropriately chosen metric. 
More precisely, let $ \xi, \eta \in L^2 ( [0,1], \R^{2n} ) $  we define the $ L^2 $ scalar 
product by 
\begin{equation}\label{eq:sc_prod}
\langle \xi, \eta \rangle := \int_0^1 \omega( df_t(p)^{-1} \xi(t) , J_t(p)df_t(p)^{-1}\eta(t) ) dt. 
\end{equation}
Remember that $ \om $ and $ J_t $ are compatible. 
For $ \xi, \eta \in H^1_{bc} ( [0,1] )$, we have 
\begin{align*}
 \langle \xi , A \eta \rangle &=  \int_0^1 \omega( df_t(p)^{-1} \xi(t) , -\partial_t (df_t(p)^{-1}\eta(t)) )dt \\
 &=  \int_0^1 \omega(\p_t(df_t(p)^{-1} \xi(t)), df_t(p)^{-1}\eta(t) ) dt \\
 &= \int_0^1 \omega( -J_t(p) df_t(p)^{-1} ( A \xi ),df_t(p)^{-1}\eta(t) ) dt \\
 &= \int_0^1 \omega ( df_t(p)^{-1}\eta(t), J_t(p) df_t(p)^{-1} ( A \xi ) ) dt \\
 &= \langle \eta, A \xi \rangle.
\end{align*}
The second equality follows by partial integration 
using the boundary conditions , i.e. 
$ df_i(p) : T_p L_i \to \R^n \times \{0\}, \; i=0,1 $.
Injectivity of the operator $ A$ is easy to prove. 
Notice the following 
\begin{align*}
 A \xi = 0 \Leftrightarrow \zeta(t) = df_t(p)^{-1} \xi(t)  = \text{const.} \in T_pM
\end{align*}
As $ \zeta(0) \in T_p L_0 $ and $ \zeta(1) \in T_p L_1 $ and 
the intersection $ T_p L_0 \cap T_p L_1 = \{ 0\} $, it 
follows that $ \zeta (t) = 0 $ and thus $ \xi \equiv 0 $.  

To simplify the notation we shall
write $ J_t $ and $ X_t $ instead of $ \tX_t $ and $ \tX_t $ further on. 
Let $ v $ be a solution of the equation \eqref{eq:per_hol}
and let $ D_v $ be the linearization of the same equation, then 
\begin{align}\label{eq:lindv}
 D_v \xi& = \partial_s \xi + J_t( v ) ( \partial_t \xi - dX_t(v) \xi ) 
 +  (dJ_t(v) \xi) (\p_t v - X_t(v)). \notag \\
 &= \partial_s \xi + J_t( v ) ( \partial_t \xi - dX_t(v) \xi ) 
 +  (dJ_t(v) \xi) J_t ( v )\p_s v. 
\end{align}
 \end{PARA}
\begin{PARA}[{\bf Proof of Theorem \ref{thm:main_thm2}}]\label{para:pr_thm2}\rm
 Let $ \sM^{\infty}_{\epsilon} $ and $ \sM^T_{\epsilon} $ be defined 
 as in \ref{def:emb_mfld}. In order to prove Theorem \ref{thm:main_thm2} 
 it is enough to prove that for $ \epsilon > 0 $ sufficiently small 
 the manifolds $ \sM^{\infty}_{\epsilon} $ and $ \sM^T_{\epsilon} $
 embed into path space by taking the restriction to the boundary. 
 In Theorem \ref{cor:main_thm3} we have proved 
 that the maps $ i^{\infty}=i^+ \times i^- $ 
 and $ i^T $ are injective immersions. 
 Locally each immersion is an embedding.
 There exists $ r_1 > 0 $ such that the restriction of the map 
 $ i^{\pm} $ ( analogously $ i^T$ ) 
 to the $ \cU_{r_1}(p)$ is an embedding, 
 where $\cU_{r_1}(p) \subset \sB^{\pm}$ (or $\cU_{r_1}(p) \subset \sB^T$ ) is
  $r_1 $ neighborhood 
 of $p$ as in \ref{ch5_rem1} .
  Now the proof for infinite strips follows directly from  
 the Theorem \ref{thm:bound_contr} and for finite strips 
 it follows from Remark \ref{rem:fin-str}. 
 More precisely it follows from Theorem \ref{thm:bound_contr} 
 that for $ \epsilon > 0 $ sufficiently small each 
 $ ( u^+,u^-) \in \sM^{\infty}_{\epsilon} $ satisfies 
 $$
 \| f_t ( u^{\pm} ) \|_{2,2} < r_1
 $$
 Thus $ u^{+} \in \cU_{r_1} (p) \subset \sB^+$ and $ u^- \in \cU_{r_1} (p) \subset \sB^- $.  
 $\hfill \square $ 
 \begin{theorem}\label{thm:bound_contr}
 There exist $ \epsilon > 0 $ and $ c > 0 $ 
 with the following significance. 
 Let $ u$ be an arbitrary half-infinite holomorphic curve as in 
 \ref{para:loc_setup} and let 
 $ v= f_t(u)\in H^2_{bc} ( \R^{\pm} \times [0,1] , \R^{2n} ) 
 $ be also as in  \ref{para:loc_setup} . If
 \begin{align}
  \| v(0)\|_{3/2} < \epsilon
 \end{align}
then 
 \begin{equation}\label{eq:main_inq}
 \|v\|_{W^{2,2} ( \R^{\pm} \times [0,1] ) }\leq c \| v(0)\|_{3/2}. 
 \end{equation}
 Here $ \| v \|_{3/2} $ denotes the norm of $ v $ in the interpolation 
 space $ H^{3/2}_{bc}$. 
\end{theorem}

\begin{proof}[Proof of Theorem \ref{thm:bound_contr}]
 We do the proof in the case of positive half-infinite strips. 
 The case of negative strips is analogous. 
 We first prove in Steps 1 and 2 that $ W^{1,2} $ norm of $ v $ is bounded 
 above by constant times $ H^1 $ norm at the boundary. 
 This follows by combining monotonicity results, 
 exponential decay and the fact that the curve 
 is contained in a local coordinate chart. \\
  \medskip
{\bf Step 1.} There exist positive constants $ c_1 $ and $ \epsilon_1 $ 
with the following significance. 
If $ \| v(0) \|_{H^1} < \epsilon_1 $ then 
$$
 \| v\|_{L^{\infty} ( \R^+ \times [0,1] ) }< c_1 \| v(0)\|_{H^1} .
$$
The claim follows from the following facts:
\begin{itemize}
 \item [i)] As $ v = f_t(u) $ we have that the $ L^{\infty} $ norm 
of a perturbed holomorphic curve $ v ,$
 $ \|v\|_{L^{\infty}(\R^+ \times [0,1] )}, $
and $ \sup\limits_{s,t} d(u(s,t), p ) $
are equivalent.
\item[ii)] In Theorems \ref{thm:mon1} and \ref{thm:mon2} we have proved that
the energy of $u, \; E(u) $ and $\sup\limits_{t\in [0,1]} d(u(0,t),p) $ 
control the distance $ \sup\limits_{s,t} d(u(s,t), p ) $.
In other words there 
exist $ \hbar $ and $ \delta $ such that the following holds:
If $ E(u) < \hbar $ and $ \sup\limits_t d(u(0,t), p) < \delta $
then 
  $$
  \sup\limits_{s,t} d(u(s,t), p) <  c \cdot  \sup\limits_t d ( u(0,t),p) 
  $$  
  \item[iii)] Both $ E(u) $ and $\sup\limits_t d(u(0,t), p) $ 
are controlled by $ \|v(0)\|_{H^1} $.
As $ u $ is contained in a local Darboux chart we have that 
\begin{align}
 \sqrt{E(u)} \leq c \| v(0)\|_{H^1([0,1])}.
\end{align}
and also $ \sup\limits_t d ( u(0,t), p) \leq c \|v(0)\|_{L^{\infty} } \leq c \| v(0)\|_{H^1} $. 
\end{itemize}
\medskip 

{\bf Step 2.} There exist $ \epsilon_2 $ and $ c_2 $ such that
every $ v $ as in the statement of the theorem with the 
property 
$ \| v(0)\|_{H^1} < \epsilon_2 $ satisfies 
$$ 
\| v\|_{W^{1,2} (\R^+ \times [0,1] )} \leq c_2\| v(0)\|_{H^1}  .
$$
\medskip 

As $ u $ decays exponentially it follows from Proposition \ref{pr:exp-dec} and 
Corollary \ref{ch1_cor0.6} that
 $ d(u(s,t), p) \leq c \sqrt{E(u)} e^{-\mu s } $ for all $ s\geq 1 $. 
 Thus it follows combining the facts from step 1 ,  $ i) $ and  $iii) $ that
 $$ 
 \| v\|_{L^{2}([1,+\infty) \times [0,1] ) }\leq c \| v(0)\|_{H^1}.
 $$
 On the compact piece $ [0,1] \times [0,1] $ we can estimate the $ L^ 2 $ 
norm of $ v $ by 
 its $ L^{\infty} $ norm, hence the claim of step 2 follows for 
 the $ L^2 $ norm of $v $ using the result of step 1). 
 The $ L^2 $ norm of $ \partial_s u $ and $ \partial_s v $ are equivalent, 
 thus using the fact $ iii) $ from Step 1 we have  
 $$ 
 \| \partial_s v \|_{L^2 ( \R^+ \times [0,1] )} 
 \leq c \sqrt{E(u)} \leq c \| v(0)\|_{H^1} .
 $$
As $ \partial_s v + J_t ( v ) (\partial_t v - X_t (v)) = 0 $ 
and $ X_t(0)=0 $ we have 
$$ 
\| \partial_t v \|_{L^2 ( \R^+ \times [0,1] )}
\leq c \Big ( \| v\|_{L^2 ( \R^+ \times [0,1] ) } + 
\| \partial_s v\|_{L^2( \R^+ \times [0,1] )} ( 1 + \|v \|_{L^{\infty}}) \Big ) .
$$
Thus $ \| v(0)\|_{H^1} $ controls also $ L^2 $ norm of $ \partial_t v $.  

In order to prove theorem \ref{thm:bound_contr} 
we still need to estimate 
$ \| \partial_s v \|_{W^{1,2}} $ and $ \| \partial_t v \|_{W^{1,2}} $. 
Notice that $ \partial_s v \in Ker ( D_v ) $,
where $ D_v $ is the linearized operator as in \eqref{eq:lindv}. \\

{\bf Step 3.}
Let $ H^1_{bc} ( \R^+ \times [0,1] ) $ be defined 
as in \eqref{eq:H1bc}.
Let $ D_0 $ be the linearization at $0 $, 
as in \ref{para:lin_op} and let $ 1 < p < 2 $. 
There exists a constant $ c_0> 0 $ such that 
every
$ 
\xi \in H^1_{bc} ( \R^+ \times [0,1], \R^{2n} )
\cap W^{1,p} (\R^+ \times [0,1] ) 
$ 
satisfies the following
\begin{equation}\label{eq:step2}
 \| \xi\|_{W^{1,p} ( \R^+ \times [0,1] )} 
 \leq c_0 \Big ( \| D_0 \xi\|_{L^p} + \| \xi(0)\|_{1/2} \Big ). 
\end{equation}

Define the space $ W^{1,p}_{bc} ( \R \times [0,1])$ as follows 
 \begin{align*}
  W^{1,p}_{bc} ( \R \times [0,1]) = \left  \{ \eta \in W^{1,p} ( \R \times [0,1], \R^{2n} ) \Big | 
 \eta ( s,i) \in \R^n \times \{0\} , \; i= 0,1 \right \} 
 \end{align*}
Then the following inequality follows 
 \begin{equation}\label{eq:st2_p1}
  \| \eta \|_{W^{1,p} ( \R \times [0,1] ) } \leq C \| D_0 \eta \|_{L^p (\R \times [0,1] ) }.
 \end{equation}
For the proof of \eqref{eq:st2_p1} have a look at the $ L^p $ 
estimates in section \ref{sec:lp}. 
Using the inequality \eqref{eq:st2_p1}, we are able to prove Step 3.
 Let $ \xi_0 = \xi(0,\cdot) 
\in [W^{1,2}_{bc}, L^2]_{1/2}$
and let $ \eta_0 (s,t) \in W^{1,2} ( \R^+\times [0,1] ) $ be the extension 
of $ \xi_0 (t) $, we can suppose w.l.o.g that $\eta_0 $ has compact support, 
thus $ \eta_0\in W^{1,p}( \R^+\times [0,1] ) $ and the following 
inequality holds 
\begin{align}\label{eq:st2_p2}
\| \eta_0\|_{W^{1,p} ( \R^+ \times [0,1] ) } 
\leq c \|\eta_0\|_{W^{1,2} ( \R^+ \times [0,1] )} \leq c \| \xi_0\|_{1/2} .
\end{align}
Let $\zeta_0(s,t)= \xi(s,t) - \eta_0(s,t) $. 
Then  $ \zeta_0 \in W^{1,p} ( \R^+ \times [0,1] ) $ and $ \zeta_0 (0, \cdot) = 0 .$
Extend $ \zeta_0 $  by $0 $ to the whole of $ \R \times [0,1] $, i.e. define
$ \zeta (s,t) $ by
\[
 \zeta(s,t) = \begin{cases}
                    \zeta_0 (s,t) , \;\;  s \geq 0 \\
                    0,\;\; s \leq 0 
             \end{cases}
\]
From the inequalities \eqref{eq:st2_p1} and \eqref{eq:st2_p2} 
we obtain 
\begin{align*}
 \| \xi\|_{W^{1,p} ( \R^+ \times [0,1] ) }&
 \leq \| \eta_0\|_{W^{1,p} ( \R^+ \times [0,1] )} + \| \zeta\|_{W^{1,p} ( \R^+ \times [0,1] )} \\
 &\leq \| \eta_0\|_{1,p} + C \| D_0 \zeta\|_{L^p} \\
 & \leq \| \eta_0\|_{1,p}  +  C\| D_0 \xi\|_{L^p} + C \| D_0 \eta_0 \|_{L^p} \\
 & \leq c_0 \Big ( \| D_0 \xi\|_{L^p ( \R^+ \times [0,1] ) } +  \| \xi(0)\|_{1/2} \Big )
\end{align*}
Here the last inequality follows as $ \| D_0 \eta_0\|_{L^p} \leq c \| \eta_0\|_{1,p} $ 
and using the inequality \eqref{eq:st2_p2}. 
\medskip

\noindent{\bf Step 4.} 
Let $ D_v $ be the linearized operator as in \eqref{eq:lindv}.  
There exists $ \delta> 0 $  and $ c_1 > 0 $ 
such that the following holds.
Assume that $ \| v (0)\|_{H^1 ( [0,1] ) } \leq \delta  $, 
then 
\begin{equation}\label{eq:step3}
 \| \xi\|_{W^{1,p} ( \R^+ \times [0,1] )} \leq c_1 
 \left ( \| D_v \xi \|_{L^p} + \| \xi(0)\|_{1/2} \right )
\end{equation}
holds for all 
$ \xi \in W^{1,2}_{bc} ( \R^+ \times [0,1] ) \cap W^{1,p} ( \R^+ \times [0,1] )$.\\
\medskip 
We will use the inequality \eqref{eq:step2} proved in Step 2
and we will prove that the norm of the 
difference $ \| (D_v- D_0) \xi \|_{L^p} $ 
is small provided that $ \| v(0)\|_{H^1([0,1])} $ 
is sufficiently small. Let 
\begin{align*}
 \| (D_v- D_0 ) \xi \|_{L^p} \leq  &\overbrace{\| ( J_t (v) - J_t (0)) \partial_t\xi\|_{L^p}}^{I}  \\
 &  + \overbrace{\| (J_t(v) dX_t(v) -J_t(0) dX_t(0)) \xi \|_{L^p}}^{II} \\ 
 &  +\overbrace{\| (dJ_t(v) \xi) J_t(v) \partial_s v\|_{L^p}}^{III}.     
\end{align*}
Obviously 
$ I, II \leq c \| v\|_{L^{\infty}} \| \xi\|_{W^{1,p}}$ 
and we have proved is Step 1
that $ \| v(0)\|_{H^1} $ controls the $ L^{\infty} $ norm of $v$.
Thus 
$$
I, II \leq c \| v(0)\|_{H^1} \| \xi \|_{W^{1,p}}.
$$
We estimate the $ III $ term in the following way. 
Remember that $ \| \partial_s v\|_{L^{\infty}} $ decays exponentially 
for $ s\geq 1 $, from Proposition \ref{pr:exp-dec} we have
\begin{align}\label{eq:pom3}
 \| \partial_s v \|_{L^{\infty} ( [s,\infty)\times[0,1])}\leq c \sqrt{E(u)} e^{-\mu s}
 \leq c \| v (0)\|_{H^1( [0,1] )} e^{-\mu s}, \;\; s \geq 1 
\end{align}
Let
$ \beta(s) $ be a smooth cut-off function with 
\[\beta(s) =\begin{cases}
             1, \;\; s\leq 1 \\
             0, \;\; s \geq 2
              \end{cases}
\]
then
\begin{align*}
 III &\leq c ( 1 +  \| v\|_{L^{\infty}}) \|( dJ_t(v)\xi) \partial_s v\|_{L^p( \R^+ \times [0,1] )} \\
 & \leq c \| (d J_t(v)\xi)  ( \beta \partial_s  v + 
 ( 1- \beta) \partial_s v) \|_{L^p( \R^+ \times [0,1] )} \\
 & \leq c \left ( \| \beta ( dJ_t(v)\xi) \partial_s v\|_{L^p} 
 + \|( d J_t(v)\xi) ( 1- \beta) \partial_s v) \|_{L^p}\right ) \\
 & \leq c  \left(  \| ( d J_t(v)\xi) \partial_s v\|_{L^p([0,2]\times [0,1] )} +
  c' ( 1 + \| v \|_{L^{\infty}} ) \| v (0)\|_{H^1( [0,1] )}
 \| \xi\|_{L^p} \right) \\ 
 &\leq c ( 1 +\| v\|_{L^{\infty}})  \left ( \| \xi\|_{L^q ( [0,2] \times [0,1] )}
 \| \partial_s v \|_{L^2 ( [0,2]\times [0,1] )}  + \| v (0)\|_{H^1( [0,1] )}
 \| \xi\|_{L^p} \right ) \\ 
 & \leq c ( 1 +\| v\|_{L^{\infty}})
 \| v(0)\|_{H^1} \| \xi\|_{W^{1,p} ( \R^+ \times [0,1] ) } 
\end{align*}
Here the second inequality follows from Step 1 and the assumption 
$ \| v(0)\|_{H^1} < \delta $.
In the penultimate inequality we have $ q = \frac{2p}{2-p} $ and the 
inequality follows from H\"older inequality whereas the last inequality 
is a corollary of the Sobolev embedding
$ W^{1,p} ( [0,2] \times [0,1] ) \hookrightarrow L^q ( [0,2] \times [0,1]) $
and Step 2. 
Thus for sufficiently small $ \delta $ and $ \| v(0)\|_{H^1} < \delta $ 
we have 
\begin{align}\label{eq:st3_pom}
 \| (D_v- D_0) \xi\|_{W^{1,p}( \R^+ \times [0,1] )} 
 \leq \frac{1}{2c_0} \| \xi\|_{W^{1,p} },
\end{align}
 where $ c_0 $ is the constant in \eqref{eq:step2}. 
 Substituting  \eqref{eq:st3_pom} in  \eqref{eq:step2}
we obtain
\begin{align}
 \| \xi\|_{W^{1,p} ( \R^+ \times [0,1] )} & \leq c_0 \left ( \| D_0 \xi\|_{L^p( \R^+ \times [0,1] )} + \| \xi(0)\|_{1/2} \right ) \notag  \\ 
 & \leq  c_0 \left ( \| D_v \xi \|_{L^p } + 
 \| (D_v - D_0)\xi \|_{L^p} +  \|\xi(0)\|_{1/2} \right ) \notag \\
 & \leq 2c_0 \left ( \| D_v \xi \|_{L^p } +  \|\xi(0)\|_{1/2} \right ).
\end{align}
Thus, we have proved Step 4. \\
\medskip 
\noindent {\bf Step 5.} 
There exist $ \epsilon_0> 0 $ and $ c_0 > 0 $ 
such that the following holds. 
If $ \| v(0) \|_{3/2} < \epsilon_0 $ then  
\begin{align}\label{eq:step5}
 \| \partial_s v \|_{W^{1,2} ( \R^+ \times [0,1] ) } < c_0 \| v(0)\|_{3/2} \notag \\
 \| \partial_t v \|_{W^{1,2} ( \R^+ \times [0,1] ) } < c_0 \| v(0)\|_{3/2} 
\end{align}
\begin{proof}
In Lemma \ref{thm:main_inq} we have proved
that the following holds for all
$ \xi \in W^{1,2}_{bc} ( \R^+ \times [0,1] ) $ 
\begin{align}\label{eq:vazna_inq}
 \| \xi\|_{W^{1,2}( \R^+ \times [0,1] ) } \leq 
 c \Big ( \| D_0 \xi \|_{L^2( \R^+ \times [0,1] )} +  \| \xi (0)\|_{1/2} \Big ). 
\end{align}
where $ D_0 $ is the linearization as in \ref{para:lin_op}. 
We want to estimate $ \| \partial_s v\|_{1,2} $. 
Notice that $ \partial_s v $ 
is an element of the kernel of the operator $ D_v .$ 
In order to estimate $ \| \partial_s v \|_{1,2} $ we 
need to estimate the difference 
$ \| (D_v - D_0) \partial_s v \|_{L^2( \R^+ \times [0,1])} $. 
We have 
\begin{align}
 \| (D_v - D_0) \partial_s v \|_{L^2 ( \R^+ \times [0,1] )} 
 \leq & \overbrace{\| ( J_t(v) - J_t(0) )\partial_s\partial_t v \|_{L^2( \R^+ \times [0,1])}}^{I}\notag   \\ 
 & + \overbrace{\| ( J_t (v) dX_t(v) - J_t(0) dX_t(0)) \partial_s v\|_{L^2}}^{II} \notag  \\
 & + \overbrace{\| ( dJ_t(v)\partial_s v) J_t(v) \partial_s v\|_{L^2 ( \R^+ \times [0,1] )}}^{III} 
\end{align}
Obviously it follows from Steps 1 and 2 that
\begin{equation}\label{eq:estI&II}
\begin{split}
I &\leq c \| v\|_{L^{\infty} } \| \partial_s v\|_{W^{1,2} ( \R^+\times [0,1] ) }
\leq c  \| v(0)\|_{H^1}\| \partial_s v\|_{W^{1,2}}\\ 
II &\leq c \| v \|_{L^{\infty} } \|\partial_s v \|_{L^2 ( \R^+ \times [0,1] )} \leq c \| v(0)\|_{H^1}^2 
\end{split}
\end{equation}
whereas  
\begin{align*}
 III \leq c  ( 1 +\| v\|_{L^{\infty} } )^2 \| \partial_s v \|^2 _{L^4 ( \R^+ \times [0,1] ) }. 
\end{align*}
Let $ \beta(s) $ be a smooth cut-off function as 
in Step 4, then 
\begin{align*}
 \| ( 1 - \beta) \partial_s v \|^2_{L^4( \R^+ \times [0,1] )} 
 \leq c\| v(0)\|_{H^1([0,1] )}^2,
\end{align*}
as $ \abs{ \partial_s v(s,t)} \leq c \| v(0)\|_{H^1( [0,1] ) } e^{-\mu s } $
for $ s\geq 1 $. 
On the other hand, from the Sobolev embedding,
$ W^{1, 4/3} ( [0,2] \times [0,1]  ) \subset L^4  ( [0,2] \times [0,1] )$,
we have 
$$
\|\beta \partial_s v \|_{L^4} \leq c \| \beta\partial_s v \|_{W^{1,4/3}} ,
$$
for some positive constant $ c $. 
Suppose that $ \epsilon_0$ is chosen such that $ \| v(0)\|_{H^1} < \delta $,
where $ \delta $ is the constant in the claim of Step 4. 
Substituting $ \beta \partial_s v $ in the 
inequality \eqref{eq:step3}  with $ p= 4/3$ we obtain
\begin{align*}
 \| \beta \partial_s v \|_{W^{1,4/3} ( \R^+ \times [0,1] ) } 
 &\leq c \Big ( \| D_v ( \beta \partial_s v ) \|_{L^{4/3}} + \| \partial_s v (0)\|_{1/2} \Big ) \\
 &\leq c \Big ( \| \beta D_v ( \partial_s v ) \|_{L^{4/3}} +  \| \dot{\beta} \partial_s v \|_{L^{4/3}} + 
 \| \partial_s v (0)\|_{1/2} \Big )\\
  &\leq c \Big ( \| \partial_s v \|_{L^{2}( [0,2] \times [0,1] )} + \| \partial_s v(0) \|_{1/2} \Big ) 
\end{align*}
Now we can estimate the $ III $ term with
\begin{align*}
 III  \leq c &  ( 1 +\| v\|_{L^{\infty}})^2 
 \cdot \| ( \beta + ( 1-\beta)) \partial_s v \|^2_{L^4} 
  \\   \leq c  &( 1 +  \| v \|_{L^{\infty}})^2 \Big ( \| \beta\partial_s v \|^2_{L^4}   + \| ( 1-\beta)\partial_s v\|_{L^4} \Big ) \\
  \leq c &  ( 1 + \| v\|_{L^{\infty}})^2 \Big ( \| \partial_s v \|_{L^{2}( [0,2] \times [0,1] )}^2 + 
  \| \partial_s v(0) \|^2_{1/2} + \| v(0)\|_{H^1([0,1] )}^2 \Big )\\
 \leq c  & ( 1 + \| v\|_{L^{\infty}})^2 \Big ( \| v(0)\|_{H^1}^2 + \| \partial_s v(0) \|^2_{1/2}  \Big ). 
\end{align*}
Finally we estimate $\| \partial_s v(0) \|_{1/2}$. As $ v $ is a solution of 
the equation \eqref{eq:per_hol} we have: 
\begin{equation}\label{eq:parus}
\begin{split}
 \| \partial_s v (0) \|_{1/2} =&\| J_t ( v(0)) ( \partial_t v(0) - X_t (v(0))) \|_{1/2} 
\\
 \leq & \overbrace{\| J_t(v(0)) \partial_t v(0) \|_{1/2}}^a +
 \overbrace{\|J_t(v(0))X_t( v(0)) \|_{1/2}}^b 
\end{split}
\end{equation}
To estimate the terms $ a$ and $ b $ we shall use the fact that
\begin{align*}
 \| fg \|_{1/2} \leq c \| f\|_{H^1} \cdot \| g\|_{1/2}
\end{align*}
for $ f \in H^1( [0,1] , \R) $ and $ g \in [W^{1,2}_{bc}, L^2]_{1/2}$. 
This holds 
as multiplication by $f $ is a continuous linear map on both
$ W^{1,2}_{bc} $ and $ L^2 $, hence it is also continuous on $ [W^{1,2}_{bc}, L^2]_{1/2} $.  
Thus, as $ X_t (0) = 0 $ it follows that 
\begin{align*}
& a \leq  c ( 1 + \| v(0)\|_{H^1} ) \| v(0)\|_{3/2} \\
& b \leq c ( 1 + \|v(0)\|_{H^1} ) \| v(0)\|_{H^1}
\end{align*}
As we can w.l.o.g. assume that $ \|v(0)\|_{H^1} $ is small 
we obtain from the previous inequality and \eqref{eq:parus}
\begin{equation}\label{eq:pom4}
 \| \partial_s v (0)\|_{1/2} \leq c \| v(0)\|_{3/2}.
\end{equation}
Substituting $ \p_s v $ in the inequality \eqref{eq:vazna_inq} and using 
the estimates for $ I,II $ and $ III $ term we obtain 
\begin{align}\label{eq:est_dsv}
 \| \partial_s v \|_{W^{1,2} ( \R^+ \times [0,1] )}
  &\leq c \Big ( \| D_0 \partial_s v \|_{L^2} + \| \partial_s v (0)\|_{1/2} \Big ) \notag \\
 & \leq c \Big ( \| D_v \partial_s v \|_{L^2} + \| ( D_v - D_0)\partial_s v\|_{L^2} +
 \| v(0)\|_{3/2} \Big )\notag  \\
 & \leq c \Big ( I + II + III + \| v(0)\|_{3/2} \Big ) \notag \\
 & \leq c  \Big ( ( 1 + \| v\|_{L^{\infty}})^2 \| v(0)\|_{H^1} + \| v(0)\|_{3/2} \Big ). 
\end{align}
Here the second inequality follows from \eqref{eq:pom4}.
From the inequality \eqref{eq:est_dsv} 
and Step 1 follows 
the inequality \eqref{eq:step5} for $ \partial_s v $. 
As $ v $ is the solution of \eqref{eq:per_hol}, we have 
\begin{align*}
 \| \partial_t v \|_{W^{1,2} ( \R^+\times [0,1] ) } \leq 
 \| X_t( v)\|_{W^{1,2}( \R^+\times [0,1] )  } +
 \| J_t(v)\partial_s v \|_{W^{1,2} ( \R^+ \times [0,1] ) }.
\end{align*}
As $ X_t(0) = 0 $ and $ X_t $ is smooth it 
follows that $ \| X_t(v)\|_{1,2} \leq c  ( 1 + \| v \|_{L^{\infty}} )\| v\|_{1,2} $. 
Also,
\begin{align*}
 \| J_t(v) \partial_s v \|_{W^{1,2}}
 \leq  &\| J_t(v)\partial_s v \|_{L^2} +  \| \partial_t J_t(v) \partial_s v \|_{L^2 } \\
 & +
 \| (dJ_t(v) \partial_t v )\partial_s v \|_{L^2}   +\| ( dJ_t(v) \partial_s v ) \partial_s v \|_{L^2}  
   \\ & + \| J_t(v) \partial^2_s v \|_{L^2} + 
 \| J_t ( v ) \partial_s \partial_t v \|_{L^2 } \\
 \leq & c ( 1 + \| v\|_{L^{\infty} }) \| \partial_s v \|_{W^{1,2} ( \R^+ \times [0,1] )} +  \| (dJ_t(v) \partial_t v )\partial_s v \|_{L^2} \\
  & + \| ( dJ_t(v) \partial_s v ) \partial_s v \|_{L^2}  
\end{align*}
Notice that 
$$
\| (dJ_t(v) \partial_t v )\partial_s v \|_{L^2}
\leq \| (dJ_t(v) X_t(v)) \partial_s v \|_{L^2}  + \|  ( dJ_t(v) J_t(v) \partial_s v ) \partial_s v \|_{L^2} 
.
$$
Using the same type of estimates as for the  $ III $ term we obtain
the analog of \eqref{eq:est_dsv} for $ \partial_t v $, i.e. 
\begin{align}
 \|\partial_t v\|_{W^{1,2} (\R^+ \times [0,1] )} \leq  c  \Big ( ( 1 +\| v\|_{L^{\infty}} )^2
 \| v(0)\|_{H^1} + \| v(0)\|_{3/2} \Big ).
\end{align}

\end{proof}
The steps 1-5 prove the theorem \ref{thm:bound_contr}.
\end{proof}
\begin{remark}\label{rem:fin-str}\rm
 The analogous statement as in Theorem \ref{thm:bound_contr} holds for 
 finite strips. Namely, 
 there exist $ \epsilon > 0 , c> 0 $ 
 such that the following holds for every $ T $.      
 If $ v =f_t(u) \in H^2_{bc}( [-T,T] \times [0,1], \R^{2n} ), $
  as in local setup \ref{para:loc_setup} satisfies
 $$
 \|v(-T)\|_{3/2} + \| v(T)\|_{3/2} < \epsilon ,
 $$ 
 then 
  $$ 
  \| v\|_{W^{2,2} ( [-T,T]\times [0,1])}  <  c ( \|v(-T)\|_{3/2} + \| v(T)\|_{3/2} ) 
  $$
  The proof is analog to the proof of proposition \ref{thm:bound_contr} 
  and we shall not repeat it. 
\end{remark}

\end{PARA}
\section{Convergence theorem}\label{SEC:conv} 
In this section we prove Theorem \ref{thm:main_thm3}.

\begin{PARA}[{\bf Hardy submanifolds}]\label{para:hardy}\rm 
In Definition \ref{def:emb_mfld} we have introduced 
manifolds $ \sM^{\infty}_{\epsilon} $ and $ \sM^T_{\epsilon} $ and we 
have proved in \ref{para:pr_thm2} that these manifolds are embedded submanifolds 
of the Hilbert manifolds of paths $ \sP^{3/2} \times \sP^{3/2} $. Denote with 
$ \sW^{\infty}_{\epsilon} $ and $ \sW^T_{\epsilon} 
$ the images of $ \sM^{\infty}_{\epsilon} $ and 
$ \sM^T_{\epsilon} $  via the maps $ i^{\infty} $ and $ i^{T} $. 
\begin{align}\label{eq:maps_i}
 i^T : \sM^T_{\epsilon} \rightarrow \sW^T_{\epsilon} 
 \subset \sP^{3/2} \times \sP^{3/2}, \;\; i^T ( u) = ( u(-T, \cdot), u(T,\cdot)) \notag \\
 i^{\infty} : \sM^{\infty}_{\epsilon} \rightarrow \sW^{\infty}_{\epsilon}\subset \sP^{3/2} \times \sP^{3/2} , \;\; ( u^-, u^+ ) \mapsto (u^- (0, \cdot), u^+ (0, \cdot)) 
\end{align}

These are {\bf Hardy} submanifolds of the path space 
and we can think of them as of those paths that extend holomorphically 
to the corresponding strips. 

\begin{align*}
 \sW_{\epsilon}^{\infty} = \Big \{ 
 (\gamma^-, \gamma^+)\in \cU_{\epsilon}(p) \times \cU_{\epsilon}(p) \; 
 \Big |& \;\; \exists  (u^+, u^- )\in \sM^{\infty}_{\epsilon}, \;\;
   u^{\pm} (0,\cdot) = \gamma^{\mp} (\cdot) \Big \} \\
 = \Big \{ (\gamma^-, \gamma^+)\in \cU_{\epsilon}(p) \times \cU_{\epsilon}(p) 
 \;\Big| & \exists u^{\pm} \in H^2_{loc} ( \R^{\pm} \times [0,1] , N ) \\ & 
 u(s,i) \in L_i, \;\; i=0,1, \; \forall s \\
 & \overline{\partial}_{J_t} u^{\pm} =0, \;\; E(u^{\pm} ) < \hbar \\ & 
 u^{\pm} (0, \cdot ) = \gamma^{\mp} (\cdot) \Big \} 
 \end{align*}
and similarly 
\begin{align*}
 \sW_{\epsilon}^{T} = \Big \{ 
 (\gamma^-, \gamma^+)\in \cU_{\epsilon}(p) \times \cU_{\epsilon}(p) \; 
 \Big |& \;\; \exists  u \in \sM^{T}_{\epsilon} ,
 \;\;  u({\pm} T,\cdot) = \gamma^{\pm} (\cdot) \Big \}.
 \end{align*}
 We prove that 
 $ \sW^T_{\epsilon} $ converge to $\sW^{\infty}_{\epsilon} $ 
 in $ C^1 $ topology. 
 
\end{PARA}
\begin{PARA}[{\bf Notion of convergence of Hilbert manifolds}]\label{para:conv}\rm
We first explain the notion of convergence of certain 
Hilbert manifolds.
Let $ W^T $ and $ W^{\infty} $ be some Hilbert manifolds.
 There are different ways one can think about $ C^1 $
 convergence $ W^T \rightarrow W^{\infty} $. 
 One way is to think of $W^T $ as of the sections
 of the normal bundle of $ W^{\infty} $ and to 
 prove that these sections converge to the zero section.
 Another way is to think of representation of 
 $ W^T $ and $ W^{\infty} $ in a local chart and to prove that locally
 $ W^{T} $ converges to $ W^{\infty} $. 
 We formulate this more precisely in the next definition.
  
\begin{definition}\label{def:conv}
 Let $ P $ be a Hilbert manifold modeled 
 on Hilbert space $ H $ and let $ W^{\infty} $ 
 and $ W^T $ be its submanifolds. 
 We say that $ W^T \rightarrow W^{\infty} , \; T\rightarrow + \infty $ 
 if for all $ x_0 \in W^{\infty} $ there exist the following:
 
\begin{itemize}
 \item[1)] A splitting $ H = H_0 \oplus H_1 $, 
 where $ H_0 $ and $ H_1 $ are closed subspaces of $H$.  
  \item[2)] A local coordinate chart $ \phi: U \rightarrow H $, 
  where $ U \subset P $ is an open neighborhood 
  of $ x_0 $, such that $ \phi(x_0) = 0 $ and
  $$ \phi(U) = \{ \xi_0 + \xi_1 : \;\; \xi_0 \in U_0 , \; \xi_1 \in U_1 \} ,$$
  where $ U_0\subset H_0 $ and $ U_1\subset H_1 $
  are open neighborhoods of $ 0 $. 
\item[3)] A smooth map $ f^{\infty} : U_0 \rightarrow H^1$ such that
  $$ \phi( U \cap W^{\infty} ) =
  \{ \xi + f^{\infty} ( \xi) : \;\; \xi\in U_0 \} .$$  
 A family of smooth maps $ f^T: U_0 \rightarrow H^1 $
 such that  for all $ T\geq T_0 $ 
 $$ \phi(U \cap W^T) = \{ \xi + f^T (\xi ): \; \xi\in U_0 \} .$$
\item[4)] The limits $\lim\limits_{T\rightarrow \infty} \|f^T - f^{\infty} \|_{C^1} = 0 $
\end{itemize}

 \end{definition}
 
  In the proof of the Theorem \ref{thm:main_thm3} we shall 
 often use the inverse function theorem, 
 we state it in the form that we shall use in the proof. 
 
 \begin{lemma}[{\bf Inverse function theorem}] \label{chlem1}
 
 Let X and Y be Banach spaces, $B_{r}(x_0) \subset X $ 
 a ball of radius $r$ and $f: B_{r}(x_0)\rightarrow Y $ 
 a continuously differentiable function 
 that satisfies the following:
 
\begin{itemize}
\item[1)] $df(x_0) $ is bijective and the operator norm 
$\|df(x_0)^{-1}\| \leq c $.
\item[2)]  $\|df(x)- df(x_0)\|\leq \frac{1}{2c}$, for all $x\in B_{r}(x_0) $.
\end{itemize}

 Then $f: B_{r}(x_0) \stackrel{diff}{\longrightarrow}f(B_{r}(x_0))\subset  Y $ 
 and $B_{r / 2c}(f(x_0))\subset f(B_{r}(x_0)) $.  
\end{lemma}
 
\end{PARA}

\begin{proof}[{\bf Proof of Theorem \ref{thm:main_thm3}}]
Let  $ \sW^{\infty}_{\epsilon} $ and $ \sW^T_{\epsilon} $ 
be as in \ref{para:hardy}. Remember that they are embedded 
submanifolds of the Hilbert manifold 
of paths $ \sP^{3/2} \times \sP^{3/2} $. 
Let  
\begin{align}\label{man_loc}
 &\cW^{\infty}_{\epsilon} = (\Phi_p \times \Phi_p)( \sW^{\infty}_{\epsilon})\notag \\
 &\cW^{T}_{\epsilon} = (\Phi_p \times \Phi_p) ( \sW^T_{\epsilon} ) 
 \end{align}
 where 
 $ 
 \Phi_p : \sU_p \rightarrow \sW_p \subset H^{3/2}_{bc} = E
 $
 is a local chart as in  \ref{ch5_rem1}. 
 We devide the proof in two parts.
 In the first part, i.e. in \ref{para:construction},
 we shall construct maps
 $$ 
f^{\infty}, f^T : B_{\rho_0}(0) \rightarrow  E , 
$$ 
which are diffeomorphisms onto their images and  
$ B_{\rho_0}(0),$ 
is an open ball of radius $ \rho_0 $ centered at 
$ 0 $ in $ E $. The set
$ \cW^{\infty}_{\epsilon}  $ is an open subset of 
$ \text{ graph } ( f^{\infty}) $ and 
also $ \cW^T_{\epsilon} $ is an open subset of $\text{graph} ( f^T) $. 
In the  second part i.e. in \ref{para:convergence}
we prove the convergence 
$$ 
f^T \stackrel{C^1}{\rightarrow} f^{\infty}, 
\;\;T \rightarrow \infty 
$$ 
on some smaller neighborhood $B_{\rho}(0)\subset B_{\rho_0}(0) $.
 And finally the convergence of maps 
$ f^T \rightarrow f^{\infty} $ will imply the convergence of submanifolds
$\cW^{T}_{\epsilon}\rightarrow \cW^{\infty}_{\epsilon} $.  
\end{proof}

\begin{PARA}[{\bf Construction of the maps $ f^{T} $ and $ f^{\infty} $ }]\label{para:construction}\rm
 
Let $ D_0 $ be the linearization at $ 0 $ 
as in \ref{para:lin_op}. 
Then $ D_0 = \partial_s + A $ 
and the operator 
$$ 
A : H^1_{bc} ( [0,1] ) \rightarrow L^2 ( [0,1] ), 
\; A \xi =  J_t(0)\partial_t \xi -J_t(0)dX_t(0)\xi 
$$ 
satisfies ~(HA) in \ref{para:bij_lin} as it is self adjoint 
with resepct to the scalar product given by \eqref{eq:sc_prod}. 
Let $ E= H^{3/2}_{bc} $, $ E^{\pm} $ be as in Remark \ref{rem:lines}, 
corresponding to 
the operator $A$, and let $ \pi^{\pm} $ 
be as in \eqref{eq:proj}.  
Abbreviate
$$ Z^{\pm} = \R^{\pm}\times [0,1], \qquad Z^T = [-T, T]\times [0,1] . $$
Let $ H^i_{bc} ( Z^{\pm} ) $ and $ H^i_{bc}( Z^T ), i=1,2 $ be defined as in \eqref{eq:H1bc}
and \eqref{eq:H2bc}.
We define maps $ \mathcal{F}^{\infty} $ and $ \mathcal{F}^ T$ as follows 
\begin{align}\label{eq:maps_f}
 & \mathcal{F}^{\infty} : H^{2}_{bc} ( Z^+) \times H^{2}_{bc} ( Z^- ) 
\longrightarrow H^{1}_{bc} ( Z^+ )
\times H^{1}_{bc} (Z^- ) \times E^+\cap E\times E^-\cap E \notag \\
 &\mathcal{F}^{\infty} (u^-, u^+)= 
 \Big (\overline{\partial}_{J_t,X_t}  u^- ,
 \overline{\partial}_{J_t, X_t}  u^+,  \pi^+(u^-(0,\cdot)),  \pi^-(u^+(0,\cdot)) \Big)
 \end{align}
 \begin{align}\label{eq:mapft}
 & \mathcal{F}^ T :H^{2}_{bc} ( Z^T) \longrightarrow H^{1}_{bc} ( Z^T )
 \times E^+ \times E^-  \notag \\
& \mathcal{F}^ T ( u) = 
\Big ( \overline{\partial}_{J_t, X_t}u, \pi^+(u(-T,\cdot)), \pi^-(u(T,\cdot)) \Big ).
\end{align}
Here $ \overline{\partial}_{J_t,X_t} $ is as in \eqref{eq:per_hol}. 

{\bf Step 1: } 
\emph{ Let $\cF^{\infty} $ and $ \cF^T $ be as in \eqref{eq:maps_f} and 
\eqref{eq:mapft} 
There exist $ c_0> 0 $ such that
the maps $ d\mathcal{F}^{\infty}(0) $ and 
$ d\mathcal{F}^ T(0)$  
are bijective and have uniformly bounded inverses }

\begin{align}\label{eq:est_inv}
  \| d\mathcal{F}^{\infty} (0)^{-1}\| \leq c_0 , \;\;\; \|d\mathcal{F}^T(0)^{-1}\| \leq c_0 ,
\end{align}
\begin{proof}
Notice that the function
 $ \mathcal{F}^{\infty} = \mathcal{F}^+ \times \mathcal{F}^-, $  
 where 
\begin{align*}
 &\mathcal{F}^{\pm} : H^{2}_{bc} ( Z^{\pm} )
 \rightarrow H^{1}_{bc} ( Z^{\pm}) \times E^{\pm}\cap E  \\
& \mathcal{F}^{\pm} ( u) = 
(\overline{ \partial}_{J_t, X_t} u, \pi^{\pm} ( u(0, \cdot) )  
\end{align*}
 The linearizations of 
 $ \cF^{\infty} $ and $ \cF^T $ at $0$ are given by:
\begin{align*}
 &d\mathcal{F}^{\infty} (0)(\hat{u}^-, \hat{u}^+ ) =
 ( D_0 \hat{u}^-, D_0 \hat{u}^+ , 
 \pi^+ (\hat{u}^-(0, \cdot)), \pi^- (\hat{u}^+ (0, \cdot))\\
 &d\mathcal{F}^T(0) ( \hat{u} ) = \Big ( D_0 \hat{u}, 
 \pi^+ (\hat{u}(-T, \cdot)), \pi^- ( \hat{u} (T, \cdot)) \Big ).
\end{align*}
The fact that the 
maps $ d\cF^{\infty}(0) $ and $ d\cF^T(0) $
are bijective and the inequality \eqref{eq:est_inv} follow 
directly from Theorem \ref{thm:main_inq}
\end{proof}

{\bf Step 2: [Quadratic estimates]} 
\emph{Let $ c_0 $ be  the constant as in Step 1. 
There exist $ r_0>0 $ such that for all $ (u^-, u^+) 
\in H^{2}_{bc} ( Z^+) \times H^{2}_{bc} ( Z^- ) $ and all 
$ u\in   H^{2}_{bc} ( Z^T),$ which satisfy 
 \begin{align}\label{eq:w22est}
  \|u^-\|_{2,2} + \|u^+\|_{2,2} < r_0, \;\; \|u\|_{2,2} < r_0
 \end{align}
the following holds.}
\begin{align}\label{eq:dif_der}
 \| d\cF^{\infty} ( u^-, u^+) - d\cF^{\infty} (0,0) \|
  \leq \frac{1}{2c_0},\;\; \;\; \| d\cF^{T}(u) -  d\cF^{T}(0) \|
  \leq \frac{1}{2c_0},
\end{align}

\begin{proof}
Notice that 
\begin{align*}
 \|d\mathcal{F}^T ( u) (\hat{u} )   - d\mathcal{F}^T ( 0) (\hat{u} ) \|
  = &  \| (D_u- D_0) \hat{u}\|_{1,2}\\
\| (d\cF^{\infty}( u^+, u^-) -  d\cF^{\infty}(0))(\hat{u}^+,\hat{u}^-)\| = & \|(D_{u^+} -D_0)(\hat{u}^+)\|_{1,2}  \\
& 
+ \|( D_{u^-} - D_0)(\hat{u}^-)\|_{1,2},
\end{align*}
Hence, we need to estimate the difference $ \| (D_u- D_0) \hat{u}\|_{1,2} $, for 
$ u \in H^2_{bc} ( I \times [0,1] )$, where $ D_u $ is the linearized operator
as in \eqref{eq:lindv}, i.e. 
$$ 
D_u \hat{u} = \partial_s \hat{u} + J_t(u) (\partial_t \hat{u} - dX_t(u) \hat{u}) +
( dJ_t(u) \hat{u} ) (\partial_t u-  X_t(u)).  
$$ 
We prove that for all 
$ u, \; \hat{u} \in H^2_{bc} ( \R^+ \times [0,1] ) $ 
with 
$ \|u\|_{W^{2,2} ( \R^+ \times [0,1] ) }\leq 1 $ 
we have 
\begin{equation}\label{eq_quad_est}
 \| (D_u - D_0) \hat{u}\|_{1,2} \leq c \| u\|_{2,2} \|\hat{u}\|_{2,2}
\end{equation}
for some positive constant $ c $. 
The same inequality holds for finite strips and the proof
is analogous.\\

\indent Let $ \beta_i: \R\rightarrow [0,1], 
\;\; i\geq 0 $ be 
smooth cut-off functions, with the properties
\begin{itemize}
 \item [a)] 
 $\; \text{sup}(\beta_i) \subset [i,i+2], \; i\geq 1, $  
 and $ \beta_0(s) = 0$ for $ s\geq 2$
 
 \item[b)] 
 $ \sum\limits_{i=0}^{+\infty} \beta_i(s) = 1 $ 
 and 
 $ \|\dot{\beta}_i (s)\|_{C^0} \leq 1 , \;i \geq 0$. 
\end{itemize}
Let 
$ \hat{u} = \sum\limits_{i=0}^{\infty} \beta_i 
\hat{u} = \sum\limits_{i=0}^{+\infty} \hat{u}_i $.
We have 
\begin{align*}
 \| ( D_u - D_0) (\sum\limits_i \hat{u}_i) \|_{W^{1,2} ( [0, +\infty) \times [0,1] )} 
&  = \| \sum\limits_{i} ( D_u - D_0) \hat{u}_i \|_{W^{1,2} ( [0,+\infty) \times [0,1])} \\
 & \leq \sum\limits_{i} \| ( D_u - D_0) \hat{u}_i \|_{W^{1,2} ( [i,i+2]\times [0,1])} 
\end{align*}
Suppose that the inequality 
\begin{equation}\label{eq_Si_est}
 \| ( D_u - D_0) \hat{u}_i \|_{W^{1,2}([i, i+2] \times[0,1])} 
 \leq c \|u\|_{W^{2,2}( [ i, i+2] \times [0,1] ) }
 \cdot \|\hat{u}_i\|_{W^{2,2}( [i,i+2] \times [0,1] )} .
\end{equation}
holds for some positive constant $ c $ and for all $i\geq 0$. 
With this assumption we have
\begin{align}
 \| ( D_u - D_0) \hat{u} \|_{2,2} &
 \leq \sum\limits_{i} \| ( D_u - D_0) \hat{u}_i \|_{W^{1,2} ( [i,i+2]\times [0,1])} \notag \\
 & \leq c \sum\limits_i  \|u\|_{W^{2,2}( [ i, i+2] \times [0,1] ) }
 \cdot \|\hat{u}_i\|_{W^{2,2}( [i,i+2] \times [0,1] )}  \notag \\ 
 & \leq  c\sqrt{\sum\limits_{i\geq 0}\|u\|^2_{W^{2,2}( [ i, i+2] \times [0,1] ) }}
 \sqrt{\sum\limits_{i\geq 0}\|\hat{u}_i\|^2_{W^{2,2}( [i,i+2] \times [0,1] )}} \notag \\
 & \leq c' \| u\|_{W^{2,2} ( [0, + \infty) \times [0,1] ) } 
 \| \hat{u}\|_{W^{2,2} ( [0, + \infty) \times [0,1] ) }
\end{align}
It remains to prove the inequality \eqref{eq_Si_est}
under the assumption 
$ \| u\|_{W^{2,2} ( \R^+ \times [0,1] )} \leq 1 $. 
Let $ \Omega = [ i, i+2] \times [0,1] $, then

\begin{align}\label{eq:dif_sob}
 \| (D_u  - D_0) \hat{u}\|_{W^{1,2}( \Omega)} \leq &
 \overbrace{\|(J_t(u) - J_t(0))\partial_t \hat{u}\|_{1,2}}^{I}  \notag \\
 & + \overbrace{\|(J_t(u)dX_t(u) - J_t(0)dX_t(0))\hat{u}\|_{1,2}}^{II} \notag\\ 
 & + \overbrace{\|(dJ_t(u)\hat{u}) \partial_t u\|_{1,2}}^{III} + 
 \overbrace{\|(dJ_t(u)\hat{u})X_t(u)\|_{1,2}}^{IV}
\end{align}

 We shall show how to estimate terms $ I $ and $ III $, 
 and analogously one can estimate the $ II $ and $ IV $ 
 term of \eqref{eq:dif_sob}. 
 
 \begin{align}
  I \leq & \overbrace{\|( J_t(u) - J_t(0)) \partial_t \hat{u} \|_{L^2}}^{A} + 
  \overbrace{\| (\partial_tJ_t(u) - \partial_tJ_t(0)) \hat{u} \|_{L^2}}^{B} 
   +  \overbrace{\|( dJ_t(u)\partial_t u) \partial_t \hat{u} \|_{L^2}}^{C} 
  \notag \\ & +
   \overbrace{\|( J_t(u) - J_t(0)) \partial^2_t \hat{u} \|_{L^2}}^{D} 
  + \overbrace{\|( dJ_t(u) \partial_s u) \partial_t \hat{u} \|_{L^2}}^{E} + 
  \overbrace{\|( J_t(u) - J_t(0)) \partial_t\partial_s \hat{u} \|_{L^2}}^{F}
\end{align}

The terms $ A,B $ $D$ and $F $ can be estimated by 
$$ 
c \| u\|_{L^{\infty} ( \Omega)} \|\hat{u}\|_{W^{2,2} (\Omega)} 
\leq c \| u\|_{W^{2,2} (\Omega)} \| \hat{u}\|_{W^{2,2} ( \Omega ) } .
$$
whereas the terms $ C $ and $ E $ can be estimated as follows 
\begin{align*}
 C &\stackrel{cs}{\leq} c ( 1 + \| u\|_{L^{\infty} } )\|\partial_t u \|_{L^4} 
 \| \partial_t \hat{u} \|_{L^4} \leq
 c' \|\partial_t u \|_{W^{1,2} ( \Omega )} \| \partial_t \hat{u} \|_{W^{1,2} ( \Omega ) } \\
    & \leq c' \| u\|_{W^{2,2} ( \Omega ) } \| \hat{u} \|_{W^{2,2} ( \Omega ) }.
\end{align*}
and the second inequality follows from 
the Sobolev embedding 
$ W^{1,2}(\Omega ) \hookrightarrow L^4( \Omega ) $. 
Thus we have proved the desired inequality for the term $I$
$$ I \leq c \| u\|_{W^{2,2}( \Omega)} \|\hat{u}\|_{W^{2,2}(\Omega)} .$$
We prove an analog inequality for $III$. 
\begin{align*}
 \| ( d J_t(u) \hat{u} ) \partial_t u \|_{W^{1,2} ( \Omega )} & \leq \overbrace{\| ( d J_t(u) \hat{u} ) \partial_t u\|_{L^2}}^a + 
 \overbrace{\| (\partial_t  d J_t(u) \hat{u} ) \partial_t u \|_{L^2}}^b  \\ 
 &+\overbrace{\|( d^2 J_t(u) \partial_t u ) \hat{u} ) \partial_t u\|_{L^2}}^d + \overbrace{\| (( d J_t(u)\partial_s u) \hat{u} ) \partial_t u \|_{L^2}}^e\\ & \overbrace{ \|( d J_t(u) \partial_t\hat{u} ) \partial_t u \|_{L^2}}^f  + \overbrace{ \|( d J_t(u) \partial_s\hat{u} ) \partial_t u \|_{L^2}}^g + \\ & + \overbrace{ \|( d J_t(u) \hat{u} ) \partial^2_t u \|_{L^2}}^h + \overbrace{ \|( d J_t(u) \hat{u} ) \partial_t\partial_s u \|_{L^2}}^k
\end{align*}
We have 
\begin{align*}
 a,b \leq c' ( 1 + \| u \|_{L^{\infty} } ) \| \hat{u} \|_{L^{\infty} } \| \partial_t u \|_{L^2} \leq c \| \hat{u} \|_{W^{2,2}( \Omega )} \| u \|_{W^{2,2} ( \Omega ) }. 
\end{align*}
and also 
\begin{align*}
  d &\leq c ( 1 + \| u \|_{L^{\infty}} ) \| \hat{u} \|_{L^{\infty}} \| \partial_t u \|_{L^4} \leq c \| u \|_{W^{2,2} ( \Omega ) } \|\hat{u}\|_{W^{2,2}( \Omega )}. \\
  e &\leq c ( 1 + \| u \|_{L^{\infty}} )\|\hat{u}\|_{L^{\infty}} \| \partial_s u \|_{L^4} \| \partial_t u \|_{L^4} \leq c \| u \|_{W^{2,2} ( \Omega ) } \| \hat{u} \|_{W^{2,2} ( \Omega )}.
\end{align*}
The terms $ f$ and $ g $ can be estimated in the same way as the term $ C $, 
and finally 
$$ h,k \leq c ( 1 +\|u\|_{L^{\infty}} ) \|\hat{u}\|_{L^{\infty}( \Omega) } \|u\|_{W^{2,2}( \Omega)} \leq c \| u \|_{W^{2,2} ( \Omega ) } \| \hat{u} \|_{W^{2,2} ( \Omega )} .$$ 
Thus, we have proved the inequality \eqref{eq_quad_est}. Take 
$ r_0 = \frac{1}{2c_0 c} $, where $ c $ is the constant from  \eqref{eq_quad_est}
and $ c_0 $ as in Step 1. For such $ r_0 $ the inequality \eqref{eq:dif_der} is fulfilled.
\end{proof}

{\bf Step 3: Constructions of the maps 
 $ f^T $ and $ f^{\infty} $  } \\
  
  In Steps 1 and 2 we have proved that the maps $ \cF^{\infty} $ and 
 $ \cF^T$ satisfy properties $ 1) $ and $ 2) $ 
 of the inverse function theorem, i.e. Lemma \ref{chlem1}.
 Let $ \rho_0 = \frac{r_0}{2c_0} $, where $ r_0 $ and $ c_0 $ 
 are the constants as in Step 2. 
 For $ \xi = (\xi^+, \xi^- ) \in B_{\rho_0} (0) $, let 
 $$ 
 u^T =  (\mathcal{F}^T ) ^{-1} (0,\xi), \;\;\; 
 (u^-, u^+) = (\mathcal{F}^{\infty})^{-1} ( 0,0,\xi^+, \xi^-)
 $$
 We define maps $ f^{\infty} $ and $ f^T$ as follows
 \begin{equation}\label{eq:maps}
  \begin{split}
   & f^{\infty} : B_{\rho_0} (0) \rightarrow  E \;\; \\
   & f^{\infty } ( \xi^+, \xi^-) =  ( f^- ( \xi^+), f^+ ( \xi^-) )= 
   (\pi^-(u^-(0,\cdot)),  \pi^+(u^+(0,\cdot)) )   \\
 & f^{T} : B_{\rho_0}(0) \rightarrow  E \;\; \\
  & f^T ( \xi^+, \xi^-) = ( f^T_- ( \xi^+, \xi^-), f^T_+ ( \xi^+,\xi^-) ) = 
 ( \pi^- ( u^T(-T, \cdot)), \pi^+ (u^T(T, \cdot)) )
  \end{split}
\end{equation}
The maps $ f^{\infty} $ and $ f^T $ are diffeomorphisms 
onto their images. The sets
$ \cW^{\infty}_{\rho_0} $ and $ \cW^T_{\rho_0} $ 
are open subsets of $ \text{graph} ( f^{\infty} ) $ and 
$ \text{graph} ( f^T ) $. 
\end{PARA}

\begin{PARA}[{\bf Convergence $ f^T \stackrel{C^1}{\rightarrow}f^{\infty} $.}]\label{para:convergence}\rm

 We have constructed maps $ f^{\pm} $   
 such that locally the stable and unstable manifolds 
 are graphs of these functions. One can think of the graph of 
 the map $ f^T $ for some fixed $T $ as of the darkened line in
 figure \ref{fig6}. This is a bit misleading as $graph(f^T) \subset E\times E $
 and the whole picture lies in $E$. Still the picture gives us 
 good intuition about the convergence phenomenon. 
 We actually have to prove that the difference of the maps $ f^T $ and 
 $ f^{\infty} $, denoted by $ \eta^T_{\pm} $ in the Figure \ref{fig6} 
 converges to $0$. 
 
\begin{figure}
\centering
\scalebox{0.5}{\input{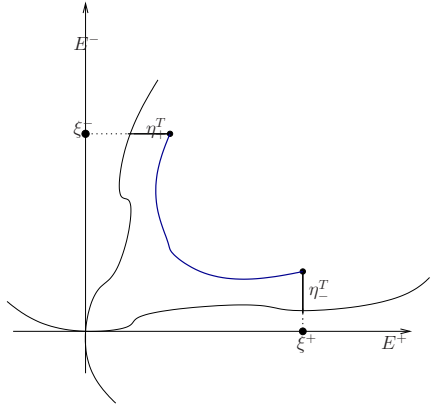}}
\caption{Convergence of submanifolds $ W^T$ to $ W^{\infty} $}
\label{fig6}
\end{figure}

We have used the inverse function theorem to 
find the maps $ f^T $ and $ f^{\infty} $
and we cannot explicitly say what are these functions. 
We construct a map that gives their difference  
$ \eta^T = f^T - f^{\infty}$.
Let $B_{\rho_0} (0)\subset E $ be the neighborhood as in Step 3, 
i.e. such that the maps $ f^{\infty} $ and $ f^T $ are defined 
on $ B_{\rho_0} (0) $. Let

\begin{align*}
  \mathcal{F}^T: B_{\rho_0}(0) \times H^{2}_{bc}(Z^T) \times E^- \times E^+  
 \longrightarrow    H^{1}(Z^T)\times E \times E  &\\
\mathcal{F}^T(\xi, u,\eta^-,\eta^+ ) = \Big( \bar{\partial}_{J_t,X_t} u ,
\xi^+ + f^+(\xi^+) + \eta^- - u(-T, \cdot), &\\ 
\xi^- + f^-(\xi^-) + \eta^+ -u(T,\cdot) \Big ),&
\end{align*}
where $ f^{\pm} $ are as in \eqref{eq:maps}.
Denote by
$ \cF^T_{\xi} := \cF^T( \xi , \cdot )  $. 
We shall prove in Lemma \ref{ch5_lem5.8} that the mapping 
$ \mathcal{F}^T_{\xi}$ is a local diffeomorphism. Then in 
Lemma \ref{lem_conv} we construct some possibly smaller neighborhood
$ B_{\rho}(0)\subset B_{\rho_0}(0)\subset E $ such that for every
$ \xi \in B_{\rho}(0) $ we can find unique $ ( u^T, \eta^T_-, \eta^T,+ ) $
such that $ \cF^T_{\xi}  ( u^T, \eta^T_-, \eta^T_+ )= 0 $ and we prove 
that $ \|\eta^T_-\|_{C^1} + \|\eta^T_+\|_{C^1} \rightarrow 0 $. 
Thus we prove that $ f^T \rightarrow f^{\infty} $ as $ T \rightarrow + \infty $. 
\end{PARA}

\begin{lemma}\label{ch5_lem5.8}
Let $ \rho_0 $ be as in Step 3.  
There exist positive constants $ C_1, r_1$ such that
 for all 
 $ u\in H^{2}_{bc}(Z^T)\; $  with  
  $$ 
  \| u\|_{2,2} < r_1 , \;\; 
  $$ 
  and all $\xi = ( \xi^+, \xi^-) \in B_{\rho_0}(0)$
  the following holds
\begin{itemize}
 \item[a)] The operator $ d\mathcal{F}^T_{\xi} (u,0,0) $ 
 is bijective and $ \|d\mathcal{F}^T_{\xi} (u,0,0)^{-1}  \|\leq C_1  $.
\item[b)] For all $ v $, $ \|v-u\| < r_1 $ 
 and for all $ \eta^{\pm} \in E^{\pm} $ we have 
  $$ 
  \|d\mathcal{F}^T_{\xi} (u,0,0) - d\mathcal{F}^T_{\xi} (v,\eta^-, \eta^+ )\| <\frac{1}{2C_1} 
  $$ 
\end{itemize}
 
\end{lemma}

{\bf Proof of Lemma \ref{ch5_lem5.8} :}

The derivative of the operator $ \mathcal{F}^T_{\xi} $ 
at the point $ ( u, \eta^-, \eta^+ )\in H^2_{bc} ( Z^T) \times E^- \times E^+ $ is given by
\begin{equation}
 d\mathcal{F}^T_{\xi} (u, \eta^-,\eta^+)
 \big (\hat{u}, \;\hat{\eta}^-,\;\hat{\eta}^+\big) 
 =\Big (D_u \hat{u},\; \hat{\eta}^- - \hat{u}(-T, \cdot),
 \;\hat{\eta}^+ - \hat{u}(T,\cdot)\Big),
\end{equation}
where the operator $D_u $ is as in \eqref{eq:lindv}. 
Notice that the linearization $ d \cF^T_{\xi} (u, \eta^-, \eta^+) $
doesn't 
depend at all on $\xi^{\pm}$ and $ \eta^{\pm} $. \\

{\bf Step A:}
 \textit{For all $ T>0,\; \xi\in B_{\rho_0}(0)$ 
 the mapping $d\mathcal{F}^T_{\xi}(0)$ is bijective. 
 Moreover, there exist a constant $c_1> 0$ such that 
\begin{equation}\label{nej_stA}
\| d\mathcal{F}^T_{\xi}\big(0\big)^{-1}\| \leq c_1
\end{equation} }

\begin{proof}
Let $ D_0 $ be the linearization of $  \bar{\partial}_{J_t,X_t} $ in $0$, i.e.
$D_0 $ is given as in \ref{para:lin_op}. 
 From Theorem \ref{thm:main_inq} we have that for all $ T$
 \begin{align}\label{nej0}
 \| \hat{u} \|_{2,2} & \leq c \Big
 ( \|D_0 \hat{u}\|_{1,2} + \| \pi^+( \hat{u} ( -T,\cdot) ) \|_{3/2} +
 \| \pi^-( \hat{u} ( T,\cdot) ) \|_{3/2} \Big ),
 \end{align}
and also 
\begin{align}\label{nej1}
 \| \hat{\eta}^- \|_{3/2} & 
 \leq \|\pi^-( \hat{\eta}^- - \hat{u} ( - T,\cdot) ) \|_{3/2}
 + \| \pi^- ( \hat{u} ( - T,\cdot) ) \|_{3/2} ,\notag \\
   \|\hat{\eta}^+\|_{3/2} &
   \leq \| \pi^+ ( \hat{\eta}^+ - \hat{u} (T, \cdot))\|_{3/2} +\|\pi^+ (\hat{u}(T, \cdot))\|_{3/2} .
\end{align}
Notice also that
\begin{equation}\label{eq:nej2}
 \begin{split}
   \| \pi^+ ( \hat{u} ( -T, \cdot) )\|_{3/2} & =
 \| \pi^+ ( \hat{u}( -T, \cdot ) - \hat{\eta}^- ) \|_{3/2}  \\
  \|\pi^-(\hat{u} (T, \cdot))\|_{3/2} &= \|\pi^-( \hat{u} (T, \cdot) - \hat{\eta}^+)\|_{3/2} 
 \end{split}
\end{equation}

Let 
$$ 
L= \| \hat{u} \|_{2,2}  + \| \hat{\eta}^-\|_{3/2} + \| \hat{\eta}^+\|_{3/2},
$$
then 
\begin{equation}\label{nej}
\begin{split}
 L  &\leq   c \Big 
 ( \|D_0 \hat{u}\|_{1,2} + \|\hat{\eta}^- - \hat{u}(-T, \cdot)\|_{3/2} 
  + \|\hat{\eta}^+ - \hat{u}(T, \cdot)\|_{3/2}  \notag \\
  & +  \|  \hat{u} ( -T,\cdot)  \|_{3/2} +
  \| \hat{u} ( T,\cdot)  \|_{3/2} \Big ) \notag \\
  & \leq c_1 ( \|D_0 \hat{u}\|_{1,2} + \|\hat{\eta}^- - \hat{u}(-T, \cdot)\|_{3/2}  + 
  \|\hat{\eta}^+ - \hat{u}(T, \cdot)\|_{3/2} \Big )
  \end{split}
\end{equation}
The first inequality in \eqref{nej} follows by 
summing \eqref{nej0} and \eqref{nej1}. 
The second inequality follows from trace inequality 
$$
  \|  \hat{u} ( -T,\cdot)  \|_{3/2} +
  \| \hat{u} ( T,\cdot)  \|_{3/2} \leq c \| \hat{u} \|_{2,2}
$$
and inequalities \eqref{nej0} and  \eqref{eq:nej2}. 
Thus we have proved that $ d\mathcal{F}^T_{\xi} (0) $ 
is injective and has closed range as it satisfies the inequality 
\eqref{nej}. We prove that it is surjective. 
Suppose that there exist a vector 
$ (  \hat{v}, \zeta^+, \zeta^-)$ orthogonal to 
$Im( d\mathcal{F}^T_{\xi} (0))$. 
Taking $ \hat{u}=0$  and varying $ \hat{\eta}^{\pm} $ 
we get $ \zeta^+\in E^+ ,\; \zeta^-  \in E^- $. 
It follows from Theorem \ref{thm:main_inq} that the operator
$$ 
\hat{u} \mapsto \Big ( D_0 \hat{u}, \pi^+ ( \hat{u}(-T, \cdot)), \pi^- ( \hat{u}(T, \cdot))  \Big ) 
$$
is bijective. Hence, for given $ \zeta^{\pm} $ there exists unique $ \hat{u} \in Ker ( D_0 )$
such that $ \pi^{\pm} ( \hat{u}( \mp T, \cdot)= \zeta^{\pm} $. 
For such $ \hat{u} $ as
 $(\hat{v}, \zeta^+, \zeta^- )$ is orthogonal to 
$ d\mathcal{F}^T_{\xi} (0)(\hat{u}, \hat{\eta}^-, \hat{\eta}^+ )$ it follows 
\begin{align*}
\langle \hat{u} ( -T, \cdot), \zeta^+ \rangle + 
\langle \hat{u} (T, \cdot),\zeta^- \rangle =
\| \zeta^+ \|^2_{3/2} + \| \zeta^- \|^2_{3/2} = 0. 
\end{align*}
Hence $ \zeta^+ = \zeta^- = 0 $. 
From the surjectivity of the operator $ D_0$ it follows that $ \hat{v} = 0 $. 
\end{proof}

{\bf Step B:} 
\textit{There exists a constant $ \tilde{r}_1 $ 
such that for all $ T$, $ \xi\in B_{\rho_0}(0) $ ,  
for all $\eta^{\pm} \in E^{\pm} $  and   
$ u \in H^2_{bc} ( Z^T) $ with:
\begin{equation}
 \|u\|_{2,2} < \tilde{r}_1
\end{equation} 
the operator $ d\mathcal{F}^T_{\xi} (u, \eta^-,\eta^+) $  
is bijective and 
\begin{equation}
\| d\mathcal{F}^T_{\xi} (u,\eta^-,\eta^+)^{-1}  \|\leq C_1= 
\frac{6c_1}{5},
\end{equation}
where $ c_1 $ is the constant from step A.} \\
\begin{proof}
\begin{align}\label{eq_nej4}
  & \Big \|\Big ( d\mathcal{F}^T_{\xi} (u,\eta^-,\eta^+)  -
  d\mathcal{F}^T_{\xi} (0)\Big )\Big ( \hat{\eta}^-, \hat{\eta}^+,\hat{u} \Big ) \Big\|
  = \Big \| \Big ( D_u \hat{u} - D_0\hat{u}, 0, 0  \Big ) \Big \| \notag \\ 
  &\leq c \| u\|_{2,2} \| \hat{u}\|_{2,2} \notag \\
  & \leq  c \| u\|_{2,2}   ( \|\hat{u}\|_{2,2} + \|\hat{\eta}^-\|_{3/2} + \|\hat{\eta}^+\|_{3/2}) 
\end{align}
The first inequality in \eqref{eq_nej4}
follows from the inequality \eqref{eq_quad_est}.
Let $ \tilde{r}_1 = \frac{1}{6c_1 c} $, where $ c_1 $ 
is the constant of the step A and $ c $ as in \eqref {eq_nej4}, 
and suppose that $ \|u\|_{2,2} < \tilde{r}_1 $. 
We have that 
 $  
 \|d\mathcal{F}^T_{\xi} (u,\eta^-,\eta^+)- d\mathcal{F}^T_{\xi} (0)\|\leq \frac{1}{6c_1} 
 $
 and hence
\begin{align}\label{eq:nej_5}
\|d\mathcal{F}^T_{\xi} (u,\eta^-,\eta^+)\cdot d\mathcal{F}^T_{\xi} (0)^{-1} - 1\|
\leq \frac{1}{6}\; , \;\;\;\forall T\geq T_0 
\end{align}
From \eqref{eq:nej_5} it follows that 
$d\mathcal{F}^T_{\xi} (u,\eta^-,\eta^+)\cdot 
d\mathcal{F}^T_{\xi} (0)^{-1}$
is invertible and thus also 
$d\mathcal{F}^T_{\xi} (u,\eta^-,\eta^+)$. 
Let  $ L=\|\hat{u}\|_{2,2} + \|\hat{\eta}^-\|_{3/2} + \|\hat{\eta}^+\|_{3/2}$, we have 
\begin{align}\label{nej4}
L  \leq  c_1 & \|d\mathcal{F}^T_{\xi} (0)(\hat{u}, \hat{\eta}^-, \hat{\eta}^+)\| \notag \\
\leq  c_1 & \Big( \|d\mathcal{F}^T_{\xi} (u,\eta^-,\eta^+)(\hat{u},\hat{\eta}^-, \hat{\eta}^+)\| 
\notag  + \|(d\mathcal{F}^T_{\xi} (u,\eta^-,\eta^+)  - 
 d\mathcal{F}^T_{\xi} (0))(\hat{u}, \hat{\eta}^-, \hat{\eta}^+)\| \Big )\notag \\
\leq c_1 & \Big( \|d\mathcal{F}^T_{\xi} (u,\eta^-,\eta^+)
(\hat{u}, \hat{\eta}^-, \hat{\eta}^+)\| +
\frac{1}{6c_1} L \Big ) \notag
\end{align}
From the previous inequality we get
\begin{equation}\label{nej5}
\|\hat{u}\|_{2,2} + \|\hat{\eta}^-\|_{3/2} + \|\hat{\eta}^+\|_{3/2} 
\leq  \frac{6c_1}{5} \|d\mathcal{F}^T_{\xi} (u,\eta^-,\eta^+)(\hat{u}, \hat{\eta}^-, \hat{\eta}^+)\|.
\end{equation}
\end{proof}

\medskip

{\bf Step C:}
\textit{The requirements $ 1) $ and $ 2) $ of Lemma \ref{ch5_lem5.8} 
are satisfied for $ r_1 = \frac{\tilde{r}_1}{2} $ and $ C_1 $ as in step B.} 
\begin{proof}
  If $ \| u\| < r_1$ we have that 
 \begin{itemize}
 \item  
 $ d\mathcal{F}^T_{ \xi} (u,0,0) $ is bijective and 
 $ \|d\mathcal{F}^T_{ \xi} (u,0,0)^{-1}\|\leq C_1= \frac{6c_1}{5} $
 \item For all $v$ such that 
 $\|u-v\|_{2,2} < r_1 $ and for all $ \eta^-\in E^-, \; \eta^+\in E^+ $ 
 we have
 \begin{align*}
  \|d\mathcal{F}^T_{ \xi} (u,0,0)-d\mathcal{F}^T_{ \xi} (v,\eta^-,\eta^+)\|
  & \leq \Big ( \|d\mathcal{F}^T_{ \xi} (u,0,0) -d\mathcal{F}^T_{ \xi} (0) \| \\ 
  & + \|d\mathcal{F}^T_{ \xi} (v,\eta^-,\eta^+) -d\mathcal{F}^T_{ \xi} (0) \|  \Big ) \\
  & \leq \frac{1}{6c_1} + \frac{1}{6c_1} =\frac{1}{3c_1} < \frac{1}{2C_1}. 
 \end{align*}
  \end{itemize}
\end{proof}
 
 \medskip
 \medskip
 
 \begin{lemma}\label{lem_conv}
 There exists $ \rho> 0 $  and $ T_1 > 0 $ such that 
 for every $ T \geq T_1$ the following holds
 \begin{itemize}
  \item [i)] For every $ \xi = ( \xi^+, \xi^-) \in B_{\rho}(0) $ 
  there exists unique 
  $ ( u^T(\xi),\; \eta^T_-(\xi), \;\eta^T_+(\xi)) \in H^2_{bc} ( Z^T ) \times E^- \times E^+ $ 
  such that
 $$ 
 \mathcal{F}^T_{\xi} ( u^T, \eta^T_-, \eta^T_+ ) = 0 .
 $$
 \item[ii)] The maps $ \eta^T= (\eta^T_-, \eta^T_+ )  $ 
 converge exponentially to $0 $ 
 $$
 \|\eta^T(\xi) \|\leq c e^{-\mu T } , \;\;
 $$
 for all $ \xi \in B_{\rho}(0) $. 
 \item[iii)] The maps $ \eta^T $ converge also in $ C^1 $ 
 norm exponentially to $0$ , i.e. for all $ \xi \in B_{\rho}(0) $ 
 the following inequality holds
 \begin{equation}\label{ch5_eq2.32}
\| d\eta^T(\xi)\|\leq c e^{-\mu T} ,  
\end{equation}
where 
$ \|d\eta^T(\xi)\|:= \sup\limits_{\|\hat{\xi}\|_{3/2} \leq 1 }
\|d\eta^T(\xi)\hat{\xi}\|_{3/2} $ .
 
 \end{itemize}

\end{lemma}

{\bf Step I) Proof of parts $ i) $ and $ ii) $ of Lemma  \ref{lem_conv} } 

\begin{proof}

Let $ \rho= \min \{\frac{r_1}{4 c_0} , \rho_0\}$
where $ r_1 $ is the constant of Lemma \ref{ch5_lem5.8}, 
$ c_0 $ the constant of \eqref{eq:est_inv} in
Step 1 and $ \rho_0 $ is the constant from Step 3. 
Let $ (\xi^+, \xi^-)\in B_{\rho}(0)\subset E $. 
Then there exist holomorphic curves 
$ u_{\mp} \in H^2_{bc} ( Z^{\pm} ) $ such that 
\begin{align*}
   & u_- (0, \cdot)= \xi^+ + f^-(\xi^+ ) , \;\;\; 
  u_+(0, \cdot)= \xi^- + f^+(\xi^- )  
\end{align*}
It follows from Step 1 and 2 in \ref{para:construction}, as the mapping 
$ \cF^{\infty} $ satisfies the requirements of the inverse function theorem,
that $ (0,0,\xi^+, \xi^-) \in B_{\rho} ( \cF^{\infty}(0,0)) \subset \cF^{\infty} ( B_{2c_0\rho} ( 0,0)) $, 
thus 
$$ 
\|u_+ \|_{2,2} + \| u_- \|_{2,2} < 2c_0 \rho \leq \frac{r_1}{2}.
$$

We construct a pregluing map $ v_T $ as follows. 
 Let $ \beta(s): \R \rightarrow [0,1] $ be a smooth cut-off 
 function such that 
\[ \beta(s)= \left\{ \begin{array}{ll}
															1 & \mbox{ if $s\leq -1$}\\
															0 & \mbox{if $ s\geq 1$}
															\end{array}
													\right. 
\]
Suppose that $ \| \beta\|_{C^2} < 2 $. 
Take 
$$
v_T= v_T(\xi) = 
\beta(s)u_-(s+T,t) + (1-\beta(s))u_+(s-T,t).
$$ 
The pregluing map $ v_T $ satisfies the following:
\begin{align*}
\|v_T\|_{2,2} &= \|\beta(s)u_+(s+T,t) + (1-\beta(s))u_-(s-T,t)\|_{2,2}\\
&\leq \|\beta(s)u_+(s+T,t)\|_{2,2} + \|(1-\beta(s))u_-(s-T,t)\|_{2,2} \\
&\leq \|\beta\|_{C^2} ( \|u_+\|_{2,2} + \|u_-\|_{2,2} ) 
< r_1
\end{align*}
and also
$$ 
v_T(\mp T,\cdot)=\xi^{\pm}(\cdot) + f^{\mp} ( \xi^{\pm}  ) .$$ 
Thus we have  
\begin{equation}
\mathcal{F}^T_{\xi} (v_T,0,0)= (w_T,0,0),
\end{equation}
where 
$w_T(s,t)=\bar{\partial}_{J_t, X_t}v_T$. 
As $u_+$ and $u_-$ are the solutions of 
$\bar{\partial}_{J_t, X_t}u=0 $ we have 
\[ w_T(s,t)= \left\{ \begin{array}{ll}
															0 & \mbox{ if $s\leq -1$}\\
														  0 & \mbox{if $ s\geq 1.$}
															\end{array}
													\right. 
\]
Hence, $w_T$ is nonzero only on the interval $[-1,1]$, 
but from exponential decay of holomorphic curves 
we know that $ u^{\pm} $ and all of their derivatives
decay exponentially, thus
\begin{equation}
\|w_T\|_{1,2} \leq c e^{-\mu (T-1)}.
\end{equation} 

The constructed pregluing map $v_T$ 
satisfies the assumptions $a) $  and $b)$ of Lemma \ref{ch5_lem5.8}.
As $ \mathcal{F}^T_{\xi} (v_T,0,0)= (w_T,0,0)$ and 
$\|w_T\|_{1,2}$ decays exponentially for some $ T_1 $ sufficiently large 
and $T\geq T_1$ we have 
$$ 
\|\mathcal{F}^T_{\xi} (v_T,0,0)\|= \|w_T\|_{1,2} \leq \frac{r_1}{2C_1},
$$ 
where $ C_1 $ is the constant in Lemma \ref{ch5_lem5.8}. 
From the inverse function theorem we have 
$$ 
B_{\frac{r_1}{2C_1}} ( \mathcal{F}^T_{\xi} (v_T,0,0) ) \subset 
\mathcal{F}^T_{\xi} ( B_{r_1} ( v_T,0,0) ) .
$$
Therefore there exist unique 
$(\eta_T^-(\xi),\eta_T^+(\xi),u_T(\xi)) \in B_{r_1} ( v_T,0,0)$ 
such that 
\begin{equation}
\mathcal{F}^T_{\xi} (u^T,\eta^T_-,\eta^T_+) =(0,0,0).
\end{equation}
The last equality proves part $ i) $ of Lemma \ref{lem_conv}. 
From the inverse function theorem it follows also 
\begin{align}\label{eq_exp2}
\|u^T-v_T\|_{2,2} + \|\eta^T_-\|_{3/2} + \|\eta^T_+\|_{3/2} 
\leq 2C_1 \|w_T\|_{1,2}\leq  c e^{-\mu T}.
\end{align}
Hence we have 
\begin{align}\label{ch5_eq2.31}
 \|\eta_T^-(\xi)\|_{3/2} + \|\eta_T^+(\xi)\|_{3/2} &\leq c e^{-\mu T}. 
\end{align}

\medskip 
\end{proof}

\medskip

{\bf Step II) Proof of $ iii) $ } 

\begin{proof}
We finish the proof of Lemma \ref{lem_conv} by proving 
part $ iii) $. Let $ v_T (\xi) $ be the pregluing map. 
From the definition of the map $ \mathcal{F}^T $ we have 
\begin{align}\label{ch5_eq1}
 E^- \ni \eta^T_-(\xi) = (u^T(\xi) - v_T(\xi))(-T, \cdot),\;\;
 E^+ \ni \eta^T_+(\xi) = (u^T(\xi) - v_T(\xi))(T,\cdot) 
\end{align}
Let 
 $$
 \hat{u}_T = du^T(\xi)(\hat{\xi}), \;\
 \hat{v}_T = dv_T(\xi)(\hat{\xi}),
 \;\; \hat{\eta}^T= d\eta^T(\xi)(\hat{\xi}).  
 $$
 Then 
\begin{align*}
  \hat{\eta}^T_-= \pi^- (  \hat{\eta}^T) = (\hat{u}_T  -  \hat{v}_T) (-T, \cdot) ,\;\;\;\;
\hat{\eta}^T_+= \pi^+(  \hat{\eta}^T) = (\hat{u}_T - \hat{v}_T) (T, \cdot)
\end{align*}
As $ \|u_T\|_{2,2} $ is sufficiently small, the following 
inequality holds for all $ \xi \in H^2_{bc} ( Z^T ) $
\begin{align}\label{eq_st2_c1_1}
 \|\xi\|_{2,2}  \leq c' \Big ( \|D_{u^T} (\xi) \|_{1,2} + 
 \|\pi^+(\xi(-T, \cdot))\|_{3/2} + \|\pi^-(\xi(T, \cdot))\|_{3/2} \Big ).
\end{align}
The previous inequality follows from Steps 1 and 2 in \ref{para:construction}. 
 Substituting  $ \xi =\hat{u}_T - \hat{v}_T $
 in \eqref{eq_st2_c1_1} we get 
 
\begin{align}\label{ch5_eq2.35}
\|\hat{u}_T-\hat{v}_T \|_{2,2} &\leq c'\Big 
( \|D_{u^T}(\hat{u}_T-\hat{v}_T )\|_{1,2} + 
\|\pi^+( \hat{\eta}^T_-)\|_{3/2} + \|\pi^-(\hat{\eta}^T_+)\|_{3/2} \Big ) \notag \\
&= c'\|D_{u^T}(\hat{u}_T-\hat{v}_T )\|_{1,2} \notag \\
& = c'\|D_{u^T}(\hat{v}_T )\|_{1,2},
\end{align}
where the last equality holds as $ u^T $ satisfies 
$ \overline{\partial}_{J_t, X_t} u^T = 0 $ and hence
$D_{u^T}(\hat{u}_T )= 0 .$ 
From the trace inequality and \eqref{ch5_eq2.35} we have

\begin{align}\label{eq_c1est}
 \|\hat{\eta}^T_-\|_{3/2} + \| \hat{\eta}^T_+\|_{3/2}
 \leq c \|D_{u^T}(\hat{v}_T )\|_{1,2}.
\end{align}

We shall prove that the right hand side of \eqref{eq_c1est}
decays exponentially. Notice that
\begin{align*}
\|D_{u^T} \hat{v}_T \|_{1,2} 
\leq \Big ( \overbrace{\|D_{v_T} \hat{v}_T\|_{1,2}}^I +
\overbrace{\|(D_{u^T}  -D_{v_T}) \hat{v}_T\|_{1,2}}^{II}\Big ) 
\end{align*}
 Then
\begin{align}\label{ch5_eq2.37}
II=\|(D_{u^T}  -D_{v_T}) \hat{v}_T\|_{1,2} 
\leq c \|u^T- v_T\|_{2,2} \|\hat{v}_T\|_{2,2} 
\leq c e^{-\delta T} \|\hat{v}_T\|_{2,2},
\end{align}
the penultimate inequality follows from 
\eqref{eq_quad_est} and the last inequality follows from 
the inequality \eqref{eq_exp2}. 
As $ \hat{v}_T=\beta du_+(\xi^+)\hat{\xi}^+ + 
(1-\beta) du_-(\xi^-)\hat{\xi}^-, $ we have  
\begin{align}\label{ch5_eq2.38}
\|\hat{v}_T\|_{2,2} 
\leq 2 \Big ( \|\overbrace{du_+(\xi^+)\hat{\xi}^+}^{w_+}\|_{2,2} +
\|\overbrace{du_-(\xi^-)\hat{\xi}^-}^{w_-}\|_{2,2} \Big )
\end{align}
Both $ w_{\pm} $ are in the kernels 
of the operators $ D_{u^{\pm} } $ and they satisfy 
\begin{align}\label{ch5_eq2.39}
 \|w_{\pm}\|_{2,2} & \leq c' \Big ( \|D_{u_{\pm}} w_{\pm} \|_{1,2} + 
 \|\pi^{\pm} ( w_{\pm}(0, \cdot )) \|_{3/2} \Big )\notag  \\
  & \leq c' \|\pi^{\pm} ( w_{\pm}(0, \cdot )) \|_{3/2} \notag \\
  & \leq c \| \hat{\xi}^{\pm} \|_{3/2} .
\end{align}

From the inequalities \eqref{ch5_eq2.37}, \eqref{ch5_eq2.38}, \eqref{ch5_eq2.39}
we have 
\begin{equation}\label{ch5_eq2.43}
 II\leq c e^{-\mu T} \big ( \|\hat{\xi}^+\|_{3/2} + \|\hat{\xi}^-\|_{3/2} \big ).
\end{equation}
For the $ I $ term we have 
\begin{align*}
  I =\|D_{v_T} \hat{v}_T \|_{H^{1}([-T,T]\times [0,1])}  
  = \|D_{v_T} \hat{v}_T\|_{H^{1}([-1,1]\times [0,1])}
  \leq c \|\hat{v}_T\|_{H^2([-1,1]\times [0,1])} 
\end{align*} 
and the penultimate inequality holds as 
\[ \hat{v}_T(s,t)=\begin{cases}
                   du_+(\xi)(\hat{\xi}^+)= 
                    w_+ ( s+ T, t) , \;\; (s,t) \in [-T, -1) \times [0,1] \\
                   du_-(\xi)(\hat{\xi}^-)=
                   w^-(s-T, t) , \;\; (s,t) \in (1,T] \times [0,1]
                  \end{cases}\]
and $ D_{u_{\pm}} w^{\pm} = 0 $, 
hence $ D_{v_T} \hat{v}_T= 0 $ on the $ ([-T, -1)\cup (1, T]) \times [0,1] $.
As $ w^{\pm} $ are in the kernel of the operators
$ D_{u_{\pm}} $ they will decay exponentially. 
In \cite{RS3} Lemma 3.1 was proved that 
$$
\| w^+ ( s, \cdot)\|_{L^2 ( (0,1))} \leq c \| w^+ ( 0, \cdot)\|_{L^2 ( (0,1)) } e^{-\mu s } . 
$$
Thus it follows that $ \| w^+ \|_{L^2 ( [ s , + \infty) \times [0,1] )} $ decays 
exponentially. The exponential decay of $ \|w^+ \|_{W^{2,2} ( [s, + \infty) \times [0,1] ) } $ 
follows from the following inequality 
\begin{align*}
 \|w^+\|_{W^{k,2} ( [s,+\infty) \times [0,1] )} \leq c \Big ( \| D_{u^+} ( w^+ ) \|_{W^{k-1,2} ( [s-1, + \infty) \times [0,1] )} + \| w^+ \|_{W^{k-1,2} ( [s-1, + \infty) \times [0,1] ) } \Big ). 
\end{align*}
This inequality follows from Lemma C.1 in \cite{RS3}. 

\end{proof}

\section{Appendix}\label{SEC:app}

\subsection{The spaces $ H^s  $ and $ H^s_0  $ and $ H^{1/2}_{00} $ }

\medskip

We shall mention some important properties of these interpolation spaces. 
For more detailed exposition of these facts we refer to \cite{LM}. 
Let $ I=[0,1] $ we define the space $ H^s (I) $ 
as the interpolation space
 $$ 
 H^s ( I)= [ H^k (I), L^2 ( I) ]_{\theta} , \;\; k ( 1-\theta) = s .
 $$  
Analogously are defined the interpolation spaces 
$ H^s(\R^n) $ . For $0<s <1 $ a function $ u \in H^s ( \R^n )= [H^1(\R^n), L^2(\R^n)]_{1-s} $ is 
 characterized by the property that $ u \in L^2 ( \R^n ) $ and 
 \begin{equation}\label{ap:eq2}
  \iint\limits_{\R^n \times \R^n } \frac{\abs{u(x) - u(y)}^2}{\abs{x-y}^{n+ 2s }} dx dy < + \infty .
  \end{equation}
  This can be proved using the Fourier transform. 
 The analogous holds for functions $ u \in H^s ( I ), \; 0 < s < 1 $.
 Namely, a function $ u \in H^s (I) $ 
 if $ u \in L^2 ( I) $ and 
 \begin{equation}\label{eq:hsomega}
  \displaystyle \iint\limits_{I\times I} 
 \frac{\abs{u(x)- u(y)}^2}{\abs{x-y}^{n+2s}} dx dy < + \infty .
 \end{equation}
 This follows from the fact that the restriction map 
 $ \text{Rest} : H^s ( \R ) \rightarrow  H^s ( I ) $ 
 is linear, bounded surjective 
 map and it has bounded right inverse. 
 Its inverse, i.e. the extension map 
 \begin{align*}
  Ext : H^s ( I ) \rightarrow H^s ( \R ) 
 \end{align*}
is given by reflection in local coordinate charts. 
  
 \begin{remark}\rm
The extension by $ 0 $ is not a bounded linear map on $ H^{1/2} ( I) $. 
Take for example $ u = 1 $. 
This function is an element of $ H^{1/2} (0,1) $ as the integral 
(\ref{eq:hsomega}) is $ 0 $. But the extended function 
\[\tilde{u}(x) =\begin{cases}  1, & \;\; x \; \in (0,1)\\
                    0, \;\; & x \notin (0,1)
                    \end{cases}
                    \]
 is not an element of $H^{1/2} ( \R), $ as the integral (\ref{ap:eq2}) diverges. 
\end{remark}

\begin{definition}
 The space $ H^s_0 (I) $ is defined as the closure
 of $ C^{\infty}_c (I ) $ in the $  H^s ( I ) -$ norm.
\end{definition}

\begin{theorem}{\bf (Extension by 0)} 
The map 
$$
u \mapsto \tilde{u} = \text{extension} \;\; \text{of} \;u \;\text{by} \;\;0 \;\; \text{outside  of} \; I
$$ 
is a continuous mapping of $ H^s (I) \rightarrow H^s (\R) $ iff $ 0 \leq  s < 1/2$. 
The same mapping is a continuous mapping of 
 $$ H^s_0(I) \rightarrow H^s (\R) $$
 if $ s > 1/2 $ and 
 $ s \neq  \;\; \text{integer} + 1/2 $. 
 \end{theorem}
  From the previous theorem we see that the spaces $ H^s_0 $ 
  change their behavior exactly for $ s= n + \frac{1}{2} , \; n\in \N_0 $. 
  In the next proposition we prove that the spaces $ H^{1/2}_0(0,1) $ and $ H^{1/2}(0,1) $ are 
  the same. 
 \begin{proposition}
 The space $ C^{\infty}_c ((0,1)) $ is dense in $ H^{1/2} ( (0,1)) $, hence 
 $$ 
 H^{1/2}_0((0,1)) = H^{1/2} ((0,1)) .
 $$
 \end{proposition}
 \begin{proof}
  The proof follows from the following facts:\\
{\bf Fact A:}
 There exist a sequence of smooth functions $f_n : \R^2 \rightarrow \R $ 
 with $ f_n = 1$ on some neighborhood of $ 0 $ , 
 but 
 $ \| f_n \|_{H^1 ( \R^2 ) } \rightarrow 0 , \; n \rightarrow \infty $.  
\\
{\bf Fact B:} 
Let $ g(x) = 1  $ on $ (-1, 1 ) $. 
 There exists a sequence of smooth functions $ g_n $ such that 
 \begin{itemize}
  \item $g_n (x) = 0 $ on some open neighborhood $ U_n $ of $0 $.
  \item $\|g_n (x ) - 1 \|_{H^{1/2} ( -1, 1)} \stackrel{n\rightarrow \infty }{\longrightarrow} 0 $
 \end{itemize}
{\bf Fact C:} 
Any polynomial $ P_k (x) = a_k x^k + \cdots a_1 x + a_0 , \; a_i \in \R, \; i=0, \cdots k $ 
can be approximated in $ H^{1/2} (0,1) $ norm by $ C^{\infty}_c (0,1) $ functions. 
\\
{\bf Proof of Fact A :} Let $ 0 < \delta < \epsilon $ ,
 define a function $ f : \R^2 \rightarrow \R $ as 
\[ f(z) = \begin{cases}
           \frac{\ln\epsilon - \ln \abs{z}}{\ln \epsilon - \ln \delta }, & \delta \leq \abs{z} \leq \epsilon \\
           1, & \abs{z} < \delta \\
           0, & \abs{z} > \epsilon.
          \end{cases}
\]
Then obviously $ f \in C( \R^2 ) $ and $ supp(f) \subset B_{\epsilon} (0) $
 and $ \abs{f'(z)} = \frac{1}{\abs{z} \ln( \epsilon/ \delta) } $ on $ \delta \leq |z| \leq \eps$. 
 We want to estimate $ \| f\|_{H^1 ( \R^2 )} $. First we have 
\begin{align*}
I_{\epsilon, \delta} = & \int\limits_{\R^2 } \abs{f'}^2 dx dy 
      = \int\limits_{\delta \leq \abs{z} \leq \epsilon} \abs{f'}^2 dx dy \\
                      =& \int\limits_{\delta \leq \abs{z} \leq \epsilon}
            \frac{dx dy}{\abs{z}^2 \ln^2( \epsilon / \delta )}=
          \int\limits_{\delta}^{\epsilon}\int\limits_0^{2\pi} \frac{rd\phi d r}{r^2 \ln^2( \epsilon/ \delta ) }
           = & \frac{2 \pi}{\ln^2 (\epsilon/ \delta) }\int\limits_{\delta}^{\epsilon} \frac{dr}{r}
        = \frac{2\pi}{\ln(\epsilon/ \delta)}
\end{align*}
Thus, if $ \delta \rightarrow 0 $ much faster than
 $ \epsilon $ we have that $ I_{\epsilon, \delta } \rightarrow 0 $. 
To obtain$ \|f\|_{L^2 } \rightarrow 0 $ we need to take
 $ \epsilon \rightarrow 0 $. Take $\epsilon_n = \frac{1}{n} $ 
and $ \delta_n =\frac{1}{n^2 } $ and take the corresponding functions $ f_n $,
 constructed as above. This is a sequence of continuous functions which satisfy 
\begin{itemize}
 \item $f_n(0) = 1$
 \item $ f_n \stackrel{H^{1} ( \R^2 )}{\longrightarrow} 0, \; n \rightarrow \infty  .$  
\end{itemize}
To get a sequence of smooth functions that satisfy the same, use molifiers to smoothen $f_n $. \\
{\bf Proof of Fact B:} It is enough to take
$ \tilde{g}_n = f_n( x,0) $, where $ f_n $ is the sequence of smooth functions as
in the proof of Fact A. 
Let $ g_n(x) = 1 - \tilde{g}_n(x) $. Then 
\begin{align*}
 \|1 - g_n (x)\|_{H^{1/2} (-1,1)} = \| \tilde{g}_n(x)\|_{H^{1/2} (-1,1)} \stackrel{tr}{\leq }\|f_n \|_{H^1 ( \R^2 ) } ,
\end{align*}
as $ \lim\limits_{n\rightarrow \infty} \|f_n\|_{H^1(\R^2 )} = 0 $ we get 
\begin{align}
 \lim\limits_{n \rightarrow \infty} \|1-g_n(x) \|_{H^{1/2} (0,1)} = 0 .
\end{align}
{\bf Proof of Fact C: } 
Let $ \left. h_n(x) = g_n( x) \cdot g_n(1-x)\right|_{(0,1)} \in  C^{\infty}_c ( 0,1) $,
where $ g_n $ is the sequence as in the proof of Fact B.  
The sequence $ h_n(x) \stackrel{H^{1/2}}{\longrightarrow }1, \; n \rightarrow \infty $.
  One can easily check that for every $ m \in \N $ 
\begin{align}
 C^{\infty}_c ( 0,1) \ni x^m h_n (x) \stackrel{H^{1/2} }{\longrightarrow } x^m , \; n \rightarrow \infty
\end{align}
As polynomials are dense in $ H^{1/2} ((0,1) $ ( $ C(\overline{\Omega} ) $ 
is dense in $ H^{1/2} (\Omega ) $ ) we get that $ C^{\infty}_c ( 0,1) $ 
is dense in $ H^{1/2} ( 0,1) $. 
 \end{proof}
Thus, we see that there is no difference between $ H^{1/2} $ and $ H^{1/2}_0 $
and that the extension by $0 $ is not a bounded linear map.
There are $ 1/2 $ interpolation spaces that allow the extension by $0 $ and
they are called Lions- Magenes spaces. These spaces are exactly 
those spaces on which our Hilbert manifold of paths $ \sP^{3/2} $ 
, defined in \eqref{eq:sp3/2}, is modeled. 
We discuss them in more details. 
\medskip
\subsection{Lions- Magenes interpolations spaces $ H^{1/2} _{00} $}

 For $ W = H^1_0 ( (0,1)) = \{ \xi \in H^1((0,1)): \xi(i)=0, \; i=0,1 \} $ 
and $ H= L^2 ((0,1)) $ we define \emph{ Lions- Magenes}  interpolation space 
\begin{align}
  H^{1/2}_{00}((0,1)) := [W, H]_{1/2} 
\end{align}
 We have seen in Theorem \ref{apthm1} that
$ [W, H]_{1/2} $ can be seen as the trace space of 
some Hilbert space or the domain $ \text{Dom}(\sqrt{A} ) $,
 where $ A $ is the operator as in the definition \ref{ap_def1.4}.
To define the space $ H^{1/2}_{00}((0,1)) $ 
one can take for example $ A = \partial_t : H^{1/2}_0 \rightarrow L^2 $.
 In the next proposition we give another equivalent 
interpretation of this space.
 \begin{proposition}\label{ap_prop3.3}
  A function $ u \in H^{1/2} ((0,1)) $ is an element 
of the space $ H^{1/2}_{00} ((0,1))= [W, H]_{1/2} $ 
if and only if 
  \begin{equation}\label{ap:eq:lm1}
   \int\limits_0^1 \frac{u^2(y)}{d(y, \partial I)} dy = 
\int\limits_0^{1/2} \frac{u^2}{y}dy + \int\limits_{1/2}^1 \frac{u^2(y)}{1-y}dy < + \infty
  \end{equation}
  and the norm 
  \begin{equation}
  \| u\|_{H^{1/2}_{00}} =
 \Big ( \|u\|^2_{H^{1/2}} + \int\limits_0^1 \frac{u^2(y)}{d(y, \partial I )} dy \Big )^{1/2}
  \end{equation}
  is equivalent to the interpolation norm. 
\end{proposition}
\begin{proof}
 Let $ u \in H^{1/2}_{00} = [W, H]_{1/2} $. The extension by $ 0 $ 
is a bounded linear operator 
\begin{align*}
 Ext_0 : H^1_0((0,1)) \rightarrow H^1 (\R) \\
Ext_0 : L^2 ((0,1)) \rightarrow L^2 ( \R).
\end{align*}
From the interpolation theorem it follows that 
the extension by $0 $ is a bounded linear operator 
$$ Ext_0 : H^{1/2}_{00} ( (0,1)) \rightarrow H^{1/2} ( \R) .$$
Let $ \tilde{u}(x) $ be the extended function i.e.
\[\tilde{u}(x)=
  \begin{cases}
    u(x), \;\; x \in (0,1) \\
    0, \;\; x\in \R \setminus ( 0,1)    
  \end{cases}
 \]
As $ \| \tilde{u}\|_{H^{1/2}(\R)} \leq c \|u\|_{H^{1/2}_{00}} $ we have 
that
\begin{align}\label{ap:eq:lm2}
 \int\limits_{\R} \int\limits_{\R} \frac{\abs{\tu(x)- \tu(y)}^2}{\abs{x-y}^2} dx dy  = 
\int\limits_{(0,1)}\int\limits_{(0,1)}  & \frac{\abs{\tu(x)- \tu(y)}^2}{\abs{x-y}^2} dx dy  + \notag \\
   + 2 \underbrace{\int\limits_{-\infty}^0 \int\limits_0^1  \frac{\abs{\tu(x)- \tu(y)}^2}{\abs{x-y}^2} dx dy}_{I}  &
+ 2 \underbrace{\int\limits_{1}^{+\infty}\int\limits_{0}^1 \frac{\abs{\tu(x)- \tu(y)}^2}{\abs{x-y}^2} dx dy}_{II}
\end{align}
Notice that 
\begin{align}\label{ap:eq:lm3}
 I &= \int\limits_{-\infty}^0 
\int\limits_0^1 \frac{\abs{u(x)}^2}{\abs{x-y}^2}dx dy  = \notag\\
   &=  \int\limits_0^1 \abs{u(x)}^2 dx\int\limits_{-\infty}^0 \frac{dy}{\abs{x-y}^2 }  = 
 \int\limits_0^1 \abs{u(x)}^2 dx\int\limits_0^{+\infty} \frac{dy}{(x+ y)^2 } \notag \\
  &=\int\limits_0^1 \abs{u(x)}^2 dx\int\limits_x^{+\infty} \frac{dt}{ t^2 } \notag \\
  & = \int\limits_0^1 \frac{\abs{u(x)}^2}{x} dx.
\end{align}
Similarly we get
\begin{align}\label{ap:eq:lm4}
 II = \int\limits_0^1 \frac{\abs{u(x)}^2}{ 1-x} dx.
\end{align}
From the equations (\ref{ap:eq:lm2}), (\ref{ap:eq:lm3}) and (\ref{ap:eq:lm4})
follows that the $ u \in H^{1/2}((0,1) $ and that the expression 
(\ref{ap:eq:lm1}) is finite.

Assume now that the function $ u\in H^{1/2} ((0,1)) $ 
satisfies (\ref{ap:eq:lm1}) and define $\tilde{u}(y) $ as follows
\[\tilde{u}(y) =\begin{cases} u(y), & \; y \; \in (0,1) \\
                              0, & \;\; y \in \R \setminus (0,1)
   
  \end{cases}\]
  From the above observation it follows that 
   $ \tilde{u} $ is an element of $ H^{1/2} ( \R ) $ 
  and thus there exists a function $ v \in H^1 ( \R^2_+ ) $
   such that $ Tr(v) = v(0,y) = \tilde{u}(y) $.
   We can suppose that $ v(x,y) = 0 $ for $ \|(x,y)\| $ big enough, 
  otherwise take a function $ \beta v $ , where $ \beta $ is a suitable cut-off 
     function. \\

  There exists a locally Lipschitz homeomorphism
  $ \Psi : [0, +\infty) \times [0,1] \rightarrow \R^2_+ $,
  such that 
  $$
  \Psi ( 0\times [0,1] ) = 0\times [0,1] ,
    \;\;\Psi ( [0, + \infty) \times 0 ) = 0 \times ( - \infty , 0] 
    $$
   and 
   $ \Psi ( [0 + \infty) \times 1 ) = 0\times [1, + \infty) $.
   Let  $ \Psi_1 $ be the mapping that stretches the strip,
   i.e. maps $ [0, + \infty) \times [0,1] \mapsto [0, +\infty) \times [-1, 1],
      \;\; \Psi_1 (x,y) = ( x, 2(y-1/2)) $.
    Let $\Psi_2: [0, + \infty) \times [-1,1] \rightarrow [0, +\infty ) \times \R $
  be a mapping that maps vertical 
 lines $(x, y) , \; 0 \leq y \leq 1$  
 on the line through $ A= ( 0, 1+x) $ and $ B= (0, x) $,
 fixing $B$, and similarly maps segment 
 $ (x,y), \; -1\leq y \leq 0 $ on 
 the line through $ B= (0,x) $ and 
 $ C= ( 0,- (1+x)) $, and it is given by
 \[\Psi_2 (x,y) = \begin{cases}
                    (1-y) ( x, 0 ) + y (0, 1+x) = ( (1-y) x, y ( 1+x)) , & \;\;0 \leq y \leq 1 \\
                    (1+y) (x,0) + y ( 0, 1+x) = ( (1+y)x, y(1+x)) , & \;\;-1 \leq  y \leq 0 
                   \end{cases} \]
  
Now $ \Psi= \Psi_2 \circ \Psi_1 $ is the desired map.
 The function $ w= v\circ  \Psi : [0, +\infty) \times[0,1] \rightarrow \R $
 and $ w \in H^1 ( [0, +\infty) \times [0,1] )$ .
 As $ v(y) =0 $ for $ y\in (-\infty,0) \cup( 1, +\infty ) $ 
we have that $ w \in L^2 ( [0,+\infty), H^1_0((0,1)) ) $,
 and the function $ u $ satisfies
\begin{align*}
 u= w|_{0\times[0,1]} \equiv u \in [L^2(0,1), H^1_0((0,1))]_{1/2}= H^{1/2}_{00} .
\end{align*}

\end{proof}
\begin{remark}\rm
 Notice that the space $ H^{1/2}_{00} ((0,1)) $ isn't closed 
in $ H^{1/2}((0,1)) = H^{1/2}_0 ((0,1) $ and it has strictly 
finer topology!
\end{remark}

\subsection{The space $H^{3/2}_{bc} $} \label{subsec_h32} 

In this section we discuss the interpolation spaces 
which are relevant for our Hilbert manifold of paths 
$ \sP^{3/2} $ introduced in  \ref{para:hil_P3/2}. 
Let 
\begin{equation*}
W= W^{2,2}_{bc}([0,1]) := \left \{ \xi \in W^{2,2} ( [0,1] , \R^{2n} ) \bigg | 
\begin{array}{ll}
 \xi(i) \in \R^n \times \{0\},\; i=0,1 \\
 \partial_t \xi(i) \in \{0\} \times \R^n, \; i=0,1 
\end{array}
 \right \}.
 \end{equation*}
and let 
\begin{equation*}
V= W^{1,2}_{bc}([0,1]) := \left \{ \xi \in W^{1,2} ( [0,1] , \R^{2n} ) \bigg | 
\begin{array}{l}
 \xi(i) \in \R^n \times \{0\}
\end{array}
 \right \}.
 \end{equation*}
 and let
$ H = L^2( [0,1], \R^{2n} )$.
Notice that 
$$
V = [W,H]_{1/2} . 
$$
We explain this fact in more details. 
Observe the following Hilbert space: 
$$
V_1 = \{ \xi \in W^{1,2} ( [0,1] , \R^{2n} ) \bigg | \xi(0) \in \R^n \times \{0 \}, \xi(1) \in \{0\} \times \R^n \}.
$$
The operator $ A_1 = i \p_t : V_1 \to H $ satisfies 
all the requirements of the Remark \ref{ap_rem1.1}. 
Let $ \Psi(t)\in U(n) = O(2n) \cap GL(n, \C) \cap Sp(2n, \R) $ be a smooth family 
such that $\Psi(0) = \Id $ and $\Psi(1) : \R^n \times \{0\} \to \{0\} \times \R^n $. 
As $ U(n) $ is connected such $ \Psi(t) $ exists. 
Notice that multiplication by $ \Psi $ defines an isometry 
between $ V $ and $ V_1 $. As the operator $ A_1 : V_1 \to H $
is bijective and self adjoint, then also its conjugate 
$$
A := \Psi^{-1} A_1 \Psi= i ( \p_t  + \Psi^{-1} \p_t \Psi ): V \to H  
$$
is bijective and self adjoint with respect to the following 
scalar product 
$$
\inner\xi\eta_{H}:= \int_0^1 \inner{\Psi(t)\xi(t)}{\Psi(t)\eta(t)}dt.  
$$

Notice that $ \text{Dom}(A) = V $ and $ \text{Dom}(A^2) = W$. 
Thus we have 
$$
[W,H]_{1/2} = \Dom( A)= V. 
$$
We define the Hilbert spaces $ H^{3/2}_{bc} $ as the following 
interpolation space 
$$ 
H^{3/2}_{bc} = [ W, V]_{1/2}.  
$$
From the previous discussion we have that $ H^{3/2}_{bc} = [W,V]_{1/2} = [W,H]_{1/4} = \Dom( \abs{A}^{3/2} ) $,
or analogously it can be defined as the set of all $ \xi \in V$ such that 
\begin{equation*}
\begin{split}
A \xi \in [V,H]_{1/2} &=[ H^1([0,1]) \times H^1_0([0,1]), L^2([0,1]) \times L^2([0,1])]_{1/2} \\
		      &= H^{1/2} ( [0,1], \R^n) \times H^{1/2}_{00} ( [0,1], \R^n) 
\end{split}
\end{equation*}
If we write $ \xi= ( \xi_1, \xi_2) ,$ where $ \xi_1 $ and $ \xi_2 $ 
denote the first and last $n $ coordinates of $\xi $, then 
$ A \xi \in [V,H]_{1/2} $ implies that $ \p_t\xi_1 \in H^{1/2}_{00}([0,1], \R^n)$
and $\p_t \xi_2 \in H^{1/2}( [0,1], \R^n) $. 
Thus the Hilbert space $ H^{3/2}_{bc} $ can be also given as 
\begin{equation*}
  H^{3/2}_{bc} =
\left \{ ( \xi_1, \xi_2 ) \in H^1( [0,1], \R^n) \times H^1_0( [0,1], \R^n)\bigg |
\begin{array}{ll} 
\partial_t \xi_1 \in H^{1/2}_{00}, \\
\partial_t\xi_2 \in H^{1/2}
\end{array}
\right  \} .
\end{equation*}

Let $ I$ be an interval in $ \overline{R} $ and denote by $ \cW(I) $ the following space: 
$$ 
\cW(I) =\{ u | u \in L^2 ( I, W) , \;\; \frac{\p^2 u}{\p s^2} \in L^2 ( I, H) \} 
$$
provided with the norm 
$$
\| u \|_{\cW(I)}^2 = \Big ( \| u \|^2_{L^2(I, W)} + \norm{\frac{\p^2 u}{\p s^2}}^2_{L^2(I,H)} \Big ). 
$$
The space $ \cW(I) $ is a Hilbert space and the space 
of smooth functions $ \C^{\infty}_c ( \overline{I}, W) $ 
is dense in $ \cW(I) $. 
It follows from intermediate derivative theorem ( Theorem 2.3 in \cite{LM})
that any function $ u \in \cW(I) $ satisfies $ \p_s u \in L^2 (I, V) $
and that the mapping 
$$
\cW(I)\rightarrow L^2 ( I, V) , \;\; u \mapsto \p_s u 
$$
is continuous linear mapping. 
Thus we have that the norm $ \| u\|_{\cW(I)} $ is equivalent to the 
following norm 
$$
\|u\|_{\sW^(I)}^2 := \Big ( \| u\|^2_{\cW(I)} + \|\p_s u\|^2_{L^2(I,V)} \Big ).
$$
With this norm, the Hilbert space $ \cW(I) $ is isometric with the 
following Hilbert space 
\begin{equation*}
W^{2,2}_{bc} ( I \times [0,1] ) :=
\left\{\xi \in W^{2,2} ( I \times [0,1], \R^{2n} ) \,\bigg|\,
\begin{array}{ll}
\xi(s,i) \in \R^n \times \{0\} , i=0,1 \\
\partial_t \xi(s,i) \in \{ 0\} \times \R^n, \; i=0,1 
\end{array}\right\}.
\end{equation*}
It follows from Trace theorem (Theorem 3.2 in \cite{LM}), that 
for $ u\in \cW(I) $ and for some $s_0 \in I $ we have 
$$
u(s_0) \in [ W, H]_{1/4} = [W,V]_{1/2}. 
$$
Thus the trace space of the Hilbert space $ \cW(I) $ is 
the Hilbert space $ H^{3/2}_{bc} $. 
The following proposition is a corollary of the 
Theorem 8.3 in \cite{LM}. 

\begin{proposition}\label{pr:trace}
Let $ W^{2,2}_{bc} ( I \times [0,1] ) $ and $ H^{3/2}_{bc} $ 
be defined as above. 
Suppose that $ s_0 \in I $ and denote with 
$r$ the restriction map 
 \begin{align}
  r: W^{2,2}_{bc} ( I \times [0,1] ) \rightarrow H^{3/2}_{bc}, \;\; r( \xi(s,t)) = \xi(s_0,t).
 \end{align}
 The linear map $ r $ is surjective and it has a continuous 
 right inverse, i.e. there exists a continuous extension operator
 \begin{align}
 \text{Ext} : H^{3/2}_{bc} \rightarrow W^{2,2}_{bc} ( I \times [0,1] ). 
 \end{align}
\end{proposition}

\section{Appendix-The Hessian of the symplectic action} 


This appendix explains why in symplectic Floer theory 
it's necessary to work with compatible (rather than tame) 
almost complex structure at least near the critical points. 
The asymptotic analysis requires that the Hessian of the 
symplectic action is self-adjoint operator for a suitable $ L^2 $
inner product. Theorem \ref{thm:comp} below shows that this is only the 
case when the almost complex structure is chosen compatible 
with the symplectic form. 

\begin{PARA}[\bf Vector space setup]\label{lin-setup}
Let $V$ be an even dimensional vector space 
and let a smooth family $ J(t) \in \End(V) $  satisfy $ J(t)^2 = -\Id $, \; $ t \in[0,1] $.
 Let $ \Lambda_0, \Lambda_1 $ be half dimensional 
 subspaces and suppose that $A(t) \in \Aut(V) $ is also a smooth 
 family. Let 
 $$ 
 W^{1,2}_{\Lambda} ( [0,1], V)
  := \{ \xi\in W^{1,2} ( [0,1], V) | \xi(i )\in \Lambda_i; \; i=0,1 \} 
 $$
 Suppose that $ \inner \cdot \cdot _t $ is a smooth family of inner 
 products on $ V$ and denote with 
 $$
 \inner \xi \eta _{L^2} = \int_0^1 \inner {\xi(t)} {\eta(t)} _t dt
 $$
 
 Observe the linear operator 
 \begin{equation}\label{eq:symop}
 \begin{split}
 &D: W^{1,2}_{\Lambda} ( [0,1], V) \to L^2 ( [0,1], V) \\
  &(D\xi) := J(t) \dot{\xi}(t) + A(t) \xi(t)
 \end{split}
 \end{equation}

 \end{PARA}
 \begin{theorem}\label{thm:comp}
 Let $ V, \Lambda_0, \Lambda_1 , D , \inner \cdot \cdot_t$
 be as in \ref{lin-setup}. 
 The operator $D$ is self adjoint for $ \inner \cdot \cdot _{L^2} $ 
 iff the following are satisfied 
\begin{itemize}
\item[i)]$J(t) $ is compatible with $ \langle \cdot, \cdot \rangle_t $, i.e.
$ \om_t := \langle J(t) \cdot, \cdot \rangle_t $ is a non degenerate skew symmetric form, i.e.
 $ \inner {J(t) \cdot}{\cdot}_t = - \inner {\cdot}{J(t)\cdot}_t $. 
\item[ii)] If $ \Phi $ is a solution of 
$$ 
J \dot{\Phi} + A \Phi =0, \Phi(0) = \Id
$$
then 
$ \Phi(t)^*\omega_t= \om_t( \Phi(t) \cdot, \Phi(t) \cdot)=\om_0 $
 for all $ t \in [0,1] $, where 
\item[iii)] $ \Lambda_0 $ is Lagrangian for $ \om_0 $ and $\Lambda_1 $ is 
Lagrangian for $ \omega_1 $. 
\end{itemize}
 \end{theorem}
 \begin{proof}
 In the special case $A=0 $ this theorem is equivalent 
 to the following statement.
 Define the operator $D $ by 
 $$ 
 D \xi:= J(t)\xi  
 $$
 Then the operator $D$ is self adjoint for $ \inner {\cdot}{\cdot}_{L^2} $ 
 if and only if there exists a skew symmetric form 
 $\om: V \times V \to \R $ such that  
  \begin{itemize}
 \item[a)] $\langle \cdot, \cdot \rangle_t = \om (\cdot, J(t) (\cdot)) $ for all $t$. 
 \item[b)] $\Lambda_0, \Lambda_1 $ are Lagrangian for $ \omega $. 
 \end{itemize}
 We prove this special case in the next three steps. 
 Assume without loss of generality $ V = \R^{2n} $ and 
 $ \inner \xi \eta _t = \xi^T Q(t) \eta $ for all $ t$. 
 Here $ Q(t) $ are symmetric matrices.  The operator $D$ is self adjoint 
 if and only if
 $$ 
 0= \langle \xi, D\eta \rangle_{L^2}- \langle D\xi, \eta \rangle_{L^2} 
 $$
 Then for all  $\xi(t), \eta(t) \in V $ 
 with $ \xi(i), \eta(i)\in \Lambda_i, \; i=0,1 $ we have
 \begin{align}\label{eq:jed1}
 0&= \int_0^1 \left (\langle \xi, Q J \dot{\eta}\rangle - \langle J \dot{\xi} , Q \eta \rangle  \right) dt \notag \\
   & = \int_0^1\langle \xi, Q J \dot{\eta} + \frac{d}{dt} ( J^T Q \eta)\rangle dt - 
           \overbrace{\langle  J(1) \xi(1), Q(1) \eta(1) \rangle + \langle  J(0) \xi(0), Q(0) \eta(0) \rangle}^{boundary \;cond.} \notag \\
   &= \int_0^1 \langle \xi, ( QJ + J^T Q ) \dot\eta \rangle + \int_0^1 \langle \frac{d}{dt}( QJ) \xi, \eta \rangle + boundary \;cond.
    \end{align}
    The second equality follows by partial integration.
   We prove that the equality \eqref{eq:jed1} implies the following
  \begin{itemize}
  \item[1.] $QJ +J^T Q=0$. 
  \item[2.] $QJ$ is constant. 
   \item[3.] $ Q(0) J(0) \Lambda_0 \perp \Lambda_0 $ and $ Q(1) J(1) \Lambda_1 \perp \Lambda_1 $
  \end{itemize}
  {\bf Step 1.} 
  \begin{equation}\label{eq:st1}
  Q(t) J(t) + J^T(t) Q(t) = 0 , \;\; \forall t .
  \end{equation}
  \begin{proof}
  Suppose that $ S(t_0) = ( QJ+ J^TQ)(t_0) \neq 0 $ 
  for some $ t_0 $ with $ 0< t_0 <1 $. 
  Choose $ \xi_0, \eta_0 $ such that 
  $$
  \langle \xi_0, S(t_0) \eta_0 \rangle=1
  $$
  Choose smooth functions $ \alpha $ and $ \beta $ with 
  compact support in $ [t_0-2\eps^2, t_0+2\eps^2] $ such that 
  $ \beta ( t_0 + t) = \frac{t}{\epsilon}, \;\; \abs{t}< \eps^2 $ and $ \alpha( t_0+t) = 1 $ 
  for $ \abs{t}< \epsilon^2 $.  
  Let $ \xi = \alpha \xi_0 $ and $ \eta= \beta \eta_0 $. 
  Thus it follows that 
  $$
   \int_0^1 \langle \partial_t ( QJ) \xi, \eta \rangle_t \leq c \eps^3 
   $$
   and 
   $$
   \int_0^1 \langle \xi, S \dot{\eta} \rangle \geq \delta \epsilon^2 
   $$
   Thus it follows that the right hand side of \eqref{eq:jed1} is different from zero, 
   what is a contradiction. 
   \end{proof}
\noindent {\bf Step 2.} By Step 1 we have 
    \begin{align}
    0 &= \langle \xi, D \eta \rangle_{L^2}  -  \langle D \xi, \eta \rangle_{L^2} \notag \\
      &= \langle Q(0) J(0) \xi(0), \eta(0)\rangle - \langle Q(1) J(1) \xi(1), \eta(1) \rangle +
       \int_0^1 \langle (\partial_t ( QJ)) \xi , \eta\rangle dt 
    \end{align}
 and the previous equality holds for all $ \xi, \eta $ with $ \xi(i), \eta(i) \in \Lambda_i $, 
 thus obviously we have 
 \begin{itemize}
 \item[i)] $\p_t ( QJ)= 0 $
 \item[ii)] $ Q(0) J(0) \Lambda_0 \perp \Lambda_0 $ and $ Q(1) \Lambda_1 \perp \Lambda_1 $. 
 \end{itemize}
 
 {\bf Step 3.} In Steps 1 and 2 we have proved that 
 \begin{itemize}
 \item[(a)] $ QJ= - (QJ)^T $ ( Step 1) \\
 \item[(b)] $ QJ\equiv \text{constant} $ ( Step 2)\\
 \item[(c)] $ Q(0) J(0) \Lambda_0 \perp \Lambda_0 $ and $ Q(1) \Lambda_1 \perp \Lambda_1 $. 
 Define 
 $$
 \om : \R^{2n} \times \R^{2n} \to \R 
 $$
 by 
 $$ 
 \om( \xi, \eta ):= \langle QJ \xi, \eta \rangle
 $$
 \end{itemize}
 This is independent of $t$ by $(b)$ 
 and skew symmetric by $(a)$ and non degenerate 
 by assumption. Moreover 
 \begin{itemize}
 \item 
 \begin{align*}
  \om( \xi, J(t) \eta ) &= \langle Q(t) J(t) \xi(t), J(t) \eta(t)\rangle \stackrel{(a)}{=} - \langle J(t)^T Q(t) \xi(t),  J(t)\eta(t) \rangle \\
  &= - \langle Q(t) \xi(t), J(t)^2 \eta(t) \rangle = \langle Q(t) \xi(t), \eta(t) \rangle = \langle \xi(t), \eta(t) \rangle_t 
  \end{align*}
  \item $\om ( \xi_0, \eta_0 ) = \langle Q(0) J(0) \xi_0, \eta_0 \rangle =0 $ for all $ \xi_0, \eta_0 \in \Lambda_0 $ and similarly 
  $\om( \xi_1, \eta_1 ) = \langle Q(1) J(1) \xi_1, \eta_1 \rangle =0 $ for all $ \xi_1, \eta_1 \in \Lambda_1 $.
   \end{itemize} 
   \end{proof}
    {\bf Step 4.} {Reducing the general case to the case $A=0$.} \\
 Conjugating the operator $ D$ with $ \Phi $, we 
 reduce the general case to the case that $ A=0 $. 
 Then the proof follows from the first three steps. 
 More precisely let $ \xi = \Phi \txi $. Then we have 
 \begin{align*}
 \tD \txi &= \Phi^{-1} D \Phi \txi \\
 	    &=\Phi^{-} ( J \Phi \dot \txi + J \dot \Phi \txi + A \Phi \txi )\\
	    &= \tJ \dot \txi + \underbrace{( \phi^{-1} J \dot{\Phi} + \Phi^{-1} A \Phi )}_{\tA} \txi
 \end{align*}
 Assume that $\tA=0$. Thus we have that the operator 
 $$
 \tD\txi= \tJ \dot{\txi} 
 $$
 and it is self adjoint with respect to 
 $$ 
 \int_0^1 \inner {\Phi\txi}{\Phi \teta}_t dt
 $$
 It follows from the first three steps that there exist a
 two form $ \om : V \times V \to \R $ skew symmetric and 
 non degenerate such that 
 \begin{align*}
 \om( \txi, \tJ(t) \teta ) &= \inner {\Phi(t) \txi}{\Phi(t) \teta}_t\\
   \om(\txi, \Phi^{-1} J \Phi \teta)&= \inner {\Phi \txi}{ \Phi \teta}_t                       
 \end{align*}
 Thus we have that 
 $\om ( \Phi^{-1} (t) \xi, \Phi^{-1}J(t) \eta ) = \inner {\xi}{\eta} _t $ 
and the smooth family of $2-$forms $ \omega_t:= \om ( \Phi(t)^{-1} \cdot, \Phi(t)^{-1} \cdot ) $ satisfies 
$$
\om_t ( \cdot, J(t) \cdot ) = \inner {\cdot}{\cdot}_t  
$$
We also have that 
\begin{itemize}
\item[1)] $ \om_t := \langle J(t) \cdot, \cdot \rangle_t = - \langle \cdot, J(t) \cdot \rangle_t$. \\
\item[2)] $ \Lambda_0 $ is Lagrangian for $ \om_0 $ and $ \Lambda_1 $ is Lagrangian for 
$ \om_1 $. 
\item[3)] If $J\dot{\Phi} + A \Phi=0 $ , $ \Phi(0) = \Id $ then 
 $$
  \om_t ( \Phi(t) \xi, \Phi(t) \eta ) = \om_0 ( \xi, \eta) 
 $$
\end{itemize}
 
 \begin{PARA} [{\bf Question}] 
 Given $ \Lambda_0 $ and $ \Lambda_1 \subset V $ and 
 $ J(t) \in \sJ(V) $ such that $ J(0) \Lambda_0 \pitchfork \Lambda_0 $ 
 and $ J(1) \Lambda_1  \pitchfork \Lambda_1$ and such that there 
 exist a non degenerate skew form $ \om_0 $ such that $J(t) $ are tame with $ \om_0 $
 for all time $t$. 
 Does there exist a non degenerate skew form $ \om : V \times V \to \R $ such that 
 \begin{itemize}
 \item[1)]  $\Lambda_0, \Lambda_1 $ are Lagrangian for $ \om$
 \item[2)] $ J(t) $ are compatible with $ \om $ for all $t$. 
 \end{itemize}
 \end{PARA}  
   The answer to this question is no.
   More precisely there exists $ \Lambda_0, \Lambda_1 $ 
   and $ J(t) $ which satisfy the above condition  and a non degenerate 
   skew form $ \om_0 $ such that $ J(t) $ 
   are tamed by $ \om_0 $ for all $t$, but there doesn't exist 
   $ \om $ such that $ J(t) $ are compatible with $\om $ 
   for all $t$ and such that $ \Lambda_i , \; i=0,1$ are Lagrangian for $ \om$. 
  
\begin{PARA}[{\bf Counterexample.}] \label{para:counterexm}\rm
   Obviously, we cannot look for counterexample in the dimension $2$, 
   as here compatibility is the same as the tame condition. Thus we 
   can suppose that $ V= \R^4$. We take $ \Lambda_0 = \R^2\times \{0\} $ 
   and $ \Lambda_1 = \{0\} \times \R^2$. The standard symplectic form 
   $ \om_0= \sum\limits_{i} dx_i \wedge dy_i $ is given by 
   $$ 
   \om_0 ( z, z') = \inner x {y'} - \inner y{x'} 
   $$ 
   Let $ B \in \R^{2\times 2} $ be any matrix with $ \text{Det}(B) \neq 0 $.
    Observe the $2-$form 
    $$ 
    \om_B ( z, z'):= \inner{x}{By'} - \inner{y}{Bx'}
    $$
    Notice that $ \om_B = \sum\limits_{i,j} b_{ij} dx_i \wedge dy_j $. 
   Any symplectic form in $ \R^4$ can be written in the form 
   $$
   \om= a\; dx_1\wedge dx_2 + c\; dy_1 \wedge dy_2 + \om_B 
   $$
   Notice that $ \Lambda_0= \R^2 \times \{0\} $ is Lagrangian 
   if and only if $ a=0 $ and similarly $ \Lambda_1 $ is Lagrangian 
   if and only if $ c=0 $. Thus both $ \Lambda_i , \; i=0,1$ are Lagrangian 
   if and only if $ \om = \om_B $. 
   
   Let $ A(t) \in \R^{2\times 2} $ with $ \text{Det}(A(t)) \neq 0 $. 
   Observe the following smooth family of $J(t) $. 
   \[J(t): = \left( \begin{array}{cc}
	 0 & - A(t)^{-1}  \\
	A(t) & 0  
\end{array} \right)\] 
Then $ J(t) $ are compatible with $ \om_0 $ if and only if $ A(t) = A(t)^T > 0 $ 
and $ J(t) $ are tame with $ \om_0 $ if and only if $ A(t) + A(t)^T > 0 $. Similarly
we have that $ J(t) $ are compatible with $ \om_B $ if and only if $ BA =  (BA)^T = A^T B^T>0$.
Consider the following three matrices  $ A_1 = \Id $
\[ A_2 = \left ( \begin{array}{cc}
				1 & 0 \\
				0 & 1+ \eps
			\end{array}
			\right) \hspace{1cm}
			A_3 = \left ( \begin{array}{cc}
				1 & \eps\\
				\eps & 1
			\end{array}
			\right) \]
These three matrices all satisfy the condition $ A= A^T$ and they are all positive definite 
for small $ \epsilon $. Observe the fourth matrix $ A_4 $ given by 
\[A_4 = \left ( \begin{array}{cc}
				1 & \epsilon\\
				0 & 1
			\end{array}
			\right)
			\]
			
Then such matrix $ A_4 $ is obviously not symmetric but it 
satisfies the condition $ A_4 + A_4^T > 0 $. 
\begin{lemma}\label{lem:aux}
Let $ A_i, \;\; i=1,2,3 $ be as above and let $ B\in \R^{2\times 2 } $ with $ \text{det}(B)\neq 0 $
be such that 
\begin{equation}\label{eq:cond}
BA_i = {A_i}^T B^T. 
\end{equation}
Then there exists a constant $ \lambda $ such that $ B = \lambda \Id $. 
\end{lemma}
\begin{proof}
Notice first that in the case of $ A_1 = \Id $ the above condition is equivalent to
$$
B=  B^T= \left ( \begin{array}{cc}
				b_{11} & b_{12}\\
				b_{12} & b_{22}
			\end{array}
			\right)
$$ 
If the equality \eqref{eq:cond} is satisfied for $ A_1 $ and $ A_2 $ 
then it also holds for their difference, as well 
as for  $ M_1=\frac{1}{\eps}( A_2 - A_1) $.
From the equality $ B M_1 = M_1^T B^T $ we obtain that  $b_{12}=0 $. 
Similarly we have that the equality \eqref{eq:cond} 
also holds for the matrix $ M_2= \frac{1}{\eps} (A_3 - A_1) $.
From the equality $ B M_2 = M_2^T B^T $ we obtain $ b_{11}= b_{22} =\lambda$.  

\end{proof}

Now we can finish the counterexample.
 Take a path  $A(t) $ going through all these matrices ( $ A_i, \; i=1, \cdots , 4$)
 and such that $ A(t) + A(t)^T > 0 $. For example we can take $ A(0) = A_1 $, 
 $ A(\frac{1}{3})= A_2 $ , $ A(\frac{2}{3})= A_3 $ and $ A(1)= A_4$. 
Then $J(t) $ is tamed by $ \om_0 $ for all $t$, but there doesn't exist 
a $2-$form $ \omega $ such that $J(t) $ are compatible with $\om $ 
and that $ \Lambda_i , \;\; i=1,2 $ are Lagrangian subspaces. 
That follows from Lemma \ref{lem:aux}. If it would exit such $\om$, 
then we would have that $ \om = \om_B $ and the matrix $B $ 
satisfies $BA= A^T B^T $ for all $t\in [0,1] $. Particularly the matrix 
$B$ satisfies that condition for $ t= \frac{i}{3} , \; i=0,1,2 $. 
Then it follows from Lemma \ref{lem:aux} that $ B = \lambda \Id $, 
but the fourth matrix  $A_4$ isn't symmetric and this is the contradiction! 
\end{PARA}

\chapter{Applications in Lagrangian Floer homology}\label{ch:hardy_floer}

One of the three main technical ingredients in Lagrangian
Floer theory is the Floer gluing theorem. (The other two 
are Floer--Gromov compactness and the linear elliptic 
Fredholm theory, including the
Fredholm-index-equals-Maslov-index theorem).
In chapter \ref{chp:hardy} we introduce a new 
approach to Lagrangian Floer gluing using nonlinear 
Hardy spaces.  The purpose of the present chapter 
is to explain how the main result in chapter \ref{chp:hardy} 
implies the relevant gluing theorems in Lagrangian 
Floer theory (see Floer~\cite{F1,F2} and Oh~\cite{OH}) 
via intersection theory in a path space.  
More precisely, we prove the following results.

\smallskip\noindent{\bf I: Boundary Map.}
The mod two count of solutions of the Floer equation 
for regular Floer data defines a map
with square zero (Theorem~\ref{thm:bo}).

\smallskip\noindent{\bf II: Chain Map.}
The mod two count of solutions of the time dependent 
Floer equation, associated to a regular homotopy of Floer data,
defines a homomorphism that intertwines the Floer boundary 
operators (Theorem~\ref{thm:isotopy1}).

\smallskip\noindent{\bf III: Chain Homotopy Equivalence.}
The induced morphism on Floer homology in~II is 
independent of the choice of the homotopy
(Theorem~\ref{thm:isotopy2}).

\smallskip\noindent{\bf IV: Catenation.}
Two composable morphisms on Floer homology as in~III 
satisfy the composition rule under catenation of homotopies 
(Theorem~\ref{thm:isotopy3}).

The exposition here is restricted to monotone 
Lagrangian submanifolds with minimal Maslov 
numbers at least three. 
On the other hand we do not impose any restrictions
on the fundamental groups or on the monotonicity factors
of the Lagrangian submanifolds and hence it is necessary 
to work with Novikov rings.  
This chapter represents joint work with Prof. D. Salamon. 

\section{Floer Homology} \label{sec:FLOER}  

This section is of expository nature.  
It discusses the basic setup of Lagrangian 
Floer theory in the monotone case (see~\cite{F1,F2,OH}).  
More precisely, we consider the following setting.

\begin{description}
\item[(H)]
{\it $(M,\om)$ is a compact symplectic manifold without 
boundary and 
$$
L_0,L_1\subset M
$$ 
are compact Lagrangian submanifolds without boundary.  
For $i=0,1$ the pair $(M,L_i)$ is monotone with minimal 
Maslov number at least three, i.e.\ for every
smooth map $u:(\D,\p\D)\to(M,L_i)$
the Maslov number $\mu(u)$ has absolute value at least three, 
and there is a constant $\tau_i>0$ such that 
$$
\int_\D u^*\om = \tau_i\mu(u)
$$
every smooth map $u:(\D,\p\D)\to(M,L_i)$.}
\end{description}

\begin{PARA}[{\bf The Floer Equation and the Energy Identity}]\label{para:FE}\rm
Fix a regular Hamiltonian function 
$H=\{H_t\}_{0\le t\le 1}\in\Hreg(M,L_0,L_1)$ 
and a smooth family of $\om$-tame almost complex structures
$J=\{J_t\}_{0\le t\le 1}\in\cJ(M,\om)$.  
The {\bf Floer equation} has the form
\begin{equation}\label{eq:FLOER}
\p_su+J_t(u)(\p_tu-X_{H_t}(u))=0,\qquad
u(s,0)\in L_0,\qquad u(s,1)\in L_1,
\end{equation}
for a smooth map ${u:\R\times[0,1]\to M}$.
The {\bf energy} of a solution $u$ of~\eqref{eq:FLOER} 
is defined by 
$$
E_H(u) 
:=
\frac12\int_{-\infty}^\infty\int_0^1
\Bigl(\abs{\p_su}_t^2+\abs{\p_tu-X_{H_t}(u)}_t^2
\Bigr)\,dtds.
$$
Here $\inner{\xi}{\eta}_t:=\tfrac12(\om(\xi,J_t\eta)+\om(\eta,J_t\xi))$
denotes the Riemannian metric determined by $\om$ and $J_t$.
If the energy is finite then the limits 
\begin{equation}\label{eq:LIMIT}
x^\pm(t) := \lim_{s\to\pm\infty}u(s,t)\in L_0\cap L_1
\end{equation}
exist and belong to $\cC(L_0,L_1;H)$ (see for example~\cite{RS3}).
The convergence is with all derivatives, uniform in $t$,
and exponential. For two solutions $x^\pm$ of~\eqref{eq:CRIT} 
denote the space of finite energy solutions of~\eqref{eq:FLOER}
and~\eqref{eq:LIMIT} by
$$
\cM(x^-,x^+;H,J) := \left\{u:\R\times[0,1]\to M\,\big|\,
\eqref{eq:FLOER},\,\eqref{eq:LIMIT},\,E_H(u)<\infty
\right\}.
$$
(When $H=0$ we abbreviate $\cM(x^-,x^+;J):=\cM(x^-,x^+;H,J)$.)
Thus $\cM(x^-,x^+;H,J)$ is the space of Floer trajectories 
from $x^-$ to $x^+$.
Every finite energy solution of~\eqref{eq:FLOER} 
and~\eqref{eq:LIMIT} satisfies the {\bf energy identity} 
\begin{equation}\label{eq:ENERGY}
E_H(u) = \int_{\R\times[0,1]} u^*\om 
- \int_0^1H_t(x^-(t))\,dt +  \int_0^1H_t(x^+(t))\,dt.
\end{equation}
\end{PARA}

\begin{PARA}[{\bf Regular Pairs}]\label{para:HJ}\rm
A family $J\in\cJ(M,\om)$ is called 
{\bf regular for $L_0$, $L_1$, $H$} if every finite 
energy solution $u:\R\times[0,1]\to M$ of~\eqref{eq:FLOER} 
is regular in the sense that the linearized operator $D_u$ 
is surjective. The set of regular families $J\in\cJ(M,\om)$ 
will be denoted by $\Jreg(M,L_0,L_1,H)$.  It is a residual subset 
of the space of $\cJ(M,\om)$ (see Floer~\cite{F1,F2}).
A pair $(H,J)\in\cH(M)\times\cJ(M,\om)$ is called
a {\bf regular pair for $(L_0,L_1)$}
if $H\in\Hreg(M,L_0,L_1)$ and $J\in\Jreg(M,L_0,L_1,H)$.
The set of regular pairs for $(L_0,L_1)$ is a residual 
subset of $\cH(M)\times\cJ(M,\om)$ and will be 
denoted by 
$$
\HJreg(M,L_0,L_1)
:=\left\{(H,J)\,\big|\,
\begin{array}{l}
H\in\Hreg(M,L_0,L_1),\\
J\in\Jreg(M,L_0,L_1,H)
\end{array}
\right\}. 
$$
When $L_0\ti L_1$ and $H=0$
we write $\Jreg(M,L_0,L_1):=\Jreg(M,L_0,L_1,H)$.
\end{PARA}

\begin{PARA}[{\bf Novikov Rings}]\label{para:Novikov}\rm
Denote the {\bf universal Novikov ring} with 
$\Z_2$ coefficients by
$$
\Lambda := \left\{
\sum_{\eps\in\R}\lambda_\eps e^{-\eps}\,\Big|\,
\lambda_\eps\in\Z_2,\,
\#\{\eps\le c\,|\,\lambda_\eps\ne 0\}<\infty
\;\forall\,c\in\R\right\}.
$$ 
Each element of $\Lambda$ can be thought of as a function 
$\R\to\Z_2:\eps\mapsto\lambda_\eps$ with finite support over
each half infinite interval $(-\infty,c]$.  The universal
Novikov ring is a field with multiplication 
$
\lambda\lambda':=\sum_\eps\sum_\delta
\lambda_\delta\lambda'_{\eps-\delta}e^{-\eps}.
$
\end{PARA}

\begin{PARA}[{\bf The Floer Chain Complex}]\label{para:Floer}\rm
Assume $(M,L_0,L_1)$ satisfy~(H) and let 
$(H,J)\in\HJreg(M,L_0,L_1)$.
The {\bf Floer chain complex} of $L_0$, $L_1$, $H$
is the vector space over $\Lambda$ generated by the
solutions of~\eqref{eq:CRIT}. It is denoted by
\begin{equation}\label{eq:CF}
\CF_*(L_0,L_1;H) := \bigotimes_{x\in\cC(L_0,L_1;H)}\Lambda x.
\end{equation}
When $H=0$ this chain complex is generated 
by the intersection points of $L_0$ and $L_1$.  
Under our assumptions the space $\cM(x^-,x^+;H,J)$ 
of Floer trajectories is a smooth manifold 
whose local dimension near $u\in\cM(x^-,x^+;H,J)$
is the Viterbo--Maslov index $\mu_H(u)$ 
(see~\cite{F1,F2,RS1,RS2,VITERBO}).
For every integer $k\ge 0$ and every constant $\eps>0$ 
denote the space of Floer trajectories from $x^-$ to $x^+$ 
with Viterbo--Maslov index $k$ and energy $\eps$ by
$$
\cM^k_\eps(x^-,x^+;H,J) := \left\{u\in\cM(x^-,x^+;H,J)\,|\,
\mu_H(u)=k,\,E_H(u)=\eps\right\}.
$$
This space is a $k$-dimensional manifold and (for $k>0$)
it carries a free and proper action of $\R$ by translation.
The quotient 
$$
\widehat{\cM}_\eps^k(x^-,x^+;H,J):=\cM_\eps^k(x^-,x^+;H,J)/\R
$$ 
is a manifold of dimension $k-1$.  
It follows from our hypotheses that
\begin{equation}\label{eq:FINITE}
\sum_{\eps\le c}\#\widehat{\cM}^1_\eps(x^-,x^+;H,J)<\infty
\qquad\forall\;c>0.
\end{equation}
Define the operator 
$
\p=\p^{H,J}:\CF_*(L_0,L_1;H)\to\CF_*(L_0,L_1;H)
$
by
\begin{equation}\label{eq:dCF}
\p^{H,J} x := \sum_{y\in\cC(L_0,L_1;H)}\sum_\eps
\#\widehat{\cM}^1_\eps(x,y;H,J)e^{-\eps}y
\end{equation}
for $x\in\cC(L_0,L_1;H)$.  The following theorems
assert that~\eqref{eq:dCF} is indeed a boundary operator 
and that the resulting Floer homology groups are invariant 
under Hamiltonian isotopy.  The original proof by 
Floer~\cite{F1,F2} was carried out under the assumption 
$\pi_2(M,L_i)=0$. Floer's results were later extended 
to the monotone setting by Oh~\cite{OH}. 
\end{PARA}

\begin{theorem}[{\bf Boundary Operator}]\label{thm:bo}
Assume~(H) and let $(H,J)$ be a regular pair for $(L_0,L_1)$.
Let $\p^{H,J}:\CF_*(L_0,L_1;H)\to\CF_*(L_0,L_1;H)$
be defined by~\eqref{eq:dCF}. Then
$
\p^{H,J}\circ\p^{H,J} =0.
$
\end{theorem}

\begin{proof}
See Section~\ref{sec:GLUING}.
\end{proof}

\begin{PARA}[{\bf Floer Homology}]\label{para:HF}\rm
The homology 
$$
\HF_*(L_0,L_1;H,J) 
:= \frac{\ker\,\p^{H,J}}{\im\,\p^{H,J}}
$$
of the chain complex in Theorem~\ref{thm:bo}
is called the {\bf Floer homology group of $(L_0,L_1)$} 
associated to the regular pair $(H,J)$. The Floer homology 
of $(L_0,L_1)$ is independent of the choice of the
regular pair $(H,J)$ up to canonical isomorphism. 
\end{PARA}
\begin{theorem}[{\bf Invariance}]\label{thm:HF}
Assume~(H). There is a collection of isomorphisms
$
\Phi^{\beta\alpha}:\HF_*(L_0,L_1;H^\alpha,J^\alpha)
\to  \HF_*(L_0,L_1;H^\beta,J^\beta),
$
one for any two regular pairs $(H^\alpha,J^\alpha)$
and $(H^\beta,J^\beta)$, satisfying
$$
\Phi^{\gamma\beta}\circ\Phi^{\beta\alpha}
= \Phi^{\gamma\alpha},\qquad
\Phi^{\alpha\alpha}=\id
$$
for all $(H^\alpha,J^\alpha), (H^\beta,J^\beta), 
(H^\gamma,J^\gamma)\in\HJreg(M,L_0,L_1)$.
\end{theorem}

\begin{PARA}[{\bf Naturality}]\label{para:NAT}\rm
Assume~(H) and let $(H,J)\in\HJreg(M,L_0,L_1)$. Let 
$
[0,1]\to\Diff(M,\om):t\mapsto\psi_t
$ 
be a Hamiltonian isotopy with corresponding family of Hamiltonian
functions $K_t:M\to\R$, $0\le t\le 1$, so that
\begin{equation}\label{eq:psiK}
\p_t\psi_t = Y_t\circ\psi_t,\qquad \iota(Y_t)\om=dK_t.
\end{equation}
We do not assume that $\psi_0$ is the identity.
Let $u:\R\times[0,1]\to M$ be a solution of~\eqref{eq:FLOER}
and define
$$
\tu(s,t) := \psi_t^{-1}(u(s,t)),\qquad
\tL_0 := \psi_0^{-1}(L_0),\qquad \tL_1:=\psi_1^{-1}(L_1),
$$
and
$$
\tH_t := (H_t-K_t)\circ\psi_t,\qquad
\tX_t := \psi_t^*(X_t-Y_t),\qquad 
\tJ_t:=\psi_t^*J_t.
$$
Then $\tH_t$ generates the Hamiltonian isotopy
$\widetilde{\phi}_t:=\psi_t^{-1}\circ\phi_t\circ\psi_0$
and $\tu$ satisfies the Floer equation
\begin{equation}\label{eq:tFLOER}
\p_s\tu + \tJ_t(\tu)(\p_t\tu-\tX_t(\tu)) = 0,\qquad
\tu(s,0)\in\tL_0,\qquad \tu(s,1)\in\tL_1.
\end{equation}
Thus pullback by $\psi_t$ induces an isomorphism
of Floer homology groups
$$
\psi^*:\HF_*(L_0,L_1;H,J)\to\HF_*(\tL_0,\tL_1;\tH,\tJ).
$$
\end{PARA}

\begin{PARA}[{\bf Regular Floer Data and Naturality}]
\label{para:FloerData}\rm
Let $(M,\om)$ be a compact symplectic manifold without boundary 
and denote by $\Freg=\Freg(M,\om)$ the set of {\bf regular Floer data} 
$(L_0,L_1,H,J)$, where ${L_0,L_1\subset M}$ are Lagrangian 
submanifolds satisfying~(H) and $(H,J)$ is regular pair
for $(L_0,L_1)$. The group $\sG=\sG(M,\om)$ of Hamiltonian 
isotopies $\psi=\{\psi_t\}_{0\le t\le1}$ of $(M,\om)$ 
(starting at any symplectomorphism) acts 
contravariantly on $\Freg$ via 
$$
\psi^*(L_0,L_1,H,J) 
:= (\psi_0^{-1}(L_0),\psi_1^{-1}(L_1),\psi^*H,\psi^*J),
$$
where 
$$
(\psi^*H)_t := (H_t-K_t)\circ\psi_t,\qquad
(\psi^*J)_t:=\psi_t^*J_t,
$$
and $K_t$ is a family of Hamiltonian 
functions generating $\psi_t$ via~\eqref{eq:psiK}.
More precisely, the homomorphism 
$
\pi_2(M,L_i)\to\pi_2(M,\psi_i^{-1}(L_i))
:u\mapsto\psi_i^{-1}\circ u
$
preserves the Maslov index and the symplectic area
for $i=0,1$. Hence the pair $(\psi_0^{-1}(L_0),\psi_1^{-1}(L_1))$
satisfies~(H) whenever $(L_0,L_1)$ does. 
Second, it follows from~\ref{para:NAT} that
$(\psi^*H,\psi^*J)\in\HJreg(M,\psi_0^{-1}(L_0),\psi_1^{-1}(L_1))$
whenever $(H,J)\in\HJreg(M,L_0,L_1)$.
\end{PARA}

\begin{theorem}[{\bf Lagrangian Seidel Homomorphism}]\label{thm:seidel}
Fix a compact symplectic manifold $(M,\om)$ without boundary. 
There is a collection of isomorphisms
$$
\psi^*:\HF_*(L_0,L_1;H,J)\to 
\HF_*(\psi_0^{-1}(L_0),\psi_1^{-1}(L_1);\psi^*H,\psi^*J),
$$
one for every $(L_0,L_1,H,J)\in\Freg$ and 
every $\psi\in\sG$, satisfying the following.

\smallskip\noindent{\bf (Functoriality)}
For all $(L_0,L_1,H,J)\in\Freg$ and $\phi,\psi\in\sG$ we have
$$
(\psi\phi)^*=\phi^*\circ\psi^*:\HF_*(L_0,L_1)\to 
\HF_*(\phi_0^{-1}(\psi_0^{-1}(L_0)),\phi_1^{-1}(\psi_1^{-1}(L_1))).
$$

\smallskip\noindent{\bf (Naturality)}
Let $(L_0,L_1)$ satisfy~(H), suppose 
$(H^\alpha,J^\alpha),(H^\beta,J^\beta)$ are regular pairs 
for $(L_0,L_1)$, and let $\psi\in\sG$.
Then the following diagram commutes
\begin{equation}\label{eq:diagiso1}
\xymatrix    
@C=20pt    
@R=20pt    
{    
\HF_*(L_0,L_1;H^\alpha,J^\alpha) 
\ar[r]^{\psi^*\qquad\quad} \ar[d]_{\Phi^{\beta\alpha}} 
& \HF_*(\psi_0^{-1}(L_0),\psi_1^{-1}(L_1);\psi^*H^\alpha,\psi^*J^\alpha)
\ar[d]^{\Phi^{\beta\alpha}}  \\ 
\HF_*(L_0,L_1;H^\beta,J^\beta) 
\ar[r]^{\psi^*\qquad\quad}
& \HF_*(\psi_0^{-1}(L_0),\psi_1^{-1}(L_1);\psi^*H^\beta,\psi^*J^\beta)
}.
\end{equation}

\smallskip\noindent{\bf (Isotopy)}
Let $(L_0,L_1,H^\alpha,J^\alpha)\in\Freg$ and  
$\phi,\psi\in\sG$ such that 
$$
\phi_0^{-1}(L_0)=\psi_0^{-1}(L_0)=:\tL_0,\qquad
\phi_1^{-1}(L_1)=\psi_1^{-1}(L_1)=:\tL_1.
$$
Define 
$$
(\tH^\beta,\tJ^\beta):=(\phi^*H^\alpha,\phi^*J^\alpha),\qquad
(\tH^\gamma,\tJ^\gamma):=(\psi^*H^\alpha,\psi^*J^\alpha).
$$
Suppose $\phi$ is isotopic to $\psi$ by a Hamiltonian isotopy 
$\{\psi^\lambda_t\}_{0\le t,\lambda\le1}$ that satisfies
$\psi^\lambda_0(L_0)=\tL_0$ and $\psi^\lambda_1(L_1)=\tL_1$
for all $\lambda$.  Then the following diagram commutes
\begin{equation}\label{eq:seidelisotopy}
\xymatrix    
@C=15pt    
@R=20pt    
{    
& \HF_*(L_0,L_1;H,J) \ar[dl]_{\phi^*} \ar[dr]^{\psi^*} & \\
\HF_*(\tL_0,\tL_1;\tH^\beta,\tJ^\beta) \ar[rr]^{\Phi^{\gamma\beta}}  
& & \HF_*(\tL_0,\tL_1;\tH^\gamma,\tJ^\gamma)  
}
\end{equation}
Here $\Phi^{\gamma\beta}$ is the isomorphism of 
Theorem~\ref{thm:HF}.
\end{theorem}

\begin{proof}
See Section~\ref{sec:ISOTOPY}.
\end{proof}

\begin{remark}\label{rmk:floer}\rm

{\bf (i)}
The Floer homology group $\HF_*(L_0,L_1)$ is a 
{\bf connected simple system} in the sense of Conley, 
i.e.\ a small category with precisely one (iso)morphism 
between any two objects.  The objects are regular 
pairs $(H,J)$ and the morphisms are the isomorphisms 
$\Phi^{\beta\alpha}$ of Theorem~\ref{thm:HF}.
More precisely one can define $\HF_*(L_0,L_1)$ 
as the set of all tuples $\{\xi^\alpha\}_\alpha$
of elements 
$$
\xi^\alpha\in\HF(L_0,L_1;H^\alpha,J^\alpha),
$$
one for each regular pair $(H^\alpha,J^\alpha)\in\HJreg(M,L_0,L_1)$,
that satisfy 
$$
\xi^\beta=\Phi^{\beta\alpha}\xi^\alpha
$$
for all $\alpha,\beta$. This is a 
vector space over the universal Novikov ring $\Lambda$.

\smallskip\noindent{\bf (ii)}
Theorem~\ref{thm:seidel} shows that Floer homology 
is a contravariant functor from the category of Lagrangian pairs 
$L_0,L_1\subset M$ satisfying~(H), where the morphisms 
from $(L_0,L_1)$ to $(\tL_0,\tL_1)$ are homotopy classes
of Hamiltonian isotopies $\{\psi_t\}_{0\le t\le1}$ of $(M,\om)$ 
satisfying 
$$
\psi_0^{-1}(L_0)=\tL_0,\qquad \psi_1^{-1}(L_1)=\tL_1,
$$
to the category of vector spaces over $\Lambda$. 

\smallskip\noindent{\bf (iii)}
The Floer chain complex is not equipped with a natural grading. 
A grading can be introduced via Seidel's notion of graded
Lagrangian submanifolds, but we shall not discuss this here.
If the Lagrangian submanifolds are oriented the Floer homology
groups are graded modulo two by the intersection indices.

\smallskip\noindent{\bf (iv)}
In favorable cases the moduli space $\cM(x,y;H,J)$ 
of Floer trajectories are orientable and the Floer 
homology groups can be defined over the integers.
However, we shall not discuss this here.

\smallskip\noindent{\bf (v)}
The definition of the Floer homology groups of pairs 
of monotone Lagrangian submanifolds can sometimes be extended 
to the case where the minimal Maslov number is two.
In this case, for $i=0,1$ and a generic almost complex 
structures $J_i$, there is a finite number of $J_i$-holomorphic
Maslov index two discs in $M$ with boundary in $L_i$,
passing through a generic point in $L_i$. 
The parity $\eps_i\in\{0,1\}$ of this number 
is independent of $J_i$ and of the chosen point in $L_i$.
The Floer homology groups can still be defined 
if either a) $\eps_0=\eps_1=0$ or b) $\eps_0+\eps_1=0$
and the factors $\tau_0=\tau_1$ in the 
definition of monotonicity agree.

\smallskip\noindent{\bf (vi)}
One can define the Floer homology groups with 
coefficients in $\Z_2$ if in~(H2) we have 
$\tau_0=\tau_1=:\tau$ and in addition the fundamental 
group of the space $\sP$ of paths from $L_0$ to 
$L_1$ is generated, modulo torsion, by $\pi_2(M,L_0)$
and $\pi_2(M,L_1)$. This the case whenever the image 
of the homomorphism 
$
\pi_1(L_i)\to\pi_1(M)
$
consists of torsion classes for $i=0,1$ (see Oh~\cite{OH}).
Here we do not make these assumptions and 
instead work with the Novikov ring~$\Lambda$.

\smallskip\noindent{\bf (vii)}
Let $(S,\sigma)$ be a compact monotone symplectic manifold
without boundary and let $\phi:S\to S$ be a symplectomorphism.
Choose 
$
M:=S\times S
$
with the product symplectic form 
$
\om:=\pi_2^*\sigma-\pi_1^*\sigma
$ 
and let $L_0:=\Delta$ (the diagonal in $M\times M$)
and $L_1:=\mathrm{graph}(\phi)$. 
Then there is a natural isomorphism
$$
\HF_*(\Delta,\mathrm{graph}(\phi))\cong\HF_*(\phi).
$$
(See~\cite{DS} for the definition of $\HF(\phi)$
and~\cite{BPS} for the isomorphism.)
\end{remark}
\begin{PARA}[{\bf Outline}]\label{para:Outline}\rm
The proof of the identity 
$
\p^2=0
$
in Theorem~\ref{thm:bo}
is based on the study of the moduli space $\cM^2(x,z;H,J)$ of 
index-$2$ solutions of the Floer equation~\eqref{eq:FLOER}
for two solutions $x,z$ of~\eqref{eq:CRIT}. 
The $1$-dimensional quotient space $\widehat{\cM}^2(x,z;H,J)$ 
will in general not be compact. It can be compactified by 
including the zero dimensional product spaces 
$\widehat{\cM}^1(x,y;H,J)\times\widehat{\cM}^1(y,z;H,J)$ 
over all solutions $y$ of~\eqref{eq:CRIT}.  
This is the content of the Floer gluing theorem
(Section~\ref{sec:GLUING}). In the present chapter
we reduce Floer gluing to intersection theory in a
path space (Theorem~\ref{thm:main_thm3}).  
This requires a monotonicity result 
for $J$-holomorphic curves with Lagrangian 
boundary conditions and a 
convergence theorem for a suitable family of nonlinear 
Hardy spaces.   
\end{PARA}

\section{Invariance} \label{sec:ISOTOPY}  

\begin{PARA}[{\bf Homotopy of Floer Data}]\label{para:Homotopy}\rm
Assume~(H) and let 
$$
(H^\alpha,J^\alpha), (H^\beta,J^\beta)\in\HJreg(M,L_0,L_1)
$$ 
be two regular pairs for $(L_0,L_1)$ (see~\ref{para:HJ}).
Choose two Hamiltonian functions
$
F,G:\R\times[0,1]\times M\to\R
$
and a smooth family $J=\{J_{s,t}\}_{(s,t)\in\R\times[0,1]}$ 
of $\om$-tame almost complex structures on $M$ 
such that
\begin{equation}\label{eq:F01}
F_{s,0}|_{L_0} = \mathrm{constant},\qquad
F_{s,1}|_{L_1} = \mathrm{constant}
\end{equation}
for every $s\in\R$ and, for some $T>0$,
$F_{s,t}=0$ for $\Abs{s}\ge T$ and
\begin{equation}\label{eq:GJends}
G_{s,t} = \left\{\begin{array}{ll}
-H^\alpha_t,&\mbox{if }s\le -T,\\
-H^\beta_t,&\mbox{if }s\ge T,
\end{array}\right.\qquad
J_{s,t} = \left\{\begin{array}{ll}
J^\alpha_t,&\mbox{if }s\le -T,\\
J^\beta_t,&\mbox{if }s\ge T.
\end{array}\right.
\end{equation}
Here we denote $F_{s,t}:=F(s,t,\cdot)$ and $G_{s,t}:=G(s,t,\cdot)$.
Such a triple $(F,G,J)$ is called a 
{\bf homotopy from $(H^\alpha,J^\alpha)$ to $(H^\beta,J^\beta)$}.
Condition~\eqref{eq:F01} guarantees that 
$L_0$ is invariant under the Hamiltonian 
isotopy generated by $F_{s,0}$ and 
$L_1$ is invariant under the Hamiltonian 
isotopy generated by $F_{s,1}$.
Equivalently, $\tL_0:=\R\times L_0$ and 
$\tL_1:=R\times L_1$ are Lagrangian
submanifolds of the symplectic manifold 
$\tM:=\R\times[0,1]\times M$ with the symplectic form 
$\tom:=\om-d^{\tM}(F\,ds+G\,dt)+cds\wedge dt$ 
(where $c>\max(\p_sG-\p_tF+\{F,G\}$).

Associated to every such homotopy is the 
{\bf time dependent Floer equation}
for a smooth map $u:\R\times[0,1]\to M$.
It has the form
\begin{equation}\label{eq:FLOERFG}
\begin{split}
&\p_su + X_{F_{s,t}}(u) + J_{s,t}(u)\left(\p_tu+X_{G_{s,t}}(u)\right)=0,\\
&u(s,0)\in L_0,\qquad u(s,1)\in L_1.
\end{split}
\end{equation}
For $s\in\R$ and $t\in[0,1]$ denote by
$
\abs{\xi}_{s,t}:=\sqrt{\om(\xi,J_{s,t}\xi)}
$ 
the Riemannian metric associated to $\om$ and $J_{s,t}$. 
The {\bf energy} of a solution $u$ of~\eqref{eq:FLOERFG}
is defined by
\begin{equation*}
\begin{split}
E_{F,G}(u) 
&:= 
\frac12\int_{-\infty}^\infty\int_0^1
\left(\Abs{\p_su+X_{F_{s,t}}(u)}_{s,t}^2
+ \Abs{\p_tu+X_{G_{s,t}}(u)}_{s,t}^2
\right)\,dtds \\
&\;= 
\int_{-\infty}^\infty\int_0^1
\om\left(\p_su+X_{F_{s,t}}(u),
\p_tu+X_{G_{s,t}}(u)\right)\,dtds.
\end{split}
\end{equation*}
If $u$ is a solution of~\eqref{eq:FLOERFG} 
with finite energy, then the limits
\begin{equation}\label{eq:LIMITFG}
x^\alpha(t) = \lim_{s\to-\infty}u(s,t),\qquad
x^\beta(t) =  \lim_{s\to\infty}u(s,t)
\end{equation}
exist and $x^\alpha\in\cC(L_0,L_1;H^\alpha)$,
$x^\beta\in\cC(L_0,L_1;H^\beta)$. The convergence 
is uniform in $t$, with all derivatives, and exponential. 
\end{PARA}

\begin{PARA}[{\bf Relative Symplectic Action}]\label{para:relAction}\rm
The {\bf relative symplectic action} of a finite energy solution
$u$ of~\eqref{eq:FLOERFG} and~\eqref{eq:LIMITFG} 
is the topological invariant $\cA_H(u)\in\R$, 
defined by 
\begin{equation}\label{eq:ACTIONFG}
\cA_H(u) := \int_{\R\times[0,1]}u^*\om
- \int_0^1H^\alpha_t(x^\alpha(t))\,dt
+ \int_0^1H^\beta_t(x^\beta(t))\,dt.
\end{equation}
It is related to the energy by the formula
\begin{equation}\label{eq:ENERGYFG}
\begin{split}
E_{F,G}(u) 
&= \cA_H(u) + \int_{-\infty}^\infty\int_0^1
\Bigl(\p_sG-\p_tF+\left\{F,G\right\}\Bigr)(u)\,dtds \\
&\quad\;
+ \int_{-\infty}^\infty\Bigl(F_{s,1}(u(s,1))
- F_{s,0}(u(s,0))\Bigr)\,ds.
\end{split}
\end{equation}
The two integrals on the right satisfy a uniform bound, 
independent of $u$.  Note that the energy agrees with the relative
symplectic action whenever $F_{s,t}=0$ and $G_{s,t}=-H_t$
for all $s$ and $t$.  
\end{PARA}

\begin{PARA}[{\bf Regular Homotopies}]\label{para:regHomotopy}\rm
A homotopy $(F,G,J)$ of Floer data from $(H^\alpha,J^\alpha)$ 
to $(H^\beta,J^\beta)$ as in~\ref{para:Homotopy} is called {\bf regular} 
if every finite energy solution of~\eqref{eq:FLOERFG} is regular 
in the sense that the linearized operator is surjective. 
The existence of a regular homotopy follows from standard
transversality theory for the Floer equation (see for example~\cite{FHS}).
Fix a regular homotopy $(F,G,J)$ from 
$(H^\alpha,J^\alpha)$ to $(H^\beta,J^\beta)$.
Then the space
$$
\cM(x^\alpha,x^\beta;F,G,J)
:= \left\{u:\R\times[0,1]\to M\,|\,
\eqref{eq:FLOERFG},\,\eqref{eq:LIMITFG},\,
\cA_H(u)<\infty\right\}
$$
of all smooth finite energy solutions of~\eqref{eq:FLOERFG}
and~\eqref{eq:LIMITFG} is a smooth manifold whose local
dimension near $u$ is given by a suitable Maslov index 
$\mu_H(u)$. The $k$-dimensional part of 
$\cM(x^\alpha,x^\beta;F,G,J)$ with action 
equal to $\eps$ is denoted 
$$
\cM^k_\eps(x^\alpha,x^\beta;F,G,J)
:= \left\{u\in\cM(x^\alpha,x^\beta;F,G,J)\,\bigg|\,
\begin{array}{l}
\mu_H(u)=k\\
\cA_H(u)=\eps
\end{array}\right\}.
$$
In the regular case the Floer--Gromov compactness theorem asserts 
that the union of the spaces $\cM^0_\eps(x^\alpha,x^\beta;F,G,J)$
over all $\eps\le c$ is a finite set for all 
$x^\alpha\in\cC(L_0,L_1;H^\alpha)$, $x^\beta\in\cC(L_0,L_1;H^\beta)$
and $c>0$.  Thus a regular homotopy $(F,G,J)$ from 
$(H^\alpha,J^\alpha)$ to $(H^\beta,J^\beta)$ determines 
a linear operator
$
\Phi^{\beta\alpha}_{F,G,J}:\CF_*(L_0,L_1;H^\alpha)
\to\CF_*(L_0,L_1;H^\beta),
$
defined by 
\begin{equation}\label{eq:PHI}
\Phi^{\beta\alpha}_{F,G,J}x^\alpha
:= \sum_{x^\beta\in\cC(L_0,L_1;H^\beta)}
\sum_{\eps>0}\#_2\cM_\eps^0(x^\alpha,x^\beta;F,G,J)
e^{-\eps}x^\beta
\end{equation}
for $x^\alpha\in\cC(L_0,L_1;H^\alpha)$.
\end{PARA}

\begin{theorem}[{\bf Chain Map}]\label{thm:isotopy1}
Assume~(H). Let $(H^\alpha,J^\alpha),\,(H^\beta,J^\beta)$
be regular pairs for $(L_0,L_1)$ and let $(F,G,J)$
be a regular homotopy from $(H^\alpha,J^\alpha)$ 
to $(H^\beta,J^\beta)$. Then 
$
\p^{H^\beta,J^\beta}\circ\Phi^{\beta\alpha}_{F,G,J}
= \Phi^{\beta\alpha}_{F,G,J}\circ\p^{H^\alpha,J^\alpha}.
$
\end{theorem}

\begin{proof}
See Section~\ref{sec:QFT} for a generalization.
\end{proof}

\begin{theorem}[{\bf Chain Homotopy Equivalence}]\label{thm:isotopy2}
Assume~(H) and let $(H^\alpha,J^\alpha)$, $(H^\beta,J^\beta)$,
$(F,G,J)$ be as in Theorem~\ref{thm:isotopy1}.
Then the induced operator 
$
\Phi^{\beta\alpha}:\HF_*(L_0,L_1;H^\alpha,J^\alpha)
\to\HF_*(L_0,L_1;H^\beta,J^\beta)
$
on Floer homology is independent of the choice of the regular 
homotopy $(F,G,J)$ from $(H^\alpha,J^\alpha)$ 
to $(H^\beta,J^\beta)$, used to define it.
\end{theorem}

\begin{proof}
See Section~\ref{sec:QFT} for a generalization.
\end{proof}
\begin{theorem}[{\bf Catenation}]\label{thm:isotopy3}
Assume~(H) and let $(H^\alpha,J^\alpha)$,
$(H^\beta,J^\beta)$, $(H^\gamma,J^\gamma)$
be regular pairs for $(L_0,L_1)$.
Let $\Phi^{\beta\alpha}$, $\Phi^{\gamma\beta}$, $\Phi^{\gamma\alpha}$, 
$\Phi^{\alpha\alpha}$ be the operators on Floer homology defined in 
Theorems~\ref{thm:isotopy1} and~\ref{thm:isotopy2}. Then
$\Phi^{\alpha\alpha} = \id$ and 
$\Phi^{\gamma\beta}\circ\Phi^{\beta\alpha}=\Phi^{\gamma\alpha}$.
\end{theorem}

\begin{proof}
See Section~\ref{sec:QFT} for a generalization.
\end{proof}

\begin{PARA}[{\bf Naturality}]\label{para:naturalFG}\rm
Let $F,G,J$ be as in Theorem~\ref{thm:isotopy1} and suppose that
$u:\R\times[0,1]\to M$ is a finite energy solution of~\eqref{eq:FLOERFG}.
Let $\R\times[0,1]\to\Diff(M,\om):(s,t)\mapsto\psi_{s,t}$ be a smooth family
of Hamiltonian symplectomorphisms such that $\p_s\psi_{s,t}=0$ for 
$\abs{s}$ sufficiently large and the Lagrangian submanifolds
$$
\tL_0 := \psi_{s,0}^{-1}(L_0),\qquad \tL_1:=\psi_{s,1}^{-1}(L_1)
$$
are independent of $s$.  Then there exist smooth families of 
Hamiltonian functions $A_{s,t},B_{s,t}:M\to\R$ for $s\in\R$ and $0\le t\le1$
such that $A_{s,0}|_{L_0}=0$, $A_{s,1}|_{L_1}=0$, $A$ has compact 
support, and
$$
\p_s\psi_{s,t}+X_{A_{s,t}}\circ\psi_{s,t}=0,\qquad
\p_t\psi_{s,t}+X_{B_{s,t}}\circ\psi_{s,t}=0.
$$
It follows that the function 
$\kappa(s,t):=\p_sB_{s,t}-\p_tA_{s,t}+\{A_{s,t},B_{s,t}\}$
on $M$ is constant for all $s$ and $t$ and vanishes 
for $\Abs{s}$ sufficiently large.  Define 
$$
\tu(s,t) := \psi_{s,t}^{-1}(u(s,t)),\qquad \tJ_{s,t}:=\psi_{s,t}^*J_{s,t},
$$
$$
\tF_{s,t} := (F_{s,t}-A_{s,t})\circ\psi_{s,t},\qquad
\tG_{s,t} := (G_{s,t}-B_{s,t})\circ\psi_{s,t}.
$$
Then $\tu$ is a solution of~\eqref{eq:FLOERFG} with $F,G,J$ replaced
by $\tF,\tG,\tJ$. Moreover, we have
$$
\p_s\tG-\p_t\tF+\{\tF,\tG\} 
= (\p_sG-\p_tF+\{F,G\})\circ\psi - \kappa.
$$ 
Hence the action of $\tu$ agrees with the action of $u$ 
and the energy of $\tu$ agrees with the energy of $u$ 
up to a global additive constant.
\end{PARA}

\begin{corollary}\label{cor:isotopy}
Assume~(H), let $(H^\alpha,J^\alpha)$
be a regular pair for $(L_0,L_1)$, and let
$\{\psi_t\}$ be a Hamiltonian isotopy such that 
\begin{equation}\label{eq:psitbc}
\psi_0(L_0)=L_0,\qquad \psi_1(L_1)=L_1.
\end{equation} 
Let $K_t$ be a family of Hamiltonian 
functions generating $\psi_t$ via~\eqref{eq:psiK} 
and define
$$
H^\beta_t := (H^\alpha_t-K_t)\circ\psi_t,\qquad
J^\beta_t := \psi_t^*J^\alpha_t.
$$  
Then $(H^\beta,J^\beta)$ is a regular pair for $(L_0,L_1)$.
If $\{\psi_t\}_{0\le t\le 1}$ is Hamiltonian isotopic to the 
constant path $\psi_{0,t}=\id$ subject to~\eqref{eq:psitbc} then
$$
\Phi^{\beta\alpha} = \psi^*:
\HF(L_0,L_1;H^\alpha,J^\alpha)\to\HF(L_0,L_1;H^\beta,J^\beta).
$$
\end{corollary}

\begin{proof}
By assumption, there exists a Hamiltonian isotopy
$$
\R\times[0,1]\times M\to M:(s,t,p)\mapsto\psi_{s,t}(p)
$$
such that 
$$
\psi_{s,t} = \left\{\begin{array}{ll}
\id,&\mbox{if }s\le 0,\\
\psi_t,&\mbox{if }s\ge1,
\end{array}\right.\qquad
\psi_{s,0}(L_0)=L_0,\qquad
\psi_{s,1}(L_1)=L_1.
$$
Choose Hamiltonian functions $A_{s,t},B_{s,t}:M\to\R$
such that $A$ has compact support, $A_{s,0}|_{L_0}=0$,
$A_{s,1}|_{L_1}=0$, $B_{s,t}=-K_t$ for $s\ge 1$, and
$$
\p_s\psi_{s,t}+X_{A_{s,t}}\circ\psi_{s,t}=0,\qquad
\p_t\psi_{s,t}+X_{B_{s,t}}\circ\psi_{s,t}=0.
$$
Define 
$$
\tF_{s,t}:=-A_{s,t}\circ\psi_{s,t},\quad
\tG_{s,t}:=(-H^\alpha_t-B_{s,t})\circ\psi_{s,t},\quad
\tJ_{s,t}:=\psi_{s,t}^*J^\alpha_t.
$$
Then $(\tF,\tG,\tJ)$ is a homotopy from $(H^\alpha,J^\alpha)$
to $(H^\beta,J^\beta)$ as in~\ref{para:Homotopy},
satisfying~\eqref{eq:F01} and~\eqref{eq:GJends}. 
Moreover, a smooth function $u:\R\times[0,1]\to M$ 
is a solution of~\eqref{eq:FLOER} with $H=H^\alpha$ 
and $J=J^\alpha$ if and only if the function
$$
\tu(s,t):=\psi_{s,t}^{-1}(u(s,t))
$$
is a solution of~\eqref{eq:FLOERFG} with $F,G,J$ replaced 
by $\tF,\tG,\tJ$ (see~\ref{para:naturalFG}).  
Since $(H^\alpha,J^\alpha)$ is a regular pair 
for $(L_0,L_1)$, this implies that $(\tF,\tG,\tJ)$ 
is a regular homotopy from $(H^\alpha,J^\alpha)$
to $(H^\beta,J^\beta)$.   It also implies that the 
homomorphism on the Floer chain complex, 
induced by $(\tF,\tG,\tJ)$ agrees with $\psi^*$:
$$
\Phi^{\beta\alpha}_{\tF,\tG,\tJ} = \psi^*:
\CF_*(L_0,L_1;H^\alpha)\to \CF_*(L_0,L_1;H^\beta).
$$
Hence the assertion of Corollary~\ref{cor:isotopy}
follows from Theorems~\ref{thm:isotopy1} 
and~\ref{thm:isotopy2}.
\end{proof}

\begin{proof}[Proof of Theorem~\ref{thm:HF}]
Let $(H^\alpha,J^\alpha)$ and $(H^\beta,J^\beta)$ 
be as in Theorem~\ref{thm:isotopy1}.  
By Theorems~\ref{thm:isotopy1} and~\ref{thm:isotopy2}
there is a unique operator
$$
\Phi^{\beta\alpha}:\HF_*(L_0,L_1;H^\alpha,J^\alpha)
\to\HF_*(L_0,L_1;H^\beta,J^\beta).
$$
By Theorem~\ref{thm:isotopy3} this is an isomorphism
with inverse $\Phi^{\alpha\beta}$ and these operators 
satisfy the requirements of Theorem~\ref{thm:HF}.
\end{proof}

\begin{proof}[Proof of Theorem~\ref{thm:seidel}]
Suppose $M,L_0,L_1$ satisfy~(H),
let $(H,J)$ be a regular pair for $(L_0,L_1)$, 
and let 
$$
\psi=\{\psi_t\}_{0\le t\le1}
$$ 
be a Hamiltonian isotopy
of $(M,\om)$, starting at any symplectomorphism. 
By~\ref{para:NAT}, pullback defines an isomorphism 
$$
\psi^*:\HF_*(L_0,L_1;H,J)\to
\HF_*(\psi_0^{-1}(L_0),\psi_1^{-1}(L_1);\psi^*H,\psi^*J).
$$
That these isomorphisms satisfy the {\it (Functoriality)} 
condition is obvious from the definitions.
That they satisfy the {\it (Naturality)} condition 
follows from~\ref{para:naturalFG} with 
$$
\psi_{s,t}:=\psi_t,\qquad A_{s,t}:=0,\qquad B_{s,t}:=-K_t. 
$$
That they satisfy the {\it  (Iso\-topy)} condition with $\phi_t=\id$
follows from Corollary~\ref{cor:isotopy}.  That they satisfy the 
{\it (Isotopy)} condition in general follows from the special case 
$\phi_t=\id$ and the {\it (Functoriality)} condition.
This proves Theorem~\ref{thm:seidel}.
\end{proof}

\section{Field Theory} \label{sec:QFT}  

It is convenient to extend the discussion of the previous section 
to a more general class of surfaces $\Sigma$ with cylindrical ends 
(replacing the strip).  In our exposition we follow the discussion 
in Seidel~\cite[pages~100-112]{SEIDELBOOK}. The complex structure~$j$ 
on the surface will be fixed. Throughout we abbreviate
$$
\R^+ := [0,\infty),\qquad\R^-:=(-\infty,0].
$$
Let $(M,\om)$ be a compact symplectic manifold without boundary
and 
$$
\cL=\cL(M,\om)
$$ 
be the set of all compact Lagrangian 
submanifolds $L\subset M$ without boundary such that $(M,L)$ 
is monotone with minimal Maslov number at least three.
Associated to $(M,\om)$ is a category $\sL=\sL(M,\om)$ 
defined as follows.

\begin{definition}[{\bf The Category of Lagrangian pairs}]
\label{def:LPAIR}\rm \hfill

\smallskip\textsc{(Objects)\,} 
An {\bf object} of $\sL$ is a map
$
I\to \cL\times\cL:i\mapsto(L_{i0},L_{i1}),
$
defined on a finite set $I$.

\smallskip\textsc{(String Cobordisms)\,}
Fix two objects 
$
(L^\alpha_0,L^\alpha_1)
=\{L^\alpha_{i0},L^\alpha_{i1}\}_{i\in I^\alpha}
$
and
$
(L^\beta_0,L^\beta_1)
=\{L^\beta_{i0},L^\beta_{i1}\}_{i\in I^\beta}
$
in $\sL(M,\om)$. A {\bf string cobordism} from 
$(L^\alpha_0,L^\alpha_1)$ to $(L^\beta_0,L^\beta_1)$ 
is a tuple 
$$
(\Sigma,L) = \bigl(
\Sigma,\{L_z\}_{z\in\p\Sigma},
\{\iota^{\alpha,-}_i\}_{i\in I^\alpha},
\{\iota^{\beta,+}_i\}_{i\in I^\beta}
\bigr)
$$
consisting of an oriented $2$-manifold $\Sigma$ with boundary,
a locally constant map $\p\Sigma\to\cL:z\mapsto L_z$, 
and orientation preserving embeddings
$$
\bigl\{\iota_i^{\alpha,-}:
\R^-\times[0,1]\to\Sigma\bigr\}_{i\in I^\alpha},\qquad
\bigl\{\iota_i^{\beta,+}:
\R^+\times[0,1]\to\Sigma\bigr\}_{i\in I^\beta},
$$ 
that satisfy the following conditions.

\smallskip\noindent{\bf (a)}
The images of the embeddings $\iota^{\alpha,-}_i$, $i\in I^\alpha$,
and  $\iota^{\beta,+}_i$, $i\in I^\beta$, are pairwise disjoint
and their complement has compact closure.

\smallskip\noindent{\bf (b)}
For $i\in I^\alpha$ and $s\le 0$ we have
$L_{\iota_i^{\alpha,-}(s,0)} = L^\alpha_{i0}$ and
$L_{\iota_i^{\alpha,-}(s,1)} = L^\alpha_{i1}$.

\smallskip\noindent{\bf (c)}
For $i\in I^\beta$ and $s\ge0$ we have
$L_{\iota_i^{\beta,+}(s,0)} = L^\beta_{i0}$ and
$L_{\iota_i^{\beta,+}(s,1)} = L^\beta_{i1}$.

\smallskip\textsc{(Morphisms)\,}
Let 
$
(\Sigma',L')
= (\Sigma',\{L'_{z'}\}_{z'\in\p\Sigma'},
\{{\iota'}^{\alpha,-}_i\}_{i\in I^\alpha},
\{{\iota'}^{\beta,+}_i\}_{i\in I^\beta})
$
be another string cobordism from 
$(L^\alpha_0,L^\alpha_1)$ to $(L^\beta_0,L^\beta_1)$. 
The string cobordisms $(\Sigma,L)$ and $(\Sigma',L')$ 
are called {\bf equivalent} if there is an 
orientation preserving diffeomorphism 
$\phi:\Sigma\to\Sigma'$ such that 
$$
L'_{\phi(z)}=L_z,\qquad 
{\iota'}^{\alpha,-}_i=\phi\circ\iota^{\alpha,-}_i,\qquad
{\iota'}^{\beta,+}_j=\phi\circ\iota^{\beta,+}_j
$$  
for $z\in\p\Sigma$, $i\in I^\alpha$, $j\in I^\beta$.
A {\bf morphism in $\sL$ from 
$(L^\alpha_0,L^\alpha_1)$ to $(L^\beta_0,L^\beta_1)$}
is an equivalence class $[\Sigma,L]$ of string cobordisms.

\smallskip\textsc{(Catenation)\,}
Let 
\begin{equation}\label{eq:Sialbe}
(\Sigma^{\alpha\beta},L^{\alpha\beta})
=\left(
\Sigma^{\alpha\beta},
\{L_z^{\alpha\beta}\}_{z\in\p\Sigma^{\alpha\beta}},
\{\iota^{\alpha,-}_i\}_{i\in I^\alpha},
\{\iota^{\beta,+}_i\}_{i\in I^\beta}
\right)
\end{equation}
be a string cobordism from 
$(L^\alpha_0,L^\alpha_1)$ to $(L^\beta_0,L^\beta_1)$
and
\begin{equation}\label{eq:Sibega}
(\Sigma^{\beta\gamma},L^{\beta\gamma}) 
= \left(\Sigma^{\beta\gamma},
\{L^{\beta\gamma}_z\}_{z\in\p\Sigma^{\beta\gamma}},
\{{\iota}^{\beta,-}_i\}_{i\in I^\beta},
\{{\iota}^{\gamma,+}_i\}_{i\in I^\gamma}
\right)
\end{equation}
be a string cobordism from 
$(L^\beta_0,L^\beta_1)$ to $(L^\gamma_0,L^\gamma_1)$.
For $T>0$ the {\bf $T$-ca\-te\-na\-tion} of these 
string cobordisms is the string cobordism 
$$
(\Sigma^{\alpha\beta},L^{\alpha\beta})
\#_T(\Sigma^{\beta\gamma},L^{\beta\gamma})
= \Bigl(
\Sigma^{\alpha\gamma}_T,
\{L^{\alpha\gamma}_z\}_{z\in\p\Sigma^{\alpha\gamma}_T},
\{\iota^{\alpha,-}_i\}_{i\in I^\alpha},
\{\iota^{\gamma,+}_i\}_{i\in I^\gamma}
\Bigr)
$$
from $\{L^\alpha_{i0},L^\alpha_{i1}\}_{i\in I^\alpha}$ to
$\{L^\gamma_{i0},L^\gamma_{i1}\}_{i\in I^\gamma}$, defined 
as follows.  The $2$-manifold $\Sigma^{\alpha\gamma}_T$
is defined as the quotient 
\begin{equation}\label{eq:SigmaT}
\begin{split}
\Sigma^{\alpha\gamma}_T 
&:= 
\frac{\Sigma^{\alpha\beta}_{2T}\sqcup\Sigma^{\beta\gamma}_{2T}}{\equiv},\\
\Sigma^{\alpha\beta}_{2T}
&:=\Sigma^{\alpha\beta}\setminus
\bigcup_{i\in I^\beta}\iota^{\beta,+}_i([2T,\infty)\times[0,1]),\\
\Sigma^{\beta\gamma}_{2T}
&:=\Sigma^{\beta\gamma}\setminus
\bigcup_{i\in I^\beta}{\iota'}^{\beta,-}_i((-\infty,-2T]\times[0,1]),
\end{split}
\end{equation}
and the equivalence relation is given by
$
\iota^{\beta,+}_i(s,t)\equiv\iota^{\beta,-}_i(s-2T,t)
$
for $i\in I^\beta$, $0<s<2T$, $0\le t\le1$.  
The Lagrangian submanifolds 
$L^{\alpha\gamma}_z\subset M$ are 
\begin{equation}\label{eq:LT}
L^{\alpha\gamma}_z:=\left\{\begin{array}{ll}
L^{\alpha\beta}_z,&\mbox{for }
z\in\p\Sigma^{\alpha\beta}\cap\Sigma^{\alpha\beta}_{2T},\\
L^{\beta\gamma}_z,&\mbox{for }
z\in\p\Sigma^{\beta\gamma}\cap\Sigma^{\beta\gamma}_{2T}.
\end{array}\right.
\end{equation}
The equivalence class of 
$(\Sigma^{\alpha\beta},L^{\alpha\beta})
\#_T(\Sigma^{\beta\gamma},L^{\beta\gamma})$
is independent of~$T$ and depends only on the 
equivalence classes of $(\Sigma^{\alpha\beta},L^{\alpha\beta})$ 
and $(\Sigma^{\beta\gamma},L^{\beta\gamma})$.

\smallskip\textsc{(Composition)\,}
The {\bf composition} in $\sL$ is defined by 
$$
[\Sigma^{\beta\gamma},L^{\beta\gamma}]
\circ[\Sigma^{\alpha\beta},L^{\alpha\beta}] 
:= [(\Sigma^{\alpha\beta},L^{\alpha\beta})
\#_T(\Sigma^{\beta\gamma},L^{\beta\gamma})].
$$
\end{definition}

\begin{PARA}[{\bf Floer Homology}]\label{para:StringHF}\rm
Let $(L^\alpha_0,L^\alpha_1)
=\{L^\alpha_{i0},L^\alpha_{i1}\}_{i\in I^\alpha}$
be an object in $\sL$ and fix a corresponding collection 
$(H^\alpha,J^\alpha) =\{H^\alpha_i,J^\alpha_i\}_{i\in I^\alpha}$
of regular pairs. Associated to these data is the Floer chain complex
$$
\CF_*(L^\alpha_0,L^\alpha_1;H^\alpha)
:= \bigotimes_{i\in I^\alpha}
\CF_*(L^\alpha_{i0},L^\alpha_{i1};H_i^\alpha)
$$
This is the vector space over $\Lambda$ generated 
by the tuples $x^\alpha=\{x^\alpha_i\}_{i\in I^\alpha}$ with
$x_i^\alpha\in\cC(L^\alpha_{i0},L^\alpha_{i1};H^\alpha_i)$.
The boundary operator 
$$
\p^\alpha:=\p^{H^\alpha,J^\alpha}:
\bigotimes_{i\in I^\alpha}\CF_*(L^\alpha_{i0},L^\alpha_{i1};H_i^\alpha)
\to \bigotimes_{i\in I^\alpha}\CF_*(L^\alpha_{i0},L^\alpha_{i1};H_i^\alpha)
$$
is induced by the boundary operators
$$
\p^{H^\alpha_i,J^\alpha_i}:\CF_*(L^\alpha_{i0},L^\alpha_{i1};H_i^\alpha)
\to \CF_*(L^\alpha_{i0},L^\alpha_{i1};H_i^\alpha)
$$
of Section~\ref{sec:FLOER}. Its Floer homology group
is denoted
$$
\HF_*(L^\alpha_0,L^\alpha_0;H^\alpha,J^\alpha)
:=  \frac{\ker\,\p^\alpha}{\im\,\p^\alpha}
= \bigotimes_{i\in I^\alpha}
\HF_*(L^\alpha_{i0},L^\alpha_{i1};H_i^\alpha,J^\alpha_i).
$$
\end{PARA}

\begin{PARA}[{\bf Floer Data on String Cobordisms}]
\label{para:StringFloerData}\rm
Fix a string cobordism $(\Sigma,L)$ from the object 
$(L^\alpha_0,L^\alpha_1)=\{L^\alpha_{i0},L^\alpha_{i1}\}_{i\in I^\alpha}$
in $\sL$ to $(L^\beta_0,L^\beta_1)=\{L^\beta_{i0},L^\beta_{i1}\}_{i\in I^\beta}$.
{\bf A set of Floer data on $\Sigma$} is a triple $(j,H,J)$ 
consisting of a complex structure $j$ on $\Sigma$,
a $1$-form ${H:T\Sigma\to\Om^0(M)}$, and a smooth 
family of $\om$-tame almost complex 
structures $J=\{J_z\}_{z\in\Sigma}$,
satisfying the following conditions.

\smallskip\noindent{\bf (a)}
$\iota^{\alpha,-}_i$ is holomorphic for $i\in I^\alpha$ and
$\iota^{\beta,+}_i$ is holomorphic for $i\in I^\beta$.

\smallskip\noindent{\bf (b)}
For $i\in I^\alpha$ there is a regular pair 
$(H^\alpha_i,J^\alpha_i)$ for $(L^\alpha_{i0},L^\alpha_{i1})$ 
such that
$$
(\iota^{\alpha,-}_i)^*H = -H^\alpha_{it}dt,\quad
J_{\iota^{\alpha,-}_i(s,t)} = J^\alpha_{it}\qquad
\mbox{for}\; s\le 0,\;0\le t\le1.
$$

\smallskip\noindent{\bf (c)}
For $i\in I^\beta$ there is a regular pair 
$(H^\beta_i,J^\beta_i)$ for $(L^\beta_{i0},L^\beta_{i1})$ 
such that
$$
(\iota^{\beta,+}_i)^*H = -H^\beta_{it}dt,\quad
J_{\iota^{\beta,+}_i(s,t)} = J^\beta_{it}\qquad
\mbox{for}\; s\le 0,\;0\le t\le1.
$$

\smallskip\noindent{\bf (d)}
The restriction $H_{z,\widehat{z}}|_{L_z}$ is constant 
for $z\in\p\Sigma$ and $\widehat{z}\in T_Z\p\Sigma$.

\smallskip\noindent
The Floer data $(j,H,J)$ are said to {\bf connect 
$(H^\alpha,J^\alpha)=\{H^\alpha_i,J^\alpha_i\}_{i\in I^\alpha}$ 
to $(H^\beta,J^\beta)=\{H^\beta_i,J^\beta_i\}_{i\in I^\beta}$}.
The tuple 
$
\cS := (\Sigma,L,j,H,J)
$
is called a {\bf framed string cobordism} from the 
tuple $\cL^\alpha=(L^\alpha_0,L^\alpha_1,H^\alpha,J^\alpha)$
to the tuple $\cL^\beta=(L^\beta_0,L^\beta_1,H^\beta,J^\beta)$.
\end{PARA}

\begin{PARA}[{\bf The Floer Equation on String Cobordisms}]
\label{para:StringFloer}\rm
Let 
$
(L^\alpha_0,L^\alpha_1)
$
and
$
(L^\beta_0,L^\beta_1)
$
be two objects in $\sL$ and $(H^\alpha,J^\alpha)$
and $(H^\beta,J^\beta)$ be two corresponding collections of regular pairs.
Fix a string cobordism $(\Sigma,L)$ from 
$(L^\alpha_0,L^\alpha_1)$ to $(L^\beta_0,L^\beta_1)$,
and a set of Floer data $(j,H,J)$ on $\Sigma$
from $(H^\alpha,J^\alpha)$ to $(H^\beta,J^\beta)$.  
Associated to these data is the Floer equation
\begin{equation}\label{eq:FLOERsigma}
\bar\p_{J,H}(u) := \tfrac12
\Bigl(d_Hu + J\circ(d_Hu)\circ j\Bigr) = 0,\qquad
u(z)\in L_z\mbox{ for }z\in\p\Sigma,
\end{equation}
for smooth maps $u:\Sigma\to M$.  
Here $d_Hu\in\Om^1(\Sigma,u^*TM)$ denotes the $1$-form
on $\Sigma$ with values in the pullback tangent bundle of $M$, 
defined by 
$$
d_Hu(z)\widehat{z} 
:= du(z)\widehat{z} + X_{H_{z,\widehat{z}}}(u(z))
$$
for $\widehat{z}\in T_z\Sigma$ and we abbreviate 
$
(J\circ(d_Hu)\circ j)(z)\widehat z 
:= J_z(u(z))d_Hu(z)j(z)\widehat z.
$
If $u:\Sigma\to M$ is a solution of~\eqref{eq:FLOERsigma}
then the functions 
$$
u^{\alpha,-}_i:=u\circ\iota^{\alpha,-}_i:
\R^-\times[0,1]\to M,\qquad i\in I^\alpha,
$$ 
satisfy the usual Floer equation~\eqref{eq:FLOER} for the 
quadruple $(L^\alpha_{i0},L^\alpha_{i1},H^\alpha_i,J^\alpha_i)$.
Similarly, the functions 
$$
u^{\beta,+}_i:=u\circ\iota^{\beta,+}_i:
\R^+\times[0,1]\to M,\qquad i\in I_\beta,
$$ 
satisfy~\eqref{eq:FLOER} for the 
quadruple $(L^\beta_{i0},L^\beta_{i1},H^\beta_i,J^\beta_i)$.
\end{PARA}

\begin{PARA}[{\bf The Energy Identity for String Cobordisms}]
\label{para:StringEnergy}\rm
The {\bf energy} of a solution $u:\Sigma\to M$
of~\eqref{eq:FLOERsigma} is defined by 
$$
E_H(u) 
:= \frac12\int_\Sigma\Abs{d_Hu}_z^2\dvol_\Sigma \\
$$
Here $\dvol_\Sigma\in\Om^2(\Sigma)$ is a volume form
compatible with the orientation and 
$\inner{\cdot}{\cdot}:=\dvol_\Sigma(\cdot,j\cdot)$
is the associated Riemannian metric on $\Sigma$. 
The term on the right in~\eqref{eq:ENERGYsigma}
is the integral of the function 
$\Sigma\to\R:z\mapsto\Abs{d_Hu(z)}_z^2$,
where $\Abs{d_Hu(z)}_z$ denotes the operator norm 
of $d_Hu(z):T_z\Sigma\to T_{u(z)}M$ 
with respect to the above metric on $\Sigma$ and the Riemannian
metric on $M$ determined by~$J_z$ and $\om$. 
This integral is independent of the choice of $\dvol_\Sigma$.  

If a solution of~\eqref{eq:FLOERsigma} has finite energy
$E_H(u)<\infty$ then the limits
\begin{equation}\label{eq:LIMITsigma}
x^\alpha_i(t) = \lim_{s\to-\infty}u^{\alpha,-}_i(s,t),\qquad
x^\beta_i(t) = \lim_{s\to\infty}u^{\beta,+}_i(s,t),
\end{equation}
exist for $i\in I^\alpha$, respectively $i\in I^\beta$, 
the convergence in~\eqref{eq:LIMITsigma} is uniform 
in~$t$, with all derivatives, and exponential, and
$x_i^\alpha\in\cC(L^\alpha_{i0},L^\alpha_{i1};H^\alpha_i)$
for $i\in I^\alpha$ and 
$x_i^\beta\in\cC(L^\beta_{i0},L^\beta_{i1};H^\beta_i)$
for $i\in I^\alpha$. The {\bf relative symplectic action} 
of a solution $u:\Sigma\to M$ of~\eqref{eq:FLOERsigma} 
and~\eqref{eq:LIMITsigma} is the number
\begin{equation}\label{eq:ACTIONsigma}
\cA_H(u)
:= \int_\Sigma u^*\om 
- \sum_{i\in I^\alpha}\int_0^1H^\alpha_{it}(x^\alpha_i(t))\,dt
+ \sum_{i\in I^\beta}\int_0^1H^\beta_{it}(x^\beta_i(t))\,dt.
\end{equation}
It is related to the energy by
\begin{equation}\label{eq:ENERGYsigma}
E_H(u) 
= \cA_H(u) + \int_\Sigma u^*\Om_H
- \int_{\p\Sigma}u^*H.
\end{equation}
The $2$-form $\Om_H\in\Om^2(\Sigma,\Om^0(M))$ is the
curvature of $H$, defined by 
$$
\Om_{H,z}(\widehat{z}_1,\widehat{z}_2)
:= dH_z(\widehat{z}_1,\widehat{z}_2) 
+ \left\{H_{z,\widehat{z}_1},H_{z,\widehat{z}_2}\right\}
\in\Om^0(M),\qquad
\widehat{z}_1,\widehat{z}_2\in T_z\Sigma.
$$ 
The value of the first term on the right at a point $p\in M$ denotes 
the differential of the $1$-form 
$$
T\Sigma\to\R:(z,\widehat{z})\mapsto H_{z,\widehat{z}}(p)
$$
and 
$
\{F,G\}:=\om(X_F,X_G)
$ 
is the Poisson bracket.
The scalar differential forms $u^*\Om_H\in\Om^2(\Sigma)$
and $u^*H\in\Om^1(\Sigma)$ are defined by evaluating at $u(z)$. 
Note that~\eqref{eq:ENERGYFG} is a special case 
of~\eqref{eq:ENERGYsigma}.  
\end{PARA}

\begin{PARA}[{\bf Regular Floer Data on String Cobordisms}]
\label{para:StringRegFloerData}\rm
The Floer data $(j,H,J)$ are called 
{\bf regular for $\Sigma,L$} if every finite energy solution 
of~\eqref{eq:FLOERsigma} is regular in the sense 
that the linearized operator is surjective. 
In this case the tuple 
$$
\cS := (\Sigma,L;j,H,J)
$$
is called a {\bf regular framed string cobordism} from 
$\cL^\alpha=(L^\alpha_0,L^\alpha_1;H^\alpha,J^\alpha)$
to $\cL^\beta=(L^\beta_0,L^\beta_1;H^\beta,J^\beta)$.
The existence of regular Floer data on any string cobordism
$(\Sigma,L)$ and for any fixed complex structure $j$ on $\Sigma$
follows from the standard transversality arguments in Floer theory
(see for example~\cite{FHS} and also~\cite{MS}).  Moreover,
it suffices to perturb $H$ in $U\times M$ for any given fixed nonempty
open subset $U\subset\Sigma$ to achieve transversality.
\end{PARA}

\begin{PARA}[{\bf Morphisms on Floer Homology}]
\label{para:StringHFmorphism}\rm
Fix a regular set of Floer data $(j,H,J)$
connecting $(H^\alpha,J^\alpha)$ to $(H^\beta,J^\beta)$. 
Fix two tuples of critical points 
$$
x^\alpha=\{x^\alpha_i\}_{i\in I^\alpha},\qquad
x^\beta=\{x^\beta_i\}_{i\in I^\beta}
$$
with 
$$
x_i^\alpha\in\cC(L^\alpha_{i0},L^\alpha_{i1};H^\alpha_i),\qquad
x_i^\beta\in\cC(L^\beta_{i0},L^\beta_{i1};H^\beta_i).
$$
Define
$$
\cM(x^\alpha,x^\beta;j,H,J)
:= \left\{u:\Sigma\to M\,\Big|\,
\eqref{eq:FLOERsigma},\,\eqref{eq:LIMITsigma},\,
\cA_H(u)<\infty
\right\}.
$$
This space is a smooth manifold whose local dimension 
near $u$ is given by a suitable Maslov index~$\mu_H(u)$. 
The $k$-dimensional part of $\cM(x^\alpha,x^\beta;j,H,J)$
with relative symplectic action $\eps$ is denoted 
\begin{equation}\label{eq:Msigma}
\cM^k_\eps(x^\alpha,x^\beta;j,H,J)
:= \left\{u\in\cM(x^\alpha,x^\beta;j,H,J)\,\Big|\,
\begin{array}{l}
\mu_H(u)=k,\\
\cA_H(u)=\eps
\end{array}\right\}.
\end{equation}
In the regular case the Floer--Gromov compactness theorem asserts 
that the union of the spaces $\cM^0_\eps(x^\alpha,x^\beta;j,H,J)$
over all $\eps\le c$ is a finite set for all 
$x^\alpha$ and $x^\beta$ and $c>0$.  
Thus $(j,H,J)$ determines a linear operator
$$
\Phi^{\beta\alpha}_{\Sigma,L;j,H,J}:
\CF_*(L^\alpha_0,L^\alpha_1;H^\alpha)
\to \CF_*(L^\beta_0,L^\beta_1;H^\beta),
$$
defined by 
\begin{equation}\label{eq:PHIsigma}
\Phi^{\beta\alpha}_{\Sigma,L;j,H,J}x^\alpha
:= \sum_{x^\beta}\sum_\eps
\#_2\cM_\eps^0(x^\alpha,x^\beta;j,H,J)
e^{-\eps}x^\beta.
\end{equation}
\end{PARA}

\begin{theorem}[{\bf Chain Map}]\label{thm:cm}
Let 
$
\cS := (\Sigma,L;j,H,J)
$
be a regular framed string cobordism from 
$\cL^\alpha=(L^\alpha_0,L^\alpha_1;H^\alpha,J^\alpha)$
to $\cL^\beta=(L^\beta_0,L^\beta_1;H^\beta,J^\beta)$.
Then 
$$
\p^\beta\circ\Phi^{\beta\alpha}_{\Sigma,L;j,H,J}
= \Phi^{\beta\alpha}_{\Sigma,L;j,H,J}\circ\p^\alpha.
$$
\end{theorem}

\begin{proof}
See Section~\ref{sec:GLUING}.
\end{proof}

\begin{PARA}[{\bf Homotopy of Floer Data}]\label{para:CHE}\rm
Let $(L^\alpha_0,L^\alpha_1)$ and $(L^\beta_0,L^\beta_1)$ be two objects 
in $\sL$ and $(H^\alpha,J^\alpha)$ and $(H^\beta,J^\beta)$ be 
two corresponding collections of regular pairs.  
Let $(\Sigma,L)$ be a string cobordism from 
$(L^\alpha_0,L^\alpha_1)$ and $(L^\beta_0,L^\beta_1)$
and let $(j_0,H_0,J_0)$ and $(j_1,H_1,J_1)$ be two 
regular sets of Floer data on $\Sigma$ from 
$(H^\alpha,J^\alpha)$ to $(H^\beta,J^\beta)$.
Denote the corresponding chain homomorphism 
between the Floer complexes by 
$$
\Phi^{\beta\alpha}_0,\Phi^{\beta\alpha}_1:
\CF_*(L^\alpha_0,L^\alpha_1;H^\alpha)
\to \CF_*(L^\beta_0,L^\beta_1;H^\beta).
$$
Choose a smooth homotopy 
$$
\{j_\lambda,H_\lambda,J_\lambda\}_{0\le\lambda\le1}
$$
of Floer data from $(j_0,H_0,J_0)$ to $(j_1,H_1,J_1)$, 
for each $\lambda$ connecting 
$(H^\alpha,J^\alpha)$ to $(H^\beta,J^\beta)$.
The homotopy can be chosen {\bf regular} in the 
sense that the linearized operator for the 
one parameter Floer equation 
\begin{equation}\label{eq:FLOERsigmaCHT}
\bar\p_{J_\lambda,H_\lambda}(u)=0,\qquad
u(z)\in L_z\mbox{ for }z\in\p\Sigma,
\end{equation}
is surjective. In this case the moduli space 
\begin{equation}\label{eq:Msila}
\cM^k_\eps(x^\alpha,x^\beta;\{j_\lambda,H_\lambda,J_\lambda\}_\lambda)
:= \left\{(\lambda,u)\,\Bigg|\,\begin{array}{l}
0\le\lambda\le 1,\,u:\Sigma\to M,\\
\eqref{eq:FLOERsigmaCHT},\,
\eqref{eq:LIMITsigma},\mu_H(u)=k,\\
\cA_{H_\lambda}(u)=\eps
\end{array}\right\}
\end{equation}
is a smooth manifold of dimension $k+1$ for every tuple 
$x^\alpha=\{x^\alpha_i\}_{i\in I^\alpha}$ with
$x_i^\alpha\in\cC(L^\alpha_{i0},L^\alpha_{i1};H^\alpha_i)$
and every $x^\beta=\{x^\beta_i\}_{i\in I^\beta}$ with
$x_i^\beta\in\cC(L^\beta_{i0},L^\beta_{i1};H^\beta_i)$.
Moreover, the usual Floer--Gromov compactness theorem 
asserts that for $k=-1$ the union of the moduli spaces
$\cM^k_\eps(x^\alpha,x^\beta;\{j_\lambda,H_\lambda,J_\lambda\}_\lambda)$
over all $\eps\le c$ is a finite set for all $c>0$ 
and all $x^\alpha,x^\beta$. Hence there is an operator 
$$
\Psi^{\beta\alpha}:
\CF_*(L^\alpha_0,L^\alpha_1;H^\alpha)
\to \CF_*(L^\beta_0,L^\beta_1;H^\beta)
$$
defined by 
\begin{equation}\label{eq:che}
\Psi^{\beta\alpha}x^\alpha 
:= \sum_{x^\beta}\sum_\eps
\#\cM^{-1}_\eps(x^\alpha,x^\beta;\{j_\lambda,H_\lambda,J_\lambda\}_\lambda)
e^{-\eps}x^\beta.
\end{equation}
\end{PARA}

\begin{theorem}[{\bf Chain Homotopy Equivalence}]\label{thm:che}
The above operators $\Phi^{\beta\alpha}_0$, $\Phi^{\beta\alpha}_1$,
and $\Psi^{\beta\alpha}$ satisfy the equation
$$
\Phi^{\beta\alpha}_1-\Phi^{\beta\alpha}_0 
= \p^\beta\circ\Psi^{\beta\alpha} 
+ \Psi^{\beta\alpha}\circ \p^\alpha.
$$
\end{theorem}

\begin{proof}
See Section~\ref{sec:GLUING}.
\end{proof}

\begin{corollary}\label{cor:che}
Let $(\Sigma,L)$, $(j,H,J)$, and $\Phi^{\beta\alpha}_{\Sigma,L;j,H,J}$  
be as in Theorem~\ref{thm:cm}. Then the induced homomorphism 
\begin{equation}\label{eq:PHIHFsigma}
\Phi^{\beta\alpha}_{\Sigma,L}:
\HF_*(L^\alpha_0,L^\alpha_1;H^\alpha,J^\alpha)
\to \HF_*(L^\beta_0,L^\beta_1;H^\beta,J^\beta)
\end{equation}
on Floer homology is independent of the choice of the regular Floer 
data from $(H^\alpha,J^\alpha)$ to $(H^\beta,J^\beta)$,
used to define it.
\end{corollary}

\begin{proof}
This follows immediately from Theorem~\ref{thm:che}.
\end{proof}

\begin{theorem}[{\bf Catenation}]\label{thm:cat}
Fix a string cobordism $(\Sigma^{\alpha\beta},L^{\alpha\beta})$ 
from the object $(L^\alpha_0,L^\alpha_1)$ to $(L^\beta_0,L^\beta_1)$
and a string cobordism $(\Sigma^{\beta\gamma},L^{\beta\gamma})$ 
from $(L^\beta_0,L^\beta_1)$ to $(L^\gamma_0,L^\gamma_1)$
and denote by $(\Sigma^{\alpha\gamma},L^{\alpha\gamma})$ 
their $T$-catenation (for any $T>0$).
Let $\Phi^{\beta\alpha}$, $\Phi^{\gamma\beta}$, 
$\Phi^{\gamma\alpha}$ be the operators on Floer homology 
associated to these cobordisms via Theorem~\ref{thm:cm} 
and Corollary~\ref{cor:che}. Then
\begin{equation}\label{eq:catenation}
\Phi^{\gamma\beta}\circ\Phi^{\beta\alpha} = \Phi^{\gamma\alpha}.
\end{equation}
\end{theorem}

\begin{proof}
See Section~\ref{sec:GLUING}.
\end{proof}

\begin{proof}[Proof of Theorems~\ref{thm:isotopy1}, 
\ref{thm:isotopy2}]
Choosing 
$$
\Sigma:=\R\times[0,1]
$$ 
we find that Theorem~\ref{thm:isotopy1} 
is a special case of Theorem~\ref{thm:cm}, and 
Theorem~\ref{thm:isotopy2} 
is a special case of Corollary~\ref{cor:che}.
\end{proof}

\begin{PARA}[{\bf The Donaldson Category}]\label{para:DON}\rm
Take 
$
\Sigma=\Delta
$ 
to be a connected genus zero surface with three cylindrical ends. 
Then the operator $\Phi^{\beta\alpha}_{\Sigma,L}$ 
in Theorem~\ref{thm:cm} and Corollary~\ref{cor:che} 
defines a homomorphism
$$
\HF_*(L_0,L_1)\otimes\HF_*(L_1,L_2)\to\HF_*(L_0,L_2).
$$
This is the {\bf Donaldson triangle product}. Associativity 
follows by splitting a genus zero surface with four cylindrical ends.  
This determines a category, where the objects are the
monotone Lagrangian submanifolds $L\in\cL(M,\om)$
with minimal Maslov number at least three and
the set of morphisms from $L_0$ to $L_1$ 
is the Floer homology group $\HF_*(L_0,L_1)$.  
Composition is given by the Donaldson triangle product.
\end{PARA}

\section{Floer Gluing} \label{sec:GLUING}  

\subsection*{Boundary Operator}

In this section we prove Theorem~\ref{thm:bo}.
The proof relies on the following version of the Floer gluing theorem.

\begin{theorem}[{\bf Floer Gluing/Boundary Operator}]
\label{thm:BO}
Assume~(H) and let $(H,J)$ be a regular pair for $(L_0,L_1)$.
Choose three Hamiltonian paths $x,y,z\in\cC(L_0,L_1;H)$ and
two Floer trajectories $u\in\cM^1(x,y;H,J)$ 
and $v\in\cM^1(y,z;H,J)$. Fix two real numbers $s_u,s_v\in\R$.
Then there exist constants $T_0>0$ and $\delta_0>0$ 
and a smooth map
\begin{equation}\label{eq:gluing}
(T_0,\infty)\to\cM^2(x,z;H,J):T\mapsto u_T
\end{equation}
satisfying the following conditions.

\smallskip\noindent{\bf (i)}
The composition of~\eqref{eq:gluing} with the 
projection 
$$
\cM^2(x,z;H,J)\to\widehat{\cM}^2(x,z;H,J)
$$
is a diffeomorphism onto its image.

\smallskip\noindent{\bf (ii)}
The functions $(s,t)\mapsto u_T(s-T,t)$ converge to $u$
and the functions $(s,t)\mapsto u_T(s+T,t)$ converge to $v$
as $T$ tends to infinity; in both cases the convergence 
is uniform with all derivatives on every compact subset 
of $\R\times[0,1]$. Moreover, 
$$
\lim_{T\to\infty}\sup_t
\left(\sup_{s\le0} d(u_T(s-T,t),u(s,t))
+ \sup_{s\ge0}d(u_T(s+T,t),v(s,t))\right)
= 0.
$$

\smallskip\noindent{\bf (iii)}
For every $T>T_0$ we have 
$$
E_H(u_T) = E_H(u) + E_H(v).
$$

\smallskip\noindent{\bf (iv)}
If $u'\in\cM^2(x,z;H,J)$ satisfies 
$$
E_H(u')=E_H(u)+E_H(v)
$$ 
and
\begin{equation}\label{eq:glueps}
\inf_{s\in\R}\sup_{0\le t\le 1}d(u'(s,t),u(s_u,t)) < \delta_0,\qquad
\inf_{s\in\R}\sup_{0\le t\le 1}d(u'(s,t),v(s_v,t)) < \delta_0,
\end{equation}
then $u'$ agrees with $u_T$ up to time shift for some $T>T_0$.
\end{theorem}

\begin{proof}
See Section~\ref{sec:PROOF}.
\end{proof}

\bigbreak

\begin{proof}[Proof of Theorem~\ref{thm:bo}]
Fix a pair $x,z\in\cC(L_0,L_1;H)$ and a constant $\eps>0$.  
By Theorem~\ref{thm:BO} and the standard compactness 
theorem for Floer trajectories, the $1$-dimensional moduli space 
$$
\widehat{\cM}^2_\eps(x,z;H,J)
:= \left\{[u]\in\widehat{\cM}^2_\eps(x,z;J)\,\big|\,E_H(u)=\eps\right\}
$$ 
admits a compactification to a compact $1$-manifold 
$\overline{\widehat{\cM}}^2_\eps(x,z;H,J)$ with boundary
$$
\p\overline{\widehat{\cM}}^2_\eps(x,z;H,J)
= \bigcup_{y\in\cC(L_0,L_1;H)}\bigcup_{0<\delta<\eps}
\widehat{\cM}^1_\delta(x,y;H,J)
\times\widehat{\cM}^1_{\eps-\delta}(y,z;H,J).
$$
Since every compact $1$-manifold has an even number 
of boundary points, this implies
$$
\sum_{y\in\cC(L_0,L_1;H)}
\sum_{0<\delta<\eps}
\#\widehat{\cM}^1_\delta(x,y;H,J)\cdot
\#\widehat{\cM}^1_{\eps-\delta}(y,z;H,J)
\in 2\Z
$$
for every $\eps>0$ and every pair of intersection 
points $x,z\in L_0\cap L_1$. This is equivalent 
to the formula 
$$
\p^{H,J}\circ\p^{H,J}=0
$$
and proves Theorem~\ref{thm:bo}.
\end{proof}

\subsection*{Chain Map}

In this section we prove Theorem~\ref{thm:cm}.

\begin{PARA}\label{para:cm}\rm
Let $(M,\om)$ be a compact symplectic manifold.
Fix two objects  
$$
(L^\alpha_0,L^\alpha_1)
=\bigl\{L^\alpha_{i0},L^\alpha_{i1}
\bigr\}_{i\in I^\alpha},\qquad
(L^\beta_0,L^\beta_1)
=\bigl\{L^\beta_{i0},L^\beta_{i1}\bigr\}_{i\in I^\beta}
$$
in $\sL(M,\om)$ and two collections
$$
(H^\alpha,J^\alpha)
=\bigl\{H^\alpha_i,J^\alpha_i
\bigr\}_{i\in I^\alpha},\qquad
(H^\beta,J^\beta)
=\bigl\{H^\beta_i,J^\beta_i
\bigr\}_{i\in I^\beta}
$$
such that $(H^\alpha_i,J^\alpha_i)$ is a regular pair
for $(L^\alpha_{i0},L^\alpha_{i1})$ when $i\in I^\alpha$ and
$(H^\beta_i,J^\beta_i)$ is a regular pair
for $(L^\beta_{i0},L^\beta_{i1})$ when $i\in I^\beta$.
Let
$$
(\Sigma,L) = \left(
\Sigma,\{L_z\}_{z\in\p\Sigma},
\{\iota^{\alpha,-}_i\}_{i\in I^\alpha},
\{\iota^{\beta,+}_i\}_{i\in I^\beta}
\right)
$$
be a string cobordism from $(L^\alpha_0,L^\alpha_1)$ to
$(L^\beta_0,L^\beta_1)$ and let $(j,H,J)$ be a regular
set of Floer data on $\Sigma$ from $(H^\alpha,J^\alpha)$
to $(H^\beta,J^\beta)$.   
\end{PARA}

\begin{theorem}[{\bf Floer Gluing/Chain Map}]\label{thm:CM}
Let $(\Sigma,L)$ and $(j,H,J)$ be as in~\ref{para:cm}.  
Fix three tuples 
$$
x^\alpha=\{x^\alpha_i\}_{i\in I^\alpha},\qquad
y^\beta=\{y^\beta_i\}_{i\in I^\beta},\qquad
z^\beta=\{z^\beta_i\}_{i\in I^\beta},
$$
with $x^\alpha_i\in\cC(L^\alpha_{i0},L^\alpha_{i1};H^\alpha_i)$
and $y^\beta_i,z^\beta_i\in\cC(L^\beta_{i0},L^\beta_{i1};H^\beta_i)$, 
and let 
$$
u\in\cM^0(x^\alpha,y^\beta;j,H,J),\qquad
v\in\cM^1(y^\beta,z^\beta;H^\beta,J^\beta).
$$
Thus $v=\{v_i\}_{i\in I^\beta}$ with 
$v_i\in\cM(y_i^\beta,z_i^\beta;H_i^\beta,J_i^\beta)$
and there is an index $i_0\in I^\beta$ such that 
$v_i(s,t)=y^\beta_i(t)=z^\beta_i(t)$ for $i\ne i_0$
and 
$$
\mu_H(v_{i_0})=1.
$$ 
Fix any nonempty open set 
$W_0\subset\Sigma\setminus\im\,\iota^{\beta,+}_{i_0}$
and a real number $s_v\in\R$.
Then there exist constants $T_0>0$ and $\delta_0>0$ 
and a smooth map
\begin{equation}\label{eq:SIGMA1}
(T_0,\infty)\to\cM^1(x^\alpha,z^\beta;j,H,J):
T\mapsto u_T
\end{equation}
satisfying the following conditions.

\smallskip\noindent{\bf (i)}
The map~\eqref{eq:SIGMA1} is a diffeomorphism onto its image.

\smallskip\noindent{\bf (ii)}
The maps $u_T$ converge to $u$ as $T$ tends to infinity and
$$
v_{i_0}(s,t) = \lim_{T\to\infty}u_T\left(\iota^{\beta,+}_{i_0}(s+T,t)\right)
$$ 
for $s\in\R$ and $0\le t\le1$.
In both cases the convergence is uniform with all 
derivatives on every compact subset of $\Sigma$, 
respectively $\R\times[0,1]$. Moreover, 
$$
\lim_{T\to\infty}
\left(\sup_{\Sigma\setminus\im\,\iota_{i_0}^{\beta,+}}d(u_T,u)
+ \sup_{s\ge0}\sup_{0\le t\le1}
d(u_T\circ\iota_{i_0}^{\beta,+}(s+T,t),v_{i_0}(s,t))\right)
= 0.
$$

\smallskip\noindent{\bf (iii)}
$\cA_H(u_T) = \cA_H(u) + \cA_H(v_{i_0})$ for every $T>T_0$.

\smallskip\noindent{\bf (iv)}
If $u'\in\cM^1(x^\alpha,z^\beta;j,H,J)$ satisfies
$$
\cA_H(u') = \cA_H(u) + \cA_H(v_{i_0})
$$
and
\begin{equation}\label{eq:SIGMA1eps}
\sup_{W_0}d(u',u) < \delta_0,\qquad
\inf_{s\ge0}\sup_{0\le t\le 1}
d(u'\circ\iota^{\beta,+}_{i_0}(s,t),v_{i_0}(s_v,t)) < \delta_0,
\end{equation}
then $u'=u_T$ for some $T>T_0$.
\end{theorem}

\begin{proof}
See Section~\ref{sec:PROOF}.
\end{proof}

\bigbreak

\begin{proof}[Proof of Theorem~\ref{thm:cm}]
Abbreviate
$$
\Phi^{\beta\alpha}:
\CF_*(L_0^\alpha,L_1^\alpha)
\to \CF_*(L_0^\beta,L_1^\beta)
$$
for the Floer chain map associated to the Floer data
$(j,H,J)$ via~\eqref{eq:PHIsigma}.
Fix two tuples $x^\alpha,z^\beta$ of critical points
and consider the $1$-dimensional manifold
$\cM^1(x^\alpha,z^\beta;j,H,J)$.  
Suppose $y^\beta$ is another tuple of critical points
and 
$$
u\in\cM^0(x^\alpha,y^\beta;j,H,J),\qquad
v\in\cM^1(y^\beta,z^\beta;H^\beta,J^\beta).
$$
are solutions of the relevant Floer equation as in 
Theorem~\ref{thm:CM}.  Then the image of the gluing
map $T\mapsto u_T$ in Theorem~\ref{thm:CM} 
is an end of the $1$-manifold $\cM^1(x^\alpha,z^\beta;j,H,J)$.  
An analogous result shows that every pair of solutions
$v\in\cM^1(x^\alpha,y^\alpha;H^\alpha,J^\alpha)$
and $u\in\cM^0(y^\alpha,z^\beta;j,H,J)$ also determines 
an end of the $1$-manifold $\cM^1(x^\alpha,z^\beta;j,H,J)$. 
Combining Theorem~\ref{thm:CM} with the standard
compactness theorem in Floer theory and our transversality
assumptions, we find that every sequence
in $\cM^1_\eps(x^\alpha,z^\beta;j,H,J)$, that does not have 
a convergent subsequence, must be contained (after 
eliminating finitely elements of the sequence) 
in the union of the images of these gluing maps.
This shows that the $1$-dimensional manifold 
$\cM^1_\eps(x^\alpha,z^\beta;j,H,J)$ admits a compactification
$\overline{\cM}^1_\eps(x^\alpha,z^\beta;j,H,J)$, which is a compact 
$1$-manifold with boundary
\begin{equation*}
\begin{split}
&
\p\overline{\cM}^1_\eps(x^\alpha,z^\beta;j,H,J) \\
&= \bigcup_{y^\alpha}
\bigcup_\delta
\widehat{\cM}^1_\delta(x^\alpha,y^\beta;H^\alpha,J^\alpha)
\times\cM^0_{\eps-\delta}(y^\beta,z^\beta;j,H,J) \\
&\qquad\cup \bigcup_{y^\beta}
\bigcup_\delta
\cM^0_\delta(x^\alpha,y^\beta;j,H,J)
\times\widehat{\cM}^1_{\eps-\delta}(y^\beta,z^\beta;H^\beta,J^\beta).
\end{split}
\end{equation*}
Since every compact $1$-manifold has an even number 
of boundary points, this implies
\begin{equation*}
\begin{split}
&\sum_{y^\alpha}\sum_\delta
\#\widehat{\cM}^1_\delta(x^\alpha,y^\alpha;H^\alpha,J^\alpha)
\cdot\#\cM^0_{\eps-\delta}(y^\alpha,z^\beta;j,H,J) \\
&- \sum_{y^\beta}\sum_\delta
\#\cM^0_\delta(x^\alpha,y^\beta;j,H,J)
\cdot\#\widehat{\cM}^1_{\eps-\delta}(y^\beta,z^\beta;H^\beta,J^\beta) \\
&\in 2\Z
\end{split}
\end{equation*}
for all $\eps$ and all $x^\alpha$, $z^\beta$. This is equivalent 
to the formula 
$$
\p^\beta\circ\Phi^{\beta\alpha}
=\Phi^{\beta\alpha}\circ\p^\alpha
$$
and proves Theorem~\ref{thm:cm}.
\end{proof}

\subsection*{Chain Homotopy Equivalence}

In this section we prove Theorem~\ref{thm:che}.

\begin{PARA}\label{para:che}\rm
Let 
$$
\cL^\alpha=(L^\alpha_0,L^\alpha_1,H^\alpha,J^\alpha),\qquad
\cL^\beta=(L^\beta_0,L^\beta_1,H^\beta,J^\beta)
$$
be as in~\ref{para:cm} and let $(\Sigma,L)$ 
be a string cobordism from $(L^\alpha_0,L^\alpha_1)$ 
to $(L^\beta_0,L^\beta_1)$. Choose two regular sets
of Floer data $(j_0,H_0,J_0)$ and $(j_1,H_1,J_1)$
on $\Sigma$ from $(H^\alpha,J^\alpha)$ to $(H^\beta,J^\beta)$,
and let $\{j_\lambda,H_\lambda,J_\lambda\}_{0\le\lambda\le1}$
be a regular homotopy of Floer data from
$(j_0,H_0,J_0)$ to $(j_1,H_1,J_1)$ as in~\ref{para:CHE}.
\end{PARA}

\begin{theorem}[{\bf Floer Gluing/Chain Homotopy Equivalence}]
\label{thm:CHE}
Let $(\Sigma,L)$ and $(j_\lambda,H_\lambda,J_\lambda)$ 
be as in~\ref{para:che}. Fix three tuples 
$$
x^\alpha=\{x^\alpha_i\}_{i\in I^\alpha},\qquad
y^\beta=\{y^\beta_i\}_{i\in I^\beta},\qquad
z^\beta=\{z^\beta_i\}_{i\in I^\beta}
$$ 
with $x^\alpha_i\in\cC(L^\alpha_{i0},L^\alpha_{i1};H^\alpha_i)$
and $y^\beta_i,z^\beta_i\in\cC(L^\beta_{i0},L^\beta_{i1};H^\beta_i)$.
Let
$$
0<\lambda<1,\quad
u\in\cM^{-1}(x^\alpha,y^\beta;j_\lambda,H_\lambda,J_\lambda),\quad
v\in\cM^1(y^\beta,z^\beta;H^\beta,J^\beta),
$$
such that $v_i(s,t)=y^\beta_i(t)=z^\beta_i(t)$ for $i\ne i_0$ 
and $\mu_H(v_{i_0})=1$. Fix any nonempty open set 
$W_0\subset\Sigma\setminus\im\iota^{\beta,+}_{i_0}$
and a real number $s_v\in\R$. Then there exist constants 
$T_0>0$ and $\delta_0>0$ and a smooth map
\begin{equation}\label{eq:CHE}
(T_0,\infty)\to\cM^0(x^\alpha,z^\beta;\{j_\lambda,H_\lambda,J_\lambda\}):
T\mapsto(\lambda_T,u_T)
\end{equation}
satisfying the following conditions.

\smallskip\noindent{\bf (i)}
The map~\eqref{eq:CHE} is a diffeomorphism onto its image.

\smallskip\noindent{\bf (ii)}
$u_T$ converges to $u$,
the maps $(s,t)\mapsto u_T\circ\iota^{\beta,+}_{i_0}(s+T,t)$ 
converge to $v_{i_0}$, and $\lambda_T$ converges to $\lambda$ 
as $T$ tends to infinity; in the first two cases cases the convergence 
is uniform with all derivatives on every compact subset 
of $\Sigma$, respectively $\R\times[0,1]$. Moreover,
$$
\lim_{T\to\infty}
\left(\sup_{\Sigma\setminus\im\,\iota_{i_0}^{\beta,+}}d(u_T,u)
+ \sup_{s\ge0}\sup_{0\le t\le1}
  d(u_T\circ\iota_{i_0}^{\beta,+}(s+T,t),v_{i_0}(s,t))\right)
= 0.
$$

\smallskip\noindent{\bf (iii)}
$\cA_H(u_T) = \cA_H(u) + \cA_H(v_{i_0})$ for every $T>T_0$.

\smallskip\noindent{\bf (iv)}
If $0<\lambda'<1$ and 
$u'\in\cM^0(x^\alpha,z^\beta;j_{\lambda'},H_{\lambda'},J_{\lambda'})$ 
satisfies
\begin{equation}\label{eq:CHEeps}
\begin{split}
&\Abs{\lambda'-\lambda}<\delta_0,\qquad
\cA_H(u') = \cA_H(u) + \cA_H(v_{i_0}),\\
&\sup_{W_0}d(u',u) < \delta_0,\qquad
\inf_{s\ge0}\sup_{0\le t\le 1}
d(u'\circ\iota^{\beta,+}_{i_0}(s,t),v_{i_0}(s_v,t)) < \delta_0,
\end{split}
\end{equation}
then $(\lambda',u')=(\lambda_T,u_T)$ for some $T>T_0$.
\end{theorem}

\begin{proof}
See Section~\ref{sec:PROOF}.
\end{proof}

\bigbreak

\begin{proof}[Proof of Theorem~\ref{thm:che}]
Fix two tuples $x^\alpha,z^\beta$ of critical points
and consider the $1$-dimensional manifold
$\cM^0(x^\alpha,z^\beta;\{j_\lambda,H_\lambda,J_\lambda\}_\lambda)$.  
Suppose $y^\beta$ is another tuple of critical points
and 
$$
0<\lambda<1,\quad
u\in\cM^{-1}(x^\alpha,y^\beta;j_\lambda,H_\lambda,J_\lambda),\quad
v\in\cM^1(y^\beta,z^\beta;H^\beta,J^\beta).
$$
are solutions of the relevant Floer equation as in 
Theorem~\ref{thm:CHE}.  Then the image of the gluing
map $T\mapsto u_T$ in Theorem~\ref{thm:CHE} 
is an end of the $1$-manifold 
$\cM^0(x^\alpha,z^\beta;\{j_\lambda,H_\lambda,J_\lambda\}_\lambda)$.  
In a similar way, every pair of solutions 
${v\in\cM^1(x^\alpha,y^\alpha;H^\alpha,J^\alpha)}$ and 
${u\in\cM^{-1}(y^\alpha,z^\beta;j_\lambda,H_\lambda,J_\lambda)}$ 
with $0<\lambda<1$ also determines an end of the $1$-manifold 
$\cM^0(x^\alpha,z^\beta;\{j_\lambda,H_\lambda,J_\lambda\}_\lambda)$.  
Combining Theorem~\ref{thm:CHE} with the standard
compactness theorem in Floer theory and our transversality
assumptions, we find that every sequence in 
$\cM^0_\eps(x^\alpha,z^\beta;\{j_\lambda,H_\lambda,J_\lambda\}_\lambda)$
that does not have a convergent subsequence, is 
contained (after eliminating finitely elements of the sequence) 
in the union of the images of these gluing maps.
This shows that the $1$-dimensional manifold 
$\cM^0_\eps(x^\alpha,z^\beta;
\{j_\lambda,H_\lambda,J_\lambda\}_\lambda)$
admits a compactification which is a compact $1$-manifold 
with boundary, denoted by $\overline{\cM}^0_\eps(x^\alpha,z^\beta;
\{j_\lambda,H_\lambda,J_\lambda\}_\lambda)$, whose boundary 
is given by
\begin{equation*}
\begin{split}
&
\p\overline{\cM}^0_\eps(x^\alpha,z^\beta;
\{j_\lambda,H_\lambda,J_\lambda\}_\lambda)\\
&=
\cM^0_\eps(x^\alpha,z^\beta;j_0,H_0,J_0)
\cup \cM^0_\eps(x^\alpha,z^\beta;j_1,H_1,J_1) \\
&\qquad
\cup\bigcup_{y^\alpha}\bigcup_\delta
\widehat{\cM}^1_\delta(x^\alpha,y^\alpha;H^\alpha,J^\alpha)
\times\cM^{-1}_{\eps-\delta}
(y^\alpha,z^\beta;\{j_\lambda,H_\lambda,J_\lambda\}_\lambda) \\
&\qquad
\cup\bigcup_{y^\beta}\bigcup_\delta
\cM^{-1}_\delta(x^\alpha,y^\beta;\{j_\lambda,H_\lambda,J_\lambda\}_\lambda)
\times\widehat{\cM}^1_{\eps-\delta}(y^\beta,z^\beta;H^\beta,J^\beta).
\end{split}
\end{equation*}
Since every compact $1$-manifold has an even number 
of boundary points, this implies
\begin{equation*}
\begin{split}
&\#\cM^0_\eps(x^\alpha,z^\beta;j_1,H_1,J_1)
- \#\cM^0_\eps(x^\alpha,z^\beta;j_0,H_0,J_0) \\
&- \sum_{y^\alpha}\sum_\delta
\#\widehat{\cM}^1_\delta(x^\alpha,y^\alpha;H^\alpha,J^\alpha)
\cdot\#\cM^{-1}_{\eps-\delta}(y^\alpha,z^\beta;
\{j_\lambda,H_\lambda,J_\lambda\}_\lambda) \\
&- \sum_{y^\beta}\sum_\delta
\#\cM^{-1}_\delta(x^\alpha,y^\beta;
\{j_\lambda,H_\lambda,J_\lambda\}_\lambda)
\cdot\#\widehat{\cM}^1_{\eps-\delta}(y^\beta,z^\beta;H^\beta,J^\beta) \\
&\in 2\Z
\end{split}
\end{equation*}
for all $\eps$ and all $x^\alpha$, $z^\beta$. This is equivalent 
to the formula 
$$
\Phi^{\beta\alpha}_1-\Phi^{\beta\alpha}_0
=\p^\beta\circ\Psi^{\beta\alpha}
+\Psi^{\beta\alpha}\circ\p^\alpha
$$
and proves Theorem~\ref{thm:che}.
\end{proof}

\subsection*{Catenation}

In this section we prove Theorem~\ref{thm:cat}.

\begin{PARA}\label{para:cat1}\rm
Let $(M,\om)$ be a compact symplectic manifold.
Fix three objects  
$$
(L^\nu_0,L^\nu_1)
=\bigl\{L^\nu_{i0},L^\nu_{i1}
\bigr\}_{i\in I^\nu},\qquad
\nu=\alpha,\beta,\gamma,
$$
in $\sL(M,\om)$ and three collections
$$
(H^\nu,J^\nu)
=\bigl\{H^\nu_i,J^\nu_i\bigr\}_{i\in I^\nu},\qquad
\nu=\alpha,\beta,\gamma,
$$
such that $(H^\nu_i,J^\nu_i)$ is a regular pair
for $(L^\nu_{i0},L^\nu_{i1})$ when $i\in I^\nu$. 
Choose regular framed string cobordisms 
$$
\cS^{\alpha\beta}=(\Sigma^{\alpha\beta},L^{\alpha\beta},
j^{\alpha\beta},H^{\alpha\beta},J^{\alpha\beta})
$$ 
from
$(L^\alpha_0,L^\alpha_1,H^\alpha,J^\alpha)$
to $(L^\beta_0,L^\beta_1,H^\beta,J^\beta)$
and 
$$
\cS^{\beta\gamma}=(\Sigma^{\beta\gamma},L^{\beta\gamma},
j^{\beta\gamma},H^{\beta\gamma},J^{\beta\gamma})
$$ 
from
$(L^\beta_0,L^\beta_1,H^\beta,J^\beta)$ to
$(L^\gamma_0,L^\gamma_1,H^\gamma,J^\gamma)$
(see~\ref{para:StringFloerData} 
and~\ref{para:StringRegFloerData}).
For $T>0$ denote by 
$
(\Sigma^{\alpha\gamma}_T,L^{\alpha\gamma}_T)
$
the $T$-catenation of the string cobordisms 
$(\Sigma^{\alpha\beta},L^{\alpha\beta})$ and 
$(\Sigma^{\beta\gamma},L^{\beta\gamma})$
(see Definition~\ref{def:LPAIR}).
This catenation is equipped with Floer data
$$
(j^{\alpha\gamma}_T,H^{\alpha\gamma}_T,J^{\alpha\gamma}_T),
$$
defined by restricting the Floer data in 
$(j^{\alpha\beta},H^{\alpha\beta},J^{\alpha\beta})$
to $(\Sigma^{\alpha\beta}_{2T},L^{\alpha\beta}_{2T})$
and restricting the Floer data
$(j^{\beta\gamma},H^{\beta\gamma},J^{\beta\gamma})$
to $(\Sigma^{\beta\gamma}_{2T},L^{\beta\gamma}_{2T})$
(see equation~\eqref{eq:SigmaT}).
\end{PARA}

\begin{PARA}\label{para:cat2}\rm
Next we formulate a gluing theorem that holds 
uniformly for Hamiltonian perturbations in 
neighborhoods of $H^{\alpha\beta}$ and $H^{\beta\gamma}$.
Thus we denote by $\cH^{\alpha\beta}$ the space of 
smooth $1$-forms
$$
h^{\alpha\beta}:T\Sigma^{\alpha\beta}\to\Om^0(M)
$$
with support in the compact set
$$
\Sigma^{\alpha\beta}\setminus
(\bigcup_{i\in I^\alpha}\im\,\iota_i^{\alpha,-}
\cup\bigcup_{i\in I^\beta}\im\,\iota_i^{\beta,+}).
$$
Likewise, we denote by $\cH^{\beta\gamma}$ the space of 
smooth $1$-forms
$$
h^{\beta\gamma}:T\Sigma^{\beta\gamma}\to\Om^0(M)
$$
with support in the compact set
$$
\Sigma^{\beta\gamma}\setminus
(\bigcup_{i\in I^\beta}\im\,\iota_i^{\beta,-}
\cup\bigcup_{i\in I^\gamma}\im\,\iota_i^{\gamma,+}).
$$
\end{PARA}

\begin{theorem}[{\bf Floer Gluing/Catenation}]\label{thm:CAT}
Let $(\Sigma^{\alpha\beta},L^{\alpha\beta})$,
$(\Sigma^{\beta\gamma},L^{\beta\gamma})$,
and $(j^{\alpha\beta},H^{\alpha\beta},J^{\alpha\beta})$,
$(j^{\beta\gamma},H^{\beta\gamma},J^{\beta\gamma})$
be as in~\ref{para:cat1}. Fix three tuples
$$
x^\alpha=\{x^\alpha_i\}_{i\in I^\alpha},\qquad
x^\beta=\{x^\beta_i\}_{i\in I^\beta},\qquad
x^\gamma=\{x^\gamma_i\}_{i\in I^\gamma}
$$ 
with $x^\nu_i\in\cC(L^\nu_{i0},L^\nu_{i1};H^\nu_i)$
for $i\in I^\nu$ and $\nu=\alpha,\beta,\gamma$, and let 
$$
u^{\alpha\beta}\in\cM^0(x^\alpha,x^\beta;
j^{\alpha\beta},H^{\alpha\beta},J^{\alpha\beta}),\qquad
u^{\beta\gamma}\in\cM^0(x^\beta,x^\gamma;
j^{\beta\gamma},H^{\beta\gamma},J^{\beta\gamma}).
$$
Fix two nonempty open sets
$$
W^{\alpha\beta}\subset\Sigma^{\alpha\beta}
\setminus\im\,\iota^{\beta,+},\qquad
W^{\beta\gamma}\subset\Sigma^{\beta\gamma}
\setminus\im\,\iota^{\beta,-},
$$
where $\im\,\iota^{\beta,\pm}
:=\bigcup_{i\in I^\beta}\im\,\iota^{\beta,\pm}_i$.
Then there exist a convex open neighborhood 
$\cH_0\subset\cH^{\alpha\beta}\times\cH^{\beta\gamma}$
of the origin, constants $T_0>0$ and $\delta_0>0$,
smooth families
\begin{equation*}
\begin{split}
&u^{\alpha\beta}_h\in\cM^0(x^\alpha,x^\beta;
j^{\alpha\beta},H^{\alpha\beta}+h^{\alpha\beta},J^{\alpha\beta}),\\
&u^{\beta\gamma}_h\in\cM^0(x^\beta,x^\gamma;
j^{\beta\gamma},H^{\beta\gamma}+h^{\beta\gamma},J^{\beta\gamma}),
\end{split}
\end{equation*}
(parametrized by $h=(h^{\alpha\beta},h^{\beta\gamma})\in\cH_0$),
and a smooth family
\begin{equation}\label{eq:CAT}
u_{h,T}\in\cM^0(x^\alpha,x^\gamma;
j^{\alpha\gamma}_T,(H+h)^{\alpha\gamma}_T,J^{\alpha\gamma}_T),\qquad 
h\in\cH_0,\;T > T_0,
\end{equation}
satisfying the following conditions.

\smallskip\noindent{\bf (i)}
For $h\in\cH_0$ and $T\ge T_0$ the solutions 
$u^{\alpha\beta}_h:\Sigma^{\alpha\beta}\to M$, 
$u^{\beta\gamma}_h:\Sigma^{\beta\gamma}\to M$ and
$u_{h,T}:\Sigma^{\alpha\gamma}_T\to M$ 
of the Floer equation are regular in the sense that
the linearized operators are bijective.
Moreover, $u^{\alpha\beta}_0=u^{\alpha\beta}$ 
and $u^{\beta\gamma}_0=u^{\beta\gamma}$. 

\smallskip\noindent{\bf (ii)}
For every $h\in\cH_0$ the maps $u^{\alpha\gamma}_{h,T}$ 
converge to $u^{\alpha\beta}_h$, uniformly with all 
derivatives on every compact subset of $\Sigma^{\alpha\beta}$, 
and they converge to $u^{\beta\gamma}_h$,
uniformly with all derivatives on every compact 
subset of $\Sigma^{\beta\gamma}$ (as $T$ tends to infinity). 
Moreover, 
$$
\lim_{T\to\infty}\sup_{h\in\cH_0}
\left(
\sup_{\Sigma^{\alpha\beta}\setminus\im\,\iota^{\beta,+}}
d(u^{\alpha\gamma}_{h,T},u^{\alpha\beta}_h)
+ \sup_{\Sigma^{\beta\gamma}\setminus\im\,\iota^{\beta,-}}
d(u^{\alpha\gamma}_{h,T},u^{\beta\gamma}_h)
\right)
=0.
$$

\smallskip\noindent{\bf (iii)}
For $h\in\cH_0$ and $T\ge T_0$ we have
$$
\cA_H(u^{\alpha\gamma}_{h,T}) 
= \cA_H(u^{\alpha\beta}_h) 
+ \cA_H(u^{\beta\gamma}_h).
$$ 

\smallskip\noindent{\bf (iv)}
If $h\in\cH_0$, $T\ge T_0$, and $u'\in\cM^0(x^\alpha,x^\gamma;
j^{\alpha\gamma}_T,(H+h)^{\alpha\gamma}_T,J^{\alpha\gamma}_T)$
satisfy
\begin{equation}\label{eq:CATeps}
\begin{split}
&\cA_H(u') = \cA_H(u^{\alpha\beta}_h) + \cA_H(u^{\beta\gamma}_h), \\
&\sup_{W^{\alpha\beta}}d(u',u^{\alpha\beta}_h) < \delta_0,\qquad
\sup_{W^{\beta\gamma}}d(u',u^{\beta\gamma}_h) < \delta_0,
\end{split}
\end{equation}
then $u'=u^{\alpha\gamma}_{h,T}$.
\end{theorem}

\begin{proof}
See Section~\ref{sec:PROOF}.
\end{proof}

\bigbreak

\begin{proof}[Proof of Theorem~\ref{thm:cat}]
On the chain level the composition 
$$
\Phi^{\gamma\beta}\circ\Phi^{\beta\alpha}:
\CF_*(L_0^\alpha,L_1^\alpha;H^\alpha)\to 
\CF_*(L_0^\gamma,L_1^\gamma;H^\gamma)
$$
is given by
$$
\Phi^{\gamma\beta}\Phi^{\beta\alpha}x^\alpha
= \sum_{x^\gamma}\sum_{\eps}n^{\alpha\gamma}_\eps(x^\alpha,x^\gamma)
e^{-\eps}x^\gamma,
$$
where the number 
$n^{\alpha\gamma}_\eps(x^\alpha,x^\gamma)\in\Z/2\Z$ is defined by
\begin{equation}\label{eq:Phialga}
\begin{split}
n^{\alpha\gamma}_\eps(x^\alpha,x^\gamma)
&:= \sum_\delta\sum_{x^\beta}n^{\alpha\beta}_\delta(x^\alpha,x^\beta)
n^{\beta\gamma}_{\eps-\delta}(x^\beta,x^\gamma),\\
n^{\alpha\beta}_\delta(x^\alpha,x^\beta)
&:=
\#_2\cM^0_\delta(x^\alpha,x^\beta;
j^{\alpha\beta},H^{\alpha\beta},J^{\alpha\beta}),\\
n^{\beta\gamma}_{\eps-\delta}(x^\beta,x^\gamma)
&:=
\#_2\cM^0_{\eps-\delta}(x^\beta,x^\gamma;
j^{\beta\gamma},H^{\beta\gamma},J^{\beta\gamma}).
\end{split}
\end{equation}
We prove in four steps that 
$\Phi^{\gamma\beta}\circ\Phi^{\beta\alpha}$
is chain homotopy equivalent to the Floer chain map
$\Phi^{\gamma\alpha}$ associated to regular Floer data 
on the catenation of the string cobordisms $\Sigma^{\alpha\beta}$
and $\Sigma^{\beta\gamma}$. 

\medskip\noindent{\bf Step~1.}
{\it Fix a sequence of real numbers
$
c_0<c_1<c_2<\cdots
$
diverging to $\infty$. Then there exist sequences of real numbers $T_\nu>0$
and $\delta_\nu>0$ and a sequence of convex open neighborhoods 
$\cH_\nu\subset\cH^{\alpha\beta}\times\cH^{\beta\gamma}$
of the origin satisfying the following conditions.}

\smallskip\noindent{\bf (a)}
{\it For every $\nu\in\N_0$ we have $T_\nu<T_{\nu+1}$, 
$\delta_{\nu+1}<\delta_\nu$, and $\cH_{\nu+1}\subset\cH_\nu$. 
Moreover, $T_\nu$ diverges to infinity.}

\smallskip\noindent{\bf (b)}
{\it The assertions of Theorem~\ref{thm:CAT} hold 
with $\cH_0,T_0,\delta_0$ replaced by $\cH_\nu,T_\nu,\delta_\nu$
for all $x^\alpha,x^\beta,x^\gamma$ and all 
$u^{\alpha\beta}$ and $u^{\beta\gamma}$
such that}
$$
\cA_H(u^{\alpha\beta})+\cA_H(u^{\beta\gamma})\le c_\nu.\
$$

\smallskip\noindent{\bf (c)}
{\it For all $x^\alpha,x^\gamma$, all $\eps\le c_\nu$,
all $h\in\cH_\nu$, and all $T\ge T_\nu$ the map
\begin{equation}\label{eq:CATh}
\begin{split}
&\bigcup_\delta\bigcup_{x^\beta}
\cM^0_\delta(x^\alpha,x^\beta;
j^{\alpha\beta},H^{\alpha\beta},J^{\alpha\beta})
\times\cM^0_{\eps-\delta}(x^\beta,x^\gamma;
j^{\beta\gamma},H^{\beta\gamma},J^{\beta\gamma}) \\
&\to 
\cM^0_\eps(x^\alpha,x^\gamma;
j^{\alpha\gamma}_T,(H+h)^{\alpha\gamma}_T,
J^{\alpha\gamma}_T):
(u^{\alpha\beta},u^{\beta\gamma})\mapsto u^{\alpha\gamma}_{h,T}
\end{split}
\end{equation}
of Theorem~\ref{thm:CAT} is bijective.}

\smallskip\noindent{\bf (d)}
{\it For all $x^\alpha,x^\gamma$, all $\eps\le c_\nu$,
all $h\in\cH_\nu$, and all $T\ge T_\nu$ we have}
$$
\cM^{-1}_\eps(x^\alpha,x^\gamma;
j^{\alpha\gamma}_T,(H+h)^{\alpha\gamma}_T,J^{\alpha\gamma}_T)
= \emptyset.
$$

\medskip\noindent
The proof is by induction on $\nu$.  For $\nu=0$ it follows from
Theorem~\ref{thm:CAT} that $\cH_0,T_0,\delta_0$ can be chosen
such that~(b) holds (because there are only finitely 
many pairs $(u^{\alpha\beta},u^{\beta\gamma})$ as in Theorem~\ref{thm:CAT}
with $\cA_H(u^{\alpha\beta})+\cA_H(u^{\beta\gamma})\le c_0$).  
After shrinking $\cH_0$ and increasing $T_0$, if necessary, 
assertion~(b) continues to be valid and we claim 
that~(c) and~(d) hold as well.  

We prove this first for~(d). Suppose otherwise.  
Then there exist critical points $x^\alpha,x^\gamma$, 
a sequence of Hamiltonian perturbations $h_k\in \cH_0$, 
and a sequence of real numbers $T_k>T_0$, 
such that $h_k$ converges to zero in the $\Cinf$-topology, 
$T_k$ diverges to $\infty$, and for every $k$
$$
\bigcup_{\eps\le c_0}\cM^{-1}_\eps(x^\alpha,x^\gamma;
j^{\alpha\gamma}_{T_k},(H+h_k)^{\alpha\gamma}_{T_k},
J^{\alpha\gamma}_{T_k})
\ne \emptyset.
$$
Choose a sequence of pairs $(\eps_k,u_k)$ such that
$$
u^{\alpha\gamma}_k\in \cM^{-1}_{\eps_k}(x^\alpha,x^\gamma;
j^{\alpha\gamma}_{T_k},(H+h_k)^{\alpha\gamma}_{T_k},
J^{\alpha\gamma}_{T_k}),\qquad \eps_k\le c_0.
$$
Then the standard Floer--Gromov compactness theorem 
asserts that a suitable subsequence of of $u^{\alpha\gamma}_k$
converges, modulo bubbling, to a catenation of 
finitely many Floer trajectories for $(H^\alpha,J^\alpha)$ 
running from $x^\alpha$ to some critical point $y^\alpha$,
an element of $\cM(y^\alpha,y^\beta;
j^{\alpha\beta},H^{\alpha\beta},J^{\alpha\beta})$
for some critical point $y^\beta$,
finitely many Floer trajectories for $(H^\beta,J^\beta)$ 
running from $y^\beta$ to some critical point $z^\beta$,
an element of $\cM(z^\beta,z^\gamma;
j^{\beta\gamma},H^{\beta\gamma},J^{\beta\gamma})$
for some critical point $z^\gamma$, and finitely 
many Floer trajectories for $(H^\gamma,J^\gamma)$ 
running from $z^\gamma$ to $x^\gamma$.  By monotonicity,
the total Fredholm index of this catenation must be
less than or equal to minus one.  On the other hand, 
by our transversality hypotheses, the the total Fredholm 
index of this catenation must be bigger than or equal to zero. 
This contradiction shows that~(d) holds after 
shrinking $\cH_0$ and increasing $T_0$, if necessary.

Next we prove that~(c) holds for a suitable pair $\cH_0,T_0$.
It follows directly from Theorem~\ref{thm:CAT}~(ii) that the
map~\eqref{eq:CATh} is injective for $h$ sufficiently close 
to zero and $T$ sufficiently large. Assume, by contradiction,
that there is a pair of critical points $x^\alpha,x^\gamma$, 
a sequence of Hamiltonian perturbations $h_k\in\cH_0$
converging to zero in the $\Cinf$ topology, a sequence 
$T_k\to\infty$, a sequence $\eps_k\le c_0$, and a sequence 
$$
u^{\alpha\gamma}_k\in \cM^0_{\eps_k}(x^\alpha,x^\gamma;
j^{\alpha\gamma}_{T_k},(H+h_k)^{\alpha\gamma}_{T_k},
J^{\alpha\gamma}_{T_k})
$$
that does not belong to the image of the map~\eqref{eq:CATh}
for $(h,T)=(h_k,T_k)$. Then it follows from monotonicity, 
transversality, and the same compactness argument as in the 
proof of~(d) that a suitable subsequence of $u^{\alpha\gamma}_k$
converges to a catenation of 
an element $u^{\alpha\beta}\in\cM(x^\alpha,x^\beta;
j^{\alpha\beta},H^{\alpha\beta},J^{\alpha\beta})$
and an element $u^{\beta\gamma}\in\cM(x^\beta,x^\gamma;
j^{\beta\gamma},H^{\beta\gamma},J^{\beta\gamma})$
for some $\delta$ and some $x^\beta$.  
Moreover there cannot be any bubbling. This means that,
for $k$ sufficiently large, the map $u^{\alpha\gamma}_k$
satisfies the assumptions of Theorem~\ref{thm:CAT}~(iv)
and hence does after all belong to the image of the 
map~\eqref{eq:CATh}, a contradiction.
Thus we have proved Step~1 for $\nu=0$. 

Now let $\nu\ge 1$ and suppose that $\cH_{\nu-1},T_{\nu-1},\delta_{\nu-1}$ 
have been constructed. Using Theorem~\ref{thm:CAT} and the above
compactness argument, we find a convex open neighborhood
$\cH_\nu\subset\cH_{\nu-1}$ of zero and real numbers 
$T_\nu>\max\{\nu,T_{\nu-1}\}$, $0<\delta_\nu<\min\{1/\nu,\delta_{\nu-1}\}$ 
such that~(b), (c), and~(d) hold. This completes the induction 
argument and the sequences satisfy~(a) by construction.  
Thus we have proved Step~1.  

\medskip\noindent
{\bf Step~2.}
{\it Let $c_\nu,T_\nu,\delta_\nu,\cH_\nu$ be as in Step~1.
Choose $h_\nu\in\cH_\nu$ such that the triple
$(j^{\alpha\gamma}_{T_\nu},(H+h_\nu)^{\alpha\gamma}_{T_\nu},J^{\alpha\gamma}_{T_\nu})$ 
is regular and denote by
$$
\Phi^{\gamma\alpha}_\nu:
\CF_*(L_0^\alpha,L_1^\alpha;H^\alpha)\to 
\CF_*(L_0^\gamma,L_1^\gamma;H^\gamma)
$$ 
the associated homomorphism on the Floer chain complex.
Then the coefficients of $\Phi^{\gamma\alpha}_\nu$ agree with 
those of $\Phi^{\gamma\beta}\Phi^{\beta\alpha}$ for 
action values $\cA_{H+h_\nu}(u)\le c_\nu$.}

\medskip\noindent
The operator $\Phi^{\alpha,\gamma}_\nu$ is given by
$$
\Phi^{\gamma\alpha}_\nu x^\alpha
= \sum_{x^\gamma}\sum_{\eps}n^{\alpha\gamma}_{\nu,\eps}(x^\alpha,x^\gamma)
e^{-\eps}x^\gamma,
$$
where the number 
$$
n^{\alpha\gamma}_{\nu,\eps}(x^\alpha,x^\gamma)
:= \#_2\cM^0_\eps(x^\alpha,x^\gamma;
j^{\alpha\gamma}_{T_\nu},
(H+h_\nu)^{\alpha\gamma}_{T_\nu},
J^{\alpha\gamma}_{T_\nu}).
$$
Since $h_\nu\in\cH_\nu$, it follows from condition~(c) in 
Step~1 that
\begin{equation*}
\begin{split}
n^{\alpha\gamma}_{\nu,\eps}(x^\alpha,x^\gamma)
&= \sum_\delta\sum_{x^\beta}
n^{\alpha\beta}_\delta(x^\alpha,x^\beta)
n^{\beta\gamma}_{\eps-\delta}(x^\beta,x^\gamma) \\
&= n^{\alpha\gamma}_\eps(x^\alpha,x^\gamma)
\end{split}
\end{equation*}
for $\eps\le c_\nu$. 
(For the last step see equation~\eqref{eq:Phialga}.)
This proves Step~2.

\medskip\noindent
{\bf Step~3.}
{\it Let $c_\nu,T_\nu,\delta_\nu,\cH_\nu$ be as in Step~1
and let $h_\nu,\Phi^{\gamma\alpha}_\nu$ be as in Step~2. 
Then there exists a sequence of chain homotopy equivalences
$\Psi_\nu^{\gamma\alpha}$ so that
$$
\Phi^{\gamma\alpha}_{\nu+1}-\Phi^{\gamma\alpha}_\nu 
= \p^\gamma\circ\Psi_\nu^{\gamma\alpha}
+ \Psi_\nu^{\gamma\alpha}\circ\p^\alpha
$$
and all coefficients of $\Psi_\nu^{\gamma\alpha}$
have the form $\lambda=\sum_{\eps>c_\nu}\lambda_\eps e^{-\eps}$.}

\medskip\noindent
Fix a string cobordism $(\Sigma^{\alpha\gamma},L^{\alpha\gamma})$ 
from $(L_0^\alpha,L_1^\alpha)$ to $(L_0^\gamma,L_1^\gamma)$ 
that is equivalent to the catenation 
$(\Sigma^{\alpha\gamma}_T,L^{\alpha\gamma}_T)
= (\Sigma^{\alpha\gamma},L^{\alpha\beta})\#_T
(\Sigma^{\beta\gamma},L^{\beta\gamma})$ 
in Definition~\ref{def:LPAIR} for all $T>0$. 
Choose a family of diffeomorphisms 
$$
\phi_T:\Sigma^{\alpha\gamma}
\to \Sigma^{\alpha\gamma}_T
= \Sigma^{\alpha\beta}\#_T\Sigma^{\beta\gamma},\qquad T\ge 1,
$$
as follows. Our fixed string cobordism is 
$\Sigma^{\alpha\gamma}:=\Sigma^{\alpha\gamma}_1$
associated to $T=1$.
Now choose a smooth function $\rho:\R\to[-1,1]$
such that $\rho'\ge 0$ and
$$
\rho(s)=\left\{\begin{array}{rl}
-1,&\mbox{for }s\le -1/2,\\
1,&\mbox{for }s\ge 1/2.
\end{array}\right.
$$
For $T\ge1$ define $\rho_T:[-1,1]\to[-T,T]$ by  
$$
\rho_T(s):=s + (T-1)\rho(s)
$$
Then $\rho_T:[-1,1]\to[-T,T]$ is a diffeomorphism 
with $\rho_T(\pm 1)=\pm T$ and $\rho_T'(s)=1$ 
for $\Abs{s}\ge 1/2$. Now define $\phi_T$ by
\begin{equation*}
\begin{split}
\phi_T(\iota_i^{\beta,+}(s+1,t))
&:= \iota_i^{\beta,+}(\rho_T(s)+T,t)),\\
\phi_T(\iota_i^{\beta,-}(s-1,t))
&:= \iota_i^{\beta,-}(\rho_T(s)-T,t))
\end{split}
\end{equation*}
for $\Abs{s}\le 1$, and by the identity on 
$\Sigma^{\alpha\beta}\setminus\bigcup_{i\in I^\beta}\im\iota^{\beta,+}_i$
and 
$\Sigma^{\beta\gamma}\setminus\bigcup_{i\in I^\beta}\im\iota^{\beta,-}_i$.
Now it follows from standard transversality arguments in Floer 
theory that there exists a smooth homotopy $\{h_T\}_{T_\nu\le T\le T_{\nu+1}}$
in $\cH_\nu$ form $h_\nu$ to $h_{\nu+1}$ such that the pullbacks
$\phi_T^*(j^{\alpha\gamma}_T,(H+h_T)^{\alpha\gamma}_T,J^{\alpha\gamma}_T)$
of the resulting Floer data on $\Sigma^{\alpha\gamma}_T$ 
to $\Sigma^{\alpha\gamma}$ under $\phi_T$ define a 
regular homotopy of Floer data as in~\ref{para:CHE}.
Denote by 
$
\Psi^{\gamma\alpha}_\nu:\CF_*(L^\alpha_0,L^\alpha_1;H^\alpha)
\to \CF_*(L^\gamma_0,L^\gamma_1;H^\gamma)
$
the homomorphism associated to this homotopy via~\eqref{eq:che}.
Then, by Theorem~\ref{thm:che}, we have 
$$
\Phi^{\gamma\alpha}_{\nu+1}-\Phi^{\gamma\alpha}_\nu
= \p^\gamma\circ\Psi^{\gamma\alpha}_\nu
+ \Psi^{\gamma\alpha}_\nu\circ\p^\alpha.
$$
By equation~\eqref{eq:che} the operator $\Psi^{\gamma\alpha}_\nu$
is given by
$$
\Psi^{\gamma\alpha}_\nu x^\alpha
= \sum_{x^\beta}\sum_\eps
N^{\alpha\gamma}_{\nu,\eps}(x^\alpha,x^\gamma)e^{-\eps}x^\gamma,
$$
where
\begin{equation}\label{eq:algaps}
N^{\alpha\gamma}_{\nu,\eps}(x^\alpha,x^\gamma)
:= \sum_{T_\nu\le T\le T_{\nu+1}}
\#_2\cM^{-1}_\eps\left(x^\alpha,x^\gamma;
j^{\alpha\gamma}_T,
(H+h_T)^{\alpha\gamma}_T,
J^{\alpha\gamma}_T\right).
\end{equation}
Since $h_T\in\cH_\nu$ and $T\ge T_\nu$ it follows from~(d) in Step~1
that the moduli space $\cM^{-1}_\eps\left(x^\alpha,x^\gamma;
j^{\alpha\gamma}_T,(H+h_T)^{\alpha\gamma}_T,J^{\alpha\gamma}_T\right)$
is empty, and hence $N^{\alpha\gamma}_{\nu,\eps}(x^\alpha,x^\gamma)=0$, 
whenever $\eps\le c_\nu$.  This proves Step~3.

\bigbreak

\medskip\noindent
{\bf Step~4.}
{\it Let $c_\nu,T_\nu,\delta_\nu,\cH_\nu$ be as in Step~1,
let $h_\nu,\Phi^{\gamma\alpha}_\nu$ be as in Step~2,
and let $\Psi^{\gamma\alpha}_\nu$ be as in Step~3.
Then the infinite sum
$$
\Psi^{\gamma\alpha} 
:= \Psi^{\gamma\alpha}_0 
+ \Psi^{\gamma\alpha}_1
+ \Psi^{\gamma\alpha}_2 
+ \cdots
$$
defines a homomorphism from
$\CF_*(L_0^\alpha,L_1^\alpha;H^\alpha)$
to $\CF_*(L_0^\gamma,L_1^\gamma;H^\gamma)$ and}
$$
\Phi^{\gamma\beta}\Phi^{\beta\alpha}
-\Phi^{\gamma\alpha}_0
= \p^\gamma\circ\Psi^{\gamma\alpha}
+ \Psi^{\gamma\alpha}\circ\p^\alpha.
$$

\medskip\noindent
For $x^\alpha$, $x^\gamma$, and $\eps$ define
$$
N^{\alpha\gamma}_\eps(x^\alpha,x^\gamma) 
:= \sum_{\nu=0}^\infty 
N^{\alpha\gamma}_{\nu,\eps}(x^\alpha,x^\gamma),
$$
where the numbers $N^{\alpha\gamma}_{\nu,\eps}(x^\alpha,x^\gamma)$
are given by~\eqref{eq:algaps}. This sum is finite and 
$\sum_\eps N^{\alpha\gamma}_\eps(x^\alpha,x^\gamma)e^{-\eps}\in\Lambda$
because $N^{\alpha\gamma}_{\nu,\eps}(x^\alpha,x^\gamma)=0$
whenever $c_\nu\ge c\ge \eps$. Hence the infinite sum
$$
\Psi^{\alpha\gamma}
:=\sum_{\nu=0}^\infty\Psi^{\alpha\gamma}_\nu:
\CF_*(L_0^\alpha,L_1^\alpha;H^\alpha)
\to \CF_*(L_0^\gamma,L_1^\gamma;H^\gamma)
$$
is well defined and given by 
$$
\Psi^{\gamma\alpha}x^\alpha
= \sum_{x^\gamma}\sum_\eps
N^{\alpha\gamma}_\eps(x^\alpha,x^\gamma)e^{-\eps}x^\gamma.
$$
Denote 
$$
n^\alpha_\eps(x^\alpha,y^\alpha)
:=\#_2\widehat{\cM}^1_\eps(x^\alpha,y^\alpha;H^\alpha,J^\alpha)
$$
and similarly for $\beta$ and $\gamma$.  Then
\begin{eqnarray*}
n_\eps^{\alpha\gamma}(x^\alpha,x^\gamma)
- n_{0,\eps}^{\alpha\gamma}(x^\alpha,x^\gamma) 
&=&
\sum_{\nu=0}^\infty
\left(n_{\nu+1,\eps}^{\alpha\gamma}(x^\alpha,x^\gamma)
- n_{\nu,\eps}^{\alpha\gamma}(x^\alpha,x^\gamma)\right) \\
&=&
\sum_{\nu=0}^\infty\sum_{y^\gamma}\sum_\delta
N^{\alpha\gamma}_{\nu,\delta}(x^\alpha,y^\gamma)
n^\gamma_{\eps-\delta}(y^\gamma,x^\gamma) \\
&&
+ \sum_{\nu=0}^\infty\sum_{y^\alpha}\sum_\delta
n^\alpha_\delta(x^\alpha,y^\alpha)
N^{\alpha\gamma}_{\nu,\eps-\delta}(y^\alpha,x^\gamma) \\
&=&
\sum_{y^\gamma}\sum_\delta
N^{\alpha\gamma}_\delta(x^\alpha,y^\gamma)
n^\gamma_{\eps-\delta}(y^\gamma,x^\gamma) \\
&&
+ \sum_{y^\alpha}\sum_\delta
n^\alpha_\delta(x^\alpha,y^\alpha)
N^{\alpha\gamma}_{\eps-\delta}(y^\alpha,x^\gamma).
\end{eqnarray*}
Here the finiteness of the first sum on the right
follows from Step~2 and the finiteness of the next 
two sums follows from Step~3.   This proves
Step~4 and Theorem~\ref{thm:cat}.
\end{proof}

\section{Truncated String Cobordism}\label{sec:truc_cob}
 In this section we carry over the results of 
 Theorem \ref{cor:main_thm3} to the case of 
 truncated surfaces. 
Fix two objects  
$$
(L^\alpha_0,L^\alpha_1)
=\bigl\{L^\alpha_{i0},L^\alpha_{i1}
\bigr\}_{i\in I^\alpha},\qquad
(L^\beta_0,L^\beta_1)
=\bigl\{L^\beta_{i0},L^\beta_{i1}\bigr\}_{i\in I^\beta}
$$
in $\sL(M,\om)$ and two collections
$$
(H^\alpha,J^\alpha)
=\bigl\{H^\alpha_i,J^\alpha_i
\bigr\}_{i\in I^\alpha},\qquad
(H^\beta,J^\beta)
=\bigl\{H^\beta_i,J^\beta_i
\bigr\}_{i\in I^\beta}
$$
such that $(H^\alpha_i,J^\alpha_i)$ is a regular pair
for $(L^\alpha_{i0},L^\alpha_{i1})$ when $i\in I^\alpha$ and
$(H^\beta_i,J^\beta_i)$ is a regular pair
for $(L^\beta_{i0},L^\beta_{i1})$ when $i\in I^\beta$.
Let
$$
(\Sigma,L) = \left(
\Sigma,\{L_z\}_{z\in\p\Sigma},
\{\iota^{\alpha,-}_i\}_{i\in I^\alpha},
\{\iota^{\beta,+}_i\}_{i\in I^\beta}
\right)
$$
be a string cobordism from $(L^\alpha_0,L^\alpha_1)$ to
$(L^\beta_0,L^\beta_1)$ and let $(j,H,J)$ be a regular
set of Floer data on $\Sigma$ from $(H^\alpha,J^\alpha)$
to $(H^\beta,J^\beta)$. 
Fix two tuples $ x^{\alpha} = \{ x_i^{\alpha}\}_{i\in I^{\alpha}} $ 
and $ y^{\beta} = \{ y_i^{\beta}\}_{i\in I^{\beta}\setminus \{i_0\}} $
with $ x^{\alpha}_i \in \cC( L_{i0}^{\alpha} , L_{i1}^{\alpha}; H^{\alpha}_i) $ 
and $ y^{\beta}_i \in \cC( L_{i0}^{\beta} , L_{i1}^{\beta}; H^{\beta}_i)$. 
Denote with $ \Sigma_0 $ a truncated surface
$$
\Sigma_0:=\Sigma\setminus
\iota^{\beta,+}_{i_0}((s_0,\infty)\times[0,1]).
$$
We define the moduli space of perturbed holomorphic curves 
whose domain is truncated surface $ \Sigma_0 $ and 
we prove that it is an infinite dimensional manifold and 
that the restriction to the non-Lagrangian boundary 
is injective immersion. 

\begin{equation*}
\begin{split}
\sM
&:=\left\{u\in W^{2,2}_\loc
(\Sigma_0,M)
\,\Bigg|\,
\begin{array}{l}
u \mbox{ satisfies }\eqref{eq:FLOERsigma},\;
E_H(u)<\infty,\\
\lim\limits_{s\to-\infty}u(\iota^{\alpha,-}_i(s,t))
=x^\alpha_i(t), i\in I^\alpha\\
\lim\limits_{s\to\infty}u(\iota^{\beta,+}_i(s,t))
=y^\beta_i(t), i\in I^\beta\setminus\{i_0\}
\end{array}\right\}.
\end{split}
\end{equation*}
\begin{theorem}
The moduli space $ \sM $ defined above
is a Hilbert manifold and the restriction map 
$$ i : \sM \to \sP, \;\; i(u) := u ( \iota^{\beta,+}_{i_0} ( s_0, \cdot)) $$ 
is an injective immersion. 
\end{theorem}
\begin{proof}
 As in the proof of theorem \ref{cor:main_thm3} 
 we consider some base manifold $ \sB $ which 
 consists of $ W^{2,2}( \Sigma_0, M ) $ maps and which satisfy 
 the corresponding two Lagrangian boundary conditions and a Hilbert 
 space bundle $ \sE $ over $\sB $. The only difference is that 
 this time the fibers consist of $ (0,1) $ forms. The set $ \sM $ 
 can be seen as the zero set of the sections $ \dbar_{J,H}= \sS$
 of the bundle $ \sE $. 
 Thus it is enough to prove that this section is transverse to zero 
 section or equivalently that the vertical differential $ D_u $ 
 is surjective. 
 Here 
 \begin{equation}
  \begin{split}
  D_u : W^{2,2}_{bc} ( \Sigma_0 , u^* TM) \to W^{1,2}_{bc} ( \Sigma_0, \Lambda^{0,1} \otimes_J u^* TM) \\
  D_u(\xi) = \frac{1}{2} \Big( \nabla \xi + \nabla_{\xi} X_H + J\circ( \nabla \xi + \nabla_{\xi} X_H) \circ j + ( \nabla_{\xi} J ) \circ d_H u \circ j \Big), 
 \end{split}
 \end{equation}
 where $ X_H(u) \in \Omega^1 ( \Sigma_0, u^* TM) $ and 
 $d_H u(z)\in\Omega^1 ( \Sigma_0, u^* TM) $ is given by $ d_H u(z)(\hat{z}) = du(z)(\hat{z}) + X_{H}(z)(\hat{z}) $.
 Let $ \eta \in W^{1,2}_{bc} ( \Sigma_0,  \Lambda^{0,1} \otimes_J u^* TM)$. We show 
 that there exists $ \xi \in  W^{2,2}_{bc} ( \Sigma_0 , u^* TM)$ such that $ D_u \xi = \eta$. 
 Let $ U_i, \; i=1, \cdots, m $ be an open cover of $ \Sigma_0 $ such that each $ U_i $ 
 is diffeomorphic to one of the following 
\begin{itemize}
 \item[1)] Strip of the form $ (-\infty,0) \times(0,1) $ or $(0,1] \times(0,1) $. 
 \item[2)] Open disc or half disk. 
\end{itemize}
Let $ \beta_i, \; i=1, \cdots, m $ be a partition of unity subordinate to the cover $ U_i $, 
$\text{supp}(\beta_i )\subset U_i $. We can assume w.l.o.g. that on each $ U_i $ we have 
local con formal coordinates $ z= s+ \i t $. In these local coordinates the $ (0,1) $ form $ \eta $ 
can be written as $\eta_i= \beta_i \eta = \eta_1^i ds + J \eta_1^i dt $, where $ \eta_1 $ 
is a vector field along u with support in $U_i $. In each local chart the operator $ D_u $ 
can be identified with the operator 
$$ 
D_i \xi :=( D_u \xi ) (\frac{\p}{\p s}) : W^{2,2}_{bc} ( U_i, u^* TM) \to W^{1,2}_{bc} ( U_i, u^* TM). 
$$
In local coordinates on $U_i $ the form $ X_H(u) = X_F(u) ds + X_G(u) dt $, where 
$ X_F(u), X_G(u) $ are vector fields along $u$. Thus, the operator $D_i $ has 
the form 
$$
D_i \xi = \nabla_s \xi + \nabla_{\xi} X_F + ( \nabla_{\xi} J) ( \p_t u + X_G) + J(u) ( \nabla_t \xi + \nabla_{\xi} X_G) 
$$
Analogously as in the proof of Theorem \ref{cor:main_thm3} we can 
construct an adequate trivialization $ \Phi $ of $ \left. u^*TM\right|_{U_i},$
which transform the operator $D_i $ to an operator of the form \eqref{eq:opD}.
In the case that $ U_i $ is a disk or half disk, we extend first 
the functions $ X_F $ and $ X_G $ in such a way that we obtain an operator on 
the strip. Thus, the same argument as in Corollary \ref{cor:surj} 
proves that each $ D_i $ is surjective. 
Particularly, this means that for each $ \eta_i = \beta_i \eta $ there exist $ \xi_i $ such 
that $\text{supp}(\xi)\subset U_i $ and $ D_u \xi_i = \eta_i $. Extend $ \xi_i $ 
by zero to entire $ \Sigma_0 $. We have
$$ D_u ( \sum_i \xi_i) = \sum_i \eta_i = \sum_i \beta_i \eta = \eta .$$
Thus for $ \xi = \sum_i \xi_i $, we have $ D_u \xi = \eta $. 
In order to prove that the mapping $ i $ is injective immersion 
we need to prove the analogous estimate as in \eqref{eq:inq_D}, 
more precisely we prove that each $ \xi \in W^{2,2}_{bc} ( \Sigma_0, u^* TM) $ 
satisfies the inequality 
\begin{equation}\label{eq:opDs}
 \| \xi \|_{2,2} \leq c \Big ( \| D_u \xi \|_{1,2} + \| \xi \circ \iota_{i_0}^{\beta,+} ( s_0, \cdot)\|_{3/2} \Big ). 
\end{equation}
Let $ \beta_i $ be the partition of unity as above 
and let $ \xi_i = \beta_i \xi $. 
As in the proof of the inequality \eqref{eq:inq_D}
( have a look at Lemma \ref{lem:inq_D}) we have that each $ \xi_i $ 
satisfies the following inequality 
$$ 
\| \xi_i \|_{2,2} \leq c \Big ( \| D_u \xi_i \|_{1,2} + \| K_i \xi \|_{1,2} + \| \xi_i \circ \iota^{\beta,+}_{i_0} ( s_0, \cdot)) \|_{3/2} \Big ), 
$$
 where $ K_i $ are some compact operators ( for example just a restriction 
 operator to some compact set). Summing these inequalities for all $ i$ we 
 obtain 
 \begin{align*}
  \| \xi \|_{2,2}& \leq c \Big ( \sum\limits_i \| D_u ( \beta_i \xi) \|_{1,2} + \sum\limits_i \| K_i \xi_i\|_{1,2} + \| \xi \circ \iota^{\beta,+}_{i_0} ( s_0, \cdot)) \|_{3/2} \Big )\\
  & \leq c \Big (  \sum\limits_i  \| \beta_i D_u(\xi)\|_{1,2}  + \sum\limits_i (\|\dot{\beta}_i\xi\|_{1,2}+ \| K_i \xi_i\|_{1,2}) + \| \xi \circ \iota^{\beta,+}_{i_0} ( s_0, \cdot)) \|_{3/2} \Big )\\
  &\leq c \Big ( \| D_u \xi \|_{1,2} + \| K\xi \|_{1,2} +  \| \xi \circ \iota^{\beta,+}_{i_0} ( s_0, \cdot)) \|_{3/2} \Big ),  
 \end{align*}
 where $ K$ is some compact operator. From unique continuation
 we have that the mapping $ \xi \mapsto ( D_u \xi,  \xi \circ \iota^{\beta,+}_{i_0} ( s_0, \cdot))$ 
 is injective and as the above inequality holds its 
 image is closed, thus from the open mapping theorem 
 its inverse is also bounded and we can omit the middle term from the 
 above inequality. Thus, we have proved the inequality 
 \eqref{eq:opDs}. This inequality implies that the mapping $ i$ 
 is an immersion, and from unique continuation of 
 perturbed holomorphic curves follows that $i$ is also injective. 

\end{proof}

\section{Proof of the Floer Gluing Theorems}\label{sec:PROOF}  

\subsection*{Boundary Operator}

In this section we prove Theorem~\ref{thm:BO}. 

\begin{lemma}\label{le:TIME}
Assume~(H) and let $(H,J)$ be a regular pair for $(L_0,L_1)$.
Let $x,y\in\cC(L_0,L_1;H)$ and fix a real number $\delta>0$.
Then, for every Floer trajectory $u\in\cM(x,y;H,J)$ with 
$
E_H(u)>\delta, 
$
there is a unique real number $T_\delta(u)$ such that 
$$
\int_{-\infty}^{T_\delta(u)}\int_0^1\Abs{\p_su}_t^2\,dtds
= \delta.
$$
Moreover, the map 
$$
T_\delta:\left\{u\in\cM(x,y;H,J)\,|\,E_H(u)>\delta\right\}\to\R
$$
is smooth.
\end{lemma}

\begin{proof}
Define the map 
$
E_H:\cM(x,y;H,J)\times\R\to(0,\infty)
$
by 
$$
E_H(u,T) := \int_{-\infty}^{T}\int_0^1\Abs{\p_su}_t^2\,dtds.
$$
This map is smooth and
$$
\frac{\p}{\p T}E_H(u,T) = \int_0^1\Abs{\p_su(T,t)}_t^2\,dt
$$
for every $u\in\cM(x,y;H,J)$ and every $T\in\R$.  
By unique continuation we have 
$\p_TE_H(u,T)>0$ for every for every $u\in\cM(x,y;H,J)$ 
and every $T\in\R$ such that $E_H(u,T)>0$. 
Hence the assertions of Lemma~\ref{le:TIME} follow from 
the intermediate value theorem (existence), the fact that 
the map ${T\mapsto E_H(u,T)}$ is strictly monotone 
unless $\p_su\equiv0$ (uniqueness), and the 
implicit function theorem (smoothness).
\end{proof}

\begin{PARA}\label{para:bo-proof2}\rm
Let $\sP=\sP^{3/2}$ be the path space defined in 
equation~\eqref{eq:hil_path}. Choose open neighborhoods 
$U,V\subset M$ of $y(0)$ as in~\ref{para:conv_thm},
choose a constant $\hbar>0$ such that the assertion of 
Theorem~\ref{thm:MON} holds with $U,V,\Lambda=\{y(0)\}$,
and choose a neighborhood $\sU\subset\sP$ of $y$ 
and a constant $T_0>0$ such that the assertions of 
Theorem~\ref{thm:main_thm3}  are satisfied with $x$ replaced by~$y$.  
Here the Hilbert manifolds $\sM^\infty(y,\sU)$ and $\sM^T(y,\sU)$
are defined by~\eqref{eq:mod_emb}
and the embeddings 
$$
\iota^\infty:\sM^\infty(y,\sU)\to\sP\times\sP,\qquad
\iota^T:\sM^T(y,\sU)\to\sP\times\sP
$$ 
for $T\ge T_0$ are defined by~\eqref{eq:iinf} and \eqref{eq:map_i}.   
\end{PARA}

\begin{proof}[Proof of Theorem~\ref{thm:BO}]
The proof has five steps. 

\medskip\noindent{\bf Step~1.}
{\it Abbreviate 
$$
u^\pm := u|_{\R^\pm\times[0,1]},\qquad
v^\pm := v|_{\R^\pm\times[0,1]}.
$$
We may assume without loss of generality that}
\begin{equation}\label{eq:deltasU}
\delta_u:=E_H(u^+)<\hbar/2,\quad 
\delta_v:=E_H(v^-)<\hbar/2,\quad
u(0,\cdot),v(0,\cdot)\in\sU. 
\end{equation}

\medskip\noindent
Choose $s_0>0$ so large that the solutions
$\tu,\tv:\R\times[0,1]\to M$
of the Floer equations, defined by
\begin{equation*}
\begin{split}
\tu(s,t) &:= u(s+s_0,t),\qquad 
\ts_u := s_u-s_0,\\
\tv(s,t) &:= v(s-s_0,t),\qquad
\ts_v := s_v+s_0,
\end{split}
\end{equation*}
satisfy the conditions in~\eqref{eq:deltasU}. 
Assume that Theorem~\ref{thm:BO}
holds for the quadruple $(\tu,\tv,\ts_u,\ts_v)$ and denote 
the resulting glued solutions of the Floer equation by 
$\tu_T\in\cM^2(x,z;H,J)$ for $T>\tT_0$. 
Then the functions 
$$
u_T(s,t) := \tu_{T-s_0}(s,t),\qquad T>T_0:=\tT_0+s_0
$$
satisfy the requirements of Theorem~\ref{thm:BO}
for the quadruple $(u,v,s_u,s_v)$. 

\medskip\noindent{\bf Step~2.}
{\it Construction of the map
$(T_0,\infty)\to\cM^2(x,z;H,J):T\mapsto u_T$.}

\medskip\noindent
By Step~1, we have $(u^+,v^-)\in\sM^\infty(y,\sU)$ and
$$
(u(0,\cdot),v(0,\cdot))
= \iota^\infty(u^+,v^-) \in\sW^\infty(y,\sU).
$$
Now define 
$$
\eps^- := E_H(u^-) = E_H(u)-\delta_u,\qquad
\eps^+:= E_H(v^+) = E_H(v)-\delta_v.
$$
Then the spaces
\begin{equation*}
\begin{split}
\sM^-(x,\eps^-)
&:=\left\{w^-\in W^{2,2}_\loc(\R^-\times[0,1],M)\,\Bigg|\,
\begin{array}{l}
w^-\mbox{ satisfies }\eqref{eq:FLOER},\\
E_H(w^-)=\eps^-,\\
\lim\limits_{s\to-\infty}w^-(s,t)=x(t)
\end{array}\right\},\\
\sM^+(z,\eps^+)
&:=\left\{w^+\in W^{2,2}_\loc(\R^-\times[0,1],M)\,\Bigg|\,
\begin{array}{l}
w^+\mbox{ satisfies }\eqref{eq:FLOER},\\
E_H(w^+)=\eps^+,\\
\lim\limits_{s\to\infty}w^+(s,t)=z(t)
\end{array}\right\}
\end{split}
\end{equation*}
are Hilbert manifolds and the restriction maps 
$$
\iota^-:\sM^-(x,\eps^-)\to\sP,\qquad
\iota^+:\sM^+(z,\eps^+)\to\sP
$$
defined by $\iota^-(w^-):=w^-(0,\cdot)$ and $\iota^+(w^+):=w^+(0,\cdot)$
are injective immersions.  The transversality condition asserts 
that the map
\begin{equation}\label{eq:iotapm}
\iota^-\times\iota^+:
\sM^-(x,\eps^-)\times\sM^+(z,\eps^+)\to\sP\times\sP
\end{equation}
intersects the submanifold 
$$
\sW^\infty(y,\sU) = \iota^\infty(\sM^\infty(y,\sU))\subset\sP\times\sP
$$
transversally in the point 
$(u(0,\cdot),v(0,\cdot))
= \left(\iota^-(u^-),\iota^+(v^+)\right)$.
Moreover, the Fredholm index is zero, so the intersection point is isolated.
Hence, by Theorem~\ref{thm:main_thm3}, the map~\eqref{eq:iotapm} 
is also transverse to the submanifold 
$$
\sW^T(y,\sU)=\iota^T(\sM^T(y,\sU))\subset\sP\times\sP
$$ 
for $T$ sufficiently large. Hence it follows from the infinite 
dimensional inverse function theorem that there is a $T_0>0$ 
such that, for every $T>T_0$, there exists a unique 
pair $(u_T^-,v_T^+)\in\sM^-(x,\eps^-)\times\sM^+(z,\eps^+)$ 
near the pair $(u^-,v^+)$
such that
$
\left(\iota^-(u_T^-),\iota^+(v_T^+)\right)\in\sW^T(y,\sU).
$
Thus, for $T>T_0$, there is a unique element 
$w_T\in\sM^T(y,\sU)$ such that 
$$
w_T(-T,t)=u_T^-(0,t),\qquad
w_T(T,t) = v_T^+(0,t)
$$
for every $t\in[0,1]$.  Now define $u_T:\R\times[0,1]\to M$ by
\begin{equation}\label{eq:uT}
u_T(s,t)
:= \left\{\begin{array}{ll}
u_T^-(s+T,t),&\mbox{if }s\le -T,\\
w_T(s,t),&\mbox{if }\Abs{s}\le T,\\
v_T^+(s-T,t),&\mbox{if }s\ge T.
\end{array}\right.
\end{equation}
Then $u_T\in\cM^2(x,z;H,J)$ for every $T\ge T_0$.
This proves Step~2.

\medskip\noindent{\bf Step~3.}
{\it The map $T\mapsto u_T$ satisfies~(i) and~(iii).}

\medskip\noindent
For $w\in\cM^2(x,z;H,J)$ denote its equivalence class
under time shift by $[w]$.  Define the map 
$T:\cM^2(x,z;H,J)\to(0,\infty)$ by 
$$
T(w) := \frac{T_{\eps^-+\delta_u+\delta_v}(w)-T_{\eps^-}(w)}{2}
$$
for $w\in\cM^2(x,z;H,J)$.  This map is smooth, by Lemma~\ref{le:TIME},
and is invariant under time shift.  Hence 
the map descends to the quotient 
$\widehat{\cM}^2(x,z;H,J)$.  
Moreover, by construction 
\begin{equation}\label{eq:energy-uT}
E_H(u_T) = E_H(u)+E_H(v) = \eps^-+\delta_u + \delta_v +\eps^+.
\end{equation}
The energy of $u_T$ on $(-\infty,T]\times[0,1]$ is equal to $\eps^-$
and the energy of $u_T$ on $[T,\infty)\times[0,1]$ is equal to $\eps^+$.
Hence $T_{\eps^-+\delta_u+\delta_v}(u_T)=T$ and $T_{\eps^-}(u_T)=-T$,
and hence $T(u_T)=T$ for every $T>T_0$.  This shows that the map 
$$
(T_0,\infty)\to\widehat{\cM}^2(x,z;H,J):T\mapsto[u_T]
$$ 
is injective, its image is an open subset of the $1$-manifold
$\widehat{\cM}^2(x,z;H,J)$, and its inverse is smooth 
by Lemma~\ref{le:TIME}.  Thus we have proved Step~3.

\medskip\noindent{\bf Step~4.}
{\it The map $T\mapsto u_T$ satisfies~(ii).}

\medskip\noindent
By construction in Step~2, 
the functions $u_T^-$ converge to $u^-:=u|_{\R^-\times[0,1]}$ 
in the Hilbert manifold $\sM^-(x,\eps^-)$ as $T$ tends to infinity,
and the functions $v_T^+$ converge to $v^+:=v|_{\R^+\times[0,1]}$ 
in the Hilbert manifold $\sM^+(z,\eps^+)$ as $T$ tends to infinity.
(See the proof of Step~2.) 
(It follows also from the construction that $u_T([-T,T]\times\{t\})$
is contained in the neighborhood $U_t=\phi_t(U)$ of $y(t)$ for 
every $t\in[0,1]$ and every $T>T_0$.)
Hence $u_T(s-T,t)=u_T^-(s,t)$ converges to $u(s,t)$ in the 
$W^{2,2}$ topology on $\R^-\times[0,1]$ and 
$u_T(s+T,t)=u_T^+(s,t)$ converges to $v(s,t)$ in the 
$W^{2,2}$ topology on $\R^+\times[0,1]$.
By elliptic bootstrapping it follows 
that the functions $u_T(-T+\cdot,\cdot)$ converges to $u$ 
uniformly with all derivatives on every 
compact subset of $(-\infty,0)\times[0,1]$.
Likewise, the functions $u_T(T+\cdot,\cdot)$ 
converge to $v$ uniformly with all derivatives on every 
compact subset of $(0,\infty)\times[0,1]$.
Since no energy is lost, by~(ii), one can now use the 
standard compactness theorem for Floer trajectories
to exclude bubbling and prove in both cases
that the convergence is in the $\Cinf$ topology on every
compact subset of $\R\times[0,1]$.  
The above convergence statement in the Hilbert manifold
$\sM^-(x,\eps^-)$ also implies that $u_T(s-T,t)$ converges 
to $u(s,t)$ uniformly on $\R^-\times[0,1]$ and hence on 
every subset of the form $(-\infty,b]\times[0,1]$.
Likewise, $u_T(s+T,t)$ converges to $v(s,t)$ uniformly 
on every subset of the form $[a,\infty)\times[0,1]$.

\smallbreak

\medskip\noindent{\bf Step~5.}
{\it The map $T\mapsto u_T$ satisfies~(iv).}

\medskip\noindent
Assume, by contradiction, that~(iv) does not hold.
Then there are sequences $w_i\in\cM^2(x,z;H,J)$ 
and $s^-_i<s_i^+$ such that
\begin{description}
\item[(a)]
$[w_i]\ne[u_T]$ for every $i$ and every $T>T_0$, 
\item[(b)]
$E_H(w_i)=E_H(u)+E_H(v)$ for every $i$, and
\item[(c)]
For $0\le t\le 1$ we have
$$
\lim\limits_{i\to\infty}w_i(s_i^-,t) = u(s_u,t),\qquad
\lim\limits_{i\to\infty}w_i(s_i^+,t) = v(s_v,t).
$$
The convergence is uniform in $t$.
\end{description}
By the standard elliptic bootstrapping, bubbling, and removal 
of singularities argument, we may assume, passing to a 
subsequence if necessary, that the sequence 
$w_i(s_i^\pm+\cdot,\cdot)$ converges, uniformly with all derivatives
on every compact subset of the complement of a finite set
in $\R\times[0,1]$, to a smooth finite energy solution 
$
u^\pm:\R\times[0,1]\to M
$ 
of~\eqref{eq:FLOER}. (See~\cite[Chapter~4]{MS}.)
By~(c) we have 
$$
u^-(0,t) = u(s_u,t),\qquad u^+(0,t)=v(s_v,t)
$$
for every $t\in[0,1]$.  Hence it follows from unique continuation 
that 
$$
u^-(s,t)=u(s_u+s,t),\qquad
u^+(s,t)=v(s_v+s,t)
$$
for all $s$ and $t$.  Since $u$ and $v$ are not related by time shift, 
it follows that $s_i^+-s_i^-$ diverges to $\infty$.  

Moreover, it follows from~(b) that there is no loss of energy.
Hence there is no bubbling and 
\begin{equation}\label{eq:uvwi}
\begin{split}
u(s,t) = \lim_{i\to\infty}w_i(r_i^-+s,t),\qquad
r_i^-:=s_i^--s_u,\\
v(s,t) =  \lim_{i\to\infty}w_i(r_i^++s,t),\qquad
r_i^+:=s_i^++s_v.
\end{split}
\end{equation}
The convergence is uniform with all derivatives on every
compact subset of $\R\times[0,1]$. This implies that, for every $T>0$,
\begin{equation*}
\begin{split}
\eps_T
&:=
E_H(u) + E_H(v) 
- \int_{-T}^T\int_0^1\Abs{\p_su}_t^2
- \int_{-T}^T\int_0^1\Abs{\p_sv}_t^2 \\
&\;= 
\lim_{i\to\infty}\left(E_H(w_i) 
- \int_{r_i^--T}^{r_i^-+T}\int_0^1\Abs{\p_sw_i}_t^2
- \int_{r_i^+-T}^{r_i^++T}\int_0^1\Abs{\p_sw_i}_t^2
\right) \\
&\;=
\lim_{i\to\infty}\left(
\int_{-\infty}^{r_i^--T}\int_0^1\Abs{\p_sw_i}_t^2
+ \int_{r_i^-+T}^{r_i^+-T}\int_0^1\Abs{\p_sw_i}_t^2
+ \int_{r_i^++T}^\infty\int_0^1\Abs{\p_sw_i}_t^2
\right).
\end{split}
\end{equation*}
Here we have used~(b).
Taking the limit $T\to\infty$, we deduce that
all three limits on the right converge to zero
as $T$ tends to infinity.  Hence
\begin{equation*}
\begin{split}
\eps^-
&= \int_{-\infty}^0\int_0^1\Abs{\p_su}_t^2 
= \lim_{T\to\infty}\int_{-T}^0\int_0^1\Abs{\p_su}_t^2 
= \lim_{T\to\infty}\lim_{i\to\infty}
\int_{r_i^--T}^{r_i^-}\int_0^1\Abs{\p_sw_i}_t^2 \\
&= \lim_{T\to\infty}\lim_{i\to\infty}
\int_{r_i^--T}^{r_i^-}\int_0^1\Abs{\p_sw_i}_t^2 
+ \lim_{T\to\infty}\lim_{i\to\infty}
\int_{-\infty}^{r_i^--T}\int_0^1\Abs{\p_sw_i}_t^2 \\
&= \lim_{i\to\infty}
\int_{-\infty}^{r_i^-}\int_0^1\Abs{\p_sw_i}_t^2,
\end{split}
\end{equation*}
and similarly
$
\eps^+
= \lim_{i\to\infty}
\int_{r_i^+}^\infty\int_0^1\Abs{\p_sw_i}_t^2.
$
Adding to $r^\pm_i$ a sequence converging to zero, if necessary, 
we may assume w.l.o.g.\ that, for every $i$, 
\begin{equation}\label{eq:sipm}
\eps^- = \int_{-\infty}^{r_i^-}\int_0^1\Abs{\p_sw_i}_t^2,\qquad
\eps^+ = \int_{r_i^+}^\infty\int_0^1\Abs{\p_sw_i}_t^2.
\end{equation}
We claim that, for large $i$,
\begin{equation}\label{eq:wiuTi}
w_i\left(T_i+\cdot,\cdot\right) = u_{T_i},\qquad
T_i:=\tfrac12\left(r_i^+-r_i^-\right)
\end{equation}
for $i$ sufficiently large, in contradiction to~(a).
To see this, note that by~\eqref{eq:sipm},
\begin{equation*}
\begin{split}
u_i^- &:=w_i(r_i^-+\cdot,\cdot)|_{\R^-\times[0,1]}\in\sM^-(x,\eps^-),\\
v_i^+ &:=w_i(r_i^++\cdot,\cdot)|_{\R^+\times[0,1]}\in\sM^+(z,\eps^+).
\end{split}
\end{equation*}
Then, by~\eqref{eq:uvwi} and~\eqref{eq:sipm}, the sequence 
$u_i^-$ converges to $u^-$ in $\sM^-(x,\eps^-)$ and the sequence 
$v_i^+$ converges to $v^+$ in $\sM^+(z,\eps^+)$.
Moreover, by~(b) and~\eqref{eq:sipm}, we have
$
\int_{r_i^-}^{r_i^+}\int_0^1\Abs{\p_sw_i}_t^2 
= \delta_u+\delta_v < \hbar
$
for every $i$.  Since 
$$
\lim_{i\to\infty}w_i(r_i^-,\cdot)=u(0,\cdot)\in\sU,\qquad
\lim_{i\to\infty}w_i(r_i^+,\cdot)=v(0,\cdot)\in\sU,
$$
where the convergence is in the topology of $\sP=\sP^{3/2}$,
it follows from the definition of $\hbar$ in the proof of Step~2, 
that
$$
(\iota^-(u_i^-),\iota^+(v_i^+))
= (w_i(r_i^-,\cdot),w_i(r_i^+,\cdot))
\in\sW^{T_i}(y,\sU)
$$
for $i$ sufficiently large.  Hence 
it follows from the definition of $u_T$ in~\eqref{eq:uT} 
that
$$
u_{T_i}(s,t)
:= 
\left\{\begin{array}{ll}
u_i^-(s+T_i,t),&\mbox{if }s\le -T_i,\\
w_i(\frac{r_i^++r_i^-}{2}+s,t),&\mbox{if }\Abs{s}\le T_i,\\
v_i^+(s-T_i,t),&\mbox{if }s\ge T_i,
\end{array}\right\}  
=
w_i\left(T_i+s,t\right).
$$
This proves equation~\eqref{eq:wiuTi}
and Theorem~\ref{thm:BO}.
\end{proof}

\subsection*{Chain Map}

In this section we prove Theorem~\ref{thm:CM}.

\begin{PARA}\label{para:cm-proof1}\rm
Let $(\Sigma,L)$ be a string cobordism in $\sL(M,\om)$
from $(L^\alpha_0,L^\alpha_1)$ to $(L^\beta_0,L^\beta_1)$,
$(j,H,J)$ be regular set of Floer data on $(\Sigma,L)$ 
from $(H^\alpha,J^\alpha)$ to $(H^\beta,J^\beta)$, and
$$
u\in\cM^0(x^\alpha,y^\beta;j,H,J),\qquad
v\in\cM^1(y^\beta,z^\beta;H^\beta,J^\beta),
$$
as in the assumptions of Theorem~\ref{thm:CM}.
Thus 
$$
v=\{v_i\}_{i\in I^\beta},\qquad
v_i\in\cM(y_i^\beta,z_i^\beta;H_i^\beta,J_i^\beta),
$$
and there is an index $i_0\in I^\beta$ such that 
$v_i(s,t)=y^\beta_i(t)=z^\beta_i(t)$ for $i\ne i_0$
and $\mu_H(v_{i_0})=1$. 
\end{PARA}

\begin{PARA}\label{para:cm-proof2}\rm
As in~\ref{para:bo-proof2}, 
let $\sP=\sP^{3/2}$ be the path space defined in 
equation~\eqref{eq:sp3/2}. Choose open neighborhoods 
$U,V\subset M$ of $y^\beta_{i_0}(0)$ as in~\ref{para:conv_thm},
choose a constant $\hbar>0$ such that the assertion of 
Theorem~\ref{thm:MON} holds with $U,V,\Lambda=\{y^\beta_{i_0}(0)\}$,
and choose a neighborhood $\sU\subset\sP$ of $y^\beta_{i_0}$ 
and a constant $T_0>0$ such that the assertions of 
Theorem~\ref{thm:main_thm3} are satisfied with $x$ replaced 
by~$y^\beta_{i_0}$.  Here the Hilbert manifolds 
$\sM^\infty(y^\beta_{i_0},\sU)$ and $\sM^T(y^\beta_{i_0},\sU)$
are defined by~\eqref{eq:mod_emb} 
and the embeddings 

$$
\iota^\infty:\sM^\infty(y^\beta_{i_0},\sU)\to\sP\times\sP,\qquad
\iota^T:\sM^T(y^\beta_{i_0},\sU)\to\sP\times\sP
$$ 
for $T\ge T_0$ are defined by ~\eqref{eq:iinf} and \eqref{eq:map_i}.   
Denote their images by 
$$
\sW^\infty(y^\beta_{i_0},\sU) 
:= \iota^\infty(\sM^\infty(y^\beta_{i_0},\sU)),\qquad
\sW^T(y^\beta_{i_0},\sU) 
:= \iota^T(\sM^T(y^\beta_{i_0},\sU)).
$$
\end{PARA}

\begin{proof}[Proof of Theorem~\ref{thm:CM}]
The proof has three steps.

\medskip\noindent{\bf Step~1.}
{\it Construction of the map
$(T_0,\infty)\to\cM^1(x^\alpha,z^\beta;j,H,J):T\mapsto u_T$.}

\medskip\noindent
Choose $s_0>0$ so large that 
\begin{equation}\label{eq:s0hbar}
\begin{split}
&\int_{s_0}^\infty\int_0^1
\Abs{\p_su(\iota^{\beta,+}_{i_0}(s,t))}^2_t\,dtds
< \hbar/2,\\
&\delta_v := \int_{-\infty}^{-s_0}\int_0^1
\Abs{\p_sv(s,t)}^2_t\,dtds
< \hbar/2,\\
&u(\iota^{\beta,+}_{i_0}(s_0,\cdot))\in\sU,\qquad
v_{i_0}(-s_0,\cdot)\in\sU.
\end{split}
\end{equation}
Then
$$
(u(\iota^{\beta,+}_{i_0}(s_0,\cdot)),v_{i_0}(-s_0,\cdot))
\in\sW^\infty(y^\beta_{i_0},\sU).
$$
Denote 
$$
\Sigma_0:=\Sigma\setminus
\iota^{\beta,+}_{i_0}((s_0,\infty)\times[0,1]),\qquad
\eps_v := E_H(v)-\delta_v.
$$
Then the spaces
\begin{equation*}
\begin{split}
\sM^-
&:=\left\{w^-\in W^{2,2}_\loc
(\Sigma_0,M)
\,\Bigg|\,
\begin{array}{l}
w^-\mbox{ satisfies }\eqref{eq:FLOERsigma},\;
E_H(w^-)<\infty,\\
\lim\limits_{s\to-\infty}w^-(\iota^{\alpha,-}_i(s,t))
=x^\alpha_i(t), i\in I^\alpha\\
\lim\limits_{s\to\infty}w^-(\iota^{\beta,+}_i(s,t))
=y^\beta_i(t), i\in I^\beta\setminus\{i_0\}
\end{array}\right\},\\
\sM^+
&:=\left\{w^+\in W^{2,2}_\loc([-s_0,\infty)\times[0,1],M)\,\Bigg|\,
\begin{array}{l}
w^+\mbox{ satisfies }\eqref{eq:FLOER},\\
E_H(w^+)=\eps_v,\\
\lim\limits_{s\to\infty}w^+(s,t)=z^\beta_{i_0}(t)
\end{array}\right\}
\end{split}
\end{equation*}
are Hilbert manifolds and the restriction maps
$\iota^\pm:\sM^\pm\to\sP$, defined by 
$$
\iota^-(w^-):=w^-(\iota^{\beta,+}_{i_0}(s_0,\cdot)),\qquad
\iota^+(w^+):=w^+(-s_0,\cdot),
$$
are injective immersions (see~\ref{cor:main_thm3}).  
By transversality, the map
\begin{equation}\label{eq:iotapm1}
\iota^-\times\iota^+:
\sM^-\times\sM^+\to\sP\times\sP
\end{equation}
intersects the submanifold $\sW^\infty(y^\beta_{i_0},\sU)$
transversally in the point 
$$
(u(\iota^{\beta,+}_{i_0}(s_0,\cdot)),v(-s_0,\cdot))
= \left(\iota^-(u|_{\Sigma_0}),
\iota^+(v|_{[-s_0,\infty)\times[0,1]})\right).
$$
Moreover, the Fredholm index is zero, 
so the intersection point is isolated.
Hence, by Theorem~\ref{thm:main_thm3}, the map~\eqref{eq:iotapm1} 
is also transverse to the submanifold $\sW^T(y^\beta_{i_0},\sU)$ 
for $T$ sufficiently large. Hence it follows from the infinite 
dimensional inverse function theorem that there is a $T_0>0$ 
such that, for every $T>T_0$, there exists a unique 
pair $(u_T^-,v_T^+)\in\sM^-\times\sM^+$ near the pair 
$(u|_{\Sigma_0},v|_{[-s_0,\infty)\times[0,1]})$
such that
$$
\left(\iota^-(u_T^-),\iota^+(v_T^+)\right)\in\sW^T(y,\sU).
$$
Thus, for $T>T_0$, there is a unique element 
$w_T\in\sM^T(y^\beta_{i_0},\sU)$ such that 
$$
w_T(-T,t) = u_T^-(\iota^{\beta,+}_{i_0}(s_0,t)),\qquad
w_T(T,t) = v_T^+(-s_0,t)
$$
for every $t\in[0,1]$.  Now define $u_T:\R\times[0,1]\to M$ by
\begin{equation}\label{eq:uT1}
\left\{\begin{array}{ll}
u_T(z) := u_T^-(z),&\mbox{if } z\in\Sigma_0,\\
u_T(\iota^{\beta,+}_{i_0}(s_0+s,t))
:=w_T(-T+s,t),&\mbox{if }\Abs{s}\le T,\\
u_T(\iota^{\beta,+}_{i_0}(s_0+2T+s,t))
:=v_T^+(-s_0+s,t),&\mbox{if }s\ge 0.
\end{array}\right.
\end{equation}
Then $u_T\in\sM^1(x^\alpha,z^\beta;j,H,J)$ for every $T\ge T_0$.
This proves Step~1.

\bigbreak

\medskip\noindent{\bf Step~2.}
{\it The map $T\mapsto u_T$ satisfies~(i), (ii), and~(iii).}

\medskip\noindent
It follows directly from the construction and the homotopy invariance
of the integral of the pullback of $\om$ under maps with Lagrangian 
boundary conditions that
$$
\int_\Sigma u_T^*\om=\int_\Sigma u^*\om + \int_{\R\times[0,1]}v^*\om
$$
for every $T>T_0$. Hence 
$
\cA_H(u_T) = \cA_H(u) + \cA_H(v)
$
for every $T\ge T_0$ and this proves~(iii).
Consider the open subset
$$
\sM_{\eps_v}
:= \left\{w\in\sM^1(x^\alpha,z^\beta;j,H,J)\,\bigg|\,
\int_{s_0}^\infty\int_0^1
\Abs{w(\iota^{\beta,+}_{i_0}(s,t)}_t^2\,dtds>\eps_v
\right\}.
$$
Then, for each $w\in\sM_{\eps_v}$, there is a unique $T=T(w)>0$
such that 
$$
\int_{s_0+2T}^\infty\int_0^1
\Abs{w(\iota^{\beta,+}_{i_0}(s,t)}_t^2\,dtds=\eps_v
$$
By Lemma~\ref{le:TIME} the map $\sM_{\eps_v}\to(0,\infty):w\mapsto T(w)$
is smooth. Moreover, by construction we have $u_T\in\sM_{\eps_v}$
and $T(u_T)=T$ for every $T\ge T_0$. Hence the map $T\mapsto u_T$
is a diffeomorphism onto its image and this proves~(i).
The proof that the map $T\mapsto u_T$ satisfies~(ii)
is almost verbatim the same as the proof of Step~4 
in the proof of Theorem~\ref{thm:BO} and will be omitted.
Thus we have proved Step~2.

\medskip\noindent{\bf Step~3.}
{\it The map $T\mapsto u_T$ satisfies~(iv).}

\medskip\noindent
Assume, by contradiction, that~(iv) does not hold.
Then there are sequences $w_i\in\cM^1(x^\alpha,z^\beta;j,H,J)$ 
and $s_i\ge0$ such that
\begin{description}
\item[(a)]
$w_i\ne u_T$ for every $i$ and every $T>T_0$, 
\item[(b)]
$\cA_H(w_i)=\cA_H(u)+\cA_H(v)$ for every $i$, and
\item[(c)]
$\lim\limits_{i\to\infty}\sup\limits_{W_0}d(w_i,u) = 0$
and
$\lim\limits_{i\to\infty}\sup\limits_{0\le t\le1}
d(w_i(\iota^{\beta,+}_{i_0}(s_i,t),v(s_v,t)) = 0$.
\end{description}
By the standard elliptic bootstrapping, bubbling, and removal 
of singularities argument, we may assume, passing to a 
subsequence if necessary, that the sequence 
$w_i$ converges uniformly with all derivatives 
on every compact subset of the complement 
of a finite set in $\Sigma$ to a smooth finite 
energy solution $\tu:\Sigma\to M$ of~\eqref{eq:FLOERsigma},
and the sequence
$w_i(\iota^{\beta,+}_{i_0}(s_i+\cdot,\cdot))$ converges, 
uniformly with all derivatives on every compact subset 
of the complement of a finite set in $\R\times[0,1]$, 
to a smooth finite energy solution 
$
\tv:\R\times[0,1]\to M
$ 
of~\eqref{eq:FLOER} 
(See~\cite[Chapter~4]{MS}.)
By~(c) we have $\tu(z)=u(z)$ for every $z\in W_0$ and
$\tv(0,t)=v(s_v,t)$ for every $t\in[0,1]$.  
Hence it follows from unique continuation that 
$\tu=u$ and $\tv(s,t)=v(s_v+s,t)$ for all $s$ and $t$.  

\bigbreak

Next we claim that $s_i$ diverges to infinity.
Otherwise, by passing to a further subsequence,
it would follow that $s_i\to s^*$ and hence 
$$
u(\iota^{\beta,+}_{i_0}(s^*,t))
= \lim_{i\to\infty}w_i(\iota^{\beta,+}_{i_0}(s_i,t))
= v(s_v,t).
$$
By unique continuation it would follow
that $u(\iota^{\beta,+}_{i_0}(s^*+s,t))=v(s_v+s,t)$ 
for every $s\ge 0$ contradicting the fact 
that
$$
\lim_{s\to\infty}u(\iota^{\beta,+}_{i_0}(s^*+s,\cdot))
= y^\beta_{i_0}
\ne z^\beta_{i_0}
= \lim_{s\to\infty}v(s_v+s,\cdot).
$$
This shows that $s_i\to\infty$, as claimed.

Next we prove that 
\begin{equation}\label{eq:Ewi}
\lim_{i\to\infty}\int_{r_i}^\infty\int_0^1
\Abs{w_i(\iota^{\beta,+}_{i_0}(s,t))}_t^2\,dtds 
= \eps_v,\qquad
r_i := s_i-s_v+s_0.
\end{equation}
First, it follows from~(b) that there is no loss of energy.
Hence there is no bubbling and 
\begin{equation}\label{eq:uvwi1}
\begin{split}
u(z) = \lim_{i\to\infty}w_i(z),\qquad
v(s_0+s,t) =  \lim_{i\to\infty}w_i(\iota^{\beta,+}_{i_0}(r_i+s,t))
\end{split}
\end{equation}
The convergence is uniform with all derivatives on every
compact subset of $\Sigma$, respectively $\R\times[0,1]$. 
This implies that, for every $T>0$,
$$
\lim_{i\to\infty}\int_{r_i}^{r_i+T}\int_0^1
\Abs{w_i(\iota^{\beta,+}_{i_0}(s,t))}_t^2\,dtds 
= \int_0^T\int_0^1\Abs{v(s_0+s,t))}_t^2\,dtds.
$$
Taking the limit $T\to\infty$ we obtain
\begin{eqnarray*}
\eps_v
&=&
\lim_{T\to\infty}
\int_0^T\int_0^1\Abs{v(s_0+s,t))}_t^2\,dtds \\
&=&
\lim_{T\to\infty}
\lim_{i\to\infty}\int_{r_i}^{r_i+T}\int_0^1
\Abs{w_i(\iota^{\beta,+}_{i_0}(s,t))}_t^2\,dtds \\
&\le &
\lim_{i\to\infty}\int_{r_i}^\infty\int_0^1
\Abs{w_i(\iota^{\beta,+}_{i_0}(s,t))}_t^2\,dtds.
\end{eqnarray*}
If the inequality is strict it would follow
that $\cA_H(w_i)>\cA_H(u)+\cA_H(v)$ for large $i$, 
contradicting~(b). Thus we have proved~\eqref{eq:Ewi}. 
It follows from~\eqref{eq:uvwi1} that, 
for $i$ sufficiently large, we have
$$
z^\beta_{i_0}(t) = \lim_{s\to\infty}
w_i(\iota^{\beta,+}_{i_0}(s,t)).
$$
Moreover, passing to a further subsequence and
adding to $r_i$ a sequence converging to zero, 
if necessary, we may assume w.l.o.g.\ that
\begin{equation}\label{eq:sipm1}
\int_{r_i}^\infty\int_0^1
\Abs{w_i(\iota^{\beta,+}_{i_0}(s,t))}_t^2\,dtds 
= \eps_v
\end{equation}
for every $i$. 

Next we claim that, for large $i$,
\begin{equation}\label{eq:wiuTi1}
w_i = u_{T_i},\qquad
T_i:=\tfrac12\left(r_i-s_0\right)
\end{equation}
for $i$ sufficiently large, in contradiction to~(a).
To see this, note that by~\eqref{eq:sipm1},
\begin{equation*}
\begin{split}
u_i^- &:=w_i|_{\Sigma_0}\in\sM^-,\\
v_i^+ &:=w_i\circ\iota^{\beta,+}_{i_0}
(s_0+r_i+\cdot,\cdot)|_{[-s_0,\infty)\times[0,1]}\in\sM^+.
\end{split}
\end{equation*}
Then, by~\eqref{eq:uvwi1} and~\eqref{eq:sipm1}, the sequence 
$u_i^-$ converges to $u^-$ in $\sM^-$ and the sequence 
$v_i^+$ converges to $v^+$ in $\sM^+$.
Moreover, by~(b) and~\eqref{eq:sipm1}, we have
\begin{eqnarray*}
\int_{s_0}^{r_i}\int_0^1\Abs{\p_sw_i}_t^2 
&=& 
\cA_H(w_i)-\cA_H(w_i|_{\Sigma_0})-\int_{r_i}^\infty\int_0^1
\Abs{w_i(\iota^{\beta,+}_{i_0}(s,t))}_t^2\,dtds \\
&=& 
\cA_H(w_i)-\cA_H(w_i|_{\Sigma_0}) - \eps_v \\
&=& 
\cA_H(u)+\cA_H(v)-\cA_H(w_i|_{\Sigma_0}) - \eps_v \\
&\to&
\cA_H(u)-\cA_H(u|_{\Sigma_0}) +\cA_H(v)-\eps_v \\
&=&
\int_{s_0}^\infty\int_0^1
\Abs{u(\iota^{\beta,+}_{i_0}(s,t))}_t^2\,dtds
+ \delta_v  \\
&<& \hbar.
\end{eqnarray*}
Here the arrow denotes the limit $i\to\infty$
and the last inequality follows from~\eqref{eq:s0hbar}.
By~\eqref{eq:uvwi1}, we have
\begin{equation*}
\begin{split}
\lim_{i\to\infty}w_i(\iota^{\beta,+}_{i_0}(s_0,\cdot))
&=u(\iota^{\beta,+}_{i_0}(s_0,\cdot))\in\sU,\\
\lim_{i\to\infty}w_i(\iota^{\beta,+}_{i_0}(r_i,\cdot))
&=v(-s_0,\cdot)\in\sU.
\end{split}
\end{equation*}
Here the convergence is in the $\Cinf$ topology 
and hence also in the topology of $\sP=\sP^{3/2}$.
Hence it follows from the definition of $\hbar$ that
$$
(\iota^-(u_i^-),\iota^+(v_i^+))
= \left(w_i(\iota^{\beta,+}_{i_0}(s_0,\cdot)),
w_i(\iota^{\beta,+}_{i_0}(r_i,\cdot))\right)
\in\sW^{T_i}(y^\beta_{i_0},\sU)
$$
for $i$ sufficiently large.  Hence 
it follows from the definition of $u_T$ in~\eqref{eq:uT1} 
that $w_i=u_{T_i}$ for $i$ sufficiently large.
This proves~\eqref{eq:wiuTi1}
and Theorem~\ref{thm:CM}.
\end{proof}

\subsection*{Chain Homotopy Equivalence}

In this section we prove Theorem~\ref{thm:CHE}.

\begin{PARA}\label{para:che-proof}\rm
Let $(\Sigma,L)$ be a string cobordism in $\sL(M,\om)$
from $(L^\alpha_0,L^\alpha_1)$ to $(L^\beta_0,L^\beta_1)$,
$(j_0,H_0,J_0)$ and $(j_1,H_1,J_1)$ be two regular sets 
of Floer data on $(\Sigma,L)$ from 
$(H^\alpha,J^\alpha)$ to $(H^\beta,J^\beta)$, 
and let $\{j_\lambda,H_\lambda,J_\lambda\}_{0\le\lambda\le1}$
be a regular homotopy of Floer data from 
$(j_0,H_0,J_0)$ to $(j_1,H_1,J_1)$.  Let
$$
0<\lambda_\infty<1,\qquad
u\in\cM^{-1}(x^\alpha,y^\beta;
j_{\lambda_\infty},H_{\lambda_\infty},J_{\lambda_\infty})
$$
and
$$
v\in\cM^1(y^\beta,z^\beta;H^\beta,J^\beta)
$$
be as in the assumptions of Theorem~\ref{thm:CHE}.
Thus there exists an index $i_0\in I^\beta$ such that
$$
v_i(s,t)=y_i^\beta(t)=z_i^\beta(t)
$$ 
for $i\ne i_0$ and 
$$
\mu_H(v_{i_0})=1.
$$ 
Choose neighborhoods $U,V$ of $y^\beta_{i_0}(0)$, 
the constant $\hbar>0$, and the neighborhood 
$\sU\subset\sP=\sP^{3/2}$ as in~\ref{para:cm-proof2}.
\end{PARA}

\begin{proof}[Proof of Theorem~\ref{thm:CHE}]
The proof is almost verbatim the same as that of 
Theorem~\ref{thm:CM}.  In Step~1 we must construct 
the required map
$$
(T_0,\infty)\to \cM^0(x^\alpha,y^\beta;
\{j_\lambda,H_\lambda,J_\lambda\}_\lambda):
T\mapsto (\lambda_T,u_T).
$$
The only difference to the proof of Step~1 in Theorem~\ref{thm:CM}
is that $\sM^-$ is now a set of pairs $(\lambda,w^-)$,
where $\lambda\in[0,1]$ and $w^-\in W^{2,2}(\Sigma_0,M)$ 
satisfies the same conditions as before with $(j,H,J)$ 
replaced by $(j_\lambda,H_\lambda,J_\lambda)$. The map
$\iota^-:\sM^-\to\sP$ is then given by 
$$
\iota^-(\lambda,w^-) := w^-(\iota^\beta_{i_0}(s_0,\cdot)).
$$
The remainder of the proof of Step~1 is verbatim the same 
as in the proof of Theorem~\ref{thm:CM}.
The proof that the resulting map $T\mapsto(\lambda_T,u_T)$
satisfies the assertions~(i), (ii), (iii), and~(iv) 
of Theorem~\ref{thm:CHE} can also be carried over 
word by word from the proof of Theorem~\ref{thm:CM}.
The details can be safely left to the reader.
This proves Theorem~\ref{thm:CHE}.
\end{proof}

\subsection*{Catenation}

In this section we prove Theorem~\ref{thm:CAT}.

\begin{PARA}\label{para:cat-proof1}\rm
Let $(\Sigma^{\alpha\beta},L^{\alpha\beta})$,
$(\Sigma^{\beta\gamma},L^{\beta\gamma})$,
and $(j^{\alpha\beta},H^{\alpha\beta},J^{\alpha\beta})$,
$(j^{\beta\gamma},H^{\beta\gamma},J^{\beta\gamma})$
be as in~\ref{para:cat1}, and let
$$
u^{\alpha\beta}\in\cM^0(x^\alpha,x^\beta;
j^{\alpha\beta},H^{\alpha\beta},J^{\alpha\beta}),\qquad
u^{\beta\gamma}\in\cM^0(x^\beta,x^\gamma;
j^{\beta\gamma},H^{\beta\gamma},J^{\beta\gamma})
$$
and $W^{\alpha\beta}\subset\Sigma^{\alpha\beta}
\setminus\im\,\iota^{\beta,+}$,
$W^{\beta\gamma}\subset\Sigma^{\beta\gamma}
\setminus\im\,\iota^{\beta,-}$
be as in the assumptions of Theorem~\ref{thm:CAT}.
Let $\cH^{\alpha\beta}$ and $\cH^{\beta\gamma}$ 
be as in~\ref{para:cat2}.
\end{PARA}

\begin{PARA}\label{para:cat-proof2}\rm
For $i\in I^\beta$ let 
$$
\sP_i:=\sP^{3/2}(H^\beta_i,J^\beta_i)
$$ 
be the path space associated to the pair
$(H^\beta_i,J^\beta_i)$ as in ~\ref{para:hil_P3/2}. 
The proof of Theorem~\ref{thm:CAT} involves the product 
space
$$
\sP := \prod_{i\in I^\beta}\sP_i
$$
Thus the elements of $\sP$ are tuples 
$\gamma=\{\gamma_i\}_{i\in I^\beta}$ with $\gamma_i\in\sP_i$. 

For each $i\in I^\beta$ choose open neighborhoods 
$U_i,V_i\subset M$ of $x^\beta_i(0)$ 
as in ~\ref{para:conv_thm}, choose a constant $\hbar>0$ such that 
the assertion of Theorem~\ref{thm:MON} holds for each $i$
with $U=U_i$, $V=V_i$, $\Lambda=\{x^\beta_i(0)\}$,
and choose neighborhoods $\sU_i\subset\sP_i$ of $x^\beta_i$ 
and a constant $T_0>0$ such that the assertions of 
Theorem~\ref{thm:main_thm3}  are satisfied with $\sU=\sU_i$ and
$x=x^\beta_i$.  The Hilbert manifolds 
$\sM^\infty(x^\beta_i,\sU_i)$ and $\sM^T(x^\beta_i,\sU_i)$
are defined by~\eqref{eq:mod_emb}
and the embeddings 
$$
\iota_i^\infty:\sM^\infty(x^\beta_i,\sU_i)\to\sP_i\times\sP_i,\qquad
\iota_i^T:\sM^T(x^\beta_i,\sU_i)\to\sP_i\times\sP_i
$$ 
for $T\ge T_0$ and $i\in I^\beta$ are defined by~\eqref{eq:map_i}.
Denote their images by 
$$
\sW^\infty_i
:= \iota^\infty(\sM^\infty(x^\beta_i,\sU_i)),\qquad
\sW^T_i
:= \iota^T(\sM^T(x^\beta_i,\sU_i)).
$$
The products 
$$
\sW^\infty := \prod_{i\in I^\beta}\sW^\infty_i,\qquad
\sW^T := \prod_{i\in I^\beta}\sW^T_i
$$
are submanifolds of $\sP\times\sP$
(of infinite dimension and infinite codimension).
By Theorem~\ref{thm:main_thm3} the submanifolds $\sW^T$ 
converge to $\sW^\infty$ in the $C^1$ topology as
$T$ tends to infinity.
\end{PARA}

\begin{proof}[Proof of Theorem~\ref{thm:CAT}]
The proof has three steps.

\medskip\noindent{\bf Step~1.}
{\it Construction of the solutions 
\begin{equation*}
\begin{split}
&u^{\alpha\beta}_h\in\cM^0(x^\alpha,x^\beta;
j^{\alpha\beta},H^{\alpha\beta}+h^{\alpha\beta},J^{\alpha\beta}),\\
&u^{\beta\gamma}_h\in\cM^0(x^\beta,x^\gamma;
j^{\beta\gamma},H^{\beta\gamma}+h^{\beta\gamma},J^{\beta\gamma}),\\
&u_{h,T}\in\cM^0(x^\alpha,x^\gamma;
j^{\alpha\gamma}_T,(H+h)^{\alpha\gamma}_T,J^{\alpha\gamma}_T)
\end{split}
\end{equation*}
of the Floer equation for 
$h\in\cH_0\subset\cH^{\alpha\beta}\times\cH^{\beta\gamma}$ 
and $T\ge T_0$.}

\medskip\noindent
Choose $s_0>0$ so large that
\begin{equation}\label{eq:T0hbar}
\begin{split}
&\int_{s_0}^\infty\int_0^1
\Abs{\p_su^{\alpha\beta}(\iota^{\beta,+}_i(s,t))}^2_t\,dtds
< \hbar/2,\\
&\int_{-\infty}^{-s_0}\int_0^1
\Abs{\p_su^{\beta,\gamma}(\iota^{\beta,-}_i(s,t))}^2_t\,dtds
< \hbar/2,\\
&u^{\alpha\beta}(\iota^{\beta,+}_i(s_0,\cdot))\in\sU_i,\qquad
u^{\beta\gamma}_i((\iota^{\beta,-}_i(-s_0,\cdot))\in\sU_i,\qquad
i\in I^\beta.
\end{split}
\end{equation}
Then
$$
\bigl(u^{\alpha\beta}(\iota^{\beta,+}_i(s_0,\cdot)),
u^{\beta\gamma}_i((\iota^{\beta,-}_i(-s_0,\cdot))\bigr)
\in\sM^\infty(x^\beta_i,\sU_i),\qquad i\in I^\beta.
$$
Define
\begin{equation*}
\begin{split}
\Sigma^{\alpha\beta}_0
&:=\Sigma^{\alpha\beta}\setminus
\bigcup_{i\in I^\beta}\iota^{\beta,+}_i((s_0,\infty)\times[0,1]),\\
\Sigma^{\beta\gamma}_0
&:=\Sigma^{\beta\gamma}\setminus
\bigcup_{i\in I^\beta}\iota^{\beta,-}_i((-\infty,-s_0)\times[0,1]).
\end{split}
\end{equation*}
(See equation~\eqref{eq:SigmaT}.)
For $h=(h^{\alpha\beta},h^{\beta\gamma})
\in\cH^{\alpha\beta}\times\cH^{\beta\gamma}$
define
\begin{equation*}
\begin{split}
\sM^{\alpha\beta}_h
&:=\left\{w^{\alpha\beta}\in W^{2,2}_\loc
(\Sigma^{\alpha\beta}_0,M)
\,\Bigg|\,
\begin{array}{l}
w^{\alpha\beta}\mbox{ satisfies }\eqref{eq:FLOERsigma}
\mbox{ for }H^{\alpha\beta}+h^{\alpha\beta},\\
E_H(w^{\alpha\beta})<\infty,\\
\lim\limits_{s\to-\infty}w^{\alpha\beta}(\iota^{\alpha,-}_i(s,t))
=x^\alpha_i(t), i\in I^\alpha
\end{array}\right\},\\
\sM^{\beta\gamma}_h
&:=\left\{w^{\beta\gamma}\in W^{2,2}_\loc
(\Sigma^{\beta\gamma}_0,M)
\,\Bigg|\,
\begin{array}{l}
w^{\beta\gamma}\mbox{ satisfies }\eqref{eq:FLOERsigma}
\mbox{ for }H^{\beta\gamma}+h^{\beta\gamma},\\
E_H(w^{\beta\gamma})<\infty,\\
\lim\limits_{s\to\infty}w^{\beta\gamma}(\iota^{\gamma,+}_i(s,t))
=x^\gamma_i(t), i\in I^\gamma
\end{array}\right\}.
\end{split}
\end{equation*}
These spaces are are Hilbert manifolds.  
The restriction maps
$$
\iota^{\alpha\beta}_h:\sM^{\alpha\beta}_h\to\sP,\qquad
\iota^{\beta\gamma}_h:\sM^{\beta\gamma}_h\to\sP, 
$$
defined by 
$$
\iota^{\alpha\beta}_h(w^{\alpha\beta})
:=\left\{w^{\alpha\beta}(\iota^{\beta,+}_i(s_0,\cdot))
\right\}_{i\in I^\beta},\quad
\iota^{\beta\gamma}_h(w^{\beta\gamma})
:=\left\{w^{\beta\gamma}(\iota^{\beta,-}_i(-s_0,\cdot))
\right\}_{i\in I^\beta}
$$
are injective immersions (see  section \ref{sec:truc_cob}). 

\bigbreak

By our transversality hypothesis, the map
\begin{equation}\label{eq:albega0}
\iota^{\alpha\beta}_0\times\iota^{\beta\gamma}_0:
\sM^{\alpha\beta}_0\times\sM^{\beta\gamma}_0\to\sP\times\sP
\end{equation}
(associated to $h=0$) intersects the submanifold $\sW^\infty$
transversally in the point 
$$
\left\{(u^{\alpha\beta}(\iota^{\beta,+}_i(s_0,\cdot)),
u^{\beta\gamma}(\iota^{\beta,-}_i(-s_0,\cdot))\right\}_{i\in I^\beta}
= \left(\iota^{\alpha\beta}_0(u^{\alpha\beta}|_{\Sigma^{\alpha\beta}_0}),
\iota^{\beta\gamma}_0(u^{\beta\gamma}|_{\Sigma^{\beta\gamma}_0})\right).
$$
Moreover, the Fredholm index is zero, 
so the intersection point is isolated.
Hence, by Theorem~\ref{thm:main_thm3} and the implicit function
theorem, there exists a number $T_1>0$ and a
convex neighborhood $\cH_0\subset\cH^{\alpha\beta}\times\cH^{\beta\gamma}$
of zero (open in the $C^2$-topology) such that,
for every $T\ge T_1$ and every $h\in\cH_0$, the map
\begin{equation}\label{eq:albegah}
\iota^{\alpha\beta}_h\times\iota^{\beta\gamma}_h:
\sM^{\alpha\beta}_h\times\sM^{\beta\gamma}_h\to\sP\times\sP
\end{equation}
is transverse to $\sW^\infty$ and $\sW^T$ in a neighborhood of the pair
\begin{equation}\label{eq:ualbega0}
(u^{\alpha\beta}|_{\Sigma^{\alpha\beta}_0},
u^{\beta\gamma}|_{\Sigma^{\beta\gamma}_0})
\in \sM^{\alpha\beta}_0\times\sM^{\beta\gamma}_0
\end{equation}
and has a unique intersection point with each of these submanifolds 
in that neighborhood. In other words, the following holds.

\smallskip\noindent{\bf (I)}
For every $h\in\cH_0$ there exists a unique pair 
$(u_h^{\alpha\beta},u_h^{\beta\gamma})
\in\sM^{\alpha\beta}_h\times\sM^{\beta\gamma}_h$ 
near the pair~\eqref{eq:ualbega0}
such that
$$
\left(\iota^{\alpha\beta}_h(u^{\alpha\beta}_h),
\iota^{\beta\gamma}_h(u^{\beta\gamma}_h)\right)
\in\sW^\infty.
$$
Thus there is a unique tuple
$(w^+_{i,h},w^-_{i,h})\in\sM^\infty(x^\beta_i,\sU_i)$, $i\in I^\beta$, 
such that, for every $t\in[0,1]$ and every $i\in I^\beta$,
$$
w^+_{i,h}(0,t) = u_h^{\alpha\beta}(\iota^{\beta,+}_i(s_0,t)),\qquad
w^-_{i,h}(0,t) = u_h^{\beta\gamma}(\iota^{\beta,-}_i(-s_0,t)). 
$$

\smallskip\noindent{\bf (II)}
For every $h\in\cH_0$ and every $T\ge T_1$ 
there exists a unique pair 
$$
(u_{h,T}^{\alpha\beta},u_{h,T}^{\beta\gamma})
\in\sM^{\alpha\beta}_h\times\sM^{\beta\gamma}_h
$$ 
near the pair~\eqref{eq:ualbega0}
such that
$$
\left(\iota^{\alpha\beta}_h(u^{\alpha\beta}_{h,T}),
\iota^{\beta\gamma}_h(u^{\beta\gamma}_{h,T})\right)
\in\sW^T.
$$
Thus there is a unique tuple
$w_{i,h,T}\in\sM^T(x^\beta_i,\sU_i)$, $i\in I^\beta$, 
such that, for every $t\in[0,1]$ and every $i\in I^\beta$,
$$
w_{i,h,T}(-T,t) = u_{h,T}^{\alpha\beta}(\iota^{\beta,+}_i(s_0,t)),\qquad
w_{i,h,T}(T,t) = u_{h,T}^{\beta\gamma}(\iota^{\beta,-}_i(-s_0,t)).
$$

\bigbreak

\medskip
For $h\in\cH_0$ define the maps
$u^{\alpha\beta}_h:\Sigma^{\alpha\beta}\to M$ 
and $u^{\beta\gamma}_h:\Sigma^{\beta\gamma}\to M$ 
by
\begin{equation}\label{eq:ualbega}
\begin{split}
u^{\alpha\beta}_h(z) 
&:= 
\left\{\begin{array}{ll}
u^{\alpha\beta}_h(z),&\mbox{if } 
z\in\Sigma^{\alpha\beta}_0,\\
w_{i,h}^+(s,t),&\mbox{if }
z=\iota^{\beta,+}_i(s+s_0,t),\,i\in I^\beta,
\end{array}\right. \\
u^{\beta\gamma}_h(z) 
&:= 
\left\{\begin{array}{ll}
u^{\beta\gamma}_h(z),&\mbox{if } 
z\in\Sigma^{\beta\gamma}_0,\\
w_{i,h}^-(s,t),&\mbox{if }
z=\iota^{\beta,-}_i(s-s_0,t),\,i\in I^\beta.
\end{array}\right.
\end{split}
\end{equation}
For $h\in\cH_0$ and $T\ge T_0:=s_0+T_1$ define the map 
$u^{\alpha\gamma}_{h,T}:\Sigma^{\alpha\gamma}_T\to M$ by
\begin{equation}\label{eq:ualgaT}
u^{\alpha\gamma}_{h,T}(z) := 
\left\{\begin{array}{ll}
u^{\alpha\beta}_{h,T-s_0}(z),&\mbox{if } 
z\in\Sigma^{\alpha\beta}_0,\\
u^{\beta\gamma}_{h,T-s_0}(z),&\mbox{if } 
z\in\Sigma^{\beta\gamma}_0,\\
w_{i,h,T-s_0}(s,t),&\mbox{if }
z=\iota^{\beta,+}_i(s+T,t)
\cong\iota^{\beta,-}(s-T,t)\\
&\mbox{and }
\abs{s}\le T-s_0,\; i\in I^\beta.
\end{array}\right.
\end{equation}
Then 
\begin{equation*}
\begin{split}
&u^{\alpha\beta}_h\in\cM^0(x^\alpha,x^\beta;
j^{\alpha\beta},H^{\alpha\beta}+h^{\alpha\beta},J^{\alpha\beta}),\\
&u^{\beta\gamma}_h\in\cM^0(x^\beta,x^\gamma;
j^{\beta\gamma},H^{\beta\gamma}+h^{\beta\gamma},J^{\beta\gamma}),\\
&u_{h,T}\in\cM^0(x^\alpha,x^\gamma;
j^{\alpha\gamma}_T,(H+h)^{\alpha\gamma}_T,J^{\alpha\gamma}_T)
\end{split}
\end{equation*}
for every $h\in\cH_0$ and every $T\ge T_0$.  
This proves Step~1.

\medskip\noindent{\bf Step~2.}
{\it The maps $u_{h,T}$ constructed in Step~1 satisfy
conditions~(i), (ii), and (iii) in Theorem~\ref{thm:CAT}.}

\medskip\noindent
That the solutions constructed in Step~1 are regular 
(i.e.\ the linearized operators are bijective) follows 
directly from transversality in the path space $\cP\times\cP$. 
This shows that the functions $u^{\alpha\gamma}_h$ and
$u^{\alpha\gamma}_{h,T}$ satisfy condition~(i) 
in Theorem~\ref{thm:CAT}. The convergence statement 
in~(ii) follows immediately from the construction. 
The action identity in~(iii) follows from the fact 
that the catenation of $u^{\alpha\beta}$ and $u^{\beta\gamma}$ 
is homotopic to $u^{\alpha\gamma}_{h,T}$ for all $h$ and $T$. 
This proves Step~2.

\medskip\noindent{\bf Step~3.}
{\it The maps $u_{h,T}$ constructed in Step~1 satisfy
condition~(iv) in Theorem~\ref{thm:CAT}, after shrinking 
$\cH_0$ and increasing $T_0$, if necessary, 
and for $\delta_0>0$ sufficiently small.}

\medskip\noindent
Suppose, by contradiction, that the 
solutions $u^{\alpha\gamma}_{h,T}$ of the Floer equation,
constructed in Step~1, do not satisfy assertion~(iv)
in Theorem~\ref{thm:CAT} for any triple $(\delta_0,\cH_0,T_0)$. 
Then there exists a sequence of perturbations 
$$
h_i=(h_i^{\alpha\beta},h_i^{\beta\gamma})
\in\cH^{\alpha\beta}\times\cH^{\beta\gamma}
$$
converging to zero in the $C^2$ topology,
a sequence $T_i\to\infty$, and a sequence 
$$
u_i\in\cM^0(x^\alpha,x^\gamma;
j^{\alpha\gamma}_{T_i},(H+h_i)^{\alpha\gamma}_{T_i},
J^{\alpha\gamma}_{T_i})
$$
such that, for every $i\in\N$, the following holds.

\smallskip\noindent{\bf (a)}
$u_i \ne u^{\alpha\gamma}_{h_i,T_i}$.

\smallskip\noindent{\bf (b)}
$\cA_H(u_i) = \cA_H(u^{\alpha\beta}) + \cA_H(u^{\beta\gamma})$.

\smallskip\noindent{\bf (c)}
$\lim_{i\to\infty}\sup_{W^{\alpha\beta}}d(u_i,u^{\alpha\beta}_h) = 0$
and $\lim_{i\to\infty}\sup_{W^{\beta\gamma}}d(u_i,u^{\beta\gamma}_h) = 0$.

\medskip\noindent
By the standard elliptic bootstrapping, 
bubbling, and removal of sin\-gu\-la\-rities argument,
we may assume as before, after passing to a subsequence
if necessary, that the restriction of $u_i$ to
$\Sigma^{\alpha\beta}$, respectively $\Sigma^{\beta\gamma}$, 
converges in $W^{2,2}_\loc$ to a smooth finite energy
solution $v^{\alpha\beta}$, respectively $v^{\beta\gamma}$, 
of the Floer equation for 
$(j^{\alpha\beta},H^{\alpha\beta},J^{\alpha\beta})$,
respectively
$(j^{\beta\gamma},H^{\beta\gamma},J^{\beta\gamma})$.
By~(c) and unique continuation, we must have $v^{\alpha\beta}=u^{\alpha\beta}$
and $v^{\beta\gamma}=u^{\beta\gamma}$.  By~(b) there is no loss 
of energy in the limit and hence no bubbling or splitting off 
of Floer trajectories can occur.  Hence the sequence 
$(h_i,u_i|_{\Sigma^{\alpha\beta}_0})$ converges in the topology 
of the fiber bundle
$$
\cM^{\alpha\beta} := \bigcup_{h\in\cH_0}\{h\}\times\cM^{\alpha\beta}_h\to\cH_0
$$
to the pair $(0,u^{\alpha\beta}|_{\Sigma^{\alpha\beta}_0})$.
Likewise, the sequence $(h_i,u_i|_{\Sigma^{\beta\gamma}_0})$
converges in the topology of 
$\cM^{\beta\gamma} := \bigcup_{h\in\cH_0}\{h\}\times\cM^{\beta\gamma}_h$
to the pair $(0,u^{\beta\gamma}|_{\Sigma^{\beta\gamma}_0})$.
Hence it follows from the construction in Step~1 
and the implicit function theorem in the path space $\sP\times\sP$
that $u_i=u^{\alpha\gamma}_{h_i,T_i}$ for $i$ sufficiently large,
in contradiction to~(a). This contradiction proves Step~3 
and Theorem~\ref{thm:CAT}.
\end{proof}




\bibliographystyle{plain}



\end{document}